\RequirePackage{fix-cm}
\documentclass[twocolumn]{svjour3}          
\smartqed  
\usepackage{graphicx}
\usepackage{epstopdf}
\usepackage{amssymb}
\usepackage{amsmath}
\usepackage{enumerate}
\usepackage[numbers]{natbib}
\usepackage{gensymb}
\usepackage{undertilde}
\usepackage[none]{hyphenat}
 \usepackage{mathptmx}      
\spnewtheorem{algorithm}{Algorithm}{\bf}{\rm}
\begin{document}

\title{Implementation techniques for the SCFO experimental optimization framework
}


\author{Gene A. Bunin, Gr\'egory Fran\c cois, Dominique Bonvin 
}


\institute{Laboratoire d'Automatique, Ecole Polytechnique F\' ed\' erale de Lausanne, Lausanne, CH-1015 \\
              \email{gene.a.bunin@ccapprox.info}           
}

\date{Submitted: \today}

\maketitle

\begin{abstract}
The material presented in this document is intended as a comprehensive, implementation-oriented supplement to the experimental optimization framework presented in \cite{Bunin2013SIAM}. The issues of physical degradation, unknown Lipschitz constants, measurement/estimation noise, gradient estimation, sufficient excitation, and the handling of soft constraints and/or a numerical cost function are all addressed, and a robust, implementable version of the sufficient conditions for feasible-side global convergence is proposed.   
\end{abstract}

\section{Overview}
\label{sec:overview}

Following the conventions that were already established in \cite{Bunin2013SIAM}, let us define an \emph{experimental optimization} problem as

\vspace{-2mm}
\begin{equation}\label{eq:mainprob}
\begin{array}{rll}
\mathop {{\rm{minimize}}}\limits_{\bf{u}} & \phi_p ({\bf{u}}) & \\
{\rm{subject}}\hspace{1mm}{\rm{to}} & g_{p,j}({\bf{u}}) \leq 0, & j = 1,...,n_{g_p} \\
 & g_{j}({\bf{u}}) \leq 0, & j = 1,...,n_{g} \\
& {\bf{u}}^L \preceq {\bf u} \preceq {\bf u}^U, &
\end{array}
\end{equation}

\noindent with the subscript $p$ used to denote \emph{experimental functions} that can only be evaluated for a given set of decision variables, ${\bf u} \in \mathbb{R}^{n_u}$, by carrying out a (presumably expensive) experiment. Apart from the standard task of finding a set ${\bf u}^*$ that is a local minimum of (\ref{eq:mainprob}), experimental optimization problems come with the additional requirement that the constraints be satisfied \emph{for all} experimental iterates ${\bf u}_0,{\bf u}_1, {\bf u}_2, ...$ that are tested in the optimization process.

In the companion work of \cite{Bunin2013SIAM}, a theoretical framework -- referred to hereafter as the SCFO (``sufficient conditions for feasibility and optimality'') framework -- was proposed, and it was proven that enforcing the SCFO

\vspace{-2mm}
\begin{equation}\label{eq:SCFO1}
g_{p,j}({\bf u}_k) + \displaystyle \mathop {\sum} \limits_{i=1}^{n_u} \kappa_{p,ji} | u_{k+1,i} - u_{k,i} | \leq 0, \;\; \forall j = 1,...,n_{g_p},
\end{equation}

\vspace{-2mm}
\begin{equation}\label{eq:SCFO2}
g_{j}({\bf u}_{k+1}) \leq 0, \;\; \forall j = 1,...,n_g,
\end{equation}

\vspace{-2mm}
\begin{equation}\label{eq:SCFO3}
{\bf u}^L \preceq {\bf u}_{k+1} \preceq {\bf u}^U,
\end{equation}

\vspace{-2mm}
\begin{equation}\label{eq:SCFO4}
\nabla g_{p,j} ({\bf u}_k)^T ({\bf u}_{k+1} - {\bf u}_k) \leq -\delta_{g_p,j}, \;\; \forall j: g_{p,j} ({\bf u}_k) \geq -\epsilon_{p,j},
\end{equation}

\vspace{-2mm}
\begin{equation}\label{eq:SCFO5}
\nabla g_{j} ({\bf u}_k)^T ({\bf u}_{k+1} - {\bf u}_k) \leq -\delta_{g,j}, \;\; \forall j: g_{j} ({\bf u}_k) \geq -\epsilon_{j},
\end{equation}

\vspace{-2mm}
\begin{equation}\label{eq:SCFO6}
\nabla \phi_{p} ({\bf u}_k)^T ({\bf u}_{k+1} - {\bf u}_k) \leq -\delta_\phi,
\end{equation}

\vspace{-2mm}
\begin{equation}\label{eq:SCFO7}
\begin{array}{l}
\nabla \phi_p ({\bf u}_k)^T ({\bf u}_{k+1} - {\bf u}_k) + \\
\displaystyle \frac{1}{2} \mathop {\sum} \limits_{i_1=1}^{n_u} \mathop {\sum} \limits_{i_2=1}^{n_u} M_{\phi,i_1 i_2} | (u_{k+1,i_1} - u_{k,i_1})(u_{k+1,i_2} - u_{k,i_2}) | \leq 0
\end{array}
\end{equation}

\noindent for the future experimental iterate ${\bf u}_{k+1}$ at every experimental iteration $k$ would generate a chain of iterates that were all feasible and converged to a point ${\bf u}_\infty$ in a finite number of experiments, with ${\bf u}_\infty \rightarrow {\bf u}^*$ in the Fritz-John-error sense as the positive projection parameters ${\boldsymbol \epsilon}_p, {\boldsymbol \epsilon}, {\boldsymbol \delta}_{g_p}, {\boldsymbol \delta}_g, \delta_\phi \downarrow {\bf 0}$. Here, the constants $\kappa_{p,ji}$ and $M_{\phi,i_1 i_2}$ are used to denote Lipschitz constants for the experimental constraints and the Lipschitz constants for the cost function derivatives, respectively, with

\vspace{-4mm}
\begin{equation}\label{eq:lipcon}
- \kappa_{p,ji} < \frac{\partial g_{p,j}}{\partial u_i} \Big |_{\bf u} < \kappa_{p,ji}, \;\; \forall {\bf u} \in \mathcal{I},
\end{equation}

\vspace{-4mm}
\begin{equation}\label{eq:lipcon2}
-M_{\phi,i_1 i_2} < \frac{\partial^2 \phi_p}{\partial u_{i_2} \partial u_{i_1} } \Big |_{\bf u} < M_{\phi,i_1 i_2}, \;\; \forall {\bf u} \in \mathcal{I},
\end{equation}

\noindent and with the experimental space $\mathcal{I}$ defined as

\vspace{-4mm}
\begin{equation}\label{eq:Idef}
\mathcal{I} = \{ {\bf u} : {\bf u}^L \preceq {\bf u} \preceq {\bf u}^U \}.
\end{equation} 

\noindent The two assumptions necessary to prove the guarantee of feasible-side global convergence are that:

\begin{enumerate}[]
\item {\bf{A1}}: The functions $\phi_p$, $g_p$, and $g$ are twice continuously differentiable ($C^2$) on an open set containing $\mathcal{I}$.
\item {\bf{A2}}: The initial experimental iterate, ${\bf{u}}_0$, is strictly feasible with respect to the experimental constraints ($g_{p,j} ({\bf{u}}_0) < 0, \; \forall j = 1,...,n_{g_p}$), feasible with respect to the numerical constraints ($g_{j} ({\bf{u}}_0) \leq 0, \; \forall j = 1,...,n_{g}$), and lies in the experimental space (${\bf u}_0 \in \mathcal{I}$).
\end{enumerate}

The actual scheme proposed to apply the SCFO (\ref{eq:SCFO1})-(\ref{eq:SCFO7}) is a project-and-filter approach, where one first projects some suggested solution ${\bf u}_{k+1}^*$, obtained by whatever means, onto the local descent space of the cost and $\epsilon$-active constraints:

\vspace{-3mm}
\begin{equation}\label{eq:proj}
\begin{array}{rl}
\bar {\bf u}_{k+1}^* := {\rm arg} \mathop {\rm minimize}\limits_{{\bf u}} & \| {\bf u} - {\bf u}_{k+1}^* \|_2^2 \vspace{1mm}  \\
{\rm{subject}}\;{\rm{to}} & \nabla g_{p,j} ({\bf u}_k)^T ({\bf u} - {\bf u}_k) \leq -\delta_{g_p,j} \vspace{1mm} \\
& \forall j : g_{p,j}({\bf u}_k) \geq -\epsilon_{p,j} \vspace{1mm} \\
 & \nabla g_{j} ({\bf u}_k)^T ({\bf u} - {\bf u}_k) \leq -\delta_{g,j} \vspace{1mm} \\
& \forall j : g_{j}({\bf u}_k) \geq -\epsilon_{j} \vspace{1mm} \\
 & \nabla \phi_{p} ({\bf u}_k)^T ({\bf u} - {\bf u}_k) \leq -\delta_{\phi} \vspace{1mm} \\
 & {\bf u}^L \preceq {\bf u} \preceq {\bf u}^U,
\end{array}
\end{equation}

\noindent and then filters this solution by applying a numerical line search to determine the filter gain value $K_k \in [0,1]$, with the next iterate defined as

\vspace{-3mm}
\begin{equation}\label{eq:inputfilter}
{\bf u}_{k+1} := {\bf u}_k + K_k \left( \bar {\bf u}_{k+1}^* - {\bf u}_{k} \right),
\end{equation}

\noindent where the goal is to maximize $K_k$ while satisfying the conditions

\vspace{-3mm}
\begin{equation}\label{eq:SCFO1i}
g_{p,j}({\bf u}_k) + K_k \displaystyle \mathop {\sum} \limits_{i=1}^{n_u} \kappa_{p,ji}|\bar u_{k+1,i}^* - u_{k,i}| \leq 0, \;\; \forall j = 1,...,n_{g_p},
\end{equation}

\vspace{-3mm}
\begin{equation}\label{eq:SCFO2i}
g_{j}({\bf u}_k + K_k ( \bar {\bf u}_{k+1}^* - {\bf u}_{k} )) \leq 0, \;\; \forall j = 1,...,n_g,
\end{equation}

\vspace{-2mm}
\begin{equation}\label{eq:SCFO7i}
\begin{array}{l}
\nabla \phi_p ({\bf u}_k)^T (\bar {\bf u}_{k+1}^* - {\bf u}_k) + \vspace{1mm} \\
\hspace{5mm} \displaystyle \frac{K_k}{2} \mathop {\sum} \limits_{i_1=1}^{n_u} \mathop {\sum} \limits_{i_2=1}^{n_u} M_{\phi,i_1 i_2} \Bigg | \begin{array}{l} (\bar u_{k+1,i_1}^* - u_{k,i_1}) \\ 
\hspace{5mm} (\bar u_{k+1,i_2}^* - u_{k,i_2}) \end{array} \Bigg | \leq 0.
\end{array}
\end{equation}

\noindent In numerical implementation, one may start with an initial choice of projection parameters that correspond to the expected sizes of the different function ranges and then proceed to decrease these values once the projection (\ref{eq:proj}) becomes infeasible. If this infeasibility persists for very small values of the projection parameters, then the user has a theoretical guarantee that the Fritz John error at the current experimental iterate is very small and that the iterate is therefore close to ${\bf u}^*$ in the Fritz-John-error sense, believed to be sufficient for most practical applications.

While the theoretical SCFO framework is very powerful, it is, as presented in \cite{Bunin2013SIAM}, only so \emph{in theory} and as such must be modified for real applications where different implementation issues may arise and make the nominal formulation of the SCFO (\ref{eq:SCFO1})-(\ref{eq:SCFO7}) either difficult or impossible to satisfy. Here, we list the key issues that we have encountered or expect to encounter:

\begin{itemize}
\item \emph{Physical degradation}: When the experimental optimization problem involves a physical system, there is a high likelihood that the system at hand will not function identically on the day that it is first used and on some day a year or two later. This essentially means that the underlying experimental functions $\phi_p$ and $g_{p,j}$ will evolve and change over time, thereby changing the optimization problem.
\item \emph{Unknown Lipschitz constants}: The Lipschitz constants $\kappa_{p,ji}$ and $M_{\phi,i_1 i_2}$ will rarely be known in practice. It follows that methods that either bypass using these constants or can reliably estimate their values are required.
\item \emph{Noise}: The function values $\phi_p ({\bf u}_k)$ and $g_{p,j} ({\bf u}_k)$ will often be corrupted by noise or estimation error in many experimental settings, thereby corrupting Condition (\ref{eq:SCFO1}), which relies on the exact value of $g_{p,j} ({\bf u}_k)$, and making the verification of cost decrease, $\phi_p ({\bf u}_{k+1}) - \phi_p ({\bf u}_{k}) < 0$, difficult.
\item \emph{Gradient estimation}: Conditions (\ref{eq:SCFO4}), (\ref{eq:SCFO6}), and (\ref{eq:SCFO7}) all rely on using gradients of experimental functions, which will be unavailable. In practice, one must work with estimates instead.
\item \emph{Sufficient excitation}: As a key requirement for estimating the gradient of an experimental function at ${\bf u}_k$ is being able to perturb the decision variables ${\bf u}$ locally around ${\bf u}_k$, there needs to be some guarantee of this being possible without violating the safety constraints.
\item \emph{Conservatism and slow convergence}: The SCFO are sufficient for feasible-side global convergence and, not surprisingly, conservative by necessity so as to enforce the guarantees that they bring. In some applications, it may be desired to trade these guarantees for faster convergence, which may involve relaxing certain Lipschitz constants or allowing certain experimental constraints to be violated temporarily.
\end{itemize}

The goal of this document is to address all of these issues and to show how the SCFO may be modified to accomodate each of these challenges. The remainder of this work will treat these issues one-by-one as follows:

\begin{itemize}
\item Section \ref{sec:degrade}: Physical Degradation,
\item Section \ref{sec:lip}: Estimation of Lipschitz Constants and \newline Potential Relaxations in the Lipschitz Bounds,
\item Section \ref{sec:bounds}: Measurement/Estimation Noise in the \newline Experimental Function Values,
\item Section \ref{sec:gradest}: Working with Bounded Gradient Estimates,
\item Section \ref{sec:suffexc}: Preserving Feasibility While Allowing for \newline Sufficient Excitation,
\item Section \ref{sec:softcon}: Incorporating Soft Constraints,
\item Section \ref{sec:numcost}: Incorporating a Numerical Cost Function.
\end{itemize}

\noindent Because all of these issues are expected to arise concurrently in real application, we have decided that the presentation not be modular -- i.e., that each section should build on the results of its predecessor. While the individual steps to modify the SCFO are, for the most part, simple and intuitive, their final accumulation yields some rather convoluted formulations, which may appear incomprehensible unless the sections are read and understood in the order that they are presented. 

The final two sections then seek to consolidate the preceding derivations by

\begin{itemize}
\item providing the full implementable SCFO procedure in Section \ref{sec:complete}, and
\item proving the theoretical safety properties of the procedure in Section \ref{sec:gcproof}.
\end{itemize}

\noindent We confess that we have chosen, given the complexity of the full modified SCFO, not to derive a modified proof of feasible-side global convergence in the current version of the document and have limited the theoretical results to feasible-side iterates only. A future version will seek to prove global convergence as well -- for the time being, we have been unable to find a satisfactory set of assumptions that allow us to prove this result.

\section{Physical Degradation}
\label{sec:degrade}

It is almost inevitable that any experimental optimization problem involving some sort of physical system will be subject to degradation effects. Here, we will use the term ``degradation'' -- also known in both the literature and industry as ``process drift'' \cite{Divoky1995,Wold2008}  -- in a fairly general sense to denote changes that occur in the system over time. While such changes can often be unwanted, such as damage due to repeated loading in a structure \cite{Lee1998} or the growth of oxide scales in a solid oxide fuel cell \cite{Brylewski2001}, they need not be. For example, in steady-state optimization problems where there are both fast and slow time-scale effects but where the influence of the fast-scale effects is greater, one may successfully optimize at a higher frequency while treating the minor effects that appear gradually on the slower time scale as a sort of ``degradation'' \cite{Bunin2012d}.

To express this phenomenon in mathematical terms, we may consider the experimental functions\footnote{While one could generalize and let the numerical constraints $g_j$ be subject to degradation as well, in our experience these are usually simple mathematical relationships that remain the same all throughout operation.} as being functions not only of the decision variables ${\bf u}$ but also of the time. Denoting this time by $\tau$, we restate the experimental optimization problem (\ref{eq:mainprob}) as

\vspace{-2mm}
\begin{equation}\label{eq:mainprobdeg}
\begin{array}{rll}
\mathop {{\rm{minimize}}}\limits_{\bf{u}} & \phi_p ({\bf{u}},\tau) & \\
{\rm{subject}}\hspace{1mm}{\rm{to}} & g_{p,j}({\bf{u}},\tau) \leq 0, & j = 1,...,n_{g_p} \\
 & g_{j}({\bf{u}}) \leq 0, & j = 1,...,n_{g} \\
& {\bf{u}}^L \preceq {\bf u} \preceq {\bf u}^U. &
\end{array}
\end{equation}

\noindent This is basically tantamount to assuming that there remains a deterministic functional relationship between the decision variables and the cost and constraints for the ``degraded'' problem at a given instant in time $\tau$, which seems reasonable as degradation is generally a deterministic process.

Note that (\ref{eq:mainprobdeg}) is equivalent in form to the nominal version (\ref{eq:mainprob}), save that the variable $\tau$ is not a decision variable but one that grows freely of any influence. We may, however, take advantage of this similarity to rederive the SCFO for (\ref{eq:mainprobdeg}) in a manner identical to how they were derived for (\ref{eq:mainprob}).

\subsection{Feasibility in the Presence of Degradation}

Denote by $\tau_k$ the time of the experiment at ${\bf u}_k$. To guarantee feasibility at the experimental iteration $k+1$ for a given experimental constraint $g_{p,j}$, it is necessary to guarantee

\vspace{-2mm}
\begin{equation}\label{eq:feastriv}
g_{p,j} ({\bf u}_{k+1},\tau_{k+1}) \leq 0.
\end{equation}

\noindent As in the nominal case, this condition is intractable since the function $g_{p,j}$ is unknown. In the same manner as before \cite{Bunin2013SIAM}, one may avoid this issue by exploiting the Lipschitz bound\footnote{Throughout the document, we will call upon Lipschitz bounds several times without giving their derivations. The interested reader is referred to \cite{Bunin:Lip} to see how they are derived.}

\vspace{-2mm}
\begin{equation}\label{eq:feas2}
\begin{array}{l}
g_{p,j} ({\bf u}_{k+1},\tau_{k+1}) \leq g_{p,j} ({\bf u}_{k},\tau_{k}) \\
\hspace{10mm}\displaystyle  + \kappa_{p,j\tau} \left( \tau_{k+1} - \tau_k \right) + \sum_{i=1}^{n_u} \kappa_{p,ji} | u_{k+1,i} - u_{k,i} |,
\end{array}
\end{equation}

\noindent where $\kappa_{p,j\tau}$ may be seen as the extension of the Lipschitz constant to the degradation effects over time, and is implicitly defined as

\vspace{-2mm}
\begin{equation}\label{eq:lipdeg}
- \kappa_{p,j\tau} \leq \frac{\partial g_{p,j}}{\partial \tau} \Big |_{({\bf u},\tau)} \leq \kappa_{p,j\tau}, \;\; \forall ({\bf u},\tau) \in \mathcal{I}_\tau,
\end{equation}

\noindent with $\mathcal{I}_\tau$ the \emph{temporal experimental space} defined as 

\vspace{-2mm}
\begin{equation}\label{eq:tempexpspace}
\mathcal{I}_\tau = \mathcal{I} \times \{ \tau : \tau_0 \leq \tau \leq \overline \tau \},
\end{equation}

\noindent i.e., as the Cartesian product of the standard experimental space with time. To maintain the boundedness of $\mathcal{I}_\tau$, we will suppose an upper bound on the time considered, i.e., that $\overline \tau < \infty$. Naturally, the definitions of the other Lipschitz constants also need to be modified for (\ref{eq:feas2}) to be valid:

\vspace{-2mm}
\begin{equation}\label{eq:lipcondeg}
- \kappa_{p,ji} < \frac{\partial g_{p,j}}{\partial u_i} \Big |_{({\bf u},\tau)} < \kappa_{p,ji}, \;\; \forall ({\bf u},\tau) \in \mathcal{I}_\tau.
\end{equation}

The Lipschitz constants for the cost function derivatives will also need to be modified accordingly:

\vspace{-2mm}
\begin{equation}\label{eq:lipcondeg2}
-M_{\phi,i_1 i_2} < \frac{\partial^2 \phi_p}{\partial u_{i_2} \partial u_{i_1} } \Big |_{({\bf u},\tau)} < M_{\phi,i_1 i_2}, \;\; \forall ({\bf u},\tau) \in \mathcal{I}_\tau.
\end{equation}

\noindent Note that we will not require the Lipschitz constants for the degradation in (\ref{eq:lipdeg}) to be strict like the others, as this will allow us to return to the no-degradation case easily by setting $\kappa_{p,j\tau} := 0$.

Prior to advancing the analysis further, let us reflect on the meaning of the Lipschitz constant defined in (\ref{eq:lipdeg}). Basically, this is a bound on the ``slope'' of the drift or degradation that can occur in the experimental function values at any time and for any choice of decision variables. This value will, of course, never be known exactly, but reasonable guesses should be possible in many of the engineering contexts where degradation would occur, especially for those applications where some sort of effort has already been put into modeling the nature of the degradation. As degradation effects also tend to be relatively slow \cite{Divoky1995}, it is expected that $\kappa_{p,j\tau}$ be relatively small for most applications.

Clearly, the version of (\ref{eq:SCFO1}) that holds in the degradation case may be derived by simply forcing the upper bound of (\ref{eq:feas2}) to be nonpositive:

\vspace{-2mm}
\begin{equation}\label{eq:SCFO1deg}
\begin{array}{l}
g_{p,j} ({\bf u}_{k},\tau_{k}) + \kappa_{p,j\tau} \left( \tau_{k+1} - \tau_k \right) \\
\hspace{30mm}\displaystyle   + \sum_{i=1}^{n_u} \kappa_{p,ji} | u_{k+1,i} - u_{k,i} | \leq 0,
\end{array}
\end{equation}

\noindent as this implies (\ref{eq:feastriv}), with strict inequality holding whenever $g_{p,j} ({\bf u}_{k},\tau_{k}) < 0$.

Note, however, that while one could always preserve feasibility for the nominal case by choosing ${\bf u}_{k+1}$ to be sufficiently close to ${\bf u}_k$, this is no longer the case here -- (\ref{eq:SCFO1deg}) is clearly impossible to satisfy for any choice of ${\bf u}_{k+1}$ if $-g_{p,j} ({\bf u}_{k},\tau_{k}) < \kappa_{p,j\tau} \left( \tau_{k+1} - \tau_k \right)$. This motivates the need for a sort of backtracking, or choosing a reference point other than ${\bf u}_k$ to use in the SCFO.

\subsection{Choice of Reference Point}

Consider modifying (\ref{eq:SCFO1deg}) by using the reference iteration ${k^*}$ instead of $k$:

\vspace{-2mm}
\begin{equation}\label{eq:SCFO1degref}
\begin{array}{l}
g_{p,j} ({\bf u}_{k^*},\tau_{k^*}) + \kappa_{p,j\tau} \left( \tau_{k+1} - \tau_{k^*} \right) \\
\hspace{30mm}\displaystyle   + \sum_{i=1}^{n_u} \kappa_{p,ji} | u_{k+1,i} - u_{k^*,i} | \leq 0,
\end{array}
\end{equation}

\noindent with the experimental iteration $k^* \in [0,k]$ chosen as

\vspace{-2mm}
\begin{equation}\label{eq:kstar}
\begin{array}{rl}
k^* := {\rm arg} \mathop {\rm maximize}\limits_{\bar k \in [0,k]} & \bar k  \\
{\rm{subject}}\;{\rm{to}} & g_{p,j} ({\bf u}_{\bar k},\tau_{\bar k}) \\
& \hspace{5mm}+ \kappa_{p,j\tau} \left( \tau_{k+1} - \tau_{\bar k} \right) \leq 0 \\
& \forall j = 1,...,n_{g_p},
\end{array}
\end{equation}

\noindent i.e., as the most recent experimental iteration for which feasibility for the experiment at $k+1$ may still be guaranteed.

The motivation for choosing the most recent iteration comes from the innate properties of the SCFO -- since the cost is lowered monotonically from one experiment to the next, the most recent experiment should, in the nominal case, be the most optimal one and thus the best to use as a reference for further improvement. The motivation for wanting to satisfy the constraint $g_{p,j} ({\bf u}_{\bar k},\tau_{\bar k}) + \kappa_{p,j\tau} \left( \tau_{k+1} - \tau_{\bar k} \right) \leq 0$ should be clear -- it is impossible to guarantee that $g_{p,j}({\bf u}_{k+1},\tau_{k+1}) \leq 0$ if this is not the case.

However, there is still no guarantee that this constraint will be satisfied for any choice of $k^*$, and for problems where degradation is particularly fast it may occur that one has no choice but to go ahead with the next experiment without a guarantee that the experimental iterate will be feasible. If this is preferable to shutting down the physical system out of safety concerns, then one may attempt to choose $k^*$ as

\vspace{-2mm}
\begin{equation}\label{eq:kstar2}
k^* := {\rm arg} \mathop {\rm minimize}\limits_{\bar k \in [0,k]} \;\; \mathop {\max} \limits_{j = 1,...,n_{g_p}} \left[ \begin{array}{l} g_{p,j} ({\bf u}_{\bar k},\tau_{\bar k}) \\ + \kappa_{p,j\tau} \left( \tau_{k+1} - \tau_{\bar k} \right) \end{array} \right] ,
\end{equation}

\noindent by essentially choosing the point that has the lowest worst-case constraint value. Because (\ref{eq:kstar2}) is sensitive to scaling, it is recommended that the different constraint functions be scaled beforehand. Following this step, the most sensible action, from the point of view of safety, would be to set ${\bf u}_{k+1} := {\bf u}_{k^*}$, as one cannot guarantee feasibility at $({\bf u}_{k+1}, \tau_{k+1})$ and changing the decision variables could potentially make things even worse. If, then, the observed value $-g_{p,j} ({\bf u}_{k^*},\tau_{k+1})$ is greater than $\kappa_{p,j\tau} \left( \tau_{k+2} - \tau_{k+1} \right)$, it will be possible for the SCFO to once more guarantee feasibility at $k+2$.

Yet another alternative for choosing $k^*$ is to define, based on \emph{a priori} knowledge, some ``safe point'' and use this point as ${\bf u}_k^*$ when (\ref{eq:kstar}) fails and when (\ref{eq:kstar2}) is deemed too risky.

Not surprisingly, some additional assumptions on the degradation, or on the existence of a safe point, will be necessary to make any rigorous theoretical claims regarding feasibility. These will be addressed in Section \ref{sec:gcproof}. As a final remark regarding the use of the reference ${\bf u}_{k^*}$, note that Conditions (\ref{eq:SCFO4})-(\ref{eq:SCFO7}) must be modified accordingly:

\vspace{-2mm}
\begin{equation}\label{eq:SCFO4ref}
\begin{array}{l}
\nabla g_{p,j} ({\bf u}_{k^*})^T ({\bf u}_{k+1} - {\bf u}_{k^*}) \leq -\delta_{g_p,j} \vspace{1mm} \\
\hspace{30mm} \forall j: g_{p,j} ({\bf u}_{k^*}) \geq -\epsilon_{p,j},
\end{array}
\end{equation}

\vspace{-2mm}
\begin{equation}\label{eq:SCFO5ref}
\nabla g_{j} ({\bf u}_{k^*})^T ({\bf u}_{k+1} - {\bf u}_{k^*}) \leq -\delta_{g,j}, \;\; \forall j: g_{j} ({\bf u}_{k^*}) \geq -\epsilon_{j},
\end{equation}

\vspace{-2mm}
\begin{equation}\label{eq:SCFO6ref}
\nabla \phi_{p} ({\bf u}_{k^*})^T ({\bf u}_{k+1} - {\bf u}_{k^*}) \leq -\delta_\phi,
\end{equation}

\vspace{-2mm}
\begin{equation}\label{eq:SCFO7ref}
\begin{array}{l}
\nabla \phi_p ({\bf u}_{k^*})^T ({\bf u}_{k+1} - {\bf u}_{k^*}) + \vspace{1mm} \\
\hspace{5mm} \displaystyle \frac{1}{2} \mathop {\sum} \limits_{i_1=1}^{n_u} \mathop {\sum} \limits_{i_2=1}^{n_u} M_{\phi,i_1 i_2} \Bigg | \begin{array}{l} (u_{k+1,i_1} - u_{{k^*},i_1}) \\ \hspace{5mm} (u_{k+1,i_2} - u_{{k^*},i_2}) \end{array} \Bigg |  \leq 0.
\end{array}
\end{equation}

\subsection{Projecting in the Presence of Degradation}

We now turn to modifying (\ref{eq:SCFO4ref}) and (\ref{eq:SCFO6ref}) to account for degradation. Following the same philosophy as before -- i.e., treating $\tau$ as an additional, albeit not controlled, variable -- we propose:

\vspace{-2mm}
\begin{equation}\label{eq:SCFO4deg}
\begin{array}{l}
\nabla g_{p,j} ({\bf u}_{k^*},\tau_{k+1})^T  \left[ \begin{array}{c} {\bf u}_{k+1} - {\bf u}_{k^*} \\ 0 \end{array} \right] \leq -\delta_{g_p,j}, \vspace{1mm} \\
\hspace{5mm} \forall j: g_{p,j} ({\bf u}_{k^*},\tau_{k^*}) + \kappa_{p,j \tau} (\tau_{k+1} - \tau_{k^*}) \geq -\epsilon_{p,j},
\end{array}
\end{equation}

\vspace{-2mm}
\begin{equation}\label{eq:SCFO6deg}
\nabla \phi_{p} ({\bf u}_{k^*},\tau_{k+1})^T  \left[ \begin{array}{c} {\bf u}_{k+1} - {\bf u}_{k^*} \\ 0 \end{array} \right] \leq -\delta_{\phi},
\end{equation}

\noindent where the degrees of freedom, ${\bf u}_{k+1}$, are used to both decrease the cost and stay away from the $\epsilon$-active constraints locally for the functions in their one-step ahead degraded state.

A couple of remarks are necessary:

\begin{itemize}
\item The gradients $\nabla g_{p,j} ({\bf u}_{k^*},\tau_{k+1})$ and $\nabla \phi_{p} ({\bf u}_{k^*},\tau_{k+1})$ are unknown and must be estimated based on the available data. Our current algorithm of choice for carrying out this task is that of \cite{Bunin2013a}. The incorporation of gradient estimates into the SCFO is discussed fully in Section \ref{sec:gradest}.
\item We choose to use $\tau_{k+1}$ instead of $\tau_{k^*}$ in the projection so as to enforce local descent for the functions as they are expected to be at the future iteration $k+1$. For the same reasons, we would prefer to use the most up-to-date $\epsilon$-activity condition $g_{p,j} ({\bf u}_{k^*},\tau_{k+1}) \geq -\epsilon_{p,j}$. However, since $g_{p,j} ({\bf u}_{k^*},\tau_{k+1})$ cannot be measured, we work instead with its upper Lipschitz bound $g_{p,j} ({\bf u}_{k^*},\tau_{k^*}) + \kappa_{p,j \tau} (\tau_{k+1} - \tau_{k^*})$.
\end{itemize}

\subsection{Monotonic Cost Improvement with Degradation}

Let us modify Condition (\ref{eq:SCFO7ref}), recalling first the Lipschitz bound

\vspace{-2mm}
\begin{equation}\label{eq:bound2Udeg}
\begin{array}{l}
\phi_p({\bf u}_{k+1},\tau_{k+1}) - \phi_p({\bf u}_{k^*},\tau_{k+1}) \leq \vspace{1mm} \\
\hspace{5mm} \nabla \phi_p({\bf u}_{k^*},\tau_{k+1})^T \left[ \begin{array}{c} {\bf u}_{k+1} - {\bf u}_{k^*} \\ 0 \end{array} \right]  \vspace{1mm} \\
\hspace{5mm} +\displaystyle \frac{1}{2} \sum_{i_1=1}^{n_u} \sum_{i_2=1}^{n_u} M_{\phi,i_1 i_2} \Bigg | \begin{array}{l} (u_{k+1,i_1} - u_{{k^*},i_1}) \\ \hspace{5mm} (u_{k+1,i_2} - u_{{k^*},i_2}) \end{array} \Bigg |,
\end{array}
\end{equation}

\noindent where we ultimately aim to achieve a decrease between $\phi_p({\bf u}_{k^*},\tau_{k+1})$ and $\phi_p({\bf u}_{k+1},\tau_{k+1})$, i.e., to change ${\bf u}$ in a manner that leads to an improvement in the cost value at the future instance $\tau_{k+1}$.

To force this decrease, it is then sufficient to enforce

\vspace{-2mm}
\begin{equation}\label{eq:costdecdeg}
\begin{array}{l}
 \nabla \phi_p({\bf u}_{k^*},\tau_{k+1})^T \left[ \begin{array}{c} {\bf u}_{k+1} - {\bf u}_{k^*} \\ 0 \end{array} \right]  \vspace{1mm} \\
\hspace{5mm} +\displaystyle \frac{1}{2} \sum_{i_1=1}^{n_u} \sum_{i_2=1}^{n_u} M_{\phi,i_1 i_2} \Bigg | \begin{array}{l} (u_{k+1,i_1} - u_{{k^*},i_1}) \\ \hspace{5mm} (u_{k+1,i_2} - u_{{k^*},i_2}) \end{array} \Bigg | \leq 0,
\end{array}
\end{equation}
 
\noindent as this would imply $\phi_p({\bf u}_{k+1},\tau_{k+1}) - \phi_p({\bf u}_{k^*},\tau_{k+1}) < 0$ when ${\bf u}_{k+1} \neq {\bf u}_{k^*}$. Note that this is nearly identical to the degradation-free case, with the only difference being in the gradient.

\subsection{The Project-and-Filter Approach in the Presence of Degradation}

The project-and-filter approach may now be modified accordingly. First, the conditions (\ref{eq:SCFO4deg}) and (\ref{eq:SCFO6deg}) are used in place of (\ref{eq:SCFO4}) and (\ref{eq:SCFO6}) in the projection, with (\ref{eq:SCFO5ref}) also replacing (\ref{eq:SCFO5}) to account for the choice of reference:

\vspace{-2mm}
\begin{equation}\label{eq:projdeg}
\begin{array}{rl}
\bar {\bf u}_{k+1}^* := \;\;\;\;\;\;\;\;\; & \vspace{1mm}\\
 {\rm arg} \mathop {\rm minimize}\limits_{{\bf u}} & \| {\bf u} - {\bf u}_{k+1}^* \|_2^2  \vspace{1mm} \\
{\rm{subject}}\;{\rm{to}} & \nabla g_{p,j} ({\bf u}_{k^*},\tau_{k+1})^T \left[ \hspace{-1mm} \begin{array}{c} {\bf u} - {\bf u}_{k^*} \\ 0 \end{array} \hspace{-1mm} \right] \leq -\delta_{g_p,j} \vspace{1mm} \\
& \forall j: \begin{array}{l} g_{p,j} ({\bf u}_{k^*},\tau_{k^*}) \\ + \kappa_{p,j\tau} (\tau_{k+1} - \tau_{k^*} ) \end{array} \geq -\epsilon_{p,j} \vspace{1mm} \\
 & \nabla g_{j} ({\bf u}_{k^*})^T ({\bf u} - {\bf u}_{k^*}) \leq -\delta_{g,j} \vspace{1mm} \\
& \forall j : g_{j}({\bf u}_{k^*}) \geq -\epsilon_{j} \vspace{1mm} \\
 & \nabla \phi_{p} ({\bf u}_{k^*},\tau_{k+1})^T  \left[ \begin{array}{c} {\bf u} - {\bf u}_{k^*} \\ 0 \end{array} \right] \leq -\delta_{\phi} \vspace{1mm} \\
 & {\bf u}^L \preceq {\bf u} \preceq {\bf u}^U.
\end{array}
\end{equation}

\noindent Noting that the future iterate is now defined as

\vspace{-2mm}
\begin{equation}\label{eq:inputfilterref}
{\bf u}_{k+1} := {\bf u}_{k^*} + K_k \left( \bar {\bf u}_{k+1}^* - {\bf u}_{k^*} \right),
\end{equation}

\noindent substitution of this law into (\ref{eq:SCFO1degref}), (\ref{eq:SCFO2}), and (\ref{eq:costdecdeg}) then yields the following constraints on $K_k \in [0,1]$:

\vspace{-2mm}
\begin{equation}\label{eq:SCFO1ideg}
\begin{array}{l}
g_{p,j} ({\bf u}_{k^*},\tau_{k^*}) + \kappa_{p,j\tau} \left( \tau_{k+1} - \tau_{k^*} \right) \\
\displaystyle   + K_k \sum_{i=1}^{n_u} \kappa_{p,ji} | \bar u_{k+1,i}^* - u_{k^*,i} | \leq 0, \;\; \forall j = 1,...,n_{g_p} \;,
\end{array}
\end{equation}

\vspace{-2mm}
\begin{equation}\label{eq:SCFO2ideg}
g_{j}({\bf u}_{k^*} + K_k ( \bar {\bf u}_{k+1}^* - {\bf u}_{k^*} )) \leq 0, \;\; \forall j = 1,...,n_g,
\end{equation}

\vspace{-2mm}
\begin{equation}\label{eq:SCFO7ideg}
\begin{array}{l}
 \nabla \phi_p({\bf u}_{k^*},\tau_{k+1})^T  \left[ \begin{array}{c} \bar {\bf u}_{k+1}^* - {\bf u}_{k^*} \\ 0 \end{array} \right]   \vspace{1mm} \\
\hspace{5mm} +\displaystyle \frac{K_k}{2} \sum_{i_1=1}^{n_u} \sum_{i_2=1}^{n_u} M_{\phi,i_1 i_2} \Bigg | \begin{array}{l} (\bar u_{k+1,i_1}^* - u_{{k^*},i_1}) \\ \hspace{5mm} ( \bar u_{k+1,i_2}^* - u_{{k^*},i_2}) \end{array} \Bigg |  \leq 0.
\end{array}
\end{equation}

\subsection{Example}

Consider the constructed example used in \cite{Bunin2013SIAM}:

\vspace{-2mm}
\begin{equation}\label{eq:exprob}
\begin{array}{rl}
\mathop {{\rm{minimize}}}\limits_{u_1,u_2} & \phi_{p}({\bf{u}}) := (u_1-0.5)^2 + (u_2-0.4)^2 \\
{\rm{subject}}\hspace{1mm}{\rm{to}} & g_{p,1}({\bf{u}}) := -6u^2_1 - 3.5u_1 + u_2 -0.6 \le 0 \vspace{1mm} \\
 & g_{p,2}({\bf{u}}) := 2u^2_1 + 0.5u_1 + u_2 -0.75 \le 0 \vspace{1mm} \\
 & g_{1}({\bf{u}}) := -u^2_1 - (u_2-0.15)^2 +0.01 \le 0 \vspace{1mm} \\
& -0.5 \leq u_1 \leq 0.5 \\
& 0 \leq u_2 \leq 0.8,
\end{array}
\end{equation}

\noindent for which we reproduce the solution obtained by the nominal SCFO methodology in Fig. \ref{fig:ex0}, referring the reader to \cite{Bunin2013SIAM} for all details regarding choice of initial experiment, Lipschitz constants, etc.

\begin{figure*}
\begin{center}
\includegraphics[width=16cm]{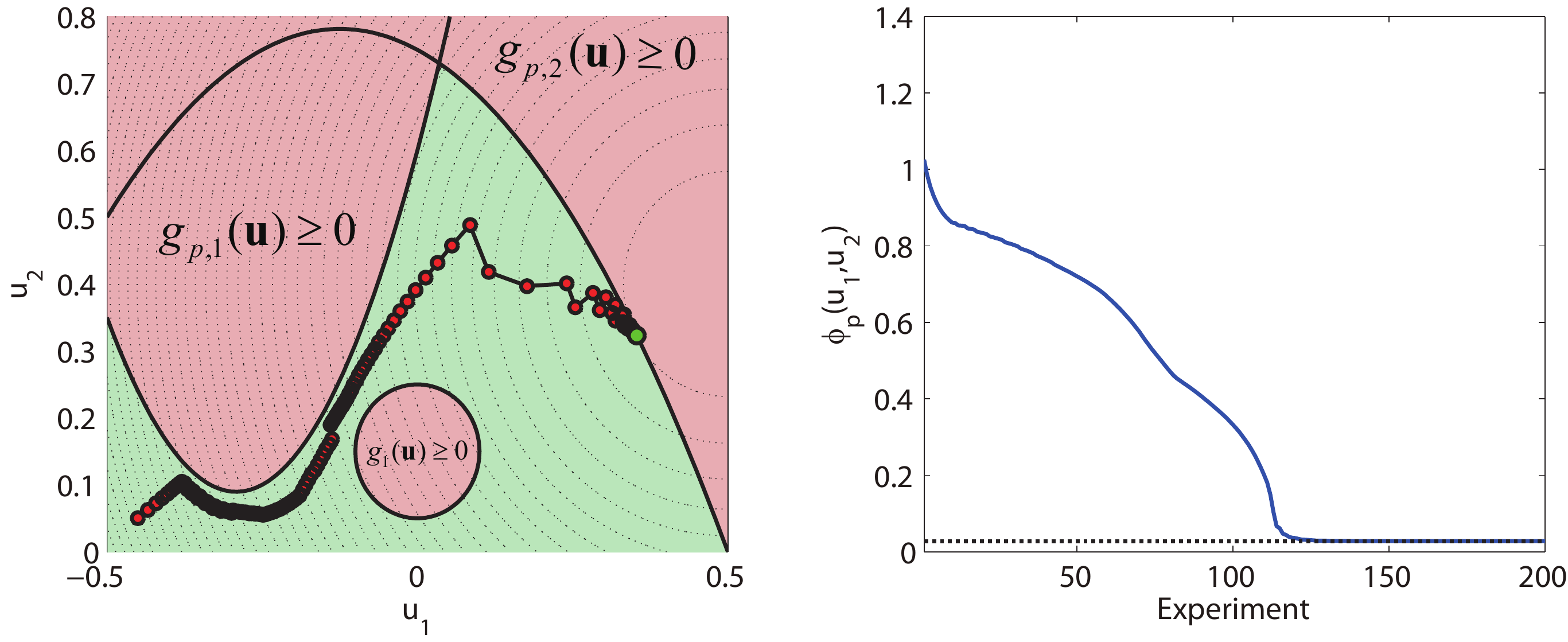}
\caption{Chain of experiments (red points) generated by applying the nominal SCFO methodology to Problem (\ref{eq:exprob}). The green point denotes the only local minimum. The dotted lines on the left plot denote the contours of the cost function, while the constant dotted line on the right denotes the cost value at the minimum.}
\label{fig:ex0}
\end{center}
\end{figure*}

To study the effects of degradation, let us suppose, for simplicity, that $\tau = k$ and that the experiments are carried out at a constant frequency. A version of (\ref{eq:exprob}) that degrades with time is constructed:

\vspace{-2mm}
\begin{equation}\label{eq:exdeg}
\begin{array}{rl}
\mathop {\rm minimize}\limits_{u_1,u_2} & (u_1 - 0.5)^2 + (u_2 - 0.4 - \frac{\tau}{500})^2 \vspace{1mm} \\
{\rm{subject}}\;{\rm{to}} & -6u_1^2 - (3.5+\frac{\tau}{500})u_1 + u_2 -0.6 \leq 0 \vspace{1mm} \\
&  2u_1^2 + 0.5u_1 + u_2 - 0.75 \pm \frac{\tau}{500} \leq 0 \vspace{1mm} \\
&  -u_1^2 - (u_2-0.15)^2 + 0.01 \leq 0 \vspace{1mm} \\
& -0.5 \leq u_1 \leq 0.5 \vspace{1mm} \\
& 0 \leq u_2 \leq 0.8,
\end{array}
\end{equation}

\noindent with $\pm$ denoting two different scenarios -- one where the feasible space with respect to $g_{p,2}$ shrinks due to degradation ($+$), and one where it expands ($-$). With the time considered limited to $\overline \tau = 200$ (i.e., to 200 experimental iterations), the Lipschitz constants over the resulting $\mathcal{I}_{\tau}$ are appropriately chosen as

\vspace{-2mm}
\begin{equation}\label{eq:exlipdeg}
\begin{array}{lll}
\kappa_{p,11} = 10, & \kappa_{p,12} = 2, & \kappa_{p,1\tau} = \frac{1}{1000}, \vspace{1mm} \\
\kappa_{p,21} = 3, & \kappa_{p,22} = 2, & \kappa_{p,2\tau} = \frac{1}{500}, \vspace{1mm} \\
M_{\phi,11} = 3, & M_{\phi,12} = 1, &  \vspace{1mm} \\
M_{\phi,21} = 1, & M_{\phi,22} = 3. &  
\end{array}
\end{equation}

Using these constants and the modifications described in (\ref{eq:projdeg})-(\ref{eq:SCFO7ideg}), with $k^*$ chosen via (\ref{eq:kstar}) and, if this is not possible, via (\ref{eq:kstar2}), the results obtained for the ($+$) and ($-$) scenarios are illustrated geometrically in Figs. \ref{fig:degA} and \ref{fig:degB}, with the corresponding cost function values given in Fig. \ref{fig:degcost}. Though it may not be obvious from the figures, we note that the scheme is able to satisfy the constraints at all iterations for this problem, sometimes doing so at the cost of optimality. In the ($+$) case, one sees that the scheme often does not have enough time to track the changing optimum due to the continuing need to return to the shrinking feasible region -- this becomes particularly pertinent after Experiment 150.

While this approximate satisfaction of the SCFO is largely satisfactory -- certainly for the ($-$) case -- it should again be emphasized that \emph{no guarantees are possible in the general degradation case without additional assumptions on the nature of the degradation}. That we cannot track the changing optimum perfectly should be evident, as being able to do so would be tantamount to solving Problem (\ref{eq:exdeg}), which changes at every experiment, in only a single experimental iteration. That we cannot guarantee feasibility throughout should be evident as well -- if we ran the optimization for longer (to, say, $k = 1000$), then the feasible space of the problem would simply shrink and vanish.

\begin{figure*}
\begin{center}
\includegraphics[width=16cm]{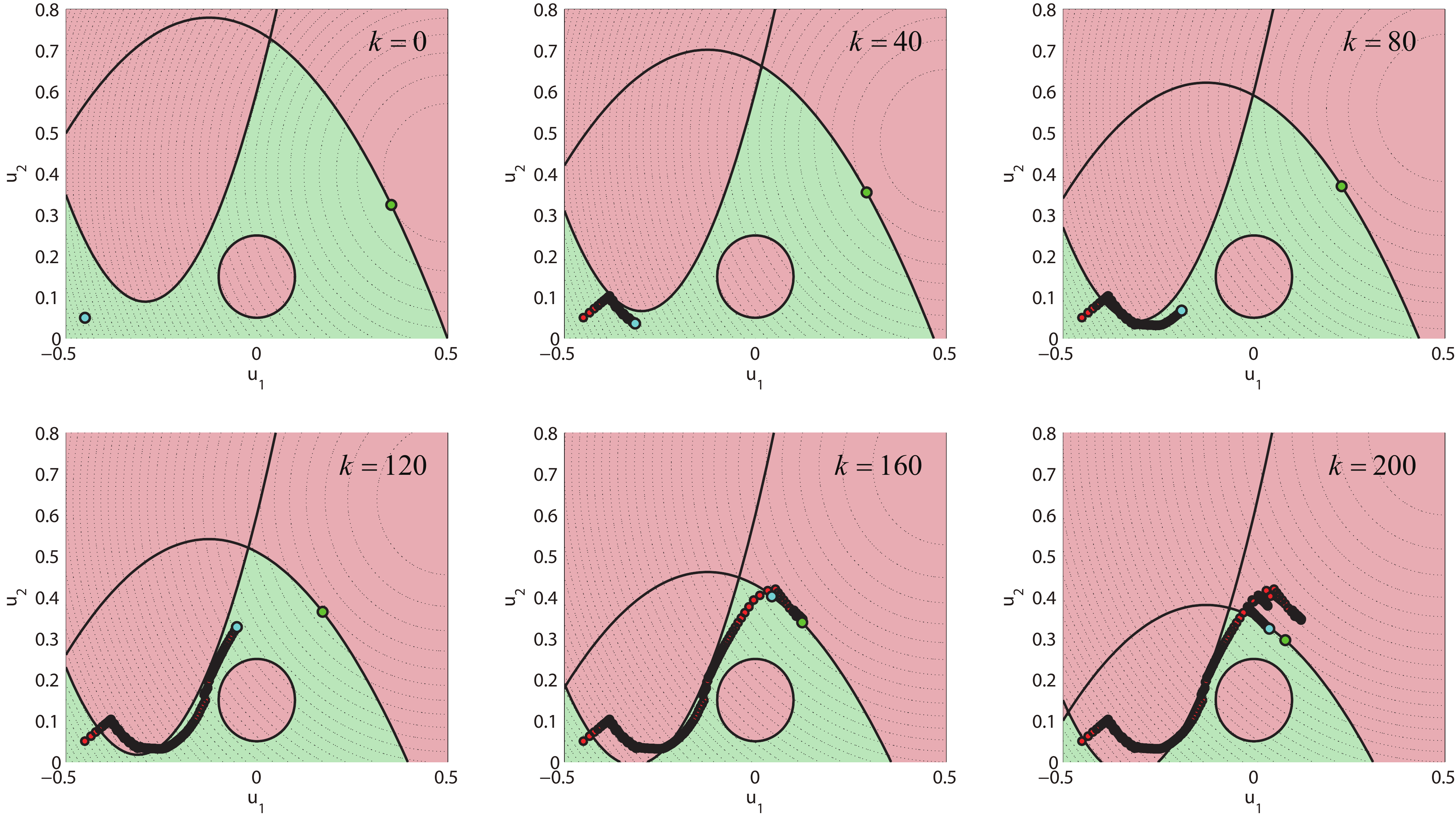}
\caption{Chain of experiments generated by applying the modified SCFO methodology to Problem (\ref{eq:exdeg}) for the ($+$) scenario. The blue point is used to highlight the current experimental iterate.}
\label{fig:degA}
\end{center}
\end{figure*}

\begin{figure*}
\begin{center}
\includegraphics[width=16cm]{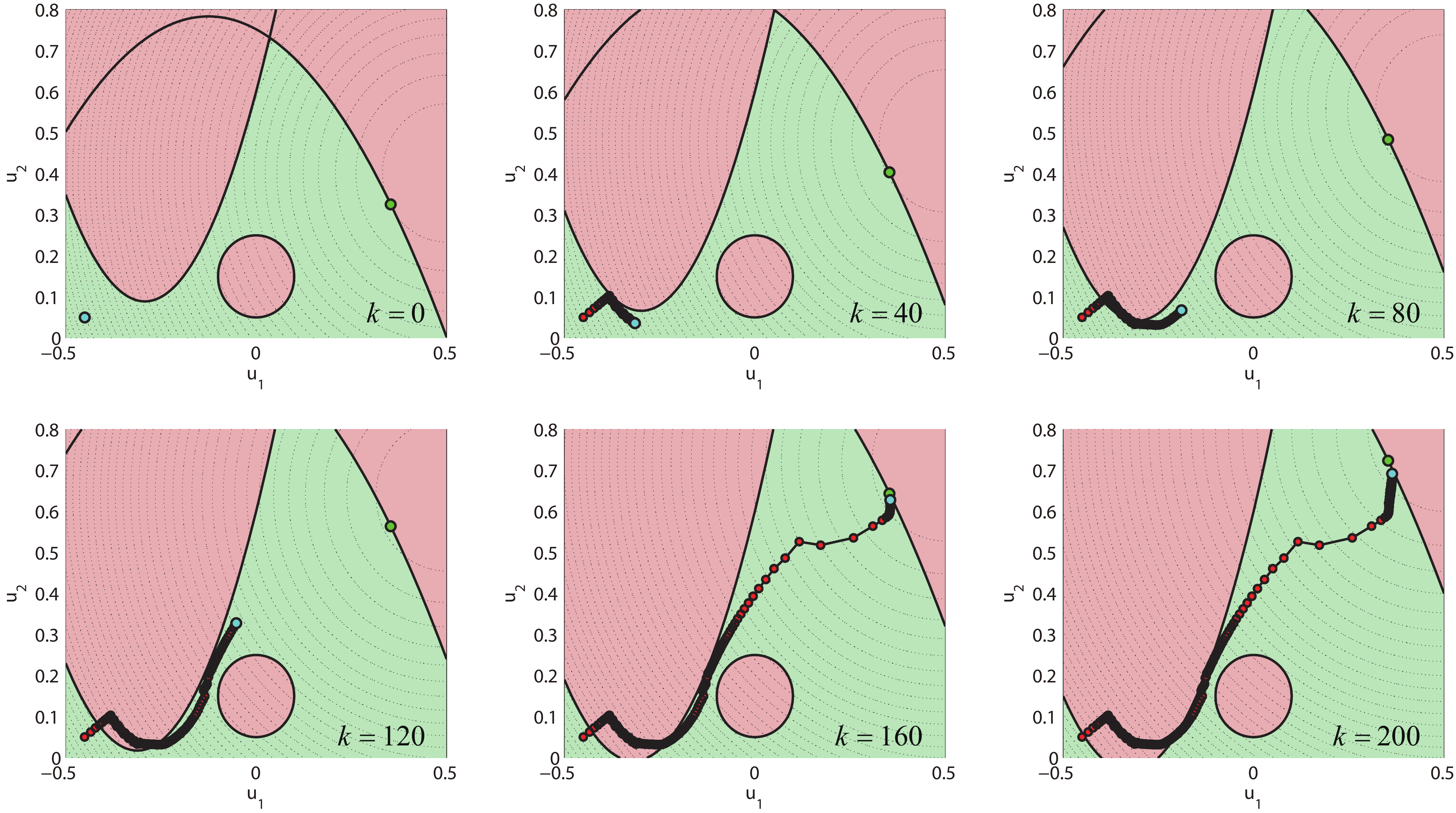}
\caption{Chain of experiments generated by applying the modified SCFO methodology to Problem (\ref{eq:exdeg}) for the ($-$) scenario.}
\label{fig:degB}
\end{center}
\end{figure*} 

\begin{figure}
\begin{center}
\includegraphics[width=8cm]{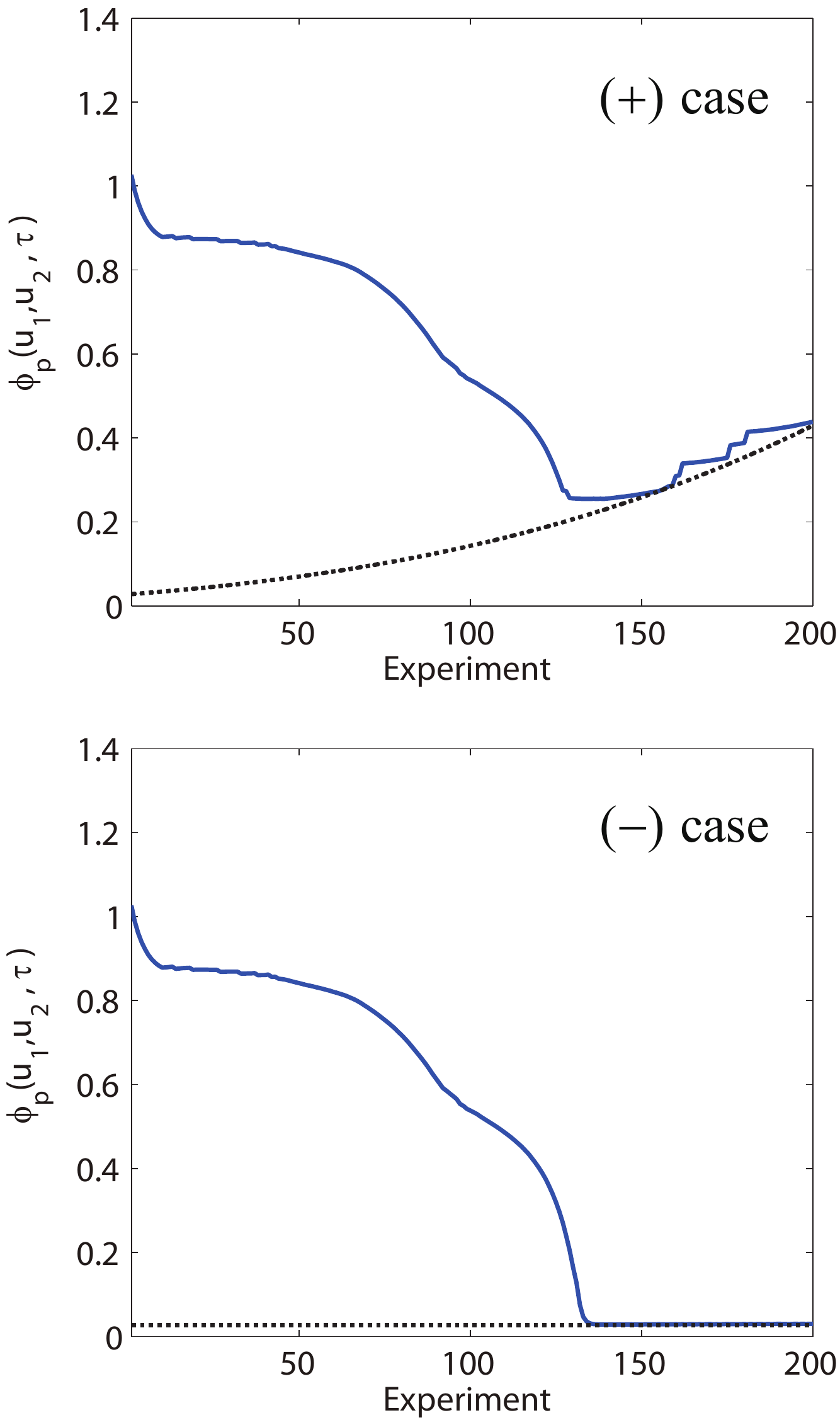}
\caption{Cost function values obtained by the modified SCFO methodology for Problem (\ref{eq:exdeg}).}
\label{fig:degcost}
\end{center}
\end{figure} 

\section{Estimating and Relaxing the Lipschitz Bounds}
\label{sec:lip}

The Lipschitz bounds and the constants that define them play a fundamental role in the SCFO, and as such are a double-edged sword. While it is precisely because of these bounds that it is possible for the SCFO to ensure feasible-side convergence and safety throughout the experimental optimization, one must nevertheless do some work to obtain proper estimates of the Lipschitz constants, which cannot be known with certainty in many practical applications. Traditionally, the problem of Lipschitz constant estimation has been largely restricted to the domain of global optimization \cite{Strongin1973,Hansen1992,Wood1996}, where these constants are typically used to build lower bounds in branch-and-bound algorithms \cite{Horst1995}. Some of the philosophy in these methods will be employed here -- notably via the Lipschitz consistency check of Section \ref{sec:consist}. However, we will largely tackle the estimation problem from scratch in a more engineering-based, holistic manner -- combining the physical meaning of the Lipschitz constant that is largely nonexistent in the mathematical context with additional mathematical ``tricks'' wherever possible so as to shed some light on the possible ways to provide the SCFO the estimates it needs. In addition to looking for ways to provide valid estimates, much of our discussion will focus on how to make these valid estimates as non-conservative as possible -- clearly, one could almost always guarantee validity by setting the constants to some enormous number, but this would lead to very slow convergence.

\subsection{Lower and Upper Lipschitz Constants}
\label{sec:twolip}

Let us start by seemingly doubling the complexity of the estimation task and distinguishing between lower and upper Lipschitz constants as follows:

\vspace{-4mm}
\begin{equation}\label{eq:lipcondegLU}
\underline \kappa_{p,ji} < \frac{\partial g_{p,j}}{\partial u_i} \Big |_{({\bf u},\tau)} < \overline \kappa_{p,ji}, \;\; \forall ({\bf u},\tau) \in \mathcal{I}_\tau,
\end{equation}

\vspace{-4mm}
\begin{equation}\label{eq:lipcondeg2LU}
\underline M_{\phi,i_1 i_2} < \frac{\partial^2 \phi_p}{\partial u_{i_2} \partial u_{i_1} } \Big |_{({\bf u},\tau)} < \overline M_{\phi,i_1 i_2}, \;\; \forall ({\bf u},\tau) \in \mathcal{I}_\tau,
\end{equation}

\vspace{-4mm}
\begin{equation}\label{eq:lipdegLU}
\underline \kappa_{p,j\tau} \leq \frac{\partial g_{p,j}}{\partial \tau} \Big |_{({\bf u},\tau)} \leq \overline \kappa_{p,j\tau}, \;\; \forall ({\bf u},\tau) \in \mathcal{I}_\tau,
\end{equation}

\noindent which, despite doubling the number of constants that require estimates, nevertheless both allows more flexibility and relaxes the original bounds -- note that one could always remove this flexibility by setting, e.g., $-\underline \kappa_{p,ji} := \overline \kappa_{p,ji} := \kappa_{p,ji}$, as had been done previously.

An additional motivation for working with both lower and upper constants is that it often allows us to exploit engineering/physical knowledge with regard to the sign of the derivative. As should already be apparent, the physical world is full of well-documented proportional or inversely proportional relationships for which the derivatives must have positive and negative signs, respectively. Consider, as some simple examples, the following:

\begin{itemize}
\item An athlete interested in controlling their weight knows that they cannot start eating more of a certain food and lose weight by doing so. The derivative between the quantity of food consumed and the athlete's weight is thus usually positive. On the contrary, it is very unlikely that the same athlete will gain weight by increasing their daily running time. Here, the derivative is negative.
\item In most power systems, increasing the amount of fuel fed while keeping everything else constant should increase the power produced. The derivative between power and fuel is thus positive.
\item In a jacketed reactor, it will often be the case -- though not always, depending on reaction -- that increasing (resp., decreasing) the temperature of the jacket will increase (resp., decrease) the temperature of the reactor. The derivative between the two temperatures is usually positive. 
\end{itemize}

\noindent When such relationships are present, one may translate them into appropriate Lipschitz constants by setting either the lower or upper constant equal to 0, or to a value abritrarily close to 0 if strict inequality is needed. For the Lipschitz constants corresponding to the derivatives, $M$, such physical links are harder to make, as they relate to the second derivatives and to the curvature of the function in question. However, even here we have encountered examples where \emph{a priori} knowledge regarding the sign of $M$ may be used -- for a simple solid oxide fuel cell stack \cite{Marchetti2009,Bunin2012d}, the relationship between the current and the power is almost always concave, which implies negative curvature and a negative $M$ ($\overline M \approx 0$). Outside of engineering applications, convex and concave relationships have been noted in fields like R\& D management \cite{Katz1982,Kuemmerle1998} and investment \cite{Teo2009}, among others.

The algorithmic usefulness of differentiating between lower and upper constants manifests itself by allowing for the SCFO framework to identify directions in which certain functions simply cannot increase, which cannot be done with the bounds employed so far in (\ref{eq:feas2}) and (\ref{eq:bound2Udeg}). Following the substitution $k \rightarrow k^*$ in (\ref{eq:feas2}), we use the results of \cite{Bunin:Lip} to formulate the relaxed bounds

\vspace{-2mm}
\begin{equation}\label{eq:feas2LU}
\hspace{-1mm}\begin{array}{l}
g_{p,j} ({\bf u}_{k+1},\tau_{k+1}) \leq g_{p,j} ({\bf u}_{k^*},\tau_{k^*}) \vspace{1mm} \\
\displaystyle  + \overline \kappa_{p,j\tau} \left( \tau_{k+1} - \tau_{k^*} \right) + \sum_{i=1}^{n_u} \mathop {\max} \left[ \begin{array}{l} \underline \kappa_{p,ji} ( u_{k+1,i} - u_{{k^*},i} ), \\ \overline \kappa_{p,ji} ( u_{k+1,i} - u_{{k^*},i} ) \end{array} \right],
\end{array}
\end{equation}

\vspace{-2mm}
\begin{equation}\label{eq:bound2UdegLU}
\begin{array}{l}
\phi_p({\bf u}_{k+1},\tau_{k+1}) - \phi_p({\bf u}_{k^*},\tau_{k+1}) \leq \vspace{1mm} \\
 \nabla \phi_p({\bf u}_{k^*},\tau_{k+1})^T \left[ \begin{array}{c} {\bf u}_{k+1} - {\bf u}_{k^*} \\ 0 \end{array} \right]  \vspace{1mm} \\
+\displaystyle \frac{1}{2} \sum_{i_1=1}^{n_u} \sum_{i_2=1}^{n_u} \mathop {\max} \left[ \begin{array}{l} \underline M_{\phi,i_1 i_2} (u_{k+1,i_1} - u_{{k^*},i_1}) \\
\hspace{20mm} (u_{k+1,i_2} - u_{{k^*},i_2}), \\ \overline M_{\phi,i_1 i_2} (u_{k+1,i_1} - u_{{k^*},i_1}) \\
\hspace{20mm}(u_{k+1,i_2} - u_{{k^*},i_2}) \end{array} \right].
\end{array}
\end{equation}

\noindent Clearly, we have opened the door for the possibility of the worst-case evolutions being \emph{negative}, as this is allowed by the maximum operator when the lower and upper constants do not have opposite signs.

Forcing the right-hand sides of (\ref{eq:feas2LU}) and (\ref{eq:bound2UdegLU}) to be negative then leads to the modified SCFO:

\vspace{-2mm}
\begin{equation}\label{eq:SCFO1LU}
\begin{array}{l}
g_{p,j} ({\bf u}_{k^*},\tau_{k^*}) + \overline \kappa_{p,j\tau} \left( \tau_{k+1} - \tau_{k^*} \right) \vspace{1mm} \\
\hspace{20mm}\displaystyle    + \sum_{i=1}^{n_u} \mathop {\max} \left[ \begin{array}{l} \underline \kappa_{p,ji} ( u_{k+1,i} - u_{{k^*},i} ), \\ \overline \kappa_{p,ji} ( u_{k+1,i} - u_{{k^*},i} ) \end{array} \right] \leq 0,
\end{array}
\end{equation}

\vspace{-2mm}
\begin{equation}\label{eq:SCFO7LU}
\begin{array}{l}
 \nabla \phi_p({\bf u}_{k^*},\tau_{k+1})^T \left[ \begin{array}{c} {\bf u}_{k+1} - {\bf u}_{k^*} \\ 0 \end{array} \right]  \vspace{1mm} \\
+\displaystyle \frac{1}{2} \sum_{i_1=1}^{n_u} \sum_{i_2=1}^{n_u} \mathop {\max} \left[ \begin{array}{l} \underline M_{\phi,i_1 i_2} (u_{k+1,i_1} - u_{{k^*},i_1}) \\
\hspace{15mm} (u_{k+1,i_2} - u_{{k^*},i_2}), \\ \overline M_{\phi,i_1 i_2} (u_{k+1,i_1} - u_{{k^*},i_1}) \\
\hspace{15mm}(u_{k+1,i_2} - u_{{k^*},i_2}) \end{array} \right] \leq 0.
\end{array}
\end{equation}

Slight modifications to (\ref{eq:kstar}) and (\ref{eq:kstar2}) are also in order, and simply consist in the substitution $\kappa_{p,j\tau} \rightarrow \overline \kappa_{p,j\tau}$:

\vspace{-2mm}
\begin{equation}\label{eq:kstarLU}
\begin{array}{rl}
k^* := {\rm arg} \mathop {\rm maximize}\limits_{\bar k \in [0,k]} & \bar k  \\
{\rm{subject}}\;{\rm{to}} & g_{p,j} ({\bf u}_{\bar k},\tau_{\bar k}) \\
& + \overline \kappa_{p,j\tau} \left( \tau_{k+1} - \tau_{\bar k} \right) \leq 0 \\
& \forall j = 1,...,n_{g_p},
\end{array}
\end{equation}

\vspace{-2mm}
\begin{equation}\label{eq:kstar2LU}
k^* := {\rm arg} \mathop {\rm minimize}\limits_{\bar k \in [0,k]} \;\; \mathop {\max} \limits_{j = 1,...,n_{g_p}} \left[ \begin{array}{l} g_{p,j} ({\bf u}_{\bar k},\tau_{\bar k}) \\ + \overline \kappa_{p,j\tau} \left( \tau_{k+1} - \tau_{\bar k} \right) \end{array} \right].
\end{equation}
 
\begin{figure*}
\begin{center}
\includegraphics[width=16cm]{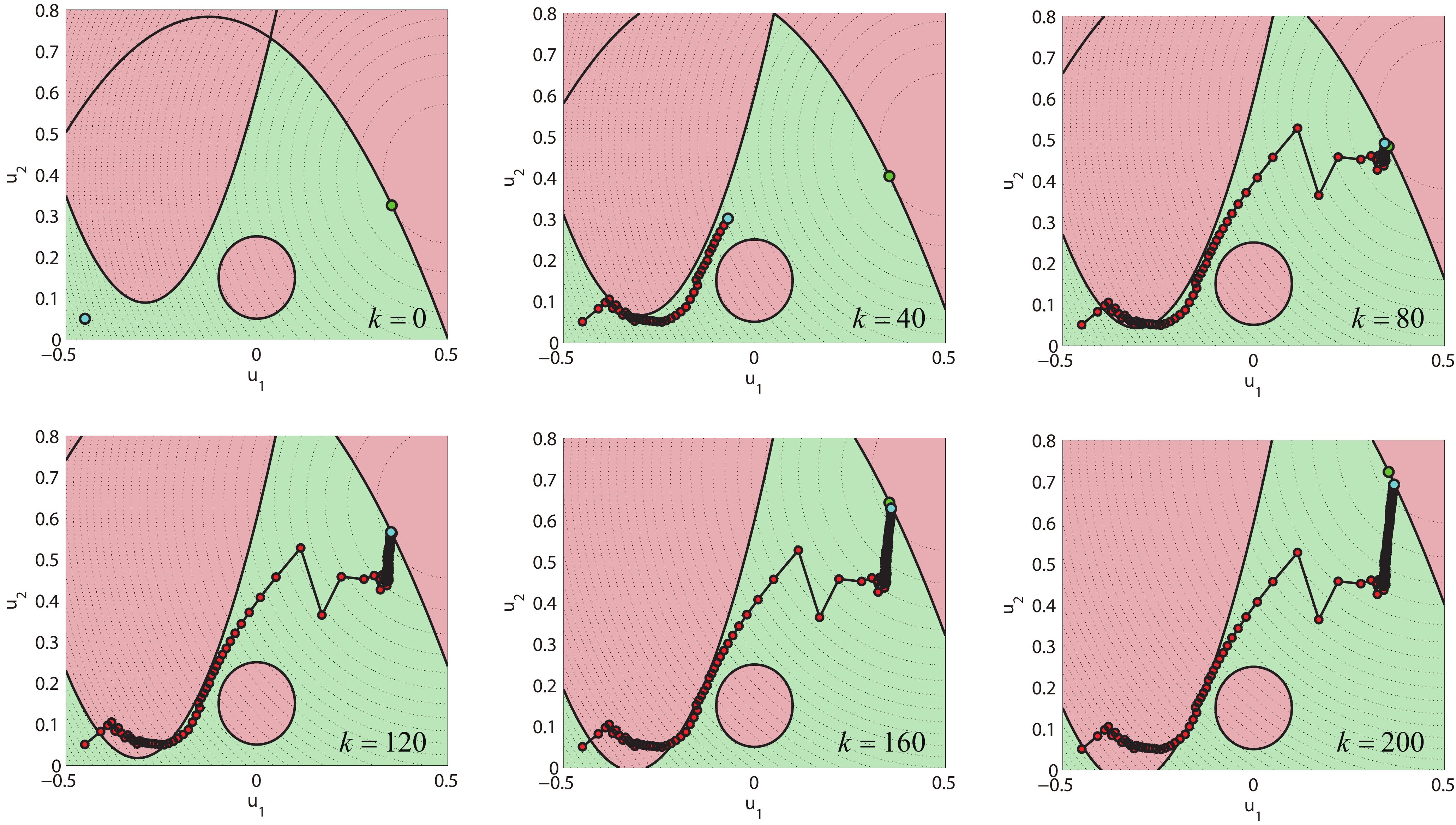}
\caption{Chain of experiments generated by applying the modified SCFO methodology to Problem (\ref{eq:exdeg}) for the ($-$) scenario with differentiation between lower and upper Lipschitz constants as given in (\ref{eq:exlipdegLU}).}
\label{fig:degC}
\end{center}
\end{figure*} 

The project-and-filter approach may benefit from the following (small) modification in the Boolean condition on $g_{p,j}$:

\vspace{-2mm}
\begin{equation}\label{eq:projdegLU}
\begin{array}{rl}
\bar {\bf u}_{k+1}^* := \;\;\;\;\;\;\;\;\; & \vspace{1mm}\\
 {\rm arg} \mathop {\rm minimize}\limits_{{\bf u}} & \| {\bf u} - {\bf u}_{k+1}^* \|_2^2  \vspace{1mm} \\
{\rm{subject}}\;{\rm{to}} & \nabla g_{p,j} ({\bf u}_{k^*},\tau_{k+1})^T \left[ \hspace{-1mm} \begin{array}{c} {\bf u} - {\bf u}_{k^*} \\ 0 \end{array} \hspace{-1mm} \right] \leq -\delta_{g_p,j} \vspace{1mm} \\
& \forall j: \begin{array}{l} g_{p,j} ({\bf u}_{k^*},\tau_{k^*}) \\ + \overline \kappa_{p,j\tau} (\tau_{k+1} - \tau_{k^*} ) \end{array} \geq -\epsilon_{p,j} \vspace{1mm} \\
 & \nabla g_{j} ({\bf u}_{k^*})^T ({\bf u} - {\bf u}_{k^*}) \leq -\delta_{g,j} \vspace{1mm} \\
& \forall j : g_{j}({\bf u}_{k^*}) \geq -\epsilon_{j} \vspace{1mm} \\
 & \nabla \phi_{p} ({\bf u}_{k^*},\tau_{k+1})^T  \left[ \begin{array}{c} {\bf u} - {\bf u}_{k^*} \\ 0 \end{array} \right] \leq -\delta_{\phi} \vspace{1mm} \\
 & {\bf u}^L \preceq {\bf u} \preceq {\bf u}^U,
\end{array}
\end{equation}

\noindent with the constraints (\ref{eq:SCFO1ideg}) and (\ref{eq:SCFO7ideg}) in the line search modified as

\vspace{-2mm}
\begin{equation}\label{eq:SCFO1idegLU}
\begin{array}{l}
g_{p,j} ({\bf u}_{k^*},\tau_{k^*}) + \overline \kappa_{p,j\tau} \left( \tau_{k+1} - \tau_{k^*} \right) \vspace{1mm} \\
\hspace{5mm} \displaystyle + K_k \sum_{i=1}^{n_u} \mathop {\max} \left[ \hspace{-.5mm} \begin{array}{l} \underline \kappa_{p,ji} ( \bar u_{k+1,i}^* - u_{{k^*},i} ), \\ \overline \kappa_{p,ji} ( \bar u_{k+1,i}^* - u_{{k^*},i} ) \end{array} \right] \leq 0, \vspace{1mm} \\
\hspace{45mm} \; \forall j = 1,...,n_{g_p},
\end{array}
\end{equation}

\vspace{-2mm}
\begin{equation}\label{eq:SCFO7idegLU}
\begin{array}{l}
 \nabla \phi_p({\bf u}_{k^*},\tau_{k+1})^T \left[ \begin{array}{c} \bar {\bf u}_{k+1}^* - {\bf u}_{k^*} \\ 0 \end{array} \right]  \vspace{1mm} \\
\displaystyle + \frac{K_k}{2} \sum_{i_1=1}^{n_u} \sum_{i_2=1}^{n_u} \mathop {\max} \left[ \begin{array}{l} \underline M_{\phi,i_1 i_2} (\bar u_{k+1,i_1}^* - u_{{k^*},i_1}) \\
\hspace{12mm} (\bar u_{k+1,i_2}^* - u_{{k^*},i_2}), \\ \overline M_{\phi,i_1 i_2} (\bar u_{k+1,i_1}^* - u_{{k^*},i_1}) \\
\hspace{12mm}(\bar u_{k+1,i_2}^* - u_{{k^*},i_2}) \end{array} \right]   \leq 0.
\end{array}
\end{equation}

To illustrate the effect of using both lower and upper Lipschitz constants, let us consider the ($-$) case of Problem (\ref{eq:exdeg}) and relax the choice of Lipschitz constants for this problem to

\vspace{-2mm}
\begin{equation}\label{eq:exlipdegLU}
\begin{array}{lll}
\underline \kappa_{p,11} = -10, & \underline \kappa_{p,12} = 0, & \underline \kappa_{p,1\tau} = -\frac{1}{1000}, \vspace{1mm} \\
\overline \kappa_{p,11} = 3, & \overline \kappa_{p,12} = 2, & \overline \kappa_{p,1\tau} = \frac{1}{1000}, \vspace{1mm} \\
\underline \kappa_{p,21} = -2, & \underline \kappa_{p,22} = 0, & \underline \kappa_{p,2\tau} = -\frac{1}{500}, \vspace{1mm} \\
\overline \kappa_{p,21} = 3, & \overline \kappa_{p,22} = 2, & \overline \kappa_{p,2\tau} = -\frac{1}{500}, \vspace{1mm} \\
\underline M_{\phi,11} = 1, & \underline M_{\phi,12} = -1, & \vspace{1mm} \\
\overline M_{\phi,11} = 3, & \overline M_{\phi,12} = 1, & \vspace{1mm} \\
\underline M_{\phi,21} = -1, & \underline M_{\phi,22} = 1, & \vspace{1mm} \\ 
\overline M_{\phi,21} = 1, & \overline M_{\phi,22} = 3. &
\end{array}
\end{equation}

Applying the modified project-and-filter approach, we obtain the results given in Figs. \ref{fig:degC} and \ref{fig:degCcost}. Comparing these to those of Figs. \ref{fig:degB} and \ref{fig:degcost}, we note that there is a significant reduction in the number of experiments needed to reach the neighborhood of the optimum, which this is mainly due to relaxing the Lipschitz constant $\kappa_{p,11} = 10$ to the lower and upper constants $\underline \kappa_{p,11} = -10$ and $\overline \kappa_{p,11} = 3$, as the latter is what appears in the maximum term as the algorithm increases $u_1$. Indeed, we note that this {\small $\sim$}3-fold reduction leads to a similar reduction in the number of experiments.

\begin{figure}
\begin{center}
\includegraphics[width=8cm]{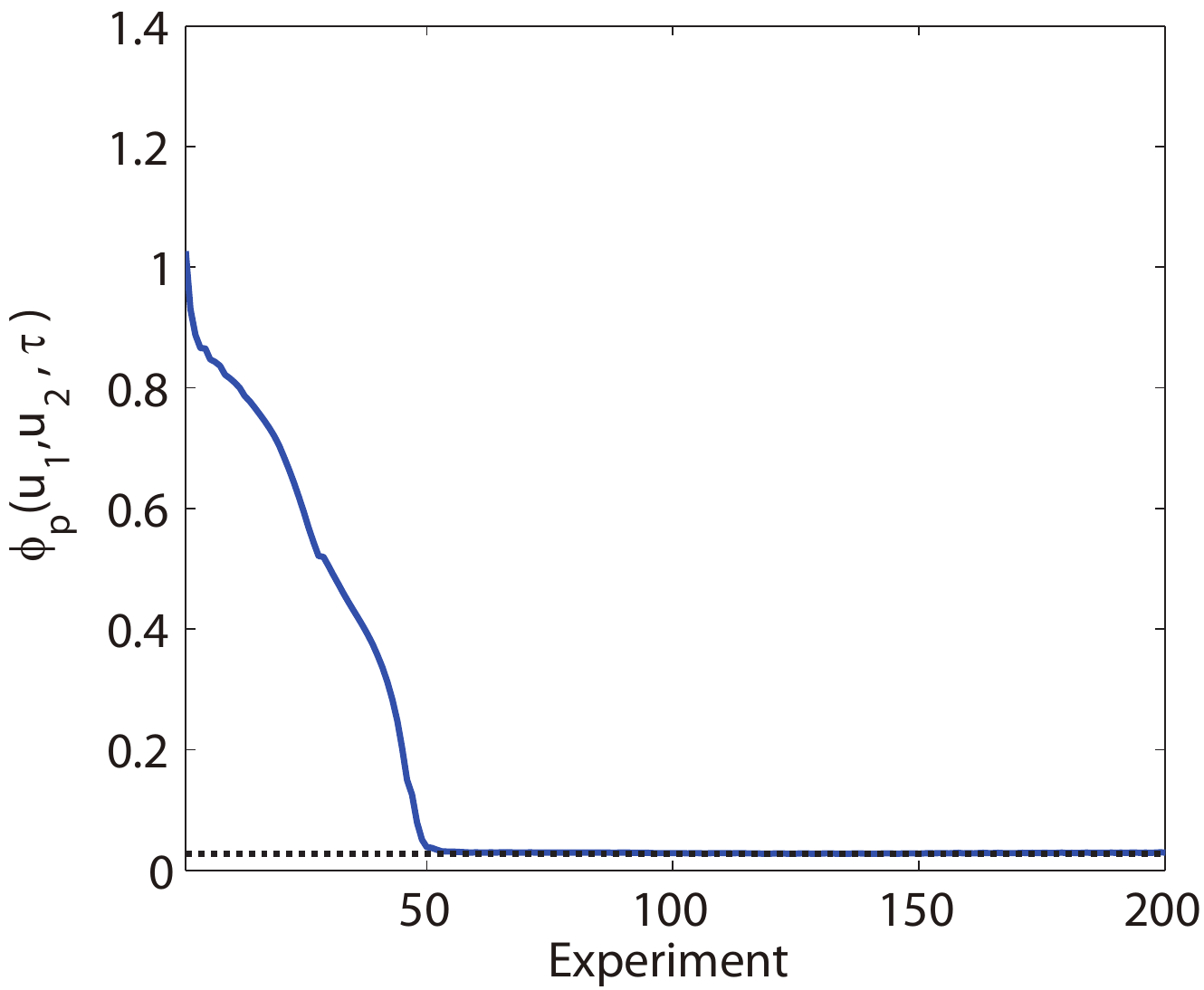}
\caption{Cost function values obtained by the modified SCFO methodology for Problem (\ref{eq:exdeg}) for the ($-$) scenario with differentiation between lower and upper Lipschitz constants as given in (\ref{eq:exlipdegLU}).}
\label{fig:degCcost}
\end{center}
\end{figure}

\subsection{Concavity Relaxations of the SCFO}
\label{sec:concave}

As mentioned previously, there may be applications where the relationships between certain decision variables and experimental functions are \emph{concave} in nature. When this is the case, the Lipschitz bounds on the experimental constraint functions may be made tighter, which in turn will make the SCFO less conservative. So as to maintain a fair degree of generality, we choose to work with \emph{partial concavity}, where the function in question only needs to be concave in certain variables \cite{Metelskii1996}.

\begin{definition}[Partial concavity of constraint function]
\label{def:ccv}
Let the decision variables ${\bf u}$ be partitioned into subvectors ${\bf v}$ and ${\bf z}$, so that $g_{p,j}({\bf u},\tau) = g_{p,j}({\bf v},{\bf z},\tau)$. The function $g_{p,j}$ will be said to be \emph{globally partially concave} in ${\bf z}$ if it is concave everywhere on $\mathcal{I}_\tau$ for ${\bf v}$ and $\tau$ fixed. Likewise, $g_{p,j}$ will be said to be globally partially concave in both ${\bf z}$ and $\tau$ if it is concave everywhere on $\mathcal{I}_\tau$ for ${\bf v}$ fixed.
\end{definition}

\begin{lemma}[Concave tightening of upper Lipschitz bound]
\label{lem:concave}
Let $I_{c,j}$ denote the set of indices defining the variables ${\bf z}$, with respect to which the experimental constraint function $g_{p,j}$ is globally partially concave. Furthermore, let $\eta_{c,j}$ denote a Boolean indicator that is equal to 1 if $g_{p,j}$ is globally partially concave in both ${\bf z}$ and $\tau$, and is equal to 0 otherwise. It follows that the following tightening of (\ref{eq:feas2LU}) is valid:

\vspace{-2mm}
\begin{equation}\label{eq:feas2LUccv}
\begin{array}{l}
g_{p,j} ({\bf u}_{k+1},\tau_{k+1}) \leq g_{p,j} ({\bf u}_{k^*},\tau_{k^*}) \vspace{1mm} \\
\hspace{24mm} \displaystyle +\eta_{c,j} \frac{\partial g_{p,j}}{\partial \tau} \Big |_{({\bf u}_{k^*},\tau_{k^*})} ( \tau_{k+1} - \tau_{{k^*}} ) \vspace{1mm} \\
\hspace{24mm} \displaystyle + (1-\eta_{c,j}) \overline \kappa_{p,j\tau} \left( \tau_{k+1} - \tau_{k^*} \right) \vspace{1mm}\\
\hspace{24mm}\displaystyle   + \sum_{i \in I_{c,j}} \frac{\partial g_{p,j}}{\partial u_i} \Big |_{({\bf u}_{k^*},\tau_{k^*})} ( u_{k+1,i} - u_{{k^*},i} ) \vspace{1mm} \\
\hspace{24mm} + \displaystyle \sum_{i \not \in I_{c,j}} \mathop {\max} \left[ \begin{array}{l} \underline \kappa_{p,ji} ( u_{k+1,i} - u_{{k^*},i} ), \vspace{1mm}\\ \overline \kappa_{p,ji} ( u_{k+1,i} - u_{{k^*},i} ) \end{array} \right],
\end{array}
\end{equation}

\noindent where $i \not \in I_{c,j}$ implicitly denotes $\{ 1,...,n_u \} \setminus I_{c,j}$.

\end{lemma}
\begin{proof}
The constraint function value at the next experimental iterate may be written as

\vspace{-2mm}
\begin{equation}\label{eq:ccvproof1}
\begin{array}{l}
g_{p,j} ({\bf v}_{k+1},{\bf z}_{k+1},\tau_{k+1}) = g_{p,j} ({\bf v}_{k^*},{\bf z}_{k^*},\tau_{k^*})  \vspace{1mm}\\
\hspace{10mm}+ g_{p,j} ({\bf v}_{k^*},{\bf z}_{k+1},\tau_{k+1}) - g_{p,j} ({\bf v}_{k^*},{\bf z}_{k^*},\tau_{k^*}) \vspace{1mm} \\
\hspace{10mm}+ g_{p,j} ({\bf v}_{k+1},{\bf z}_{k+1},\tau_{k+1}) - g_{p,j} ({\bf v}_{k^*},{\bf z}_{k+1},\tau_{k+1}).
\end{array}
\end{equation}

\noindent We will proceed to obtain the desired result by bounding the two addends, $g_{p,j} ({\bf v}_{k^*},{\bf z}_{k+1},\tau_{k+1}) - g_{p,j} ({\bf v}_{k^*},{\bf z}_{k^*},\tau_{k^*})$ and $g_{p,j} ({\bf v}_{k+1},{\bf z}_{k+1},\tau_{k+1}) - g_{p,j} ({\bf v}_{k^*},{\bf z}_{k+1},\tau_{k+1})$, separately.

First, suppose that $\eta_{c,j} = 1$. If this is so, then we have, by the first-order condition of concavity:

\vspace{-2mm}
\begin{equation}\label{eq:ccvproof2}
\begin{array}{l}
g_{p,j} ({\bf v}_{k^*},{\bf z}_{k+1},\tau_{k+1}) - g_{p,j} ({\bf v}_{k^*},{\bf z}_{k^*},\tau_{k^*}) \vspace{1mm}\\
\displaystyle \hspace{20mm}\leq \frac{\partial g_{p,j}}{\partial \tau} \Big |_{({\bf u}_{k^*},\tau_{k^*})} ( \tau_{k+1} - \tau_{{k^*}} ) \vspace{1mm} \\
\displaystyle \hspace{23mm} \displaystyle   + \sum_{i \in I_{c,j}} \frac{\partial g_{p,j}}{\partial u_i} \Big |_{({\bf u}_{k^*},\tau_{k^*})} ( u_{k+1,i} - u_{{k^*},i} ).
\end{array}
\end{equation}

In the case that $\eta_{c,j} = 0$, consider the decomposition

\vspace{-2mm}
\begin{equation}\label{eq:ccvproof3}
\begin{array}{l}
g_{p,j} ({\bf v}_{k^*},{\bf z}_{k+1},\tau_{k+1}) - g_{p,j} ({\bf v}_{k^*},{\bf z}_{k^*},\tau_{k^*}) = \vspace{1mm}  \\
\displaystyle \hspace{10mm}g_{p,j} ({\bf v}_{k^*},{\bf z}_{k+1},\tau_{k^*}) - g_{p,j} ({\bf v}_{k^*},{\bf z}_{k^*},\tau_{k^*}) \vspace{1mm} \\
\displaystyle \hspace{10mm}+g_{p,j} ({\bf v}_{k^*},{\bf z}_{k+1},\tau_{k+1}) - g_{p,j} ({\bf v}_{k^*},{\bf z}_{k+1},\tau_{k^*}),
\end{array}
\end{equation}

\noindent where the first addend may be bounded, again, by exploiting concavity:

\vspace{-2mm}
\begin{equation}\label{eq:ccvproof4}
\begin{array}{l}
 g_{p,j} ({\bf v}_{k^*},{\bf z}_{k+1},\tau_{k^*}) - g_{p,j} ({\bf v}_{k^*},{\bf z}_{k^*},\tau_{k^*}) \vspace{1mm} \\
\hspace{20mm}\displaystyle   \leq \sum_{i \in I_{c,j}} \frac{\partial g_{p,j}}{\partial u_i} \Big |_{({\bf u}_{k^*},\tau_{k^*})} ( u_{k+1,i} - u_{{k^*},i} ),
\end{array}
\end{equation}

\noindent while the second may be bounded  by using the Lipschitz bound

\vspace{-2mm}
\begin{equation}\label{eq:ccvproof5}
\begin{array}{l}
g_{p,j} ({\bf v}_{k^*},{\bf z}_{k+1},\tau_{k+1}) - g_{p,j} ({\bf v}_{k^*},{\bf z}_{k+1},\tau_{k^*}) \vspace{1mm} \\
\hspace{40mm}\displaystyle   \leq \overline \kappa_{p,j\tau} \left( \tau_{k+1} - \tau_{k^*} \right).
\end{array}
\end{equation}

Adding (\ref{eq:ccvproof4}) and (\ref{eq:ccvproof5}) thus yields

\vspace{-2mm}
\begin{equation}\label{eq:ccvproof6}
\begin{array}{l}
g_{p,j} ({\bf v}_{k^*},{\bf z}_{k+1},\tau_{k+1}) - g_{p,j} ({\bf v}_{k^*},{\bf z}_{k^*},\tau_{k^*}) \vspace{1mm}\\
\displaystyle \hspace{20mm}\leq \overline \kappa_{p,j\tau} \left( \tau_{k+1} - \tau_{k^*} \right) \vspace{1mm} \\
\displaystyle \hspace{23mm} \displaystyle   + \sum_{i \in I_{c,j}} \frac{\partial g_{p,j}}{\partial u_i} \Big |_{({\bf u}_{k^*},\tau_{k^*})} ( u_{k+1,i} - u_{{k^*},i} )
\end{array}
\end{equation}

\noindent for the case when $\eta_{c,j} = 0$. 

To account for both $\eta_{c,j} = 1$ and $\eta_{c,j} = 0$ with a single bound, it is clear that we can simply combine (\ref{eq:ccvproof2}) and (\ref{eq:ccvproof6}) as

\vspace{-2mm}
\begin{equation}\label{eq:ccvproof7}
\begin{array}{l}
g_{p,j} ({\bf v}_{k^*},{\bf z}_{k+1},\tau_{k+1}) - g_{p,j} ({\bf v}_{k^*},{\bf z}_{k^*},\tau_{k^*}) \vspace{1mm}\\
\displaystyle \hspace{20mm}\leq \eta_{c,j} \frac{\partial g_{p,j}}{\partial \tau} \Big |_{({\bf u}_{k^*},\tau_{k^*})} ( \tau_{k+1} - \tau_{{k^*}} ) \vspace{1mm} \\
\displaystyle \hspace{23mm}+ (1-\eta_{c,j}) \overline \kappa_{p,j\tau} \left( \tau_{k+1} - \tau_{k^*} \right) \vspace{1mm} \\
\displaystyle \hspace{23mm} \displaystyle   + \sum_{i \in I_{c,j}} \frac{\partial g_{p,j}}{\partial u_i} \Big |_{({\bf u}_{k^*},\tau_{k^*})} ( u_{k+1,i} - u_{{k^*},i} ).
\end{array}
\end{equation}

To bound the second addend in (\ref{eq:ccvproof1}), we revert to the Lipschitz bound

\vspace{-2mm}
\begin{equation}\label{eq:ccvproof8}
\begin{array}{l}
g_{p,j} ({\bf v}_{k+1},{\bf z}_{k+1},\tau_{k+1}) - g_{p,j} ({\bf v}_{k^*},{\bf z}_{k+1},\tau_{k+1}) \vspace{1mm}\\
\hspace{22mm}\leq \displaystyle \sum_{i \not \in I_{c,j}} \mathop {\max} \left[ \begin{array}{l} \underline \kappa_{p,ji} ( u_{k+1,i} - u_{{k^*},i} ), \vspace{1mm}\\ \overline \kappa_{p,ji} ( u_{k+1,i} - u_{{k^*},i} ) \end{array} \right].
\end{array}
\end{equation}

Applying (\ref{eq:ccvproof7}) and (\ref{eq:ccvproof8}) to (\ref{eq:ccvproof1}), or simply adding (\ref{eq:ccvproof7}) to (\ref{eq:ccvproof8}) and rearranging, then yields (\ref{eq:feas2LUccv}). \qed

\end{proof}

\begin{figure*}
\begin{center}
\includegraphics[width=16cm]{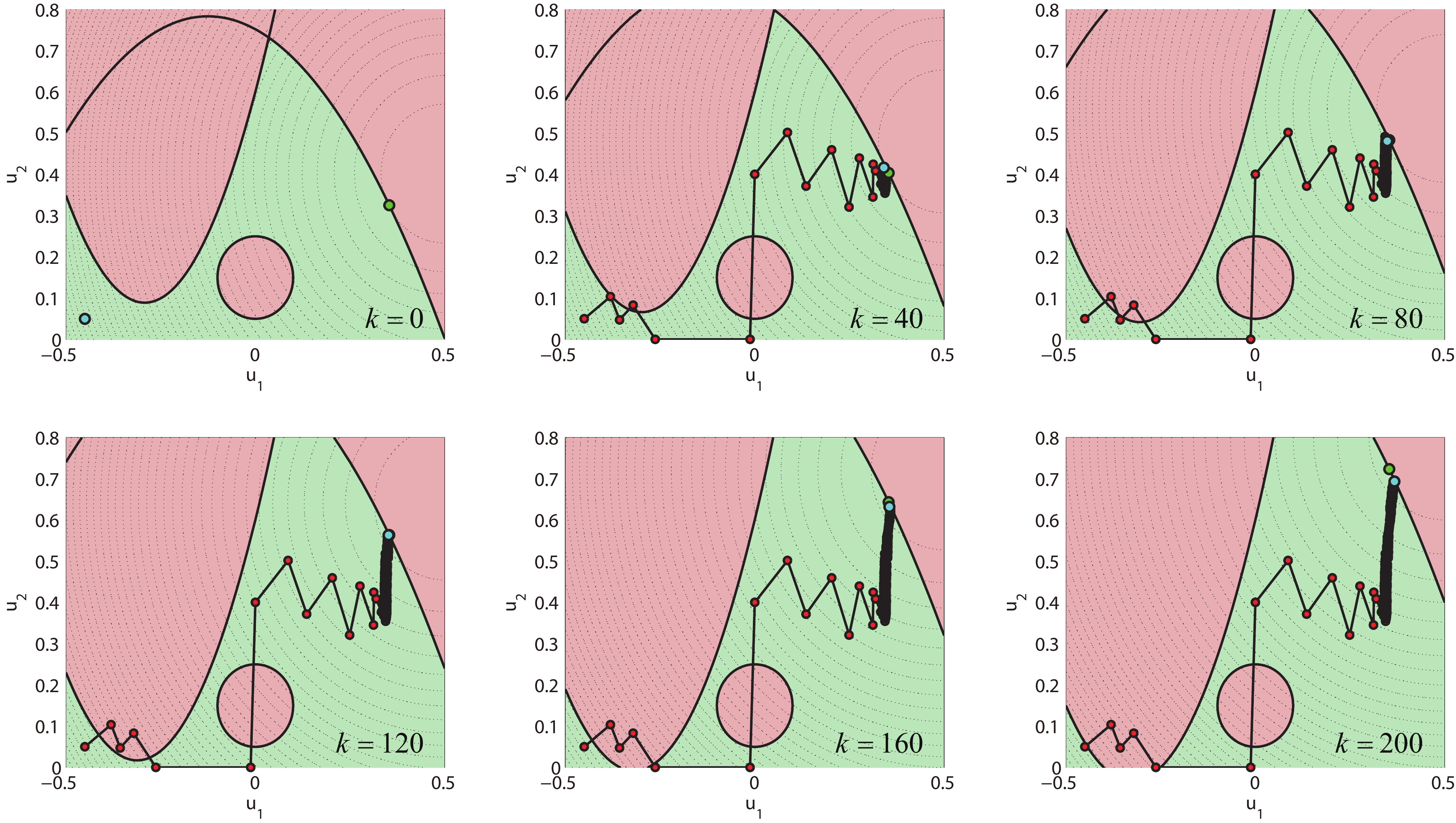}
\caption{Chain of experiments generated by applying the modified SCFO methodology to Problem (\ref{eq:exdeg}) for the ($-$) scenario with concavity relaxations included.}
\label{fig:degD}
\end{center}
\end{figure*} 

Because

\vspace{-2mm}
\begin{equation}\label{eq:ccvtight}
\begin{array}{l}
\displaystyle \frac{\partial g_{p,j}}{\partial u_i} \Big |_{({\bf u}_{k^*},\tau_{k^*})} ( u_{k+1,i} - u_{{k^*},i} ) \\
\hspace{20mm} \leq \mathop {\max}  \left[ \begin{array}{l} \underline \kappa_{p,ji} ( u_{k+1,i} - u_{{k^*},i} ), \vspace{1mm}\\ \overline \kappa_{p,ji} ( u_{k+1,i} - u_{{k^*},i} ) \end{array} \right]
\end{array}
\end{equation}

\noindent and

\vspace{-2mm}
\begin{equation}\label{eq:ccvtight2}
\frac{\partial g_{p,j}}{\partial \tau} \Big |_{({\bf u}_{k^*},\tau_{k^*})} ( \tau_{k+1} - \tau_{{k^*}} ) \leq \overline \kappa_{p,j\tau} \left( \tau_{k+1} - \tau_{k^*} \right),
\end{equation}

\noindent one sees that (\ref{eq:feas2LUccv}) is indeed a tighter bound than (\ref{eq:feas2LU}).

The modified SCFO (\ref{eq:SCFO1LU}) can thus be relaxed by forcing the right-hand side of (\ref{eq:feas2LUccv}) to be nonpositive:

\vspace{-2mm}
\begin{equation}\label{eq:SCFO1ccv}
\begin{array}{l}
\displaystyle g_{p,j} ({\bf u}_{k^*},\tau_{k^*})  +\eta_{c,j} \frac{\partial g_{p,j}}{\partial \tau} \Big |_{({\bf u}_{k^*},\tau_{k^*})} ( \tau_{k+1} - \tau_{{k^*}} ) \vspace{1mm} \\
 \displaystyle + (1-\eta_{c,j}) \overline \kappa_{p,j\tau} \left( \tau_{k+1} - \tau_{k^*} \right) \vspace{1mm}\\
\displaystyle   + \sum_{i \in I_{c,j}} \frac{\partial g_{p,j}}{\partial u_i} \Big |_{({\bf u}_{k^*},\tau_{k^*})} ( u_{k+1,i} - u_{{k^*},i} ) \vspace{1mm} \\
+ \displaystyle \sum_{i \not \in I_{c,j}} \mathop {\max} \left[ \begin{array}{l} \underline \kappa_{p,ji} ( u_{k+1,i} - u_{{k^*},i} ), \vspace{1mm}\\ \overline \kappa_{p,ji} ( u_{k+1,i} - u_{{k^*},i} ) \end{array} \right] \leq 0.
\end{array}
\end{equation}

In applying the project-and-filter approach, we may, using the same arguments as above, relax the Boolean condition in the projection:

\vspace{-2mm}
\begin{equation}\label{eq:projdegLUccv}
\begin{array}{rl}
\bar {\bf u}_{k+1}^* := \;\;\;\;\;\;\;\;\; & \vspace{1mm}\\
 {\rm arg} \mathop {\rm minimize}\limits_{{\bf u}} & \| {\bf u} - {\bf u}_{k+1}^* \|_2^2  \vspace{1mm} \\
{\rm{subject}}\;{\rm{to}} & \nabla g_{p,j} ({\bf u}_{k^*},\tau_{k+1})^T \left[ \hspace{-1mm} \begin{array}{c} {\bf u} - {\bf u}_{k^*} \\ 0 \end{array} \hspace{-1mm} \right] \leq -\delta_{g_p,j} \vspace{1mm} \\
& \hspace{-5mm} \forall j: \begin{array}{l} g_{p,j} ({\bf u}_{k^*},\tau_{k^*}) \vspace{1mm} \\ \displaystyle + \eta_{c,j} \frac{\partial g_{p,j}}{\partial \tau} \Big |_{({\bf u}_{k^*},\tau_{k^*})} ( \tau_{k+1} - \tau_{{k^*}} ) \vspace{1mm}  \\ + (1 - \eta_{c,j}) \overline \kappa_{p,j\tau} (\tau_{k+1} - \tau_{k^*} ) \geq -\epsilon_{p,j} \end{array} \vspace{1mm} \\
 & \nabla g_{j} ({\bf u}_{k^*})^T ({\bf u} - {\bf u}_{k^*}) \leq -\delta_{g,j} \vspace{1mm} \\
& \forall j : g_{j}({\bf u}_{k^*}) \geq -\epsilon_{j} \vspace{1mm} \\
 & \nabla \phi_{p} ({\bf u}_{k^*},\tau_{k+1})^T  \left[ \begin{array}{c} {\bf u} - {\bf u}_{k^*} \\ 0 \end{array} \right] \leq -\delta_{\phi} \vspace{1mm} \\
 & {\bf u}^L \preceq {\bf u} \preceq {\bf u}^U,
\end{array}
\end{equation}

The feasibility condition (\ref{eq:SCFO1idegLU}) is then be modified as

\vspace{-2mm}
\begin{equation}\label{eq:SCFO1idegLUccv}
\begin{array}{l}
\displaystyle g_{p,j} ({\bf u}_{k^*},\tau_{k^*}) +\eta_{c,j} \frac{\partial g_{p,j}}{\partial \tau} \Big |_{({\bf u}_{k^*},\tau_{k^*})} ( \tau_{k+1} - \tau_{{k^*}} ) \vspace{1mm} \\
 \displaystyle + (1-\eta_{c,j}) \overline \kappa_{p,j\tau} \left( \tau_{k+1} - \tau_{k^*} \right) \vspace{1mm}\\
\displaystyle   + K_k \sum_{i \in I_{c,j}} \frac{\partial g_{p,j}}{\partial u_i} \Big |_{({\bf u}_{k^*},\tau_{k^*})} ( \bar u_{k+1,i}^* - u_{{k^*},i} ) \vspace{1mm} \\
+ K_k \displaystyle \sum_{i \not \in I_{c,j}} \mathop {\max} \left[ \begin{array}{l} \underline \kappa_{p,ji} ( \bar u_{k+1,i}^* - u_{{k^*},i} ), \vspace{1mm}\\ \overline \kappa_{p,ji} ( \bar u_{k+1,i}^* - u_{{k^*},i} ) \end{array} \right] \leq 0, \vspace{1mm} \\
\hspace{0mm} \forall j = 1,...,n_{g_p}.
\end{array}
\end{equation}

Further modifications to (\ref{eq:kstarLU}) and (\ref{eq:kstar2LU}) are also necessary to account for the possible relaxations:

\vspace{-2mm}
\begin{equation}\label{eq:kstarLUccv}
\begin{array}{rl}
k^* := {\rm arg} \mathop {\rm maximize}\limits_{\bar k \in [0,k]} & \bar k \vspace{1mm} \\
{\rm{subject}}\;{\rm{to}} & g_{p,j} ({\bf u}_{\bar k},\tau_{\bar k}) \vspace{1mm} \\
& \displaystyle + \eta_{c,j} \frac{\partial g_{p,j}}{\partial \tau} \Big |_{({\bf u}_{\bar k},\tau_{\bar k})} ( \tau_{k+1} - \tau_{{\bar k}} ) \vspace{1mm} \\
&  + (1- \eta_{c,j}) \overline \kappa_{p,j\tau} \left( \tau_{k+1} - \tau_{\bar k} \right) \leq 0, \vspace{1mm} \\
& \forall j = 1,...,n_{g_p},
\end{array}
\end{equation}

\vspace{-2mm}
\begin{equation}\label{eq:kstar2LUccv}
\begin{array}{l}
k^* := \\
{\rm arg} \mathop {\rm minimize}\limits_{\bar k \in [0,k]} \hspace{-1mm} \mathop {\max} \limits_{j = 1,...,n_{g_p}} \hspace{-1mm} \left[ \hspace{-1mm} \begin{array}{l} g_{p,j} ({\bf u}_{\bar k},\tau_{\bar k}) \vspace{1mm} \\
\displaystyle + \eta_{c,j} \frac{\partial g_{p,j}}{\partial \tau} \Big |_{({\bf u}_{\bar k},\tau_{\bar k})} ( \tau_{k+1} - \tau_{{\bar k}} ) \vspace{1mm} \\
+ (1-  \eta_{c,j}) \overline \kappa_{p,j\tau} \left( \tau_{k+1} - \tau_{\bar k} \right) \end{array} \hspace{-1mm} \right].
\end{array}
\end{equation}

To demonstrate the potential benefits of these relaxations, let us again consider the ($-$) case of Problem (\ref{eq:exdeg}) and apply the modified SCFO. For $g_{p,2}$, it is clear that we can assume global partial concavity -- in fact, linearity -- in both $u_2$ and $\tau$, with $I_{c,2} := \{ 2 \}$ and $\eta_{c,2} := 1$. For $g_{p,1}$, there is actually a choice due to the bilinear term $\tau u_1$, which is not concave. More specifically, we can either set $I_{c,1} := \{ 1,2 \}, \eta_{c,1} := 0$ or $I_{c,1} := \{ 2 \}, \eta_{c,1} := 1$, as partial concavity in $\tau$ and $u_1$ is mutually exclusive. Since the decision variable $u_1$ has a much greater influence on the function value than $\tau$, we choose the former as this is expected to lead to more significant relaxations in the SCFO. The results are given in Figs. \ref{fig:degD} and \ref{fig:degDcost}, and show that the number of experiments needed to arrive in the neighborhood of the optimum is reduced drastically (compare with Fig. \ref{fig:degC}), as the relaxations in the feasibility condition allow much bigger steps and thus faster progress.

\begin{figure}
\begin{center}
\includegraphics[width=8cm]{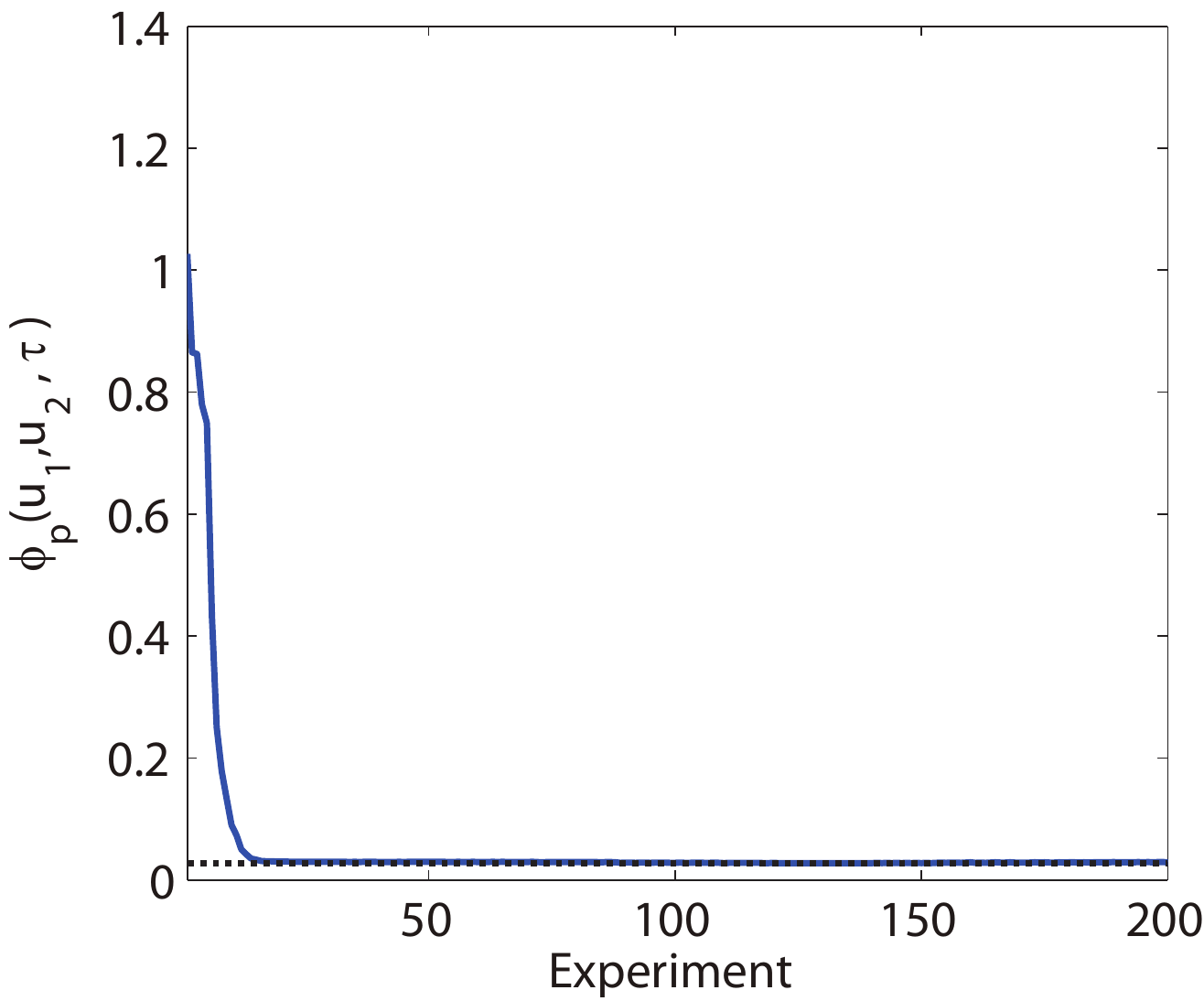}
\caption{Cost function values obtained by the modified SCFO methodology for Problem (\ref{eq:exdeg}) for the ($-$) scenario with concavity relaxations included.}
\label{fig:degDcost}
\end{center}
\end{figure} 

While impressive, such relaxations are of course of limited use in practice, as in many experimental problems one may not know about the existence of concave relationships even when they are present. A natural case where one might try to apply these relaxations would be when a sufficiently accurate surrogate model of the experimental function $g_{p,j}$ is available and exhibits obvious concave behavior in certain variables (e.g., the power-current relationship in the fuel cell system \cite{Marchetti2009}). When such a model is not available, an alternative may be to regress the available experimental data with concave models as a means of ``discovering'' or validating the concavity assumption \cite{Boyd2008,Ubhaya2009,Bunin2013a}. However, as falsely assuming concavity could lead to constraint violations due to unjustified relaxations in the Lipschitz bounds, one should always remain careful when making this assumption in practice.

\subsection{Using All Available Data}

So far we have only considered stating the feasibility condition (\ref{eq:SCFO1ccv}) with respect to the future experimental iteration ${k+1}$ and the reference iteration ${k^*}$. However, as the bound (\ref{eq:feas2LUccv}) applies everywhere on $\mathcal{I}_\tau$, it follows that any other iteration from 0 to $k$ may be used in place of $k^*$ without endangering the validity of the bound. Because every one of the $k+1$ bounds is valid, it follows that their minimum is valid as well, and so we may tighten (\ref{eq:feas2LUccv}) further:

\vspace{-2mm}
\begin{equation}\label{eq:feas2LUccvall}
\begin{array}{l}
g_{p,j} ({\bf u}_{k+1},\tau_{k+1}) \leq  \vspace{1mm} \\
\hspace{0mm} \mathop {\min} \limits_{\bar k = 0,...,k} \left[ \begin{array}{l} g_{p,j} ({\bf u}_{\bar k},\tau_{\bar k}) \displaystyle +\eta_{c,j} \frac{\partial g_{p,j}}{\partial \tau} \Big |_{({\bf u}_{\bar k},\tau_{\bar k})} ( \tau_{k+1} - \tau_{{\bar k}} ) \vspace{1mm} \\
 \displaystyle + (1-\eta_{c,j}) \overline \kappa_{p,j\tau} \left( \tau_{k+1} - \tau_{\bar k} \right) \vspace{1mm}\\
\displaystyle   + \sum_{i \in I_{c,j}} \frac{\partial g_{p,j}}{\partial u_i} \Big |_{({\bf u}_{\bar k},\tau_{\bar k})} ( u_{k+1,i} - u_{{\bar k},i} ) \vspace{1mm} \\
 + \displaystyle \sum_{i \not \in I_{c,j}} \mathop {\max} \left[ \begin{array}{l} \underline \kappa_{p,ji} ( u_{k+1,i} - u_{\bar k,i} ), \vspace{1mm}\\ \overline \kappa_{p,ji} ( u_{k+1,i} - u_{\bar k,i} ) \end{array} \right] \end{array} \right].
\end{array}
\end{equation}

Forcing the right-hand side to be negative then leads to the relaxation of Condition (\ref{eq:SCFO1ccv}):

\vspace{-2mm}
\begin{equation}\label{eq:SCFO1all}
\begin{array}{l}
\mathop {\min} \limits_{\bar k = 0,...,k} \left[ \hspace{-1mm} \begin{array}{l} g_{p,j} ({\bf u}_{\bar k},\tau_{\bar k}) \vspace{1mm} \\
\displaystyle +\eta_{c,j} \frac{\partial g_{p,j}}{\partial \tau} \Big |_{({\bf u}_{\bar k},\tau_{\bar k})} ( \tau_{k+1} - \tau_{{\bar k}} ) \vspace{1mm} \\
 \displaystyle + (1-\eta_{c,j}) \overline \kappa_{p,j\tau} \left( \tau_{k+1} - \tau_{\bar k} \right) \vspace{1mm}\\
\displaystyle   + \sum_{i \in I_{c,j}} \frac{\partial g_{p,j}}{\partial u_i} \Big |_{({\bf u}_{\bar k},\tau_{\bar k})} ( u_{k+1,i} - u_{{\bar k},i} ) \vspace{1mm} \\
 + \displaystyle \sum_{i \not \in I_{c,j}} \mathop {\max} \left[ \begin{array}{l} \underline \kappa_{p,ji} ( u_{k+1,i} - u_{\bar k,i} ), \vspace{1mm}\\ \overline \kappa_{p,ji} ( u_{k+1,i} - u_{\bar k,i} ) \end{array} \right] \end{array} \right] \leq 0
\end{array},
\end{equation}

\noindent with the corresponding modification in the project-and-filter approach:

\vspace{-2mm}
\begin{equation}\label{eq:SCFO1idegLUccvall}
\hspace{-4mm} \begin{array}{l}
\mathop {\min} \limits_{\bar k = 0,...,k} \left[ \hspace{-1mm} \begin{array}{l} g_{p,j} ({\bf u}_{\bar k},\tau_{\bar k}) \vspace{1mm} \\
 \displaystyle +\eta_{c,j} \frac{\partial g_{p,j}}{\partial \tau} \Big |_{({\bf u}_{\bar k},\tau_{\bar k})} ( \tau_{k+1} - \tau_{{\bar k}} ) \vspace{1mm} \\
 \displaystyle + (1-\eta_{c,j}) \overline \kappa_{p,j\tau} \left( \tau_{k+1} - \tau_{\bar k} \right) \vspace{1mm}\\
\displaystyle   + \sum_{i \in I_{c,j}} \frac{\partial g_{p,j}}{\partial u_i} \Big |_{({\bf u}_{\bar k},\tau_{\bar k})} ( u_{k^*,i}  +\\
\hspace{17mm}  K_k (\bar u_{k+1,i}^* - u_{k^*,i} ) - u_{{\bar k},i} ) \vspace{1mm} \\
 + \displaystyle \sum_{i \not \in I_{c,j}} \mathop {\max} \left[ \begin{array}{l} \underline \kappa_{p,ji} ( u_{k^*,i} + \\ \hspace{2mm} K_k (\bar u_{k+1,i}^* - u_{k^*,i} ) - u_{\bar k,i} ), \vspace{1mm}\\ \overline \kappa_{p,ji} ( u_{k^*,i} + \\ \hspace{2mm} K_k (\bar u_{k+1,i}^* - u_{k^*,i} ) - u_{\bar k,i} ) \end{array} \right] \end{array} \hspace{-1mm} \right]  \leq 0.
\end{array}
\end{equation}

In general, one would expect the relaxation to be useful when an algorithm revisits parts of the experimental space where it has been before, as the Lipschitz bounds essentially construct small feasibility-guaranteeing regions around each ${\bf u}_{\bar k}$ \cite{Bunin:12c}. We list some scenarios where an algorithm might visit the same region twice:

\begin{itemize}
\item Errors in the gradients (Section \ref{sec:gradest}) lead to haphazard behavior until better gradient estimates are obtained.
\item The algorithm converges in a zig-zag manner.
\item After the algorithm converges to the optimum of (\ref{eq:mainprob}), the cost function changes -- due to, e.g., a change in market demand -- but the feasible region remains the same. As such, the algorithm must now navigate to a new optimum, possibly revisiting regions of the experimental space where it has been already.
\item Additional experiments have been carried out prior to ${\bf u}_0$, and may be included in the data set.
\end{itemize}

We illustrate the potential usefulness of this relaxation by considering an example of the third point. Let us imagine an experimental optimization where the problem to be solved is originally (\ref{eq:exprob}) -- we will, for simplicity, make degradation effects negligible -- but which then changes to

\vspace{-2mm}
\begin{equation}\label{eq:exprobmod}
\begin{array}{rl}
\mathop {{\rm{minimize}}}\limits_{u_1,u_2} & (u_1+0.25)^2 + (u_2-0.6)^2 \\
{\rm{subject}}\hspace{1mm}{\rm{to}} & -6u^2_1 - 3.5u_1 + u_2 -0.6 \le 0 \vspace{1mm} \\
 & 2u^2_1 + 0.5u_1 + u_2 -0.75 \le 0 \vspace{1mm} \\
 & -u^2_1 - (u_2-0.15)^2 +0.01 \le 0 \vspace{1mm} \\
& -0.5 \leq u_1 \leq 0.5 \\
& 0 \leq u_2 \leq 0.8
\end{array}
\end{equation}

\noindent for $k > 50$ (i.e., the cost function changes after 50 experimental iterations). We will assume that the new cost function has not been evaluated during the first 51 experiments, but that the constraint functions, which remain the same, have and may be used in the relaxation (\ref{eq:SCFO1idegLUccvall}).

The results for the case when this relaxation is not employed are given first (Fig. \ref{fig:extwocost1}), and show that the algorithm successfully converges to the first optimum and then to the second. However, because the initial point for the second problem is very close to an experimental constraint, the initial steps during the second phase are forced to be small so as to not violate $g_{p,2}({\bf u}) \leq 0$. We now compare this to the case where the relaxation is used (Fig. \ref{fig:extwocost2}) and see that, while there is no change in convergence behavior with respect to the first optimum, the second phase is radically different, the algorithm being allowed to take a very big step away from the constraint and to the neighborhood of one of the previous experimental iterates. The result is that the optimum is approached much quicker for the second cost function when this relaxation is used.

\begin{figure*}
\begin{center}
\includegraphics[width=16cm]{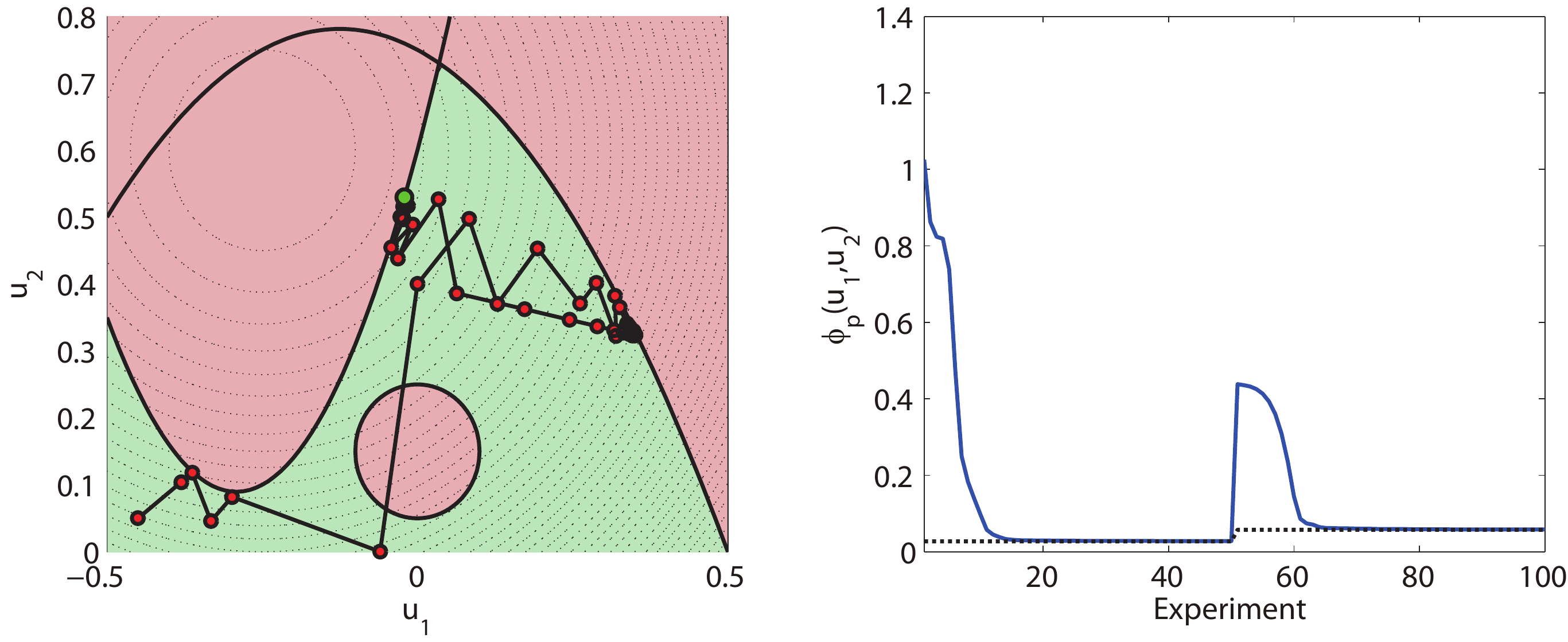}
\caption{Solution generated by the modified SCFO methodology for the switching cost Problem (\ref{eq:exprob})/(\ref{eq:exprobmod}) when the relaxed condition (\ref{eq:SCFO1idegLUccvall}) is not used. The cost contours and minimum given in the left-hand-side plot are only those of the second cost function, while the minimum cost value in the right-hand-side plot switches accordingly after the 50$^{\rm th}$ experiment.}
\label{fig:extwocost1}
\end{center}
\end{figure*}

\begin{figure*}
\begin{center}
\includegraphics[width=16cm]{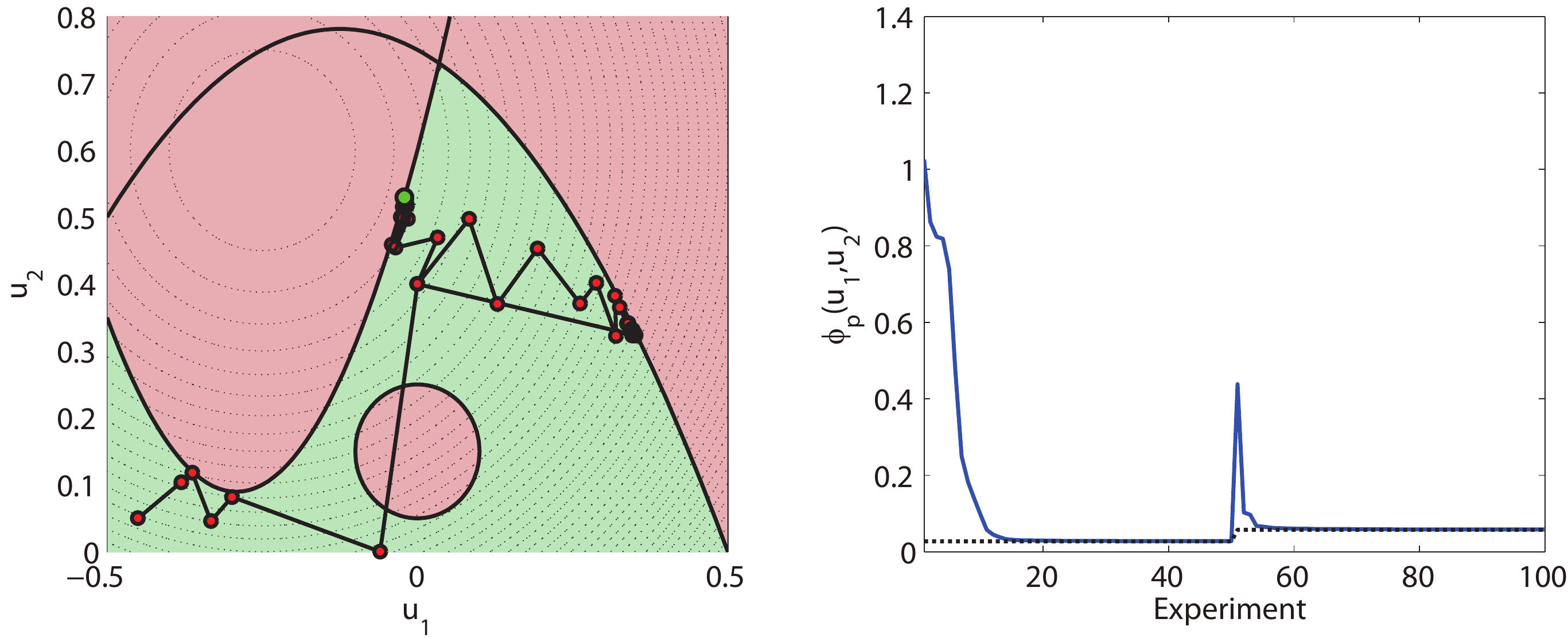}
\caption{Solution generated by the modified SCFO methodology for the switching cost Problem (\ref{eq:exprob})/(\ref{eq:exprobmod}) when the relaxed condition (\ref{eq:SCFO1idegLUccvall}) is used.}
\label{fig:extwocost2}
\end{center}
\end{figure*}

\subsection{Additional Safeguards to Guarantee Cost Decrease}
\label{sec:necess}

The Lipschitz constants on the cost function derivatives, $M_\phi$, are expected to be more difficult to estimate than the Lipschitz constants on the function $\kappa$, as the former correspond to lower and upper bounds on the second, rather than first, derivatives. As such, it may occur that the choice of $M_\phi$ is not sufficiently conservative to guarantee (\ref{eq:SCFO7LU}), which could then lead to the scheme not converging due to a lack of cost decrease from one experiment to the next. We present here one additional condition that may act as a support to (\ref{eq:SCFO7idegLU}) and possibly help enforce convergence when (\ref{eq:SCFO7idegLU}) does not due to a poor choice of Lipschitz constants or due to errors in the cost function gradient.

We start by defining additional Lipschitz constants for the cost function:

\vspace{-2mm}
\begin{equation}\label{eq:lipcost}
\underline \kappa_{\phi,i} \leq \frac{\partial \phi_{p}}{\partial u_i} \Big |_{({\bf u},\tau)} \leq \overline \kappa_{\phi,i}, \;\; \forall ({\bf u},\tau) \in \mathcal{I}_\tau,
\end{equation}

\vspace{-2mm}
\begin{equation}\label{eq:lipcostdeg}
\underline \kappa_{\phi,\tau} \leq \frac{\partial \phi_{p}}{\partial \tau} \Big |_{({\bf u},\tau)} \leq \overline \kappa_{\phi,\tau}, \;\; \forall ({\bf u},\tau) \in \mathcal{I}_\tau,
\end{equation}

\noindent noting that we do not require these to be strict. Since the SCFO aim to decrease the cost function value between $({\bf u}_{k^*},\tau_{k+1})$ and $({\bf u}_{k+1},\tau_{k+1})$, it follows that we should cross out of consideration all ${\bf u}_{k+1}$ where it is guaranteed that $\phi_p ({\bf u}_{k+1},\tau_{k+1}) \geq \phi_p ({\bf u}_{k^*},\tau_{k+1})$. Since the value of $\phi_p ({\bf u}_{k+1},\tau_{k+1})$ cannot be known without testing the experimental point ${\bf u}_{k+1}$, one may proceed by working with its \emph{lower bound}

\vspace{-2mm}
\begin{equation}\label{eq:costlipbound}
\begin{array}{l}
\phi_{p} ({\bf u}_{k+1},\tau_{k+1}) \geq \phi_{p} ({\bf u}_{\bar k},\tau_{\bar k}) \\
\hspace{3mm}\displaystyle  + \underline \kappa_{\phi,\tau} \left( \tau_{k+1} - \tau_{\bar k} \right) + \sum_{i=1}^{n_u} \mathop {\min} \left[ \begin{array}{l} \underline \kappa_{\phi,i} ( u_{k+1,i} - u_{\bar k,i} ), \\ \overline \kappa_{\phi,i} ( u_{k+1,i} - u_{\bar k,i} ) \end{array} \right],
\end{array}
\end{equation}

\noindent which, however, may be tightened further by supposing the cost function to be globally partially \emph{convex} in certain variables, in a manner analogous to what was done for the constraints in Section \ref{sec:concave}. For the sake of completeness, we state here the analogous definition and lemma.

\begin{definition}[Partial convexity of cost function]
\label{def:cvx}
Let the decision variables ${\bf u}$ be partitioned into subvectors ${\bf v}$ and ${\bf z}$, so that $\phi_{p}({\bf u},\tau) = \phi_{p}({\bf v},{\bf z},\tau)$. The function $\phi_{p}$ will be said to be globally \emph{partially convex} in ${\bf z}$ if it is convex everywhere on $\mathcal{I}_\tau$ for ${\bf v}$ and $\tau$ fixed. Likewise, $\phi_{p}$ will be said to be globally partially convex in both ${\bf z}$ and $\tau$ if it is convex everywhere on $\mathcal{I}_\tau$ for ${\bf v}$ fixed.
\end{definition}

\begin{lemma}[Convex tightening of lower Lipschitz bound]
Let $I_{v,\phi}$ denote the set of indices defining the variables ${\bf z}$, with respect to which the experimental cost function $\phi_{p}$ is globally partially convex. Furthermore, let $\eta_{v,\phi}$ denote a Boolean indicator that is equal to 1 if $\phi_{p}$ is globally partially convex in both ${\bf z}$ and $\tau$, and is equal to 0 otherwise. It follows that the following tightening of (\ref{eq:costlipbound}) is valid for any $\bar k = 0,...,k$:

\vspace{-2mm}
\begin{equation}\label{eq:costlipboundccv}
\begin{array}{l}
\phi_{p} ({\bf u}_{k+1},\tau_{k+1}) \geq \phi_{p} ({\bf u}_{\bar k},\tau_{\bar k}) \vspace{1mm} \\
\hspace{24mm} \displaystyle +\eta_{v,\phi} \frac{\partial \phi_{p}}{\partial \tau} \Big |_{({\bf u}_{\bar k},\tau_{\bar k})} ( \tau_{k+1} - \tau_{{\bar k}} ) \vspace{1mm} \\
\hspace{24mm} \displaystyle + (1-\eta_{v,\phi}) \underline \kappa_{\phi,\tau} \left( \tau_{k+1} - \tau_{\bar k} \right) \vspace{1mm}\\
\hspace{24mm}\displaystyle   + \sum_{i \in I_{v,\phi}} \frac{\partial \phi_{p}}{\partial u_i} \Big |_{({\bf u}_{\bar k},\tau_{\bar k})} ( u_{k+1,i} - u_{{\bar k},i} ) \vspace{1mm} \\
\hspace{24mm} + \displaystyle \sum_{i \not \in I_{v,\phi}} \mathop {\min} \left[ \begin{array}{l} \underline \kappa_{\phi,i} ( u_{k+1,i} - u_{{\bar k},i} ), \vspace{1mm}\\ \overline \kappa_{\phi,i} ( u_{k+1,i} - u_{{\bar k},i} ) \end{array} \right].
\end{array}
\end{equation}

\noindent where $i \not \in I_{v,\phi}$ implicitly denotes $\{ 1,...,n_u \} \setminus I_{v,\phi}$.

\end{lemma}

\begin{proof}
The proof is identical to that of Lemma \ref{lem:concave}, save that many things are reversed to account for lower bounding instead of upper bounding and for convexity instead of concavity. We do not consider it necessary to repeat the steps here and leave this to the reader. \qed
\end{proof}

It follows from (\ref{eq:costlipboundccv}) that any choice of ${\bf u}_{k+1}$ for which

\vspace{-4mm}
\begin{equation}\label{eq:costhigh}
\begin{array}{l}
\displaystyle \phi_{p} ({\bf u}_{\bar k},\tau_{\bar k})  +\eta_{v,\phi} \frac{\partial \phi_{p}}{\partial \tau} \Big |_{({\bf u}_{\bar k},\tau_{\bar k})} ( \tau_{k+1} - \tau_{{\bar k}} ) \vspace{1mm} \\
\displaystyle + (1-\eta_{v,\phi}) \underline \kappa_{\phi,\tau} \left( \tau_{k+1} - \tau_{\bar k} \right) \vspace{1mm}\\
\displaystyle   + \sum_{i \in I_{v,\phi}} \frac{\partial \phi_{p}}{\partial u_i} \Big |_{({\bf u}_{\bar k},\tau_{\bar k})} ( u_{k+1,i} - u_{{\bar k},i} ) \vspace{1mm} \\
+ \displaystyle \sum_{i \not \in I_{v,\phi}} \mathop {\min} \left[ \begin{array}{l} \underline \kappa_{\phi,i} ( u_{k+1,i} - u_{{\bar k},i} ), \vspace{1mm}\\ \overline \kappa_{\phi,i} ( u_{k+1,i} - u_{{\bar k},i} ) \end{array} \right] \geq \phi_p ({\bf u}_{k^*},\tau_{k+1})
\end{array}
\end{equation}

\noindent is invalid, as this guarantees that $\phi_p ({\bf u}_{k+1},\tau_{k+1}) \geq \phi_p ({\bf u}_{k^*},\tau_{k+1})$.

Likewise, as the actual value of $\phi_p ({\bf u}_{k^*},\tau_{k+1})$ is not available, we must instead work with its upper bound. Denoting by $I_{c,\phi}$ and $\eta_{c,\phi}$ the cost \emph{concavity} indicators, we may employ a bound analogous to (\ref{eq:feas2LUccv}):

\vspace{-4mm}
\begin{equation}\label{eq:feas2LUccvcost}
\begin{array}{l}
\phi_{p} ({\bf u}_{k^*},\tau_{k+1}) \leq \phi_{p} ({\bf u}_{\tilde k},\tau_{\tilde k}) \vspace{1mm} \\
\hspace{24mm} \displaystyle +\eta_{c,\phi} \frac{\partial \phi_{p}}{\partial \tau} \Big |_{({\bf u}_{\tilde k},\tau_{\tilde k})} ( \tau_{k+1} - \tau_{{\tilde k}} ) \vspace{1mm} \\
\hspace{24mm} \displaystyle + (1-\eta_{c,\phi}) \overline \kappa_{\phi,\tau} \left( \tau_{k+1} - \tau_{\tilde k} \right) \vspace{1mm}\\
\hspace{24mm}\displaystyle   + \sum_{i \in I_{c,\phi}} \frac{\partial \phi_{p}}{\partial u_i} \Big |_{({\bf u}_{\tilde k},\tau_{\tilde k})} ( u_{k^*,i} - u_{{\tilde k},i} ) \vspace{1mm} \\
\hspace{24mm} + \displaystyle \sum_{i \not \in I_{c,\phi}} \mathop {\max} \left[ \begin{array}{l} \underline \kappa_{\phi,i} ( u_{k^*,i} - u_{{\tilde k},i} ), \vspace{1mm}\\ \overline \kappa_{\phi,i} ( u_{k^*,i} - u_{{\tilde k},i} ) \end{array} \right],
\end{array}
\end{equation}

\noindent which is valid for any $\tilde k = 0,...,k$.

This bound may then be used in an implementable version of (\ref{eq:costhigh}):

\vspace{-4mm}
\begin{equation}\label{eq:costhigh2}
\begin{array}{l}
\displaystyle \phi_{p} ({\bf u}_{\bar k},\tau_{\bar k})  +\eta_{v,\phi} \frac{\partial \phi_{p}}{\partial \tau} \Big |_{({\bf u}_{\bar k},\tau_{\bar k})} ( \tau_{k+1} - \tau_{{\bar k}} ) \vspace{1mm} \\
\displaystyle + (1-\eta_{v,\phi}) \underline \kappa_{\phi,\tau} \left( \tau_{k+1} - \tau_{\bar k} \right) \vspace{1mm}\\
\displaystyle   + \sum_{i \in I_{v,\phi}} \frac{\partial \phi_{p}}{\partial u_i} \Big |_{({\bf u}_{\bar k},\tau_{\bar k})} ( u_{k+1,i} - u_{{\bar k},i} ) \vspace{1mm} \\
+ \displaystyle \sum_{i \not \in I_{v,\phi}} \mathop {\min} \left[ \begin{array}{l} \underline \kappa_{\phi,i} ( u_{k+1,i} - u_{{\bar k},i} ), \vspace{1mm}\\ \overline \kappa_{\phi,i} ( u_{k+1,i} - u_{{\bar k},i} ) \end{array} \right] \\
\hspace{22mm}\geq \phi_{p} ({\bf u}_{\tilde k},\tau_{\tilde k}) \vspace{1mm} \\
\hspace{24mm} \displaystyle +\eta_{c,\phi} \frac{\partial \phi_{p}}{\partial \tau} \Big |_{({\bf u}_{\tilde k},\tau_{\tilde k})} ( \tau_{k+1} - \tau_{{\tilde k}} ) \vspace{1mm} \\
\hspace{24mm} \displaystyle + (1-\eta_{c,\phi}) \overline \kappa_{\phi,\tau} \left( \tau_{k+1} - \tau_{\tilde k} \right) \vspace{1mm}\\
\hspace{24mm}\displaystyle   + \sum_{i \in I_{c,\phi}} \frac{\partial \phi_{p}}{\partial u_i} \Big |_{({\bf u}_{\tilde k},\tau_{\tilde k})} ( u_{k^*,i} - u_{{\tilde k},i} ) \vspace{1mm} \\
\hspace{24mm} + \displaystyle \sum_{i \not \in I_{c,\phi}} \mathop {\max} \left[ \begin{array}{l} \underline \kappa_{\phi,i} ( u_{k^*,i} - u_{{\tilde k},i} ), \vspace{1mm}\\ \overline \kappa_{\phi,i} ( u_{k^*,i} - u_{{\tilde k},i} ) \end{array} \right],
\end{array}
\end{equation}

\noindent as any ${\bf u}_{k+1}$ that satisfies (\ref{eq:costhigh2}) must also satisfy (\ref{eq:costhigh}), and thus may be considered as invalid.

As in the previous subsection, we may take $\bar k$ to be any of the past experimental iterations. Since satisfying (\ref{eq:costhigh2}) with respect to any single iteration is sufficient to fathom a given ${\bf u}_{k+1}$ from consideration, one may take the maximum of the left-hand side of (\ref{eq:costhigh2}) over all experimental iterations:

\vspace{-2mm}
\begin{equation}\label{eq:costhighmax0}
\begin{array}{l}
\mathop {\max} \limits_{\bar k = 0,...,k}\left[ \begin{array}{l}
\displaystyle \phi_{p} ({\bf u}_{\bar k},\tau_{\bar k})  +\eta_{v,\phi} \frac{\partial \phi_{p}}{\partial \tau} \Big |_{({\bf u}_{\bar k},\tau_{\bar k})} ( \tau_{k+1} - \tau_{{\bar k}} ) \vspace{1mm} \\
\displaystyle + (1-\eta_{v,\phi}) \underline \kappa_{\phi,\tau} \left( \tau_{k+1} - \tau_{\bar k} \right) \vspace{1mm}\\
\displaystyle   + \sum_{i \in I_{v,\phi}} \frac{\partial \phi_{p}}{\partial u_i} \Big |_{({\bf u}_{\bar k},\tau_{\bar k})} ( u_{k+1,i} - u_{{\bar k},i} ) \vspace{1mm} \\
+ \displaystyle \sum_{i \not \in I_{v,\phi}} \mathop {\min} \left[ \begin{array}{l} \underline \kappa_{\phi,i} ( u_{k+1,i} - u_{{\bar k},i} ), \vspace{1mm}\\ \overline \kappa_{\phi,i} ( u_{k+1,i} - u_{{\bar k},i} ) \end{array} \right] 
\end{array} \right] \vspace{1mm}\\
\hspace{22mm}\geq \phi_{p} ({\bf u}_{\tilde k},\tau_{\tilde k}) \vspace{1mm} \\
\hspace{24mm} \displaystyle +\eta_{c,\phi} \frac{\partial \phi_{p}}{\partial \tau} \Big |_{({\bf u}_{\tilde k},\tau_{\tilde k})} ( \tau_{k+1} - \tau_{{\tilde k}} ) \vspace{1mm} \\
\hspace{24mm} \displaystyle + (1-\eta_{c,\phi}) \overline \kappa_{\phi,\tau} \left( \tau_{k+1} - \tau_{\tilde k} \right) \vspace{1mm}\\
\hspace{24mm}\displaystyle   + \sum_{i \in I_{c,\phi}} \frac{\partial \phi_{p}}{\partial u_i} \Big |_{({\bf u}_{\tilde k},\tau_{\tilde k})} ( u_{k^*,i} - u_{{\tilde k},i} ) \vspace{1mm} \\
\hspace{24mm} + \displaystyle \sum_{i \not \in I_{c,\phi}} \mathop {\max} \left[ \begin{array}{l} \underline \kappa_{\phi,i} ( u_{k^*,i} - u_{{\tilde k},i} ), \vspace{1mm}\\ \overline \kappa_{\phi,i} ( u_{k^*,i} - u_{{\tilde k},i} ) \end{array} \right].
\end{array}
\end{equation}

The same may be done for the right-hand side, save that one chooses the $\tilde k$ that \emph{minimizes} the upper bound:

\vspace{-2mm}
\begin{equation}\label{eq:costhighmax1}
\begin{array}{l}
\mathop {\max} \limits_{\bar k = 0,...,k}\left[ \begin{array}{l}
\displaystyle \phi_{p} ({\bf u}_{\bar k},\tau_{\bar k})  + \eta_{v,\phi} \frac{\partial \phi_{p}}{\partial \tau} \Big |_{({\bf u}_{\bar k},\tau_{\bar k})} ( \tau_{k+1} - \tau_{{\bar k}} ) \vspace{1mm} \\
\displaystyle + (1-\eta_{v,\phi}) \underline \kappa_{\phi,\tau} \left( \tau_{k+1} - \tau_{\bar k} \right) \vspace{1mm}\\
\displaystyle   + \sum_{i \in I_{v,\phi}} \frac{\partial \phi_{p}}{\partial u_i} \Big |_{({\bf u}_{\bar k},\tau_{\bar k})} ( u_{k+1,i} - u_{{\bar k},i} ) \vspace{1mm} \\
+ \displaystyle \sum_{i \not \in I_{v,\phi}} \mathop {\min} \left[ \begin{array}{l} \underline \kappa_{\phi,i} ( u_{k+1,i} - u_{{\bar k},i} ), \vspace{1mm}\\ \overline \kappa_{\phi,i} ( u_{k+1,i} - u_{{\bar k},i} ) \end{array} \right] 
\end{array} \right] \vspace{1mm}\\
\hspace{0mm}\displaystyle \geq \mathop {\min}_{\tilde k = 0,...,k} \left[ \begin{array}{l} \displaystyle \phi_{p} ({\bf u}_{\tilde k},\tau_{\tilde k})  +\eta_{c,\phi} \frac{\partial \phi_{p}}{\partial \tau} \Big |_{({\bf u}_{\tilde k},\tau_{\tilde k})} ( \tau_{k+1} - \tau_{{\tilde k}} ) \vspace{1mm} \\
\hspace{0mm} \displaystyle + (1-\eta_{c,\phi}) \overline \kappa_{\phi,\tau} \left( \tau_{k+1} - \tau_{\tilde k} \right) \vspace{1mm}\\
\hspace{0mm}\displaystyle   + \sum_{i \in I_{c,\phi}} \frac{\partial \phi_{p}}{\partial u_i} \Big |_{({\bf u}_{\tilde k},\tau_{\tilde k})} ( u_{k^*,i} - u_{{\tilde k},i} ) \vspace{1mm} \\
\hspace{0mm} + \displaystyle \sum_{i \not \in I_{c,\phi}} \mathop {\max} \left[ \begin{array}{l} \underline \kappa_{\phi,i} ( u_{k^*,i} - u_{{\tilde k},i} ), \vspace{1mm}\\ \overline \kappa_{\phi,i} ( u_{k^*,i} - u_{{\tilde k},i} ) \end{array} \right] \end{array} \right].
\end{array}
\end{equation}

Since $\phi_p ({\bf u}_{k+1},\tau_{k+1}) \geq \phi_p ({\bf u}_{k^*},\tau_{k+1})$ is guaranteed for any ${\bf u}_{k+1}$ that satisfies (\ref{eq:costhighmax1}), we obtain a necessary condition for cost decrease by reversing the inequality sign:

\vspace{-2mm}
\begin{equation}\label{eq:costhighmax}
\begin{array}{l}
\mathop {\max} \limits_{\bar k = 0,...,k}\left[ \begin{array}{l}
\displaystyle \phi_{p} ({\bf u}_{\bar k},\tau_{\bar k})  +\eta_{v,\phi} \frac{\partial \phi_{p}}{\partial \tau} \Big |_{({\bf u}_{\bar k},\tau_{\bar k})} ( \tau_{k+1} - \tau_{{\bar k}} ) \vspace{1mm} \\
\displaystyle + (1-\eta_{v,\phi}) \underline \kappa_{\phi,\tau} \left( \tau_{k+1} - \tau_{\bar k} \right) \vspace{1mm}\\
\displaystyle   + \sum_{i \in I_{v,\phi}} \frac{\partial \phi_{p}}{\partial u_i} \Big |_{({\bf u}_{\bar k},\tau_{\bar k})} ( u_{k+1,i} - u_{{\bar k},i} ) \vspace{1mm} \\
+ \displaystyle \sum_{i \not \in I_{v,\phi}} \mathop {\min} \left[ \begin{array}{l} \underline \kappa_{\phi,i} ( u_{k+1,i} - u_{{\bar k},i} ), \vspace{1mm}\\ \overline \kappa_{\phi,i} ( u_{k+1,i} - u_{{\bar k},i} ) \end{array} \right] 
\end{array} \right] \vspace{1mm}\\
\hspace{0mm}\displaystyle \leq \mathop {\min}_{\tilde k = 0,...,k} \left[ \begin{array}{l} \displaystyle \phi_{p} ({\bf u}_{\tilde k},\tau_{\tilde k})  +\eta_{c,\phi} \frac{\partial \phi_{p}}{\partial \tau} \Big |_{({\bf u}_{\tilde k},\tau_{\tilde k})} ( \tau_{k+1} - \tau_{{\tilde k}} ) \vspace{1mm} \\
\hspace{0mm} \displaystyle + (1-\eta_{c,\phi}) \overline \kappa_{\phi,\tau} \left( \tau_{k+1} - \tau_{\tilde k} \right) \vspace{1mm}\\
\hspace{0mm}\displaystyle   + \sum_{i \in I_{c,\phi}} \frac{\partial \phi_{p}}{\partial u_i} \Big |_{({\bf u}_{\tilde k},\tau_{\tilde k})} ( u_{k^*,i} - u_{{\tilde k},i} ) \vspace{1mm} \\
\hspace{0mm} + \displaystyle \sum_{i \not \in I_{c,\phi}} \mathop {\max} \left[ \begin{array}{l} \underline \kappa_{\phi,i} ( u_{k^*,i} - u_{{\tilde k},i} ), \vspace{1mm}\\ \overline \kappa_{\phi,i} ( u_{k^*,i} - u_{{\tilde k},i} ) \end{array} \right] \end{array} \right].
\end{array}
\end{equation}

Substituting the filtering law into (\ref{eq:costhighmax}) then leads to the following condition on $K_k$ in the project-and-filter approach:

\vspace{-2mm}
\begin{equation}\label{eq:costhighmaxPF}
\begin{array}{l}
\mathop {\max} \limits_{\bar k = 0,...,k}\left[ \begin{array}{l}
\displaystyle \phi_{p} ({\bf u}_{\bar k},\tau_{\bar k})  +\eta_{v,\phi} \frac{\partial \phi_{p}}{\partial \tau} \Big |_{({\bf u}_{\bar k},\tau_{\bar k})} ( \tau_{k+1} - \tau_{{\bar k}} ) \vspace{1mm} \\
\displaystyle + (1-\eta_{v,\phi}) \underline \kappa_{\phi,\tau} \left( \tau_{k+1} - \tau_{\bar k} \right) \vspace{1mm}\\
\displaystyle   + \sum_{i \in I_{v,\phi}} \frac{\partial \phi_{p}}{\partial u_i} \Big |_{({\bf u}_{\bar k},\tau_{\bar k})} ( u_{k^*,i} + \\ \hspace{13mm}  K_k(\bar u_{k+1,i}^* - u_{k^*,i}) - u_{{\bar k},i} ) \vspace{1mm} \\
+ \displaystyle \sum_{i \not \in I_{v,\phi}} \mathop {\min} \left[ \begin{array}{l} \underline \kappa_{\phi,i} ( u_{k^*,i} +\\ \hspace{3mm} K_k(\bar u_{k+1,i}^* - u_{k^*,i}) - u_{{\bar k},i} ), \vspace{1mm}\\ \overline \kappa_{\phi,i} ( u_{k^*,i} +\\ \hspace{3mm} K_k(\bar u_{k+1,i}^* - u_{k^*,i}) - u_{{\bar k},i} ) \end{array} \right] 
\end{array} \right] \vspace{2mm}\\
\hspace{0mm}\displaystyle \leq \mathop {\min}_{\tilde k = 0,...,k} \left[ \begin{array}{l} \displaystyle \phi_{p} ({\bf u}_{\tilde k},\tau_{\tilde k})  +\eta_{c,\phi} \frac{\partial \phi_{p}}{\partial \tau} \Big |_{({\bf u}_{\tilde k},\tau_{\tilde k})} ( \tau_{k+1} - \tau_{{\tilde k}} ) \vspace{1mm} \\
\hspace{0mm} \displaystyle + (1-\eta_{c,\phi}) \overline \kappa_{\phi,\tau} \left( \tau_{k+1} - \tau_{\tilde k} \right) \vspace{1mm}\\
\hspace{0mm}\displaystyle   + \sum_{i \in I_{c,\phi}} \frac{\partial \phi_{p}}{\partial u_i} \Big |_{({\bf u}_{\tilde k},\tau_{\tilde k})} ( u_{k^*,i} - u_{{\tilde k},i} ) \vspace{1mm} \\
\hspace{0mm} + \displaystyle \sum_{i \not \in I_{c,\phi}} \mathop {\max} \left[ \begin{array}{l} \underline \kappa_{\phi,i} ( u_{k^*,i} - u_{{\tilde k},i} ), \vspace{1mm}\\ \overline \kappa_{\phi,i} ( u_{k^*,i} - u_{{\tilde k},i} ) \end{array} \right] \end{array} \right].
\end{array}
\end{equation}

Since we are now supposing the very practical possibility of the scheme failing to satisfy the sufficient conditions for cost decrease -- due to, for example, an improper choice of Lipschitz constants -- it makes sense to modify the way in which we choose the reference point as well, since the last experimental iterate for which feasibility may be guaranteed is no longer guaranteed to be the best as the cost is not guaranteed to decrease from experiment to experiment. What we propose is to take as a reference the latest experimental iterate for which feasibility may be guaranteed and which \emph{cannot be proven to have a cost function value greater to that at another provably feasible point}.

The issue that arises here is that it is not sufficient to simply compare the cost function values for the different experiments, since $\phi_p ({\bf u}_{k-1},\tau_{k-1}) < \phi_p ({\bf u}_{k},\tau_{k})$ need not imply $\phi_p ({\bf u}_{k-1},\tau_{k}) < \phi_p ({\bf u}_{k},\tau_{k})$ due to degradation effects. As such, one could not say that ${\bf u}_{k-1}$ is a better reference point than ${\bf u}_{k}$ by simply comparing their measured values at different time instants. A robust approach would be to use an upper bound on $\phi_p ({\bf u}_{k-1},\tau_{k})$ in the comparison instead. As one example, consider the bound

\vspace{-4mm}
\begin{equation}\label{eq:costupperbound}
\phi_p ({\bf u}_{k-1},\tau_{k}) \leq \phi_p ({\bf u}_{k-1},\tau_{k-1}) + \overline \kappa_{\phi,\tau} (\tau_{k}-\tau_{k-1}). 
\end{equation}

\noindent Clearly, if the right-hand side is inferior to $\phi_p ({\bf u}_{k},\tau_{k})$, then this implies $\phi_p ({\bf u}_{k-1},\tau_{k}) < \phi_p ({\bf u}_{k},\tau_{k})$ and we can safely say that ${\bf u}_{k-1}$ is a better reference than ${\bf u}_k$ because it has a lower cost function value at time instant $\tau_k$.

The above example is special in that we do not need a lower or upper bound on $\phi_p ({\bf u}_{k},\tau_{k})$, as this value is current and does not require accounting for degradation. In general, however, we may want to be able to compare the \emph{current} cost function values for any two past experimental iterates. For example, in choosing a more appropriate reference between ${\bf u}_{k-1}$ and ${\bf u}_{k-2}$, we will give our preference to ${\bf u}_{k-1}$, since it comes later, provided that it cannot be proven to have a cost function value greater than that at ${\bf u}_{k-2}$, i.e., provided that we cannot prove $\phi_p ({\bf u}_{k-1},\tau_{k}) > \phi_p ({\bf u}_{k-2},\tau_{k})$.

To prove this would entail using \emph{two} bounds and showing that the lower bound on $\phi_p ({\bf u}_{k-1},\tau_{k})$ is strictly greater than the upper bound on $\phi_p ({\bf u}_{k-2},\tau_{k})$. For a general $\bar k$, such bounds are obtained as

\vspace{-4mm}
\begin{equation}\label{eq:costupperbound2}
\begin{array}{l}
\displaystyle \phi_p ({\bf u}_{\bar k},\tau_{k}) \leq \phi_p ({\bf u}_{\bar k},\tau_{\bar k}) + \eta_{c,\phi} \frac{\partial \phi_{p}}{\partial \tau} \Big |_{({\bf u}_{\bar k},\tau_{\bar k})} ( \tau_{k} - \tau_{{\bar k}} ) \vspace{1mm} \\
\hspace{35mm}+ (1- \eta_{c,\phi}) \overline \kappa_{\phi,\tau} \left( \tau_{k} - \tau_{\bar k} \right),
\end{array}
\end{equation}

\vspace{-4mm}
\begin{equation}\label{eq:costlowerbound2}
\begin{array}{l}
\displaystyle \phi_p ({\bf u}_{\bar k},\tau_{k}) \geq \phi_p ({\bf u}_{\bar k},\tau_{\bar k}) + \eta_{v,\phi} \frac{\partial \phi_{p}}{\partial \tau} \Big |_{({\bf u}_{\bar k},\tau_{\bar k})} ( \tau_{k} - \tau_{{\bar k}} ) \vspace{1mm} \\
\hspace{35mm}+ (1-  \eta_{v,\phi}) \underline \kappa_{\phi,\tau} \left( \tau_{k} - \tau_{\bar k} \right).
\end{array}
\end{equation}

Using these results, we proceed to replace (\ref{eq:kstarLUccv}) by

\vspace{-4mm}
\begin{equation}\label{eq:kstarLUccvcost}
\begin{array}{rl}
k^* := \;\;\;\;\;\;\;\;\;\;\;\;\;\;& \vspace{1mm} \\
{\rm arg} \mathop {\rm maximize}\limits_{\bar k \in [0,k]} & \bar k \vspace{1mm} \\
{\rm{subject}}\;{\rm{to}} & g_{p,j} ({\bf u}_{\bar k},\tau_{\bar k}) \vspace{1mm} \\
& \displaystyle + \eta_{c,j} \frac{\partial g_{p,j}}{\partial \tau} \Big |_{({\bf u}_{\bar k},\tau_{\bar k})} ( \tau_{k+1} - \tau_{{\bar k}} ) \vspace{1mm} \\
&  + (1- \eta_{c,j}) \overline \kappa_{p,j\tau} \left( \tau_{k+1} - \tau_{\bar k} \right) \leq 0, \vspace{1mm} \\
& \forall j = 1,...,n_{g_p}
\end{array}
\end{equation}

$$
\begin{array}{rl}
& \displaystyle \phi_p ({\bf u}_{\bar k},\tau_{\bar k}) + \eta_{v,\phi} \frac{\partial \phi_{p}}{\partial \tau} \Big |_{({\bf u}_{\bar k},\tau_{\bar k})} ( \tau_{k} - \tau_{{\bar k}} ) \vspace{1mm} \\
&+ (1-  \eta_{v,\phi}) \underline \kappa_{\phi,\tau} \left( \tau_{k} - \tau_{\bar k} \right) \leq \vspace{1mm} \\
& \mathop {\min} \limits_{\tilde k \in {\bf k}_f} \left[ \begin{array}{l}  \phi_p ({\bf u}_{\tilde k},\tau_{\tilde k}) + \vspace{1mm} \\
\displaystyle \eta_{c,\phi} \frac{\partial \phi_{p}}{\partial \tau} \Big |_{({\bf u}_{\tilde k},\tau_{\tilde k})} ( \tau_{k} - \tau_{{\tilde k}} ) \vspace{1mm} \\
+ (1-  \eta_{c,\phi}) \overline \kappa_{\phi,\tau} \left( \tau_{k} - \tau_{\tilde k} \right) \end{array} \right],
\end{array}
$$

\begin{equation}\label{eq:kfeas}
{\bf k}_f = \left\{ \bar k : \begin{array}{l} g_{p,j} ({\bf u}_{\bar k},\tau_{\bar k}) \vspace{1mm} \\
 + \displaystyle \eta_{c,j} \frac{\partial g_{p,j}}{\partial \tau} \Big |_{({\bf u}_{\bar k},\tau_{\bar k})} ( \tau_{k+1} - \tau_{{\bar k}} ) \vspace{1mm} \\ + (1-  \eta_{c,j}) \overline \kappa_{p,j\tau} \left( \tau_{k+1} - \tau_{\bar k} \right)  \leq 0, \vspace{1mm} \\ \forall j = 1,...,n_{g_p} \end{array} \right\},
\end{equation}

\noindent which chooses the most recent point while guaranteeing that feasibility may be maintained for this point and that the corresponding cost function value at time instant $\tau_k$ will not be greater than the value for any other points for which feasibility is guaranteed. In the case where this problem has no solution, a possible implementation, apart from simply terminating the procedure, is to fall back on (\ref{eq:kstar2LUccv}) in hopes of finding a feasible point.

To illustrate the potential usefulness of (\ref{eq:costhighmaxPF}), let us consider the following modified version of (\ref{eq:exprob}):

\vspace{-2mm}
\begin{equation}\label{eq:exprobunc}
\begin{array}{rl}
\mathop {{\rm{minimize}}}\limits_{u_1,u_2} & \phi_{p}({\bf{u}}) := (u_1-0.2)^2 + (u_2-0.4)^2 \\
{\rm{subject}}\hspace{1mm}{\rm{to}} & g_{p,1}({\bf{u}}) := -6u^2_1 - 3.5u_1 + u_2 -0.6 \le 0 \vspace{1mm} \\
 & g_{p,2}({\bf{u}}) := 2u^2_1 + 0.5u_1 + u_2 -0.75 \le 0 \vspace{1mm} \\
 & g_{1}({\bf{u}}) := -u^2_1 - (u_2-0.15)^2 +0.01 \le 0 \vspace{1mm} \\
& -0.5 \leq u_1 \leq 0.5 \\
& 0 \leq u_2 \leq 0.8,
\end{array}
\end{equation}

\noindent where we again, for simplicity, ignore degradation effects. Because this problem has an unconstrained optimum, it follows that the Lipschitz constants on the cost derivatives will play a more crucial role -- when an optimum is constrained, errors in the sufficient condition (\ref{eq:SCFO7idegLU}) for cost decrease are often compensated for by the conservatism introduced by the feasibility-guaranteeing conditions, which lower $K_k$ sufficiently even when the cost decrease condition does not.

We will suppose that an improper choice of $M_\phi$ has been made, with

\vspace{-2mm}
\begin{equation}\label{eq:exlipunc}
\begin{array}{ll}
\underline M_{\phi,11} = 0.1, & \underline M_{\phi,12} = -2, \\
\overline M_{\phi,11} = 0.5, & \overline M_{\phi,12} = -1.5,\\
\underline M_{\phi,21} = -2, & \underline M_{\phi,22} = 0.1, \\ 
\overline M_{\phi,21} = -1.5, & \overline M_{\phi,22} = 0.5,
\end{array}
\end{equation}

\noindent which clearly does not satisfy (\ref{eq:lipcondeg2LU}) for the $\phi_p$ given in (\ref{eq:exprobunc}). Employing (\ref{eq:kstarLUccvcost}) to select $k^*$ but not employing (\ref{eq:costhighmaxPF}), we see that the scheme fails to converge to the optimum (Fig. \ref{fig:exbadlip}). What essentially happens here is that the scheme finds a good reference at approximately ${\bf u} = (0.15,0.4)$ and then proceeds to take overly big steps in the positive $u_1$ direction. As no sort of feedback is applied to tell the algorithm that smaller steps are needed, it simply continues to take the same large steps due to the improper choice of $M_\phi$ and then going back to the best point as its reference.

\begin{figure*}
\begin{center}
\includegraphics[width=16cm]{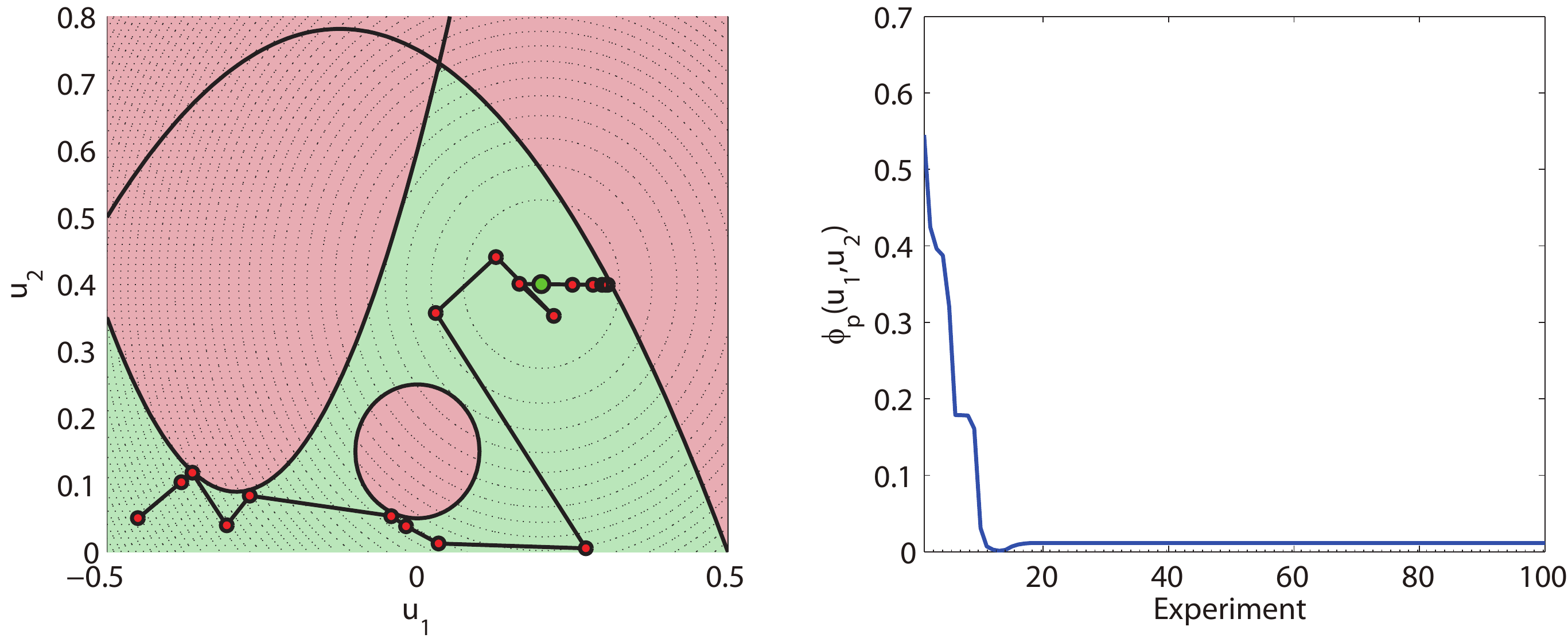}
\caption{Solution generated by the modified SCFO methodology for the unconstrained-optimum problem (\ref{eq:exprobunc}) when the necessary condition (\ref{eq:costhighmaxPF}) is not enforced and the Lipschitz constants $M_\phi$ are improperly chosen.}
\label{fig:exbadlip}
\end{center}
\end{figure*}

By contrast, enforcing (\ref{eq:costhighmaxPF}) with the Lipschitz constants

\vspace{-2mm}
\begin{equation}\label{eq:exlip2unc}
\begin{array}{ll}
\underline \kappa_{\phi,1} = -1.4, & \underline \kappa_{\phi,2} = -0.8, \\
\overline \kappa_{\phi,1} = 0.6, & \overline \kappa_{\phi,2} = 0.8
\end{array}
\end{equation}

\noindent and the convexity relaxations defined by $I_{v,\phi} = \{ 1,2 \}$ actually allows convergence to the optimum despite the poor choice of $M_\phi$ (Fig. \ref{fig:exbadlip2}). The reason is simple: although the algorithm plans to take an overly large step, it does not because doing so would lead do an experimental iterate for which the cost is guaranteed to be higher than that at the reference. As such, the steps are gradually shortened as the experimental iterates are driven to the optimum.

\begin{figure*}
\begin{center}
\includegraphics[width=16cm]{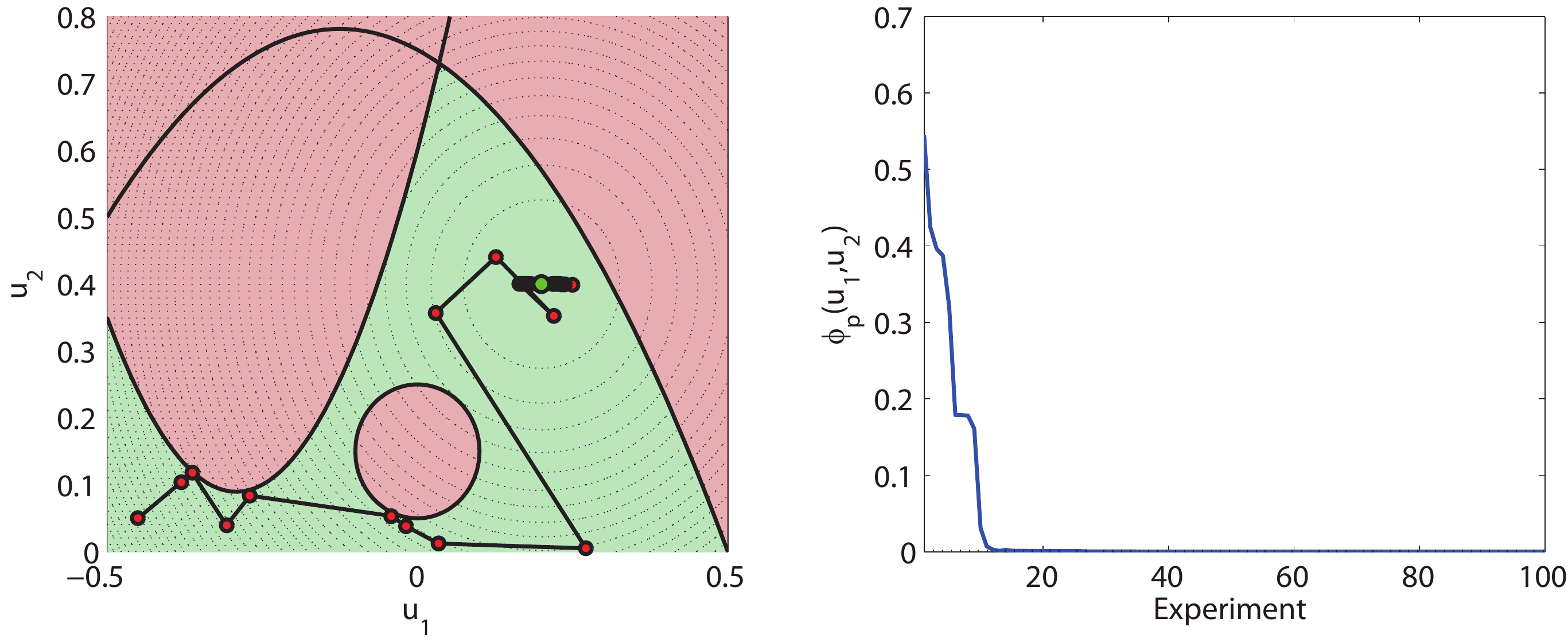}
\caption{Solution generated by the modified SCFO methodology for the unconstrained-optimum problem (\ref{eq:exprobunc}) when the necessary condition (\ref{eq:costhighmaxPF}) is enforced.}
\label{fig:exbadlip2}
\end{center}
\end{figure*}

We finish by reflecting on the implementability of this technique. Clearly, the price to pay for employing it comes via the need to provide \emph{even more} Lipschitz constants -- those for the cost function -- which may be seen as undesirable. However, as the constants $\kappa_\phi$ are bounds on the \emph{first-order} derivatives, they are expected to be much easier to obtain. Additionally, choosing very conservative values for $\kappa_\phi$ would not be detrimental to the performance of the scheme and would only reduce the usefulness of (\ref{eq:costhighmaxPF}). As before, we give the standard warnings associated with concavity and convexity relaxations -- i.e., that one should only use these when they are well justified and the concave/convex relationships are well documented. While it might even be possible to prove that the necessary condition for cost decrease (\ref{eq:costhighmaxPF}) becomes \emph{sufficient} asymptotically as more and more regions are explored and fathomed from consideration until only the cost descent regions are left, we do not attempt to do so here and only propose (\ref{eq:costhighmaxPF}) as a useful supplement to (\ref{eq:SCFO7idegLU}).

\subsection{Methods of Estimating Lipschitz Constants}

Here, we discuss the possible means of estimating the Lipschitz constants for the experimental functions ($\kappa$) and the cost function derivatives ($M_\phi$). Three approaches are proposed and may be classified as being based on:

\begin{itemize}
\item physical laws,
\item conservative model-based estimates,
\item consistency checks.
\end{itemize}

\noindent We note that these three are not mutually exclusive and may often be complementary.

\subsubsection{The Physical Meaning of a Lipschitz Constant}

As already illustrated with some examples in Section \ref{sec:twolip}, the sign of the Lipschitz constants $\kappa$ may be known with certainty in some applications due to underlying physical principles or simply due to engineering experience. As each Lipschitz constant is essentially a lower or upper bound on a given derivative, or sensitivity, of an experimental function, such knowledge allows either the lower or upper constant to be set to 0.

In other cases one may be able to limit not only the sign but also the magnitude --  consider, for example, a heat exchanger where a hot element comes into contact with a colder one, thereby heating it. In many such set-ups, one may raise or lower the temperature of the heating element to raise or lower the temperature of the heated element, but usually by no more than the change in temperature in the heating element -- i.e., one cannot raise the temperature of the heating element by 5\degree C and expect the temperature of the heated element to go up by \emph{more than} 5\degree C. If we now imagine the temperature of the heated element to be an experimental function and the temperature of the heating element to be a decision variable, then it is clear that the upper Lipschitz constant that relates the two cannot exceed 1. 

Finally, if a certain decision variable is known to have a negligible effect on the function, its corresponding Lipschitz constants may be set to very small negative and positive values (or to 0, if strict inequality is not needed). 

For the cost function derivative Lipschitz constants $M_\phi$ the physical significance may be less obvious, and there will likely be less \emph{a priori} knowledge to exploit. As already mentioned, the curvature of the cost function with respect to certain decision variables may be known for certain applications where the function's convexity or concavity is well documented. If this is so, then the signs of the different $M_\phi$ may be restricted. Furthermore, if the function is known to be linear in certain variables, then all of the second derivatives that include this variable are 0, and so the corresponding Lipschitz constants may be set to very small negative and positive values.

\subsubsection{Model-Based Estimation}

In many engineering applications, one will often have a model of the process or system being optimized, which, though not a perfect description, will nevertheless capture the basic trends and behaviors. Such models will often be parameterized by some finite set of parameters ${\boldsymbol \theta}$ and will approximate the true experimental functions as

\vspace{-2mm}
\begin{equation}\label{eq:parmodel}
\phi_{\hat p} ({\bf u},\tau, \theta) \approx \phi_{p} ({\bf u},\tau),
\end{equation}

\noindent with $\hat p$ used to denote the model (we use the cost function as the example).

As the model is a known numerical function, one may compute its Lipschitz constants directly. In the general case, we may suppose the parameters $\theta$ to belong to an uncertain set $\Theta$, and may compute the model constants as

\vspace{-2mm}
\begin{equation}\label{eq:parmodellip}
\begin{array}{l}
\displaystyle \underline \kappa_{\hat \phi,i} = \mathop {\min} \limits_{\footnotesize \begin{array}{c} {\bf u},\tau \in \mathcal{I}_\tau \\ \theta \in \Theta \end{array}} \; \frac{\partial \phi_{\hat p}}{\partial u_i} \Big |_{({\bf u},\tau,\theta)}, \vspace{1mm} \\
\displaystyle \overline \kappa_{\hat \phi,i} = \mathop {\max} \limits_{\footnotesize \begin{array}{c} {\bf u},\tau \in \mathcal{I}_\tau \\ \theta \in \Theta \end{array}} \; \frac{\partial \phi_{\hat p}}{\partial u_i} \Big |_{({\bf u},\tau,\theta)},
\end{array}
\end{equation}

\vspace{-2mm}
\begin{equation}\label{eq:parmodellip2}
\begin{array}{l}
\displaystyle \underline \kappa_{\hat \phi,\tau} = \hspace{-.3mm} \mathop {\min} \limits_{\footnotesize \begin{array}{c} {\bf u},\tau \in \mathcal{I}_\tau \\ \theta \in \Theta \end{array}} \frac{\partial \phi_{\hat p}}{\partial \tau} \Big |_{({\bf u},\tau,\theta)}, \vspace{1mm} \\
\displaystyle  \overline \kappa_{\hat \phi,\tau} = \mathop {\max} \limits_{\footnotesize \begin{array}{c} {\bf u},\tau \in \mathcal{I}_\tau \\ \theta \in \Theta \end{array}}  \frac{\partial \phi_{\hat p}}{\partial \tau} \Big |_{({\bf u},\tau,\theta)},
\end{array}
\end{equation}

\noindent with the $M_{\hat \phi}$ values computed in an analogous manner. Note that such computations may not be easy -- if the model used is highly nonlinear/nonconvex in ${\bf u}$ and $\tau$, robust global optimization techniques may be required. However, this is likely to be a secondary concern as all the relevant computations may be done offline prior to experimental optimization.

To account for potential model inaccuracies, one may add an additional layer of safety to the estimates by setting the actual Lipschitz constant estimates as

\vspace{-2mm}
\begin{equation}\label{eq:conservest}
\begin{array}{l}
\underline \kappa_{\phi,i} := \underline \kappa_{\hat \phi,i} - 0.5(\overline \kappa_{\hat \phi,i} - \underline \kappa_{\hat \phi,i} ), \vspace{1mm} \\
\overline \kappa_{\phi,i} := \overline \kappa_{\hat \phi,i} + 0.5(\overline \kappa_{\hat \phi,i} - \underline \kappa_{\hat \phi,i} ), \vspace{1mm} \\
\underline \kappa_{\phi,\tau} := \underline \kappa_{\hat \phi,\tau} - 0.5(\overline \kappa_{\hat \phi,\tau} - \underline \kappa_{\hat \phi,\tau} ), \vspace{1mm} \\
\overline \kappa_{\phi,\tau} := \overline \kappa_{\hat \phi,\tau} + 0.5(\overline \kappa_{\hat \phi,\tau} - \underline \kappa_{\hat \phi,\tau} ),
\end{array}
\end{equation}

\noindent with analogous settings for $M_{\phi}$.

Note, as well, that the model need not be available \emph{a priori} and could also be constructed from a set of initial experiments, as is standardly done in the response-surface methodology \cite{Myers2009}. Such models are typically linear or quadratic and are easier to work with, although they may not capture all of the involved characteristics of the process or system in consideration. This data-driven approach is particularly well-suited for experimental functions that are very difficult to model and for which the underlying physics are not well documented or well explained.

\subsubsection{Consistency Checks}
\label{sec:consist}

We will consider a given choice of Lipschitz constants \emph{consistent} if they satisfy the Lipschitz bounds for all pairs of available experimental iterates. In particular, letting $k_1$ and $k_2$ denote any two iterations from the interval $[0,k]$, we know that for a consistent choice of Lipschitz constants the following relations must hold:

\vspace{-2mm}
\begin{equation}\label{eq:lipcheck1U}
\begin{array}{l}
g_{p,j} ({\bf u}_{k_2},\tau_{k_2}) \leq g_{p,j} ({\bf u}_{k_1},\tau_{k_1}) \\
\hspace{20mm} \displaystyle + \mathop {\max} \left[ \begin{array}{l} \underline \kappa_{p,j\tau} \left( \tau_{k_2} - \tau_{k_1} \right) \\ 
\overline \kappa_{p,j\tau} \left( \tau_{k_2} - \tau_{k_1} \right) \end{array} \right]  \\
\hspace{20mm}\displaystyle + \sum_{i=1}^{n_u} \mathop {\max} \left[ \begin{array}{l} \underline \kappa_{p,ji} ( u_{k_2,i} - u_{k_1,i} ), \\ \overline \kappa_{p,ji} ( u_{k_2,i} - u_{{k_1},i} ) \end{array} \right],
\end{array}
\end{equation}

\vspace{-2mm}
\begin{equation}\label{eq:lipcheck1L}
\begin{array}{l}
g_{p,j} ({\bf u}_{k_2},\tau_{k_2}) \geq g_{p,j} ({\bf u}_{k_1},\tau_{k_1}) \\
\hspace{20mm} \displaystyle + \mathop {\min} \left[ \begin{array}{l} \underline \kappa_{p,j\tau} \left( \tau_{k_2} - \tau_{k_1} \right) \\ 
\overline \kappa_{p,j\tau} \left( \tau_{k_2} - \tau_{k_1} \right) \end{array} \right]  \\
\hspace{20mm}\displaystyle + \sum_{i=1}^{n_u} \mathop {\min} \left[ \begin{array}{l} \underline \kappa_{p,ji} ( u_{k_2,i} - u_{k_1,i} ), \\ \overline \kappa_{p,ji} ( u_{k_2,i} - u_{{k_1},i} ) \end{array} \right],
\end{array}
\end{equation}

\vspace{-2mm}
\begin{equation}\label{eq:lipcheck1costU}
\begin{array}{l}
\phi_{p} ({\bf u}_{k_2},\tau_{k_2}) \leq \phi_{p} ({\bf u}_{k_1},\tau_{k_1}) \\
\hspace{20mm} \displaystyle + \mathop {\max} \left[ \begin{array}{l} \underline \kappa_{\phi,\tau} \left( \tau_{k_2} - \tau_{k_1} \right) \\ 
\overline \kappa_{\phi,\tau} \left( \tau_{k_2} - \tau_{k_1} \right) \end{array} \right]  \\
\hspace{20mm}\displaystyle + \sum_{i=1}^{n_u} \mathop {\max} \left[ \begin{array}{l} \underline \kappa_{\phi,i} ( u_{k_2,i} - u_{k_1,i} ), \\ \overline \kappa_{\phi,i} ( u_{k_2,i} - u_{{k_1},i} ) \end{array} \right],
\end{array}
\end{equation}

\vspace{-2mm}
\begin{equation}\label{eq:lipcheck1costL}
\begin{array}{l}
\phi_{p} ({\bf u}_{k_2},\tau_{k_2}) \geq \phi_{p} ({\bf u}_{k_1},\tau_{k_1}) \\
\hspace{20mm} \displaystyle + \mathop {\min} \left[ \begin{array}{l} \underline \kappa_{\phi,\tau} \left( \tau_{k_2} - \tau_{k_1} \right) \\ 
\overline \kappa_{\phi,\tau} \left( \tau_{k_2} - \tau_{k_1} \right) \end{array} \right]  \\
\hspace{20mm}\displaystyle + \sum_{i=1}^{n_u} \mathop {\min} \left[ \begin{array}{l} \underline \kappa_{\phi,i} ( u_{k_2,i} - u_{k_1,i} ), \\ \overline \kappa_{\phi,i} ( u_{k_2,i} - u_{{k_1},i} ) \end{array} \right],
\end{array}
\end{equation}

\vspace{-2mm}
\begin{equation}\label{eq:lipcheck2U}
\begin{array}{l}
 \phi_{p} ({\bf u}_{k_2},\tau_{k_2}) \leq \\
\phi_{p} ({\bf u}_{k_1},\tau_{k_1}) + \nabla \phi_p({\bf u}_{k_1},\tau_{k_1})^T \left[ \hspace{-.5mm} \begin{array}{c} {\bf u}_{k_2} - {\bf u}_{k_1} \\ 0 \end{array} \hspace{-.5mm} \right]  \vspace{1mm} \\
+ \mathop {\max} \left[  \begin{array}{l} \underline \kappa_{\phi,\tau} (\tau_{k_2} - \tau_{k_1}), \\ \overline \kappa_{\phi,\tau} (\tau_{k_2} - \tau_{k_1})  \end{array} \right] \vspace{1mm} \\
+\displaystyle \frac{1}{2} \sum_{i_1=1}^{n_u} \sum_{i_2=1}^{n_u} \mathop {\max} \left[ \begin{array}{l} \underline M_{\phi,i_1 i_2} (u_{k_2,i_1} - u_{k_1,i_1}) \\
\hspace{15mm} (u_{k_2,i_2} - u_{{k_1},i_2}), \\ \overline M_{\phi,i_1 i_2} (u_{k_2,i_1} - u_{{k_1},i_1}) \\
\hspace{15mm}(u_{k_2,i_2} - u_{{k_1},i_2}) \end{array} \right],
\end{array}
\end{equation}

\vspace{-2mm}
\begin{equation}\label{eq:lipcheck2L}
\begin{array}{l}
 \phi_{p} ({\bf u}_{k_2},\tau_{k_2}) \geq \\
\phi_{p} ({\bf u}_{k_1},\tau_{k_1}) + \nabla \phi_p({\bf u}_{k_1},\tau_{k_1})^T \left[ \hspace{-.5mm} \begin{array}{c} {\bf u}_{k_2} - {\bf u}_{k_1} \\ 0 \end{array} \hspace{-.5mm} \right]  \vspace{1mm} \\
+ \mathop {\min} \left[  \begin{array}{l} \underline \kappa_{\phi,\tau} (\tau_{k_2} - \tau_{k_1}), \\ \overline \kappa_{\phi,\tau} (\tau_{k_2} - \tau_{k_1})  \end{array} \right] \vspace{1mm} \\
+\displaystyle \frac{1}{2} \sum_{i_1=1}^{n_u} \sum_{i_2=1}^{n_u} \mathop {\min} \left[ \begin{array}{l} \underline M_{\phi,i_1 i_2} (u_{k_2,i_1} - u_{k_1,i_1}) \\
\hspace{15mm} (u_{k_2,i_2} - u_{{k_1},i_2}), \\ \overline M_{\phi,i_1 i_2} (u_{k_2,i_1} - u_{{k_1},i_1}) \\
\hspace{15mm}(u_{k_2,i_2} - u_{{k_1},i_2}) \end{array} \right],
\end{array}
\end{equation}

\noindent where the last two bounds may be derived by considering the decomposition

\vspace{-2mm}
\begin{equation}\label{eq:consder}
\begin{array}{l}
\phi_{p} ({\bf u}_{k_2},\tau_{k_2}) - \phi_{p} ({\bf u}_{k_1},\tau_{k_1}) = \vspace{1mm} \\
\hspace{10mm} \phi_{p} ({\bf u}_{k_2},\tau_{k_2}) - \phi_{p} ({\bf u}_{k_2},\tau_{k_1}) \vspace{1mm} \\
\hspace{10mm} + \phi_{p} ({\bf u}_{k_2},\tau_{k_1}) - \phi_{p} ({\bf u}_{k_1},\tau_{k_1}),
\end{array}
\end{equation}

\noindent with the linear Lipschitz bound then applied to the first addend, $\phi_{p} ({\bf u}_{k_2},\tau_{k_2}) - \phi_{p} ({\bf u}_{k_2},\tau_{k_1})$, and the quadratic Lipschitz bound applied to the second, $\phi_{p} ({\bf u}_{k_2},\tau_{k_1}) - \phi_{p} ({\bf u}_{k_1},\tau_{k_1})$.

It then follows that if there exist $k_1$ and $k_2$ such that at least one of the relations above does not hold, the Lipschitz constants are not consistent with the data and as such should be modified. One possible algorithm to carry out this operation is now provided.
\newline
\newline
{\bf Algorithm 1 -- Lipschitz Consistency Check}
\begin{enumerate}
\item Set $a := 1$.
\item Check the validity of (\ref{eq:lipcheck1costU}) and (\ref{eq:lipcheck1costL}) for every combination $( k_1, k_2 ) \in \{ 0,...,k\} \times \{ 0,...,k\}$. If these inequalities are satisfied for every $( k_1, k_2 )$, then terminate. Otherwise, proceed to Step 3.
\item If $a \leq 5$, then for each $i = 1,...,n_u$ set

\vspace{-2mm}
$$
\begin{array}{rcl}
\underline \kappa_{\phi,i} & := & 2^{-{\rm sign}\;\underline \kappa_{\phi,i}}\underline \kappa_{\phi,i}, \vspace{1mm} \\
\overline \kappa_{\phi,i} & := & 2^{{\rm sign}\;\overline \kappa_{\phi,i}}\overline \kappa_{\phi,i}, \vspace{1mm} \\
\underline \kappa_{\phi,\tau} & := & 2^{-{\rm sign}\;\underline \kappa_{\phi,\tau}}\underline \kappa_{\phi,\tau}, \vspace{1mm} \\
\overline \kappa_{\phi,\tau} & := & 2^{{\rm sign}\;\overline \kappa_{\phi,\tau}}\overline \kappa_{\phi,\tau}.
\end{array}
$$

If $5 < a \leq 10$, set

\vspace{-2mm}
$$
\begin{array}{rcl}
\underline \kappa_{\phi,i} & := & -2 \mathop {\max} \left[ | \underline \kappa_{\phi,i} |, | \overline \kappa_{\phi,i} | \right], \vspace{1mm} \\ 
\overline \kappa_{\phi,i} & := & 2 \mathop {\max} \left[ | \underline \kappa_{\phi,i} |, | \overline \kappa_{\phi,i} | \right], \vspace{1mm} \\ 
\underline \kappa_{\phi,\tau} & := & -2 \mathop {\max} \left[ | \underline \kappa_{\phi,\tau} |, | \overline \kappa_{\phi,\tau} | \right], \vspace{1mm} \\
\overline \kappa_{\phi,\tau} & := & 2 \mathop {\max} \left[ | \underline \kappa_{\phi,\tau} |, | \overline \kappa_{\phi,\tau} | \right].
\end{array}
$$

If $10 < a$, set

\vspace{-2mm}
$$
\begin{array}{rcl}
\underline \kappa_{\phi,i} & := & 2^{a-10} \underline \kappa_{\phi,i}, \vspace{1mm} \\
\overline \kappa_{\phi,i} & := & 2^{a-10} \overline \kappa_{\phi,i}, \vspace{1mm} \\
\underline \kappa_{\phi,\tau} & := & 2^{a-10} \underline \kappa_{\phi,\tau}, \vspace{1mm} \\
\overline \kappa_{\phi,\tau} & := & 2^{a-10} \overline \kappa_{\phi,\tau}.
\end{array}
$$

\item Set $a := a+1$ and return to Step 2.

\end{enumerate}

Identical algorithms may be used to verify (\ref{eq:lipcheck1U})-(\ref{eq:lipcheck1L}) for $\kappa_{p,ji}$ (for each $j = 1,...,n_{g_p}$) and to verify (\ref{eq:lipcheck2U})-(\ref{eq:lipcheck2L}) for $M_{\phi,i_1 i_2}$, with the latter using the values of $\underline \kappa_{\phi,\tau}$ and $\overline \kappa_{\phi,\tau}$ that have been fixed by Algorithm 1.

Clearly, the essence of the algorithm is to make the Lipschitz constant estimates more conservative until the corresponding Lipschitz bounds are consistent with the data. In the version proposed above, this is done in three phases. The first ($a \leq 5$) allows all constants to retain their signs while making their values lower or higher, as appropriate. As mentioned earlier, information regarding the sign of a derivative can be very useful and so it is for this reason that we attempt to retain it here. If this is not sufficient due to an error in the sign for some constant, the second phase ($5 < a \leq 10$) makes the lower and upper bounds equal in magnitude as this magnitude continues to be increased. If after 10 augmentations the new constants are still not consistent, which would very likely be due to scaling errors -- the derivative assumed to vary on a magnitude smaller than the one on which it actually does -- the final phase ($ 10 < a$) proceeds to augment the estimates at a much quicker (geometric) rate until they become consistent.

While this algorithm is not perfect and comes with several fundamental weaknesses, such as augmenting all the constants for a given function via a single factor of 2, it nevertheless provides a basic and useful tool for the pure knowledge-free case, until other methods of estimating the Lipschitz constants become possible. In the case where the available estimates are already reasonably close to valid, we would expect Algorithm 1 to only refine these estimates slightly without needing to proceed to the second phase.

\subsection{Exploiting Locality}

We have, up to now, only employed global definitions of the Lipschitz constants and the concavity/convexity properties of the experimental functions. However, as the SCFO only depend on these being valid on a given subregion of the experimental space, it follows that one can relax the resulting bounds so that they only hold over the subregion in question.

Let us start by defining the \emph{local temporal experimental space with respect to $\bar k$}, $\mathcal{I}_{\tau}^{\bar k}$, as

\vspace{-2mm}
\begin{equation}\label{eq:locspace}
\mathcal{I}_{\tau}^{\bar k} = {\rm bbox} \left( {\bf u}_{\bar k}, {\bf u}_{k^*}, \bar {\bf u}_{k+1}^* \right) \times \left\{ \tau : \begin{array}{l} \mathop {\min} \left[ \tau_{\bar k}, \tau_{k^*} \right] \\ \hspace{5mm} \leq \tau \leq \tau_{k+1} \end{array} \right\},
\end{equation}

\noindent i.e., as the Cartesian product of the (minimal) bounding box of the points ${\bf u}_{\bar k}$, ${\bf u}_{k^*}$, and $\bar {\bf u}_{k+1}^*$ with the shortest time interval that includes $\tau_{\bar k}$, $\tau_{k^*}$, and $\tau_{k+1}$. Clearly, $\mathcal{I}_{\tau}^{\bar k} \subseteq \mathcal{I}_{\tau}$. Furthermore, $({\bf u}_{k+1},\tau_{k+1}) \in \mathcal{I}_\tau^{\bar k}$ follows from the convexity of $\mathcal{I}_\tau^{\bar k}$ and the fact that $({\bf u}_{k^*},\tau_{k+1})$ and $(\bar {\bf u}_{k+1}^*,\tau_{k+1})$, the convex combination of which defines $({\bf u}_{k+1},\tau_{k+1})$, are both in  $\mathcal{I}_\tau^{\bar k}$.

The motivation for defining $\mathcal{I}_\tau^{\bar k}$ this way stems from the fact that the upper and lower bounds on the evolution of an experimental function between any two points in $\mathcal{I}_\tau^{\bar k}$ only need to be written using the Lipschitz constants and convexity/concavity properties relevant to $\mathcal{I}_\tau^{\bar k}$. Since these bounds will generally involve the past experimental iterate ${\bf u}_{\bar k}$, the reference iterate ${\bf u}_{k^*}$, the target iterate $\bar {\bf u}_{k+1}^*$, and the actual future iterate ${\bf u}_{k+1}$, together with their relevant time values $\tau$, the definition of  $\mathcal{I}_\tau^{\bar k}$ essentially considers the smallest box space over which such bounds can be applied\footnote{One could be even more minimalist and use something other than a box (e.g., a convex hull), which in some cases would simplify down to a line \cite{Bunin:Lip}. However, to discuss such possibilities rigorously would require additional development that we believe would complicate the message without adding anything extremely useful from the implementation perspective.}. 

The local Lipschitz constants with respect to $\bar k$ are then defined as

\vspace{-2mm}
\begin{equation}\label{eq:lipcondegLUloc}
\hspace{-2mm}\begin{array}{l}
\displaystyle \underline \kappa_{p,ji} \leq \underline \kappa_{p,ji}^{\bar k} < \frac{\partial g_{p,j}}{\partial u_i} \Big |_{({\bf u},\tau)} < \overline \kappa_{p,ji}^{\bar k} \leq \overline \kappa_{p,ji}, \\
\hspace{50mm} \forall ({\bf u},\tau) \in \mathcal{I}_\tau^{\bar k},
\end{array}
\end{equation}

\vspace{-2mm}
\begin{equation}\label{eq:lipcondeg2LUloc}
\hspace{-2mm}\begin{array}{l}
\displaystyle \underline M_{\phi,i_1 i_2} \hspace{-.1mm} \leq \hspace{-.1mm} \underline M_{\phi,i_1 i_2}^{\bar k} \hspace{-.1mm} < \hspace{-.1mm} \frac{\partial^2 \phi_p}{\partial u_{i_2} \partial u_{i_1} } \Big |_{({\bf u},\tau)} \hspace{-.1mm} < \hspace{-.1mm} \overline M_{\phi,i_1 i_2}^{\bar k} \hspace{-.1mm} \leq \hspace{-.1mm} \overline M_{\phi,i_1 i_2}, \\
\hspace{50mm} \forall ({\bf u},\tau) \in \mathcal{I}_\tau^{\bar k},
\end{array}
\end{equation}

\vspace{-2mm}
\begin{equation}\label{eq:lipdegLUloc}
\hspace{-2mm}\begin{array}{l}
\displaystyle \underline \kappa_{p,j\tau} \leq \underline \kappa_{p,j\tau}^{\bar k} \leq \frac{\partial g_{p,j}}{\partial \tau} \Big |_{({\bf u},\tau)} \leq \overline \kappa_{p,j\tau}^{\bar k} \leq \overline \kappa_{p,j\tau}, \\
 \hspace{50mm} \forall ({\bf u},\tau) \in \mathcal{I}_\tau^{\bar k},
\end{array}
\end{equation}

\vspace{-2mm}
\begin{equation}\label{eq:lipcostloc}
\hspace{-2mm}\underline \kappa_{\phi,i} \leq \underline \kappa_{\phi,i}^{\bar k} \leq \frac{\partial \phi_{p}}{\partial u_i} \Big |_{({\bf u},\tau)} \leq \overline \kappa_{\phi,i}^{\bar k} \leq \overline \kappa_{\phi,i}, \;\; \forall ({\bf u},\tau) \in \mathcal{I}_\tau^{\bar k},
\end{equation}

\vspace{-2mm}
\begin{equation}\label{eq:lipcostdegloc}
\hspace{-2mm}\underline \kappa_{\phi,\tau} \leq \underline \kappa_{\phi,\tau}^{\bar k} \leq \frac{\partial \phi_{p}}{\partial \tau} \Big |_{({\bf u},\tau)} \leq \overline \kappa_{\phi,\tau}^{\bar k} \leq \overline \kappa_{\phi,\tau}, \;\; \forall ({\bf u},\tau) \in \mathcal{I}_\tau^{\bar k}.
\end{equation}

Extending this locality to the concavity/convexity assumptions, let us now state the local versions of Definitions \ref{def:ccv} and \ref{def:cvx}.

\begin{definition}[Local partial concavity]
\label{def:ccvloc}
Let the decision variables ${\bf u}$ be partitioned into subvectors ${\bf v}$ and ${\bf z}$, so that $g_{p,j}({\bf u},\tau) = g_{p,j}({\bf v},{\bf z},\tau)$. The function $g_{p,j}$ will be said to be \emph{locally partially concave with respect to $\bar k$} in ${\bf z}$ if it is concave everywhere on $\mathcal{I}_\tau^{\bar k}$ for ${\bf v}$ and $\tau$ fixed. Likewise, $g_{p,j}$ will be said to be locally partially concave with respect to $\bar k$ in both ${\bf z}$ and $\tau$ if it is concave everywhere on $\mathcal{I}_\tau^{\bar k}$ for ${\bf v}$ fixed.
\end{definition}

\begin{definition}[Local partial convexity]
\label{def:cvxloc}
Let the decision variables ${\bf u}$ be partitioned into subvectors ${\bf v}$ and ${\bf z}$, so that $\phi_{p}({\bf u},\tau) = \phi_{p}({\bf v},{\bf z},\tau)$. The function $\phi_{p}$ will be said to be \emph{locally partially convex with respect to $\bar k$} in ${\bf z}$ if it is convex everywhere on $\mathcal{I}_\tau^{\bar k}$ for ${\bf v}$ and $\tau$ fixed. Likewise, $\phi_{p}$ will be said to be locally partially convex with respect to $\bar k$ in both ${\bf z}$ and $\tau$ if it is convex everywhere on $\mathcal{I}_\tau^{\bar k}$ for ${\bf v}$ fixed.
\end{definition}

Accordingly, let $I_{c,j}^{\bar k}$ denote the indices of the variables ${\bf u}$ in which $g_{p,j}$ is locally partially concave with respect to $\bar k$, and let $\eta_{c,j}^{\bar k}$ denote a Boolean indicator that is equal to 1 if $g_{p,j}$ is locally partially concave with respect to $\bar k$ in both the variables indexed by $I_{c,j}^{\bar k}$ and in $\tau$, and is equal to 0 otherwise. Using analogous extensions for convexity, let $I_{v,\phi}^{\bar k}$ denote the indices of the variables ${\bf u}$ in which $\phi_{p}$ is locally partially convex with respect to $\bar k$, and let $\eta_{v,\phi}^{\bar k}$ denote a Boolean indicator that is equal to 1 if $\phi_p$ is locally partially convex with respect to $\bar k$ in both the variables indexed by $I_{v,\phi}^{\bar k}$ and in $\tau$, and is equal to 0 otherwise. While the language here may appear complicated and verbose, we note that these definitions are just trivial localizations of what has already been presented in the previous sections.

In some cases -- in particular, those related only to degradation -- it will also be of interest to take a local space defined by only two iterations, $\bar k_1, \bar k_2 \in [0,k+1]$. Denoting this space by $\mathcal{I}_{\tau}^{\bar k_1, \bar k_2}$ and assuming, without loss of generality, that $\bar k_1 \leq \bar k_2$, we define it as

\vspace{-2mm}
\begin{equation}\label{eq:locspace2}
\mathcal{I}_{\tau}^{\bar k_1,\bar k_2} = {\rm bbox} \left( {\bf u}_{\bar k_1}, {\bf u}_{\bar k_2} \right) \times \left\{ \tau : \tau_{\bar k_1} \leq \tau \leq \tau_{\bar k_2} \right\}.
\end{equation}

\noindent As this will be particularly relevant with respect to degradation, we define the appropriate degradation-related Lipschitz constants as

\vspace{-2mm}
\begin{equation}\label{eq:lipdegLUloc2}
\begin{array}{l}
\displaystyle \underline \kappa_{p,j\tau} \leq \underline \kappa_{p,j\tau}^{\bar k_1, \bar k_2} \leq \frac{\partial g_{p,j}}{\partial \tau} \Big |_{({\bf u},\tau)} \leq \overline \kappa_{p,j\tau}^{\bar k_1, \bar k_2} \leq \overline \kappa_{p,j\tau}, \\
\hspace{45mm} \forall ({\bf u},\tau) \in \mathcal{I}_\tau^{\bar k_1, \bar k_2},
\end{array}
\end{equation}

\vspace{-2mm}
\begin{equation}\label{eq:lipcostdegloc2}
\begin{array}{l}
\displaystyle \underline \kappa_{\phi,\tau} \leq \underline \kappa_{\phi,\tau}^{\bar k_1, \bar k_2} \leq \frac{\partial \phi_{p}}{\partial \tau} \Big |_{({\bf u},\tau)} \leq \overline \kappa_{\phi,\tau}^{\bar k_1 ,\bar k_2} \leq \overline \kappa_{\phi,\tau}, \\
\hspace{45mm} \forall ({\bf u},\tau) \in \mathcal{I}_\tau^{\bar k_1,\bar k_2}.
\end{array}
\end{equation}

We are now equipped with all of the necessary notation to state the local versions of the SCFO. Although one could go through and rederive all of the previous results with the local constants and local assumptions, such a task would be rather space-consuming and redundant, as it would not introduce anything insightful or unexpected -- put otherwise, we find it sufficient to let the reader accept on trust that all of the following bounds and inequalities, which are obtained by simply replacing global constants and assumptions by the local ones, will hold.

Considering first Condition (\ref{eq:SCFO1idegLUccvall}), we state its local version:

\vspace{-2mm}
\begin{equation}\label{eq:SCFO1idegLUccvallloc}
\hspace{-6mm}\begin{array}{l}
\mathop {\min} \limits_{\bar k = 0,...,k} \left[ \hspace{-1mm} \begin{array}{l} g_{p,j} ({\bf u}_{\bar k},\tau_{\bar k}) \vspace{1mm} \\
\displaystyle +\eta_{c,j}^{\bar k} \frac{\partial g_{p,j}}{\partial \tau} \Big |_{({\bf u}_{\bar k},\tau_{\bar k})} ( \tau_{k+1} - \tau_{{\bar k}} ) \vspace{1mm} \\
 \displaystyle + (1-\eta_{c,j}^{\bar k}) \overline \kappa_{p,j\tau}^{\bar k} \left( \tau_{k+1} - \tau_{\bar k} \right) \vspace{1mm}\\
\displaystyle   + \sum_{i \in I_{c,j}^{\bar k}} \frac{\partial g_{p,j}}{\partial u_i} \Big |_{({\bf u}_{\bar k},\tau_{\bar k})} ( u_{k^*,i} + \\
\hspace{15mm}  K_k (\bar u_{k+1,i}^* - u_{k^*,i} ) - u_{{\bar k},i} ) \vspace{1mm} \\
 + \displaystyle \sum_{i \not \in I_{c,j}^{\bar k}} \mathop {\max} \left[ \begin{array}{l} \underline \kappa_{p,ji}^{\bar k} ( u_{k^*,i} + \\ \hspace{2mm} K_k (\bar u_{k+1,i}^* - u_{k^*,i} ) - u_{\bar k,i} ), \vspace{1mm}\\ \overline \kappa_{p,ji}^{\bar k} ( u_{k^*,i} + \\ \hspace{2mm} K_k (\bar u_{k+1,i}^* - u_{k^*,i} ) - u_{\bar k,i} ) \end{array} \right] \end{array} \hspace{-1mm} \right] \leq 0.
\end{array}
\end{equation}

With respect to the cost descent conditions, we have the following local version of (\ref{eq:SCFO7idegLU}):

\vspace{-2mm}
\begin{equation}\label{eq:SCFO7idegLUloc}
\begin{array}{l}
 \nabla \phi_p({\bf u}_{k^*},\tau_{k+1})^T  \left[ \begin{array}{c} \bar {\bf u}_{k+1}^* - {\bf u}_{k^*} \\ 0 \end{array} \right]  \vspace{1mm} \\
\displaystyle + \frac{K_k}{2} \sum_{i_1=1}^{n_u} \sum_{i_2=1}^{n_u} \mathop {\max} \left[ \begin{array}{l} \underline M_{\phi,i_1 i_2}^{k^*} (\bar u_{k+1,i_1}^* - u_{{k^*},i_1}) \\
\hspace{11mm} (\bar u_{k+1,i_2}^* - u_{{k^*},i_2}), \\ \overline M_{\phi,i_1 i_2}^{k^*} (\bar u_{k+1,i_1}^* - u_{{k^*},i_1}) \\
\hspace{11mm}(\bar u_{k+1,i_2}^* - u_{{k^*},i_2}) \end{array} \right]  \leq 0.
\end{array}
\end{equation}

For the supplementary condition to help enforce cost decrease, given by (\ref{eq:costhighmaxPF}), the local version is as follows:

\vspace{-2mm}
\begin{equation}\label{eq:costhighmaxPFloc}
\begin{array}{l}
\mathop {\max} \limits_{\bar k = 0,...,k}\left[ \begin{array}{l}
\displaystyle \phi_{p} ({\bf u}_{\bar k},\tau_{\bar k}) \vspace{1mm} \\
\displaystyle  +\eta_{v,\phi}^{\bar k} \frac{\partial \phi_{p}}{\partial \tau} \Big |_{({\bf u}_{\bar k},\tau_{\bar k})} ( \tau_{k+1} - \tau_{{\bar k}} ) \vspace{1mm} \\
\displaystyle + (1-\eta_{v,\phi}^{\bar k}) \underline \kappa_{\phi,\tau}^{\bar k} \left( \tau_{k+1} - \tau_{\bar k} \right) \vspace{1mm}\\
\displaystyle   + \sum_{i \in I_{v,\phi}^{\bar k}} \frac{\partial \phi_{p}}{\partial u_i} \Big |_{({\bf u}_{\bar k},\tau_{\bar k})} ( u_{k^*,i} + \\ \hspace{15mm}  K_k(\bar u_{k+1,i}^* - u_{k^*,i}) - u_{{\bar k},i} ) \vspace{1mm} \\
+ \displaystyle \sum_{i \not \in I_{v,\phi}^{\bar k}} \mathop {\min} \left[ \begin{array}{l} \underline \kappa_{\phi,i}^{\bar k} ( u_{k^*,i} +\\ \hspace{2mm} K_k(\bar u_{k+1,i}^* - u_{k^*,i}) - u_{{\bar k},i} ), \vspace{1mm}\\ \overline \kappa_{\phi,i}^{\bar k} ( u_{k^*,i} +\\ \hspace{2mm} K_k(\bar u_{k+1,i}^* - u_{k^*,i}) - u_{{\bar k},i} ) \end{array} \right] 
\end{array} \right]
\end{array}
\end{equation}

$$
\begin{array}{l}
\hspace{8mm}\displaystyle \leq \mathop {\min}_{\tilde k = 0,...,k} \left[ \begin{array}{l} \displaystyle \phi_{p} ({\bf u}_{\tilde k},\tau_{\tilde k}) \vspace{1mm} \\
\displaystyle  +\eta_{c,\phi}^{\tilde k} \frac{\partial \phi_{p}}{\partial \tau} \Big |_{({\bf u}_{\tilde k},\tau_{\tilde k})} ( \tau_{k+1} - \tau_{{\tilde k}} ) \vspace{1mm} \\
\hspace{0mm} \displaystyle + (1-\eta_{c,\phi}^{\tilde k}) \overline \kappa_{\phi,\tau}^{\tilde k} \left( \tau_{k+1} - \tau_{\tilde k} \right) \vspace{1mm}\\
\hspace{0mm}\displaystyle   + \sum_{i \in I_{c,\phi}^{\tilde k}} \frac{\partial \phi_{p}}{\partial u_i} \Big |_{({\bf u}_{\tilde k},\tau_{\tilde k})} ( u_{k^*,i} - u_{{\tilde k},i} ) \vspace{1mm} \\
\hspace{0mm} + \displaystyle \sum_{i \not \in I_{c,\phi}^{\tilde k}} \mathop {\max} \left[ \begin{array}{l} \underline \kappa_{\phi,i}^{\tilde k} ( u_{k^*,i} - u_{{\tilde k},i} ), \vspace{1mm}\\ \overline \kappa_{\phi,i}^{\tilde k} ( u_{k^*,i} - u_{{\tilde k},i} ) \end{array} \right] \end{array} \right].
\end{array}
$$

For the choice of reference point, given in (\ref{eq:kstarLUccvcost}) and (\ref{eq:kstar2LUccv}), note that one cannot make relaxations over $\mathcal{I}_{\tau}^{\bar k}$ since this space is defined, in part, by $k^*$, which here would be undetermined. However, we may instead work with $\mathcal{I}_{\tau}^{\bar k_1, \bar k_2}$, as defined in (\ref{eq:locspace2}), to refine (\ref{eq:kstarLUccvcost}) and (\ref{eq:kstar2LUccv}) as

\vspace{-2mm}
\begin{equation}\label{eq:kstarLUccvcostloc}
\begin{array}{rl}
k^* := \;\;\;\;\;\;\;\;\;\;\;\;\;\;& \vspace{1mm} \\
{\rm arg} \mathop {\rm maximize}\limits_{\bar k \in [0,k]} & \bar k \vspace{1mm} \\
{\rm{subject}}\;{\rm{to}} & g_{p,j} ({\bf u}_{\bar k},\tau_{\bar k}) \vspace{1mm} \\
& \displaystyle + \eta_{c,j} \frac{\partial g_{p,j}}{\partial \tau} \Big |_{({\bf u}_{\bar k},\tau_{\bar k})} ( \tau_{k+1} - \tau_{{\bar k}} ) \vspace{1mm} \\
&  + (1-  \eta_{c,j}) \overline \kappa_{p,j\tau} \left( \tau_{k+1} - \tau_{\bar k} \right) \leq 0, \vspace{1mm} \\
& \forall j = 1,...,n_{g_p} \vspace{1mm} \\
& \displaystyle \phi_p ({\bf u}_{\bar k},\tau_{\bar k}) + \eta_{v,\phi}^{\bar k, k} \frac{\partial \phi_{p}}{\partial \tau} \Big |_{({\bf u}_{\bar k},\tau_{\bar k})} ( \tau_{k} - \tau_{{\bar k}} ) \vspace{1mm} \\
&+ (1-  \eta_{v,\phi}^{\bar k, k}) \underline \kappa_{\phi,\tau}^{\bar k, k} \left( \tau_{k} - \tau_{\bar k} \right) \leq \vspace{1mm} \\
& \mathop {\min} \limits_{\tilde k \in {\bf k}_f} \left[ \begin{array}{l}  \phi_p ({\bf u}_{\tilde k},\tau_{\tilde k}) + \vspace{1mm} \\
\displaystyle \eta_{c,\phi}^{\tilde k, k} \frac{\partial \phi_{p}}{\partial \tau} \Big |_{({\bf u}_{\tilde k},\tau_{\tilde k})} ( \tau_{k} - \tau_{{\tilde k}} ) \vspace{1mm} \\
+ (1-  \eta_{c,\phi}^{\tilde k, k}) \overline \kappa_{\phi,\tau}^{\tilde k, k} \left( \tau_{k} - \tau_{\tilde k} \right) \end{array} \right],
\end{array}
\end{equation}

\vspace{-2mm}
\begin{equation}\label{eq:kfeas2}
{\bf k}_f = \left\{ \bar k : \begin{array}{l} g_{p,j} ({\bf u}_{\bar k},\tau_{\bar k}) \vspace{1mm} \\  + \displaystyle  \eta_{c,j} \frac{\partial g_{p,j}}{\partial \tau} \Big |_{({\bf u}_{\bar k},\tau_{\bar k})} ( \tau_{k+1} - \tau_{{\bar k}} ) \vspace{1mm} \\ + (1-  \eta_{c,j}) \overline \kappa_{p,j\tau} \left( \tau_{k+1} - \tau_{\bar k} \right)  \leq 0, \vspace{1mm} \\ \forall j = 1,...,n_{g_p} \end{array} \right\},
\end{equation}

\vspace{-2mm}
\begin{equation}\label{eq:kstar2LUccvloc}
\begin{array}{l}
k^*  := \vspace{1mm} \\
\displaystyle {\rm arg} \mathop {\rm minimize}\limits_{\bar k \in [0,k]}  \mathop {\max} \limits_{j = 1,...,n_{g_p}}  \left[  \begin{array}{l} g_{p,j} ({\bf u}_{\bar k},\tau_{\bar k}) \vspace{1mm} \\
\displaystyle + \eta_{c,j} \frac{\partial g_{p,j}}{\partial \tau} \Big |_{({\bf u}_{\bar k},\tau_{\bar k})} \\
\hspace{15mm}( \tau_{k+1} - \tau_{{\bar k}} ) \vspace{1mm} \\
+ (1-  \eta_{c,j}) \overline \kappa_{p,j\tau} \vspace{1mm} \\
\hspace{15mm}\left( \tau_{k+1} - \tau_{\bar k} \right) \end{array}  \right].
\end{array}
\end{equation}

At this point, it may be fair to ask if these refinements are actually of practical interest, as very specific knowledge is not assumed to be available for the experimental functions, and both the local Lipschitz constants and the local structural assumptions are examples of specific knowledge as they characterize how the function behaves in different portions of the experimental space. At the time of writing, it is not clear if one could routinely and reliably obtain these local relaxations for most experimental optimization problems. However, it is possible to envision a scenario where the function is particularly steep in one region of the experimental space, where it has large Lipschitz constants, and somewhat flat in another, where its Lipschitz constants are small. Likewise, one could also imagine a function that is very nonlinear in some regions (i.e., has large $M$ values) and close to linear otherwise (i.e., has $M$ values close to 0). If a model of the function is available and can predict such trends approximately, then one may obtain the local constants for the model. Returning to the example given in (\ref{eq:parmodellip}) and (\ref{eq:parmodellip2}), the appropriate local analogues would be

\vspace{-2mm}
\begin{equation}\label{eq:parmodelliploc}
\begin{array}{l}
\displaystyle \underline \kappa_{\hat \phi,i}^{\bar k} = \mathop {\min} \limits_{\footnotesize \begin{array}{c} {\bf u},\tau \in \mathcal{I}_\tau^{\bar k} \\ \theta \in \Theta \end{array}} \; \frac{\partial \phi_{\hat p}}{\partial u_i} \Big |_{({\bf u},\tau,\theta)},\\
\displaystyle \overline \kappa_{\hat \phi,i}^{\bar k} = \mathop {\max} \limits_{\footnotesize \begin{array}{c} {\bf u},\tau \in \mathcal{I}_\tau^{\bar k} \\ \theta \in \Theta \end{array}} \; \frac{\partial \phi_{\hat p}}{\partial u_i} \Big |_{({\bf u},\tau,\theta)},
\end{array}
\end{equation}

\vspace{-2mm}
\begin{equation}\label{eq:parmodellip2loc}
\begin{array}{l}
\displaystyle \underline \kappa_{\hat \phi,\tau}^{\bar k} = \hspace{-.3mm} \mathop {\min} \limits_{\footnotesize \begin{array}{c} {\bf u},\tau \in \mathcal{I}_\tau^{\bar k} \\ \theta \in \Theta \end{array}} \frac{\partial \phi_{\hat p}}{\partial \tau} \Big |_{({\bf u},\tau,\theta)},\\
\displaystyle \overline \kappa_{\hat \phi,\tau}^{\bar k} = \mathop {\max} \limits_{\footnotesize \begin{array}{c} {\bf u},\tau \in \mathcal{I}_\tau^{\bar k} \\ \theta \in \Theta \end{array}}  \frac{\partial \phi_{\hat p}}{\partial \tau} \Big |_{({\bf u},\tau,\theta)},
\end{array}
\end{equation}

\noindent followed by the local analogue to (\ref{eq:conservest})

\vspace{-2mm}
\begin{equation}\label{eq:conservestloc}
\begin{array}{l}
\underline \kappa_{\phi,i}^{\bar k} := \underline \kappa_{\hat \phi,i}^{\bar k} - 0.5(\overline \kappa_{\hat \phi,i}^{\bar k} - \underline \kappa_{\hat \phi,i}^{\bar k} ), \vspace{1mm} \\
\overline \kappa_{\phi,i}^{\bar k} := \overline \kappa_{\hat \phi,i}^{\bar k} + 0.5(\overline \kappa_{\hat \phi,i}^{\bar k} - \underline \kappa_{\hat \phi,i}^{\bar k} ), \vspace{1mm} \\
\underline \kappa_{\phi,\tau}^{\bar k} := \underline \kappa_{\hat \phi,\tau}^{\bar k} - 0.5(\overline \kappa_{\hat \phi,\tau}^{\bar k} - \underline \kappa_{\hat \phi,\tau}^{\bar k} ), \vspace{1mm} \\
\overline \kappa_{\phi,\tau}^{\bar k} := \overline \kappa_{\hat \phi,\tau}^{\bar k} + 0.5(\overline \kappa_{\hat \phi,\tau}^{\bar k} - \underline \kappa_{\hat \phi,\tau}^{\bar k} ).
\end{array}
\end{equation}

Such relaxations could be particularly useful when a certain portion of the experimental space possesses large Lipschitz constants -- thus requiring small steps of the SCFO, and thereby slow progress -- but is not actually the part of the experimental space that is explored during optimization.

To illustrate what a difference locality can make, we again consider the ($-$) case of the problem with degradation (\ref{eq:exdeg}). To ease comparison, we remove the concavity assumptions on the constraints and set $I_{c,1}^{\bar k} := I_{c,2}^{\bar k} := \varnothing$, since not doing this would lead to very fast progress even without the locality relaxations. The Lipschitz constants are chosen correctly, as in (\ref{eq:exlipdegLU}), and the same convexity assumptions as before are made on the cost, with $I_{v,\phi}^{\bar k} := \{1,2 \}$ for any $\bar k$. $\eta_{c,1}^{\bar k} := 0$, $\eta_{c,2}^{\bar k} := 1$, and $ \eta_{v,\phi}^{\bar k} := 1$ are used for all $\bar k$. The upper Lipschitz constant on the cost with respect to degradation is taken as $\overline \kappa_{\phi,\tau} = 2/500$. The results for this scenario are given in Fig. \ref{fig:degE}, and are essentially comparable to what was already obtained in Fig. \ref{fig:degC}, the only difference between the two being certain concavity/convexity assumptions with respect to degradation, which contribute little in this case.

\begin{figure*}
\begin{center}
\includegraphics[width=16cm]{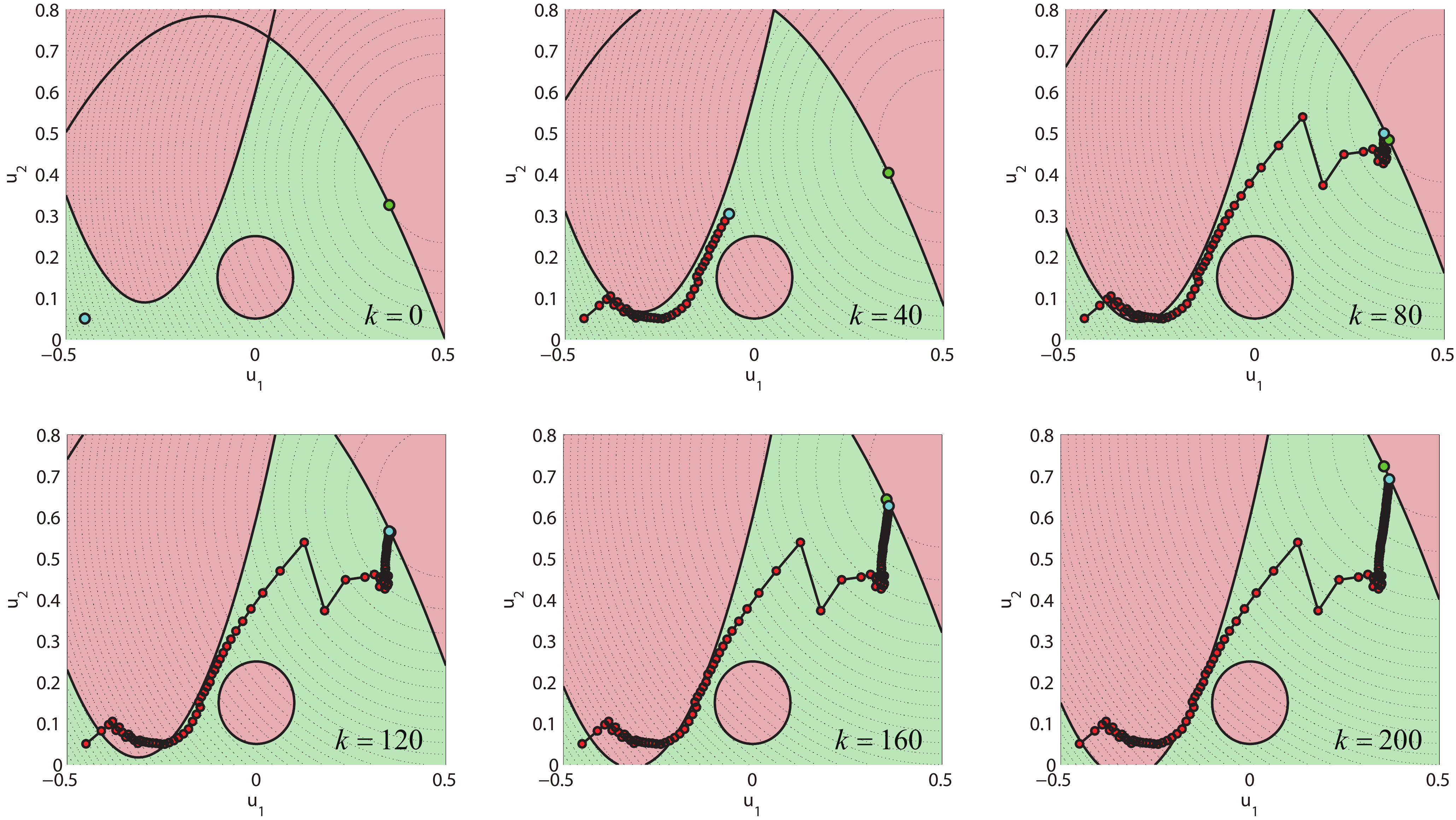}
\caption{Chain of experiments generated by applying the modified SCFO methodology to Problem (\ref{eq:exdeg}) for the ($-$) scenario without local relaxations included.}
\label{fig:degE}
\end{center}
\end{figure*} 

Consider now applying local relaxations to the relevant Lipschitz constants for the two experimental constraints, for which we will, for simplicity, minimize and maximize the derivatives of the true functions over the relevant local spaces (as mentioned, in practice one would likely perform such minimizations/maximizations for a model):

\vspace{-2mm}
\begin{equation}\label{eq:liploc1}
\underline \kappa_{p,11}^{\bar k} := {\rm arg} \mathop {\rm minimize} \limits_{u_1,u_2,\tau \in \mathcal{I}_{\tau}^{\bar k}} \;\;  -12u_1-3.5-\frac{\tau}{500} - 10^{-3},
\end{equation}

\vspace{-2mm}
\begin{equation}\label{eq:liploc2}
\overline \kappa_{p,11}^{\bar k} := {\rm arg} \mathop {\rm maximize} \limits_{u_1,u_2,\tau \in \mathcal{I}_{\tau}^{\bar k}} \;\; -12u_1-3.5-\frac{\tau}{500} + 10^{-3},
\end{equation}

\vspace{-2mm}
\begin{equation}\label{eq:liploc4}
\underline \kappa_{p,21}^{\bar k} := {\rm arg} \mathop {\rm minimize} \limits_{u_1,u_2,\tau \in \mathcal{I}_{\tau}^{\bar k}} \;\; 4u_1+0.5 - 10^{-3},
\end{equation}

\vspace{-2mm}
\begin{equation}\label{eq:liploc5}
\overline \kappa_{p,21}^{\bar k} := {\rm arg} \mathop {\rm maximize} \limits_{u_1,u_2,\tau \in \mathcal{I}_{\tau}^{\bar k}} \;\; 4u_1+0.5 + 10^{-3},
\end{equation}

\noindent with $10^{-3}$ added or subtracted to enforce strict inequality. While other relaxations are also possible, we do not make them here as their effect on performance is likely to be negligible.

The results for the implementation using these local versions are given in Fig. \ref{fig:degF}, with the cost values corresponding to the two versions given in Fig. \ref{fig:degEFcost}. One notes a drastic improvement in convergence speed.

\begin{figure*}
\begin{center}
\includegraphics[width=16cm]{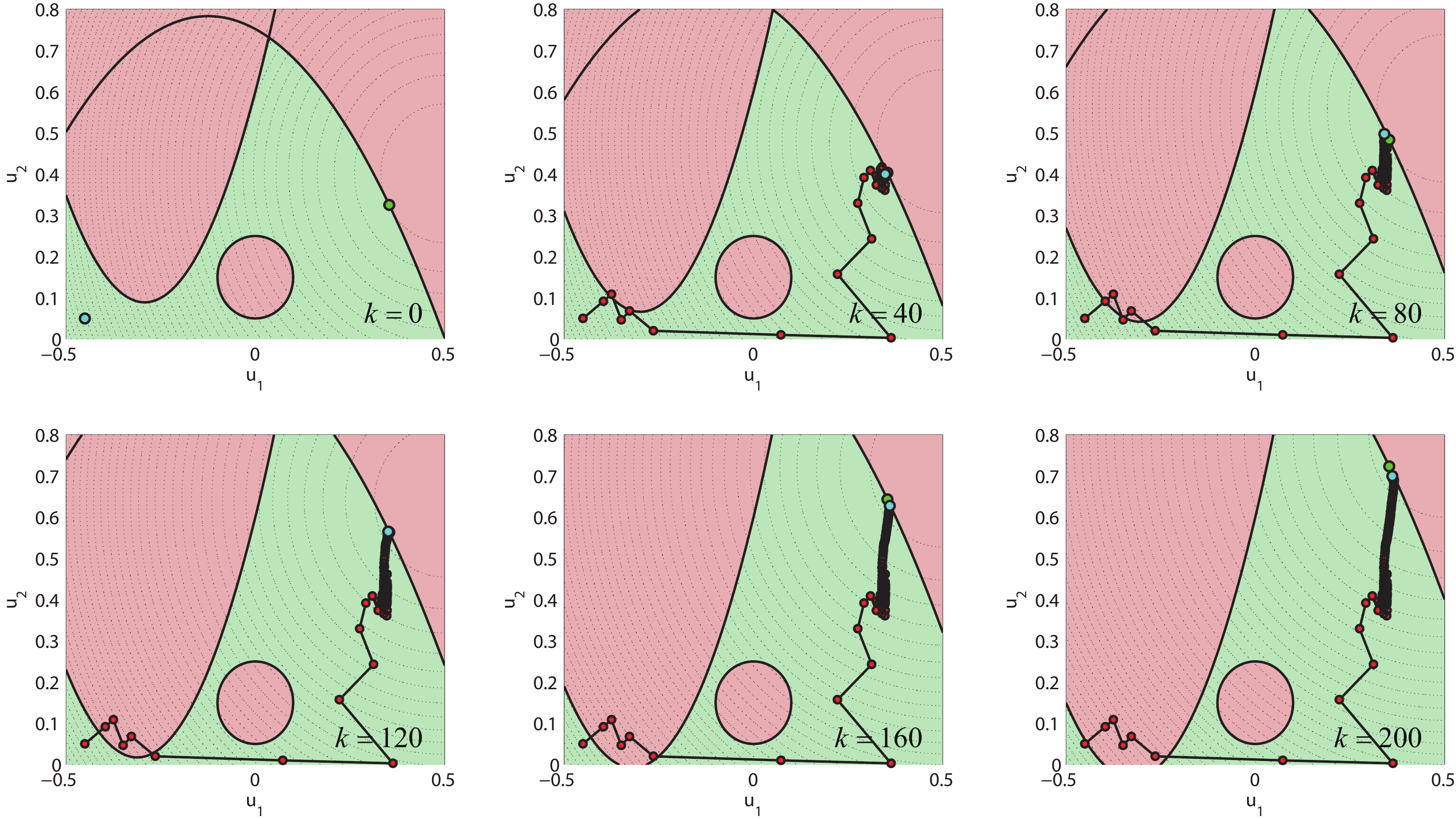}
\caption{Chain of experiments generated by applying the modified SCFO methodology to Problem (\ref{eq:exdeg}) for the ($-$) scenario with local relaxations of certain Lipschitz constants included.}
\label{fig:degF}
\end{center}
\end{figure*}

\section{Lower and Upper Bounds on the Cost and Constraint Values}
\label{sec:bounds}

In most experimental settings, what one observes after applying a given ${\bf u}_{\bar k}$ at time $\tau_{\bar k}$ are not the exact experimental function values $\phi_p({\bf u}_{\bar k},\tau_{\bar k})$ and $g_{p.j}({\bf u}_{\bar k},\tau_{\bar k})$ but rather some noise-corrupted or estimated $\hat \phi_p({\bf u}_{\bar k},\tau_{\bar k})$ and $\hat g_{p,j}({\bf u}_{\bar k},\tau_{\bar k})$. The error in the obtained measurement or estimate is commonly expressed with the aid of the additive noise model (see, e.g., \cite{Hotelling1941,Myers2009,Marchetti2010,More2011,Larson2012}):

\vspace{-2mm}
\begin{equation}\label{eq:costnoise}
\hat \phi_p({\bf u}_{\bar k},\tau_{\bar k}) = \phi_p({\bf u}_{\bar k},\tau_{\bar k}) + w_{\phi,{\bar k}},
\end{equation}

\vspace{-2mm}
\begin{equation}\label{eq:connoise}
\hat g_{p,j}({\bf u}_{\bar k},\tau_{\bar k}) = g_{p,j}({\bf u}_{\bar k},\tau_{\bar k}) + w_{j,{\bar k}},
\end{equation}

\noindent where the noise components, $w_{\phi}$ and $w_j$, will be assumed to be independent.

In the general case where $w_{\phi}$ and $w_j$ may change sign, it should be clear that applying the SCFO with the estimated values $\hat \phi_p({\bf u}_{\bar k},\tau_{\bar k})$ and $\hat g_{p,j}({\bf u}_{\bar k},\tau_{\bar k})$ does not lead to a robust implementation, as this can falsely under or overpredict the true function values. To make these conditions robust against noise, it is then necessary to somehow derive lower and upper bounds on the true values of $\phi_p({\bf u}_{\bar k},\tau_{\bar k})$ and $g_{p,j}({\bf u}_{\bar k},\tau_{\bar k})$ with respect to the observed $\hat \phi_p({\bf u}_{\bar k},\tau_{\bar k})$ and $\hat g_{p,j}({\bf u}_{\bar k},\tau_{\bar k})$.

We will now discuss three possible ways to obtain the needed robust bounds, using the cost function $\phi_p$ as the example throughout.

\subsection{Bounding the True Function Value by Bounding the Noise}

The natural requirement for deriving lower and upper bounds on $\phi_p({\bf u}_{\bar k},\tau_{\bar k})$  is that the stochastic terms $w_\phi$ be bounded with a certain probability. Noise bounds that are expected to hold with, for example, a probability of 99\% would need to satisfy

\vspace{-4mm}
\begin{equation}\label{eq:probbound1}
\mathbb{P} (\underline w_{\phi,\bar k} < w_{\phi,{\bar k}} < \overline w_{\phi,\bar k}) = 0.99.
\end{equation}

To obtain the bounds $\underline w_{\phi,\bar k}$ and $\overline w_{\phi,\bar k}$, one should at least know the mean ($w_{ave, \bar k}$) and variance ($\sigma_{\bar k}$) of $w_\phi$. If the probability density function (pdf) of $w_{\phi, \bar k}$ is not known, the most ``brute'' way to determine a sufficiently low (resp., high) $\underline w_\phi$ (resp., $\overline w_\phi$) is to use Chebyshev's inequality \cite{Ferentinos1982,More2011}:

\vspace{-4mm}
\begin{equation}\label{eq:probbound2}
\begin{array}{l}
\mathbb{P} (\underline w_{\phi,\bar k} < w_{\phi,{\bar k}} < \overline w_{\phi,\bar k}) \geq \vspace{1mm} \\
\hspace{6mm} \displaystyle \frac{4[(w_{ave,\bar k}-\underline w_{\phi,\bar k})(\overline w_{\phi,\bar k} - w_{ave,\bar k}) - \sigma_{\bar k}^2]}{(\overline w_{\phi,\bar k} - \underline w_{\phi,\bar k})^2} = 0.99,
\end{array}
\end{equation}

\noindent where one may solve the algebraic equation so as to obtain the appropriate $\underline w_{\phi,\bar k}$ and $\overline w_{\phi,\bar k}$.

\begin{figure}
\begin{center}
\includegraphics[width=7cm]{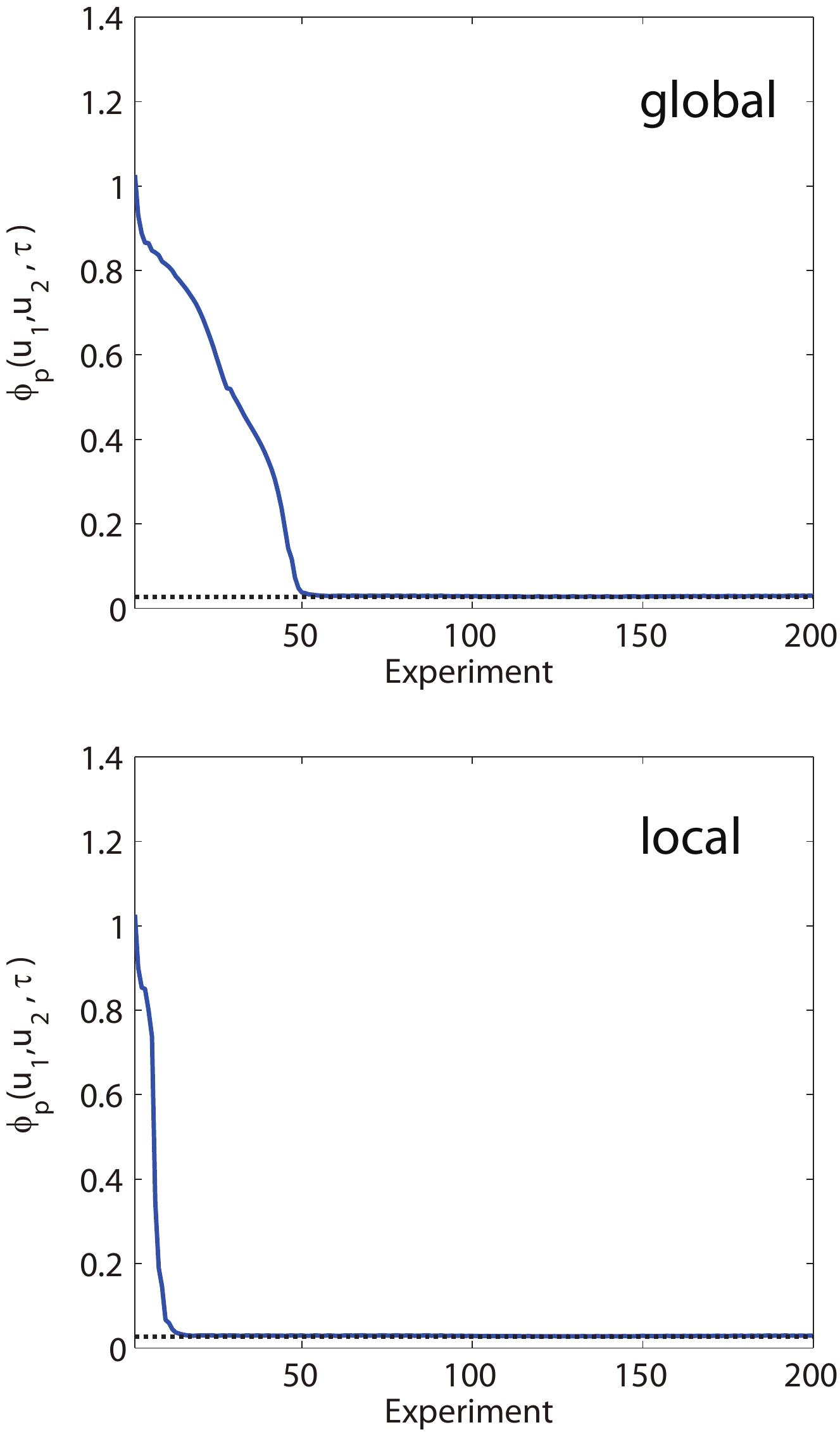}
\caption{Cost function values obtained by the modified SCFO methodology for Problem (\ref{eq:exdeg}) for the ($-$) scenario when global Lipschitz constants (top) and local Lipschitz constants (bottom) are used.}
\label{fig:degEFcost}
\end{center}
\end{figure}

Given the generality of Chebyshev's inequality, such bounds will tend to be fairly conservative, and tighter bounds can naturally be obtained if the pdf is known. Taking the well-known example of zero-mean Gaussian noise, one may simply set the bounds as

\vspace{-2mm}
\begin{equation}\label{eq:probbound3}
\underline w_{\phi,\bar k} := -3\sigma_{\bar k},\; \overline w_{\phi,\bar k} = 3\sigma_{\bar k},
\end{equation}

\noindent which yield a probability slightly higher than 99\% in (\ref{eq:probbound1}). For the general pdf, one may solve (\ref{eq:probbound1}) numerically to obtain an appropriate pair of $\underline w_{\phi,\bar k}$ and $\overline w_{\phi,\bar k}$.

With these bounds available, we may then rearrange (\ref{eq:costnoise}) and apply the bounds to obtain

\vspace{-2mm}
\begin{equation}\label{eq:costbound1}
\hat \phi_p({\bf u}_{\bar k},\tau_{\bar k}) - \overline w_{\phi,\bar k} < \phi_p({\bf u}_{\bar k},\tau_{\bar k}) < \hat \phi_p({\bf u}_{\bar k},\tau_{\bar k}) - \underline w_{\phi,\bar k},
\end{equation}

\noindent which effectively bounds the true value to lie in the interval $(\hat \phi_p({\bf u}_{\bar k},\tau_{\bar k}) - \overline w_{\phi,\bar k}, \hat \phi_p({\bf u}_{\bar k},\tau_{\bar k}) - \underline w_{\phi,\bar k})$ with some desired probability.

\subsection{Reduced-Variance Bounds in the Case of Repeated Measurements}

One may attempt to tighten the bounds of (\ref{eq:costbound1}) in the case when the same variables ${\bf u}_{\bar k}$ have been applied for multiple iterations -- this may occur, for example, if one chooses to temporarily stop experimental optimization or, in cases where the adaptation must be fast, one is forced to ``skip'' an iteration to allow for the algorithm to have enough time to carry out its computations (for an example of the latter, we refer the reader to Case Study 4.2 in \cite{Bunin2013p}). Let us denote by $\bar {\bf k}$ a set of iterations, with cardinality $N$, where the decision variables are equal to ${\bf u}_{\bar k}$, i.e.:

\vspace{-2mm}
\begin{equation}\label{eq:repset}
\bar {\bf k} \subseteq \{ \tilde k : {\bf u}_{\tilde k} = {\bf u}_{\bar k} \},
\end{equation}

\noindent where the cardinality of the right-hand side will be denoted by $\overline N$, with $N \leq \overline N$.

The corresponding $N$ measurement identities (\ref{eq:costnoise}) then become

\vspace{-2mm}
\begin{equation}\label{eq:costnoiserep}
\hat \phi_p({\bf u}_{\tilde k},\tau_{\tilde k}) = \phi_p({\bf u}_{\bar k},\tau_{\tilde k}) + w_{\phi,\tilde k}, \;\; \forall \tilde k \in \bar {\bf k}.
\end{equation}

To proceed further, let us recall the bounds (\ref{eq:costupperbound2})-(\ref{eq:costlowerbound2}) and state their more general (local) versions for the two iterations $\bar k, \tilde k \in [0,k]$ over the space $\mathcal{I}_\tau^{\bar k,\tilde k}$ as defined in (\ref{eq:locspace2}):

\vspace{-2mm}
\begin{equation}\label{eq:costupperbound2loc}
\begin{array}{l}
\displaystyle \phi_p ({\bf u}_{\bar k},\tau_{\tilde k}) \leq \phi_p ({\bf u}_{\bar k},\tau_{\bar k}) \vspace{1mm} \\
\hspace{20mm} \displaystyle + \eta_{c,\phi}^{\bar k, \tilde k} \frac{\partial \phi_{p}}{\partial \tau} \Big |_{({\bf u}_{\bar k},\tau_{\bar k})} ( \tau_{\tilde k} - \tau_{{\bar k}} ) \vspace{1mm} \\
\hspace{20mm}+ (1-  \eta_{c,\phi}^{\bar k, \tilde k}) \mathop {\max} \left[ \begin{array}{l} \underline \kappa_{\phi,\tau}^{\bar k, \tilde k} \left( \tau_{\tilde k} - \tau_{\bar k} \right), \vspace{1mm} \\ \overline \kappa_{\phi,\tau}^{\bar k, \tilde k} \left( \tau_{\tilde k} - \tau_{\bar k} \right) \end{array} \right],
\end{array}
\end{equation}

\vspace{-2mm}
\begin{equation}\label{eq:costlowerbound2loc}
\begin{array}{l}
\displaystyle \phi_p ({\bf u}_{\bar k},\tau_{\tilde k}) \geq \phi_p ({\bf u}_{\bar k},\tau_{\bar k}) \vspace{1mm} \\
\hspace{20mm} \displaystyle +  \eta_{v,\phi}^{\bar k, \tilde k} \frac{\partial \phi_{p}}{\partial \tau} \Big |_{({\bf u}_{\bar k},\tau_{\bar k})} ( \tau_{\tilde k} - \tau_{{\bar k}} ) \vspace{1mm} \\
\hspace{20mm}+ (1- \eta_{v,\phi}^{\bar k, \tilde k}) \mathop {\min} \left[ \begin{array}{l} \underline \kappa_{\phi,\tau}^{\bar k, \tilde k} \left( \tau_{\tilde k} - \tau_{\bar k} \right), \vspace{1mm} \\ \overline \kappa_{\phi,\tau}^{\bar k, \tilde k} \left( \tau_{\tilde k} - \tau_{\bar k} \right) \end{array} \right].
\end{array}
\end{equation}

Applying these bounds to (\ref{eq:costnoiserep}) then yields

\vspace{-2mm}
\begin{equation}\label{eq:costnoiserep2}
\begin{array}{l}
\displaystyle \phi_p ({\bf u}_{\bar k},\tau_{\bar k}) +  \eta_{v,\phi}^{\bar k, \tilde k} \frac{\partial \phi_{p}}{\partial \tau} \Big |_{({\bf u}_{\bar k},\tau_{\bar k})} ( \tau_{\tilde k} - \tau_{{\bar k}} )  \vspace{1mm} \\
+ (1-  \eta_{v,\phi}^{\bar k, \tilde k}) \mathop {\min} \left[ \begin{array}{l} \underline \kappa_{\phi,\tau}^{\bar k, \tilde k} \left( \tau_{\tilde k} - \tau_{\bar k} \right), \vspace{1mm} \\ \overline \kappa_{\phi,\tau}^{\bar k, \tilde k} \left( \tau_{\tilde k} - \tau_{\bar k} \right) \end{array} \right] + w_{\phi,\tilde k} \vspace{1mm} \\
\hspace{7mm} \leq \hat \phi_p ({\bf u}_{\tilde k},\tau_{\tilde k}) \leq \vspace{1mm} \\
\hspace{12mm}\displaystyle \phi_p ({\bf u}_{\bar k},\tau_{\bar k})  + \eta_{c,\phi}^{\bar k, \tilde k} \frac{\partial \phi_{p}}{\partial \tau} \Big |_{({\bf u}_{\bar k},\tau_{\bar k})} ( \tau_{\tilde k} - \tau_{{\bar k}} ) \vspace{1mm} \\
\hspace{12mm}+ (1- \eta_{c,\phi}^{\bar k, \tilde k}) \mathop {\max} \left[ \begin{array}{l} \underline \kappa_{\phi,\tau}^{\bar k, \tilde k} \left( \tau_{\tilde k} - \tau_{\bar k} \right), \vspace{1mm} \\ \overline \kappa_{\phi,\tau}^{\bar k, \tilde k} \left( \tau_{\tilde k} - \tau_{\bar k} \right) \end{array} \right] + w_{\phi,\tilde k},
\end{array}
\end{equation}

\noindent which, upon rearrangement with respect to $\phi_p ({\bf u}_{\bar k},\tau_{\bar k})$, yields

\vspace{-2mm}
\begin{equation}\label{eq:costnoiserep3}
\begin{array}{l}
\displaystyle \hat \phi_p ({\bf u}_{\tilde k},\tau_{\tilde k}) - \eta_{c,\phi}^{\bar k, \tilde k} \frac{\partial \phi_{p}}{\partial \tau} \Big |_{({\bf u}_{\bar k},\tau_{\bar k})} ( \tau_{\tilde k} - \tau_{{\bar k}} ) \vspace{1mm} \\
- (1-  \eta_{c,\phi}^{\bar k, \tilde k}) \mathop {\max} \left[ \begin{array}{l} \underline \kappa_{\phi,\tau}^{\bar k, \tilde k} \left( \tau_{\tilde k} - \tau_{\bar k} \right), \vspace{1mm} \\ \overline \kappa_{\phi,\tau}^{\bar k, \tilde k} \left( \tau_{\tilde k} - \tau_{\bar k} \right) \end{array} \right] - w_{\phi,\tilde k} \vspace{1mm} \\
\hspace{7mm} \leq \phi_p ({\bf u}_{\bar k},\tau_{\bar k}) \leq \vspace{1mm} \\
\hspace{12mm}\displaystyle \hat \phi_p ({\bf u}_{\tilde k},\tau_{\tilde k})  -  \eta_{v,\phi}^{\bar k, \tilde k} \frac{\partial \phi_{p}}{\partial \tau} \Big |_{({\bf u}_{\bar k},\tau_{\bar k})} ( \tau_{\tilde k} - \tau_{{\bar k}} ) \vspace{1mm} \\
\hspace{12mm}- (1-  \eta_{v,\phi}^{\bar k, \tilde k}) \mathop {\min} \left[ \begin{array}{l} \underline \kappa_{\phi,\tau}^{\bar k, \tilde k} \left( \tau_{\tilde k} - \tau_{\bar k} \right), \vspace{1mm} \\ \overline \kappa_{\phi,\tau}^{\bar k, \tilde k} \left( \tau_{\tilde k} - \tau_{\bar k} \right) \end{array} \right] - w_{\phi,\tilde k}.
\end{array}
\end{equation}

Noting that (\ref{eq:costnoiserep3}) represents $N$ pairs of inequalities, we may sum them up to obtain

\vspace{-2mm}
\begin{equation}\label{eq:costnoiserep4}
\begin{array}{l}
\displaystyle \sum_{\tilde k \in \bar {\bf k}} \hat \phi_p ({\bf u}_{\tilde k},\tau_{\tilde k}) - \sum_{\tilde k \in \bar {\bf k}}  \eta_{c,\phi}^{\bar k, \tilde k} \frac{\partial \phi_{p}}{\partial \tau} \Big |_{({\bf u}_{\bar k},\tau_{\bar k})} ( \tau_{\tilde k} - \tau_{{\bar k}} ) \vspace{1mm} \\
\displaystyle - \sum_{\tilde k \in \bar {\bf k}} (1-  \eta_{c,\phi}^{\bar k, \tilde k}) \mathop {\max} \left[ \begin{array}{l} \underline \kappa_{\phi,\tau}^{\bar k, \tilde k} \left( \tau_{\tilde k} - \tau_{\bar k} \right), \vspace{1mm} \\ \overline \kappa_{\phi,\tau}^{\bar k, \tilde k} \left( \tau_{\tilde k} - \tau_{\bar k} \right) \end{array} \right] - \sum_{\tilde k \in \bar {\bf k}} w_{\phi,\tilde k} \vspace{1mm} \\
\hspace{7mm} \leq N \phi_p ({\bf u}_{\bar k},\tau_{\bar k}) \leq \vspace{1mm} \\
\hspace{2mm}\displaystyle \sum_{\tilde k \in \bar {\bf k}} \hat \phi_p ({\bf u}_{\tilde k},\tau_{\tilde k})  - \sum_{\tilde k \in \bar {\bf k}}  \eta_{v,\phi}^{\bar k, \tilde k} \frac{\partial \phi_{p}}{\partial \tau} \Big |_{({\bf u}_{\bar k},\tau_{\bar k})} ( \tau_{\tilde k} - \tau_{{\bar k}} ) \vspace{1mm} \\
\displaystyle \hspace{2mm}- \sum_{\tilde k \in \bar {\bf k}} (1-  \eta_{v,\phi}^{\bar k, \tilde k}) \mathop {\min} \left[ \begin{array}{l} \underline \kappa_{\phi,\tau}^{\bar k, \tilde k} \left( \tau_{\tilde k} - \tau_{\bar k} \right), \vspace{1mm} \\ \overline \kappa_{\phi,\tau}^{\bar k, \tilde k} \left( \tau_{\tilde k} - \tau_{\bar k} \right) \end{array} \right] - \sum_{\tilde k \in \bar {\bf k}} w_{\phi,\tilde k},
\end{array}
\end{equation}

\noindent which then yields

\vspace{-2mm}
\begin{equation}\label{eq:costnoiserep5}
\begin{array}{l}
\displaystyle \frac{1}{N} \sum_{\tilde k \in \bar {\bf k}} \hat \phi_p ({\bf u}_{\tilde k},\tau_{\tilde k}) \vspace{1mm} \\
\displaystyle - \frac{1}{N} \sum_{\tilde k \in \bar {\bf k}} \eta_{c,\phi}^{\bar k, \tilde k} \frac{\partial \phi_{p}}{\partial \tau} \Big |_{({\bf u}_{\bar k},\tau_{\bar k})} ( \tau_{\tilde k} - \tau_{{\bar k}} ) \vspace{1mm} \\
\displaystyle - \frac{1}{N} \sum_{\tilde k \in \bar {\bf k}} (1-  \eta_{c,\phi}^{\bar k, \tilde k}) \mathop {\max} \left[ \begin{array}{l} \underline \kappa_{\phi,\tau}^{\bar k, \tilde k} \left( \tau_{\tilde k} - \tau_{\bar k} \right), \vspace{1mm} \\ \overline \kappa_{\phi,\tau}^{\bar k, \tilde k} \left( \tau_{\tilde k} - \tau_{\bar k} \right) \end{array} \right] \vspace{1mm} \\
\displaystyle - \frac{1}{N} \sum_{\tilde k \in \bar {\bf k}} w_{\phi,\tilde k} \vspace{1mm} \\
\hspace{7mm} \leq \phi_p ({\bf u}_{\bar k},\tau_{\bar k}) \leq \vspace{1mm} \\
\hspace{12mm}\displaystyle \frac{1}{N} \sum_{\tilde k \in \bar {\bf k}} \hat \phi_p ({\bf u}_{\tilde k},\tau_{\tilde k}) \vspace{1mm}  \\
\hspace{12mm} \displaystyle - \frac{1}{N} \sum_{\tilde k \in \bar {\bf k}} \eta_{v,\phi}^{\bar k, \tilde k} \frac{\partial \phi_{p}}{\partial \tau} \Big |_{({\bf u}_{\bar k},\tau_{\bar k})} ( \tau_{\tilde k} - \tau_{{\bar k}} ) \vspace{1mm} \\
\displaystyle \hspace{12mm}- \frac{1}{N} \sum_{\tilde k \in \bar {\bf k}} (1-  \eta_{v,\phi}^{\bar k, \tilde k}) \mathop {\min} \left[ \begin{array}{l} \underline \kappa_{\phi,\tau}^{\bar k, \tilde k} \left( \tau_{\tilde k} - \tau_{\bar k} \right), \vspace{1mm} \\ \overline \kappa_{\phi,\tau}^{\bar k, \tilde k} \left( \tau_{\tilde k} - \tau_{\bar k} \right) \end{array} \right] \vspace{1mm} \\
\hspace{12mm} \displaystyle - \frac{1}{N} \sum_{\tilde k \in \bar {\bf k}} w_{\phi,\tilde k}.
\end{array}
\end{equation}

In a manner analogous to (\ref{eq:probbound1}), we may define bounds, $\underline W_{\phi,\bar k}$ and $\overline W_{\phi, \bar k}$, satisfying the condition

\vspace{-2mm}
\begin{equation}\label{eq:probbound1multi}
\mathbb{P} \left( \underline W_{\phi,\bar k} < \frac{1}{N} \sum_{\tilde k \in \bar {\bf k}} w_{\phi,\tilde k} < \overline W_{\phi,\bar k} \right) = 0.99,
\end{equation}

\noindent which finally allows us to state:

\vspace{-2mm}
\begin{equation}\label{eq:costbound2}
\begin{array}{l}
\displaystyle \frac{1}{N} \sum_{\tilde k \in \bar {\bf k}} \hat \phi_p ({\bf u}_{\tilde k},\tau_{\tilde k})  - \frac{1}{N} \sum_{\tilde k \in \bar {\bf k}}  \eta_{c,\phi}^{\bar k, \tilde k} \frac{\partial \phi_{p}}{\partial \tau} \Big |_{({\bf u}_{\bar k},\tau_{\bar k})} ( \tau_{\tilde k} - \tau_{{\bar k}} ) \vspace{2mm} \\
\displaystyle - \frac{1}{N} \sum_{\tilde k \in \bar {\bf k}} (1-  \eta_{c,\phi}^{\bar k, \tilde k}) \mathop {\max} \left[ \begin{array}{l} \underline \kappa_{\phi,\tau}^{\bar k, \tilde k} \left( \tau_{\tilde k} - \tau_{\bar k} \right), \vspace{1mm} \\ \overline \kappa_{\phi,\tau}^{\bar k, \tilde k} \left( \tau_{\tilde k} - \tau_{\bar k} \right) \end{array} \right]  - \overline W_{\phi, \bar k} \vspace{2mm} \\
\hspace{0mm} \leq \phi_p ({\bf u}_{\bar k},\tau_{\bar k}) \leq \vspace{2mm} \\
\displaystyle \frac{1}{N} \sum_{\tilde k \in \bar {\bf k}} \hat \phi_p ({\bf u}_{\tilde k},\tau_{\tilde k})  -  \frac{1}{N} \sum_{\tilde k \in \bar {\bf k}}  \eta_{v,\phi}^{\bar k, \tilde k} \frac{\partial \phi_{p}}{\partial \tau} \Big |_{({\bf u}_{\bar k},\tau_{\bar k})} ( \tau_{\tilde k} - \tau_{{\bar k}} ) \vspace{2mm} \\
- \displaystyle \frac{1}{N} \sum_{\tilde k \in \bar {\bf k}} (1-  \eta_{v,\phi}^{\bar k, \tilde k}) \mathop {\min} \left[ \begin{array}{l} \underline \kappa_{\phi,\tau}^{\bar k, \tilde k} \left( \tau_{\tilde k} - \tau_{\bar k} \right), \vspace{1mm} \\ \overline \kappa_{\phi,\tau}^{\bar k, \tilde k} \left( \tau_{\tilde k} - \tau_{\bar k} \right) \end{array} \right]  - \underline W_{\phi, \bar k}.
\end{array}
\end{equation}
\vspace{-3mm}

As before, $\underline W_{\phi, \bar k}$ and $\overline W_{\phi, \bar k}$ may be obtained in different ways depending on what is known about the noise. For the case when only the mean and variance are known, one may turn to generalized multivariate versions of Chebyshev's inequality \cite{Birnbaum1947}. For the known general pdf, Monte Carlo methods may be used to compute these bounds by numerically generating the pdf of the summation. For the oft-assumed case of white Gaussian noise, we have the well known variance reduction of $\sigma^2/N$ and the high-probability bounds of $\underline W_{\phi, \bar k} := -3\sigma/\sqrt{N}$ and $\overline W_{\phi,\bar k} := 3\sigma/\sqrt{N}$.

For the particular case when $\overline N=1$ (and $\bar {\bf k} = \bar k$), we see that (\ref{eq:costbound2}) simply yields (\ref{eq:costbound1}). Otherwise, different bounds are obtained. The motivation for using (\ref{eq:costbound2}) comes from its ability to filter out the variance of the noise as $N \rightarrow \infty$, although it is possible that this offers no advantage in the presence of significant degradation, and it is not difficult to show that the lower and upper bounds of (\ref{eq:costbound2}) may actually \emph{diverge} as $N \rightarrow \infty$ when degradation effects are present due to the corresponding summations being potentially unbounded. When degradation is negligible, one does expect the lower and upper bounds to converge to the true value, however. The white Gaussian noise case is a simple illustration as both $\underline W_{\phi,\bar k}$ and $\overline W_{\phi, \bar k}$ would tend to 0 as $N \rightarrow \infty$ and the mean of the measured values would tend to the true value. 

It is also interesting to note that one may essentially obtain $\displaystyle \sum_{N=1}^{\overline N} \left( \begin{array}{c} \overline N \\ N \end{array} \right)$ different bounds by playing with the definition of $\bar {\bf k}$ and considering all possible sets of the repeated measurements, and while taking the largest set possible -- with $N := \overline N$ -- may seem like the intuitive best option, this need not be the case. As such, one may enumerate all the possible $\bar {\bf k}$ and take the tightest bounds that result.

\subsection{Refinement Using Lipschitz Bounds}

Let us now denote the best lower and upper bounds obtained by (\ref{eq:costbound2}) -- considering all possible $\bar {\bf k}$, as these also include (\ref{eq:costbound1}) -- as

\vspace{-4mm}
\begin{equation}\label{eq:bestbounds}
\underline \phi_p ({\bf u}_{\bar k},\tau_{\bar k}) < \phi_p ({\bf u}_{\bar k},\tau_{\bar k}) < \overline \phi_p ({\bf u}_{\bar k},\tau_{\bar k}),
\end{equation}

\noindent where $\underline \phi_p ({\bf u}_{\bar k},\tau_{\bar k})$ is simply the maximum of the left-hand sides of (\ref{eq:costbound2}) for different $\bar {\bf k}$, while $\overline \phi_p ({\bf u}_{\bar k},\tau_{\bar k})$ is the minimum of the right-hand sides.

At this point, we may attempt to tighten these bounds further by coupling them with the local concavity/convexity-refined Lipschitz bounds that are valid for a given two iterations $\bar k$ and $\tilde k$:

\vspace{-4mm}
\begin{equation}\label{eq:lipboundcostU}
\begin{array}{l}
\phi_{p} ({\bf u}_{\bar k},\tau_{\bar k}) \leq \phi_{p} ({\bf u}_{\tilde k},\tau_{\tilde k}) \vspace{.7mm} \\
\displaystyle \hspace{15mm} + \eta_{c,\phi}^{\bar k, \tilde k} \frac{\partial \phi_{p}}{\partial \tau} \Big |_{({\bf u}_{\tilde k},\tau_{\tilde k})} ( \tau_{\bar k} - \tau_{{\tilde k}} ) \vspace{.7mm} \\
\displaystyle  \hspace{15mm} + (1-\eta_{c,\phi}^{\bar k, \tilde k})  \mathop {\max} \left[ \begin{array}{l} \underline \kappa_{\phi,\tau}^{\bar k, \tilde k} ( \tau_{\bar k} - \tau_{\tilde k} ), \vspace{.7mm} \\ \overline \kappa_{\phi,\tau}^{\bar k, \tilde k} ( \tau_{\bar k} - \tau_{\tilde k} ) \end{array} \right] \vspace{.7mm} \\
 \displaystyle \hspace{15mm} + \sum_{i \in I_{c,\phi}^{\bar k, \tilde k}} \frac{\partial \phi_p}{\partial u_i} \Big |_{({\bf u}_{\tilde k},\tau_{\tilde k})} ( u_{\bar k,i} - u_{\tilde k,i} ) \vspace{.7mm} \\
\displaystyle \hspace{15mm} + \sum_{i \not \in I_{c,\phi}^{\bar k, \tilde k}} \mathop {\max} \left[ \begin{array}{l} \underline \kappa_{\phi,i}^{\bar k, \tilde k} ( u_{\bar k,i} - u_{\tilde k,i} ), \vspace{.7mm} \\ \overline \kappa_{\phi,i}^{\bar k, \tilde k} ( u_{\bar k,i} - u_{\tilde k,i} ) \end{array} \right],
\end{array}
\end{equation}

\vspace{-2mm}
\begin{equation}\label{eq:lipboundcostL}
\begin{array}{l}
\phi_{p} ({\bf u}_{\bar k},\tau_{\bar k}) \geq \phi_{p} ({\bf u}_{\tilde k},\tau_{\tilde k}) \vspace{.7mm} \\
\displaystyle \hspace{15mm} + \eta_{v,\phi}^{\bar k, \tilde k} \frac{\partial \phi_{p}}{\partial \tau} \Big |_{({\bf u}_{\tilde k},\tau_{\tilde k})} ( \tau_{\bar k} - \tau_{{\tilde k}} ) \vspace{.7mm} \\
\displaystyle  \hspace{15mm} + (1-\eta_{v,\phi}^{\bar k, \tilde k})  \mathop {\min} \left[ \begin{array}{l} \underline \kappa_{\phi,\tau}^{\bar k, \tilde k} ( \tau_{\bar k} - \tau_{\tilde k} ), \vspace{.7mm} \\ \overline \kappa_{\phi,\tau}^{\bar k, \tilde k} ( \tau_{\bar k} - \tau_{\tilde k} ) \end{array} \right] \vspace{.7mm} \\
 \displaystyle \hspace{15mm} + \sum_{i \in I_{v,\phi}^{\bar k, \tilde k}} \frac{\partial \phi_p}{\partial u_i} \Big |_{({\bf u}_{\tilde k},\tau_{\tilde k})} ( u_{\bar k,i} - u_{\tilde k,i} ) \vspace{.7mm} \\
\displaystyle \hspace{15mm} + \sum_{i \not \in I_{v,\phi}^{\bar k, \tilde k}} \mathop {\min} \left[ \begin{array}{l} \underline \kappa_{\phi,i}^{\bar k, \tilde k} ( u_{\bar k,i} - u_{\tilde k,i} ), \vspace{.7mm} \\ \overline \kappa_{\phi,i}^{\bar k, \tilde k} ( u_{\bar k,i} - u_{\tilde k,i} ) \end{array} \right].
\end{array}
\end{equation}

From (\ref{eq:bestbounds}), it should be clear that

\vspace{-3mm}
\begin{equation}\label{eq:lipboundcostU2}
\begin{array}{l}
\phi_{p} ({\bf u}_{\bar k},\tau_{\bar k}) < \overline \phi_{p} ({\bf u}_{\tilde k},\tau_{\tilde k}) \vspace{.7mm} \\
\displaystyle \hspace{15mm} + \eta_{c,\phi}^{\bar k, \tilde k} \frac{\partial \phi_{p}}{\partial \tau} \Big |_{({\bf u}_{\tilde k},\tau_{\tilde k})} ( \tau_{\bar k} - \tau_{{\tilde k}} ) \vspace{.7mm} \\
\displaystyle  \hspace{15mm} + (1-\eta_{c,\phi}^{\bar k, \tilde k})  \mathop {\max} \left[ \begin{array}{l} \underline \kappa_{\phi,\tau}^{\bar k, \tilde k} ( \tau_{\bar k} - \tau_{\tilde k} ), \vspace{.7mm} \\ \overline \kappa_{\phi,\tau}^{\bar k, \tilde k} ( \tau_{\bar k} - \tau_{\tilde k} ) \end{array} \right] \vspace{.7mm}  \\
 \displaystyle \hspace{15mm} + \sum_{i \in I_{c,\phi}^{\bar k, \tilde k}} \frac{\partial \phi_p}{\partial u_i} \Big |_{({\bf u}_{\tilde k},\tau_{\tilde k})} ( u_{\bar k,i} - u_{\tilde k,i} ) \vspace{.7mm} \\
\displaystyle \hspace{15mm} + \sum_{i \not \in I_{c,\phi}^{\bar k, \tilde k}} \mathop {\max} \left[ \begin{array}{l} \underline \kappa_{\phi,i}^{\bar k, \tilde k} ( u_{\bar k,i} - u_{\tilde k,i} ), \vspace{.7mm} \\ \overline \kappa_{\phi,i}^{\bar k, \tilde k} ( u_{\bar k,i} - u_{\tilde k,i} ) \end{array} \right],
\end{array}
\end{equation}

\vspace{-2mm}
\begin{equation}\label{eq:lipboundcostL2}
\begin{array}{l}
\phi_{p} ({\bf u}_{\bar k},\tau_{\bar k}) > \underline \phi_{p} ({\bf u}_{\tilde k},\tau_{\tilde k}) \vspace{.7mm} \\
\displaystyle \hspace{15mm} + \eta_{v,\phi}^{\bar k, \tilde k} \frac{\partial \phi_{p}}{\partial \tau} \Big |_{({\bf u}_{\tilde k},\tau_{\tilde k})} ( \tau_{\bar k} - \tau_{{\tilde k}} ) \vspace{.7mm} \\
\displaystyle  \hspace{15mm} + (1-\eta_{v,\phi}^{\bar k, \tilde k})  \mathop {\min} \left[ \begin{array}{l} \underline \kappa_{\phi,\tau}^{\bar k, \tilde k} ( \tau_{\bar k} - \tau_{\tilde k} ), \vspace{.7mm} \\ \overline \kappa_{\phi,\tau}^{\bar k, \tilde k} ( \tau_{\bar k} - \tau_{\tilde k} ) \end{array} \right] \vspace{.7mm} \\
 \displaystyle \hspace{15mm} + \sum_{i \in I_{v,\phi}^{\bar k, \tilde k}} \frac{\partial \phi_p}{\partial u_i} \Big |_{({\bf u}_{\tilde k},\tau_{\tilde k})} ( u_{\bar k,i} - u_{\tilde k,i} ) \vspace{.7mm} \\
\displaystyle \hspace{15mm} + \sum_{i \not \in I_{v,\phi}^{\bar k, \tilde k}} \mathop {\min} \left[ \begin{array}{l} \underline \kappa_{\phi,i}^{\bar k, \tilde k} ( u_{\bar k,i} - u_{\tilde k,i} ), \vspace{.7mm} \\ \overline \kappa_{\phi,i}^{\bar k, \tilde k} ( u_{\bar k,i} - u_{\tilde k,i} ) \end{array} \right]
\end{array}
\end{equation}

\noindent hold as well, which yields an additional set of $k$ pairs of lower/upper bounds that are themselves a function of the lower/upper bounds obtained for the other $k$ measurements.

As illustrated in Fig. \ref{fig:liprefine}, these bounds have the potential to carry a very nice ``chain effect'', in that a tight bound for one measurement can be effectively propagated to nearby measurements for which the initial bounds may not be as tight. This depends, of course, on using valid Lipschitz constants, although we note that one is not obliged to use the same constants here as for the SCFO in general, and more conservative values may be chosen specifically for this refinement without having the same detrimental effects on performance as they might if this were done for the SCFO. Additionally, as refining certain bounds can lead to better refinement in others, this bounding procedure should be iterated until no more improvement is obtained.

\begin{figure*}
\begin{center}
\includegraphics[width=12cm]{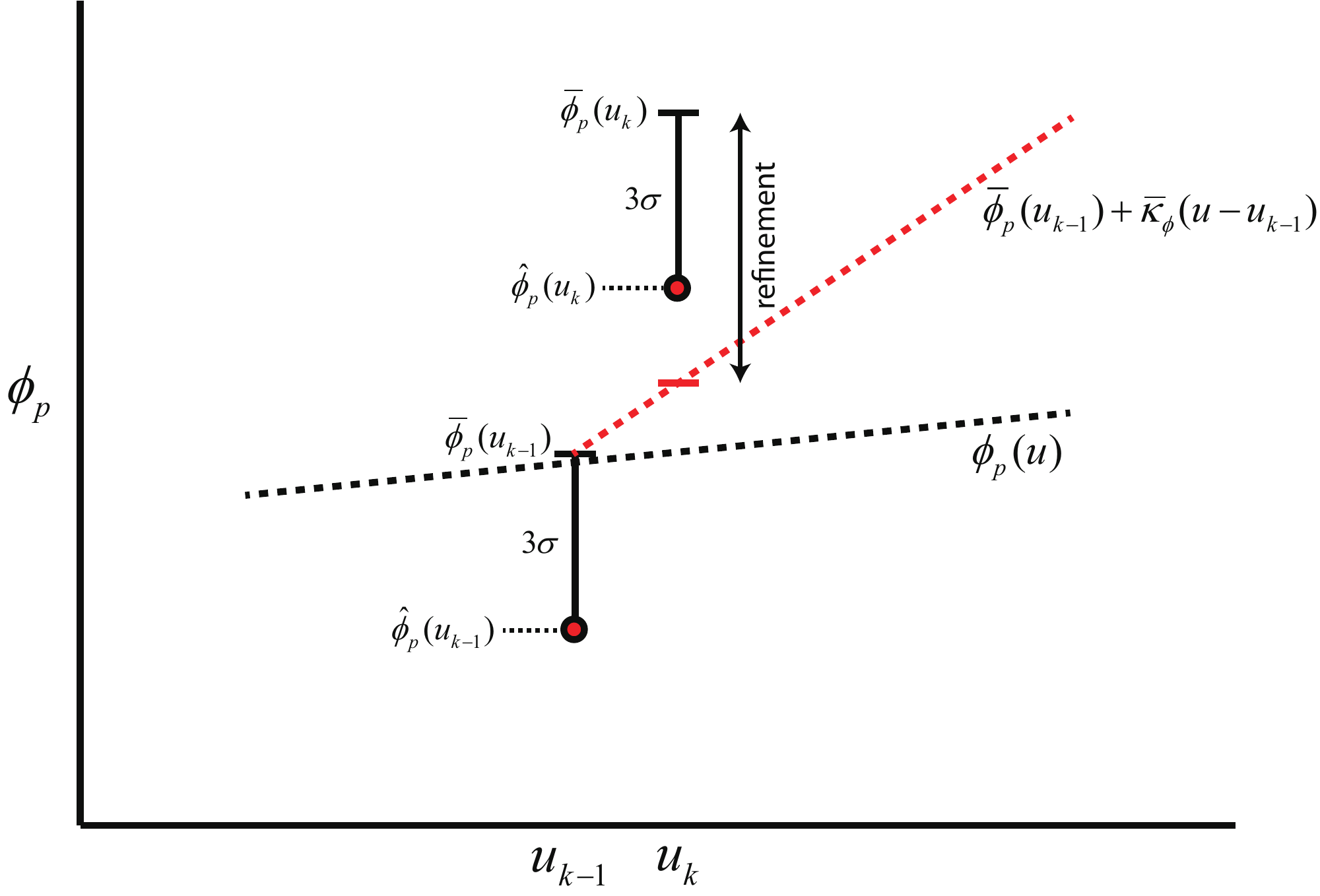}
\caption{An example of how a bound on the true experimental function value may be refined by exploiting the Lipschitz bound. A simple one-dimensional case with neither degradation effects nor convexity/concavity refinements is shown here, with a strong negative corruption due to noise in the measurement at $u_{k-1}$ and a strong positive corruption due to noise at $u_k$. White zero-mean Gaussian noise is assumed, and only the upper bound is considered. If the bound (\ref{eq:costbound1}) is used independently for both $u_{k-1}$ and $u_k$, one sees that the upper bound for $\phi_p (u_k)$ ends up being very high. However, as the upper bound on $\phi_p(u_{k-1})$ is relatively tight, this tightness may be propagated to the next point over with the help of the Lipschitz bound, and a much better upper bound may be obtained.}
\label{fig:liprefine}
\end{center}
\end{figure*}

\subsection{An Algorithm to Compute Lower and Upper Bounds for Experimental Function Values}

To consolidate the results presented so far, we now provide an algorithm that uses all of the techniques presented to obtain the best possible lower and upper bounds on the true values $\phi_p ({\bf u}_{\bar k},\tau_{\bar k})$ for all $\bar k \in [0,k]$.
\newline
\newline
{\bf Algorithm 2 -- Lower and Upper Bounds for Experimental Function Values}
\begin{enumerate}
\item Set $\bar k := 0$. Choose $\Delta_r, \underline \Delta_r \in \mathbb{R}_{++}$ such that $\Delta_r > \underline \Delta_r$.
\item If $\bar k > k$, go to Step 6. Otherwise, proceed to Step 3.
\item Let $\overline N$ denote the number of iterates with the decision variable values ${\bf u}_{\bar k}$. Set $\underline \phi_p ({\bf u}_{\bar k},\tau_{\bar k}) := -\infty$ and $\overline \phi_p ({\bf u}_{\bar k},\tau_{\bar k}) := \infty$. 
\item For $N = 1,...,\overline N$:
\begin{enumerate}
\item Generate the $\left( \begin{array}{c} \overline N \\ N \end{array} \right)$ index sets that correspond to different combinations of iterations with the same decision variable values. For each set:
\begin{enumerate}
\item Define $\bar {\bf k}$ as the corresponding index set and compute the lower and upper bounds as given by (\ref{eq:costbound2}). Set these bounds as the candidate values $\underline \phi_{p,test}$ and $\overline \phi_{p,test}$.
\item If $\underline \phi_{p,test} >  \underline \phi_p ({\bf u}_{\bar k},\tau_{\bar k})$, set $\underline \phi_p ({\bf u}_{\bar k},\tau_{\bar k}) := \underline \phi_{p,test}$. If $\overline \phi_{p,test} < \overline \phi_p ({\bf u}_{\bar k},\tau_{\bar k})$, set $\overline \phi_p ({\bf u}_{\bar k},\tau_{\bar k}) := \overline \phi_{p,test}$.
\end{enumerate}
\end{enumerate}
\item Set $\bar k := \bar k + 1$ and return to Step 2.
\item If $\Delta_r \leq \underline \Delta_r$, terminate. Otherwise, set $\Delta_r := 0, \; \bar k := 0$ and go to Step 7.
\item If $\bar k > k$, return to Step 6. Otherwise, proceed to Step 8.
\item For ${\tilde k} := \{ 0,...,k \} \setminus \bar k$:

\begin{enumerate}
\item Compute the lower and upper bounds as given by (\ref{eq:lipboundcostU2}) and (\ref{eq:lipboundcostL2}). Set these bounds as the candidate values $\underline \phi_{p,test}$ and $\overline \phi_{p,test}$.
\item Set 

\vspace{-2mm}
\begin{equation}\label{eq:Deltar}
\Delta_r := \mathop {\max} \left[  \begin{array}{l} \Delta_r,  \underline \phi_{p,test} - \underline \phi_p ({\bf u}_{\bar k},\tau_{\bar k}), \vspace{1mm} \\ \overline \phi_p ({\bf u}_{\bar k},\tau_{\bar k}) - \overline \phi_{p,test} \end{array} \right].
\end{equation}

\item If $\underline \phi_{p,test} > \underline \phi_p ({\bf u}_{\bar k},\tau_{\bar k})$, set $\underline \phi_p ({\bf u}_{\bar k},\tau_{\bar k}) := \underline \phi_{p,test}$. If $\overline \phi_{p,test} < \overline \phi_p ({\bf u}_{\bar k},\tau_{\bar k})$, set $\overline \phi_p ({\bf u}_{\bar k},\tau_{\bar k}) := \overline \phi_{p,test}$.

\end{enumerate}

\item Set $\bar k := \bar k+1$ and return to Step 7.

\end{enumerate}

Steps 2-5 essentially generate the initial lower and upper bounds, while Steps 6-9 carry out the Lipschitz refinements. For each refinement cycle, $\Delta_r$ represents the best refinement obtained. When this best refinement becomes inferior to some (preferably small) tolerance $\underline \Delta_r$, the algorithm terminates and the resulting lower and upper bounds are outputted. In the case when $\overline N$ is large, an ``incomplete'' version of the algorithm that only considers a limited number of subsets in the definition of $\bar {\bf k}$ may be used so as to avoid computational issues due to the combinatorial nature of $\bar {\bf k}$.

We finish this discussion on obtaining lower/upper bounds by noting that identical procedures may be followed to obtain the corresponding bounds for the constraints, $\underline g_{p,j} ({\bf u}_{\bar k},\tau_{\bar k})$ and $\overline g_{p,j} ({\bf u}_{\bar k},\tau_{\bar k})$.

\subsection{Accounting for Measurement and Estimation Error in the SCFO}

We now go through the modifications necessary to make the SCFO robust against measurement/estimation error.

\subsubsection{Accounting for Error in the Feasibility Conditions}

From $g_{p,j} ({\bf u}_{\bar k}, \tau_{\bar k}) < \overline g_{p,j} ({\bf u}_{\bar k}, \tau_{\bar k})$, it follows that enforcing

\vspace{-2mm}
\begin{equation}\label{eq:SCFO1idegLUccvalllocMN}
\hspace{-4mm}\begin{array}{l}
\mathop {\min} \limits_{\bar k = 0,...,k} \left[ \hspace{-1mm} \begin{array}{l} \overline g_{p,j} ({\bf u}_{\bar k},\tau_{\bar k}) \vspace{1mm} \\
 \displaystyle +\eta_{c,j}^{\bar k} \frac{\partial g_{p,j}}{\partial \tau} \Big |_{({\bf u}_{\bar k},\tau_{\bar k})} ( \tau_{k+1} - \tau_{{\bar k}} ) \vspace{1mm} \\
 \displaystyle + (1-\eta_{c,j}^{\bar k}) \overline \kappa_{p,j\tau}^{\bar k} \left( \tau_{k+1} - \tau_{\bar k} \right) \vspace{1mm}\\
\displaystyle   + \sum_{i \in I_{c,j}^{\bar k}} \frac{\partial g_{p,j}}{\partial u_i} \Big |_{({\bf u}_{\bar k},\tau_{\bar k})} ( u_{k^*,i} + \\
\hspace{15mm}  K_k (\bar u_{k+1,i}^* - u_{k^*,i} ) - u_{{\bar k},i} ) \vspace{1mm} \\
 + \displaystyle \sum_{i \not \in I_{c,j}^{\bar k}} \mathop {\max} \left[ \begin{array}{l} \underline \kappa_{p,ji}^{\bar k} ( u_{k^*,i} + \\ \hspace{2mm} K_k (\bar u_{k+1,i}^* - u_{k^*,i} ) - u_{\bar k,i} ), \vspace{1mm}\\ \overline \kappa_{p,ji}^{\bar k} ( u_{k^*,i} + \\ \hspace{2mm} K_k (\bar u_{k+1,i}^* - u_{k^*,i} ) - u_{\bar k,i} ) \end{array} \right] \end{array} \hspace{-1mm} \right] \leq 0
\end{array}
\end{equation}
\vspace{-1.9mm}

\noindent enforces (\ref{eq:SCFO1idegLUccvallloc}). Naturally, since $g_{p,j} ({\bf u}_{\bar k}, \tau_{\bar k}) < \overline g_{p,j} ({\bf u}_{\bar k}, \tau_{\bar k})$, the price one pays for this form is more conservative steps, as $K_k$ will naturally need to be smaller to satisfy (\ref{eq:SCFO1idegLUccvalllocMN}).

As before, some additional assumptions are needed on the past measurements for (\ref{eq:SCFO1idegLUccvalllocMN}) to be enforcable. Namely, in the general case where the lower and upper Lipschitz constants may have opposite signs and the summation terms are always positive for any positive value of $K_k$, it is sufficient that 

\vspace{-2mm}
\begin{equation}\label{eq:minfeas}
\begin{array}{l}
\overline g_{p,j} ({\bf u}_{k^*},\tau_{k^*}) \displaystyle +\eta_{c,j}^{k^*} \frac{\partial g_{p,j}}{\partial \tau} \Big |_{({\bf u}_{k^*},\tau_{k^*})} ( \tau_{k+1} - \tau_{{k^*}} ) \vspace{1mm} \\
\hspace{20mm} \displaystyle + (1-\eta_{c,j}^{k^*}) \overline \kappa_{p,j\tau}^{k^*} \left( \tau_{k+1} - \tau_{k^*} \right) \leq 0
\end{array}
\end{equation}

\noindent to guarantee that one can choose $K_k$ in a manner so as to guarantee (\ref{eq:SCFO1idegLUccvalllocMN}) and thus $g_{p,j} ({\bf u}_{k+1}, \tau_{k+1}) \leq 0$ -- in the worst case, one can always choose $K_k := 0$.

To avoid premature convergence, the following modified version of the projection (\ref{eq:projdegLUccv}) is also proposed:

\vspace{-2mm}
\begin{equation}\label{eq:projdeg2}
\begin{array}{rl}
\bar {\bf u}_{k+1}^* := \;\;\;\;\;\;\;\;\; & \vspace{1mm}\\
 {\rm arg} \mathop {\rm minimize}\limits_{{\bf u}} & \| {\bf u} - {\bf u}_{k+1}^* \|_2^2  \vspace{1mm} \\
{\rm{subject}}\;{\rm{to}} & \nabla g_{p,j} ({\bf u}_{k^*},\tau_{k+1})^T \left[ \hspace{-1mm} \begin{array}{c} {\bf u} - {\bf u}_{k^*} \\ 0 \end{array} \hspace{-1mm} \right] \leq -\delta_{g_p,j} \vspace{1mm} \\
& \hspace{-6mm} \forall j: \begin{array}{l} \overline g_{p,j} ({\bf u}_{k^*},\tau_{k^*}) \vspace{1mm} \\ \displaystyle + \eta_{c,j} \frac{\partial g_{p,j}}{\partial \tau} \Big |_{({\bf u}_{k^*},\tau_{k^*})} ( \tau_{k+1} - \tau_{{k^*}} ) \vspace{1mm}  \\ + (1 - \eta_{c,j}) \overline \kappa_{p,j\tau} (\tau_{k+1} - \tau_{k^*} ) \geq -\epsilon_{p,j} \end{array} \vspace{1mm} \\
 & \nabla g_{j} ({\bf u}_{k^*})^T ({\bf u} - {\bf u}_{k^*}) \leq -\delta_{g,j} \vspace{1mm} \\
& \forall j : g_{j}({\bf u}_{k^*}) \geq -\epsilon_{j} \vspace{1mm} \\
 & \nabla \phi_{p} ({\bf u}_{k^*},\tau_{k+1})^T  \left[ \begin{array}{c} {\bf u} - {\bf u}_{k^*} \\ 0 \end{array} \right] \leq -\delta_{\phi} \vspace{1mm} \\
 & {\bf u}^L \preceq {\bf u} \preceq {\bf u}^U,
\end{array}
\end{equation}

\noindent where we have simply substituted $\overline g_{p,j} ({\bf u}_{k^*}, \tau_{k^*})$ for $g_{p,j} ({\bf u}_{k^*}, \tau_{k^*})$ in the Boolean condition. Alternatively, one could propose to use the estimate $\hat g_{p,j} ({\bf u}_{k^*}, \tau_{k^*})$, but it is easy to show that this choice may lead to failure for certain noise distributions. For example, the distribution with a mean smaller than $-\epsilon_{p,j}$ and a very small variance would essentially lead to the projection never being applied with respect to the constraint $g_{p,j}$, which could in turn lead to premature convergence if this constraint were to become active.

\subsubsection{Accounting for Error in the Necessary Conditions for Cost Decrease}

To make the condition (\ref{eq:costhighmaxPFloc}) implementable in the presence of errors in the measurement of $ \phi_p ({\bf u}_{\bar k},\tau_{\bar k})$, we propose

\vspace{-2mm}
\begin{equation}\label{eq:costhighmaxPFlocMN}
\begin{array}{l}
\mathop {\max} \limits_{\bar k = 0,...,k}\left[ \begin{array}{l}
\displaystyle \underline \phi_{p} ({\bf u}_{\bar k},\tau_{\bar k}) \vspace{1mm} \\
\displaystyle  +\eta_{v,\phi}^{\bar k} \frac{\partial \phi_{p}}{\partial \tau} \Big |_{({\bf u}_{\bar k},\tau_{\bar k})} ( \tau_{k+1} - \tau_{{\bar k}} ) \vspace{1mm} \\
\displaystyle + (1-\eta_{v,\phi}^{\bar k}) \underline \kappa_{\phi,\tau}^{\bar k} \left( \tau_{k+1} - \tau_{\bar k} \right) \vspace{1mm}\\
\displaystyle   + \sum_{i \in I_{v,\phi}^{\bar k}} \frac{\partial \phi_{p}}{\partial u_i} \Big |_{({\bf u}_{\bar k},\tau_{\bar k})} ( u_{k^*,i} + \\ \hspace{15mm}  K_k(\bar u_{k+1,i}^* - u_{k^*,i}) - u_{{\bar k},i} ) \vspace{1mm} \\
+ \displaystyle \sum_{i \not \in I_{v,\phi}^{\bar k}} \mathop {\min} \left[ \begin{array}{l} \underline \kappa_{\phi,i}^{\bar k} ( u_{k^*,i} +\\ \hspace{2mm} K_k(\bar u_{k+1,i}^* - u_{k^*,i}) - u_{{\bar k},i} ), \vspace{1mm}\\ \overline \kappa_{\phi,i}^{\bar k} ( u_{k^*,i} +\\ \hspace{2mm} K_k(\bar u_{k+1,i}^* - u_{k^*,i}) - u_{{\bar k},i} ) \end{array} \right] 
\end{array} \right] \vspace{2mm} \\
\hspace{8mm}\displaystyle \leq \mathop {\min}_{\tilde k = 0,...,k} \left[ \begin{array}{l} \displaystyle \overline \phi_{p} ({\bf u}_{\tilde k},\tau_{\tilde k}) \vspace{1mm} \\
\displaystyle  +\eta_{c,\phi}^{\tilde k} \frac{\partial \phi_{p}}{\partial \tau} \Big |_{({\bf u}_{\tilde k},\tau_{\tilde k})} ( \tau_{k+1} - \tau_{{\tilde k}} ) \vspace{1mm} \\
\hspace{0mm} \displaystyle + (1-\eta_{c,\phi}^{\tilde k}) \overline \kappa_{\phi,\tau}^{\tilde k} \left( \tau_{k+1} - \tau_{\tilde k} \right) \vspace{1mm}\\
\hspace{0mm}\displaystyle   + \sum_{i \in I_{c,\phi}^{\tilde k}} \frac{\partial \phi_{p}}{\partial u_i} \Big |_{({\bf u}_{\tilde k},\tau_{\tilde k})} ( u_{k^*,i} - u_{{\tilde k},i} ) \vspace{1mm} \\
\hspace{0mm} + \displaystyle \sum_{i \not \in I_{c,\phi}^{\tilde k}} \mathop {\max} \left[ \begin{array}{l} \underline \kappa_{\phi,i}^{\tilde k} ( u_{k^*,i} - u_{{\tilde k},i} ), \vspace{1mm}\\ \overline \kappa_{\phi,i}^{\tilde k} ( u_{k^*,i} - u_{{\tilde k},i} ) \end{array} \right] \end{array} \right]
\end{array}
\end{equation}

\noindent as its robust version, with the necessity of (\ref{eq:costhighmaxPFloc}) implying the necessity of (\ref{eq:costhighmaxPFlocMN}) as the left-hand side of the inequality is made lower and the right-hand side higher.

The result, as should be expected, is that less of the decision variable space can be ruled out in the robust version as larger values of $K_k$ are permitted -- i.e., less of the decision space can be robustly proven to have cost function values superior to the cost at the reference point.

\subsubsection{Accounting for Error in the Choice of Reference Point}

In choosing the reference point via the optimizations in (\ref{eq:kstarLUccvcostloc}) or (\ref{eq:kstar2LUccvloc}), we propose the following robust versions:

\vspace{-2mm}
\begin{equation}\label{eq:kstarLUccvcostlocMN}
\begin{array}{rl}
k^* := \;\;\;\;\;\;\;\;\;\;\;\;\;\;& \vspace{1mm} \\
{\rm arg} \mathop {\rm maximize}\limits_{\bar k \in [0,k]} & \bar k \vspace{1mm} \\
{\rm{subject}}\;{\rm{to}} & \overline g_{p,j} ({\bf u}_{\bar k},\tau_{\bar k}) \vspace{1mm} \\
& \displaystyle +  \eta_{c,j} \frac{\partial g_{p,j}}{\partial \tau} \Big |_{({\bf u}_{\bar k},\tau_{\bar k})} ( \tau_{k+1} - \tau_{{\bar k}} ) \vspace{1mm} \\
&  + (1-  \eta_{c,j}) \overline \kappa_{p,j\tau} \left( \tau_{k+1} - \tau_{\bar k} \right) \leq 0 \vspace{1mm} \\
& \forall j = 1,...,n_{g_p} \vspace{1mm} \\
& \displaystyle \underline \phi_p ({\bf u}_{\bar k},\tau_{\bar k}) +  \eta_{v,\phi}^{\bar k, k} \frac{\partial \phi_{p}}{\partial \tau} \Big |_{({\bf u}_{\bar k},\tau_{\bar k})} ( \tau_{k} - \tau_{{\bar k}} ) \vspace{1mm} \\
&+ (1-  \eta_{v,\phi}^{\bar k, k}) \underline \kappa_{\phi,\tau}^{\bar k, k} \left( \tau_{k} - \tau_{\bar k} \right) \leq  \vspace{1mm} \\
& \mathop {\min} \limits_{\tilde k \in {\bf k}_f} \left[ \begin{array}{l} \overline \phi_p ({\bf u}_{\tilde k},\tau_{\tilde k}) + \vspace{1mm} \\
\displaystyle  \eta_{c,\phi}^{\tilde k, k} \frac{\partial \phi_{p}}{\partial \tau} \Big |_{({\bf u}_{\tilde k},\tau_{\tilde k})} ( \tau_{k} - \tau_{{\tilde k}} ) \vspace{1mm} \\
+ (1-  \eta_{c,\phi}^{\tilde k, k}) \overline \kappa_{\phi,\tau}^{\tilde k, k} \left( \tau_{k} - \tau_{\tilde k} \right) \end{array} \right],
\end{array}
\end{equation}

\vspace{-2mm}
\begin{equation}\label{eq:kfeas3}
{\bf k}_f = \left\{ \bar k : \begin{array}{l} \overline g_{p,j} ({\bf u}_{\bar k},\tau_{\bar k}) \vspace{1mm} \\  + \displaystyle  \eta_{c,j} \frac{\partial g_{p,j}}{\partial \tau} \Big |_{({\bf u}_{\bar k},\tau_{\bar k})} ( \tau_{k+1} - \tau_{{\bar k}} ) \vspace{1mm} \\ + (1- \eta_{c,j}) \overline \kappa_{p,j\tau} \left( \tau_{k+1} - \tau_{\bar k} \right)  \leq 0, \vspace{1mm} \\ \forall j = 1,...,n_{g_p} \end{array} \right\},
\end{equation}

\vspace{-2mm}
\begin{equation}\label{eq:kstar2LUccvlocMN}
\begin{array}{l}
k^*  := \\
\displaystyle {\rm arg} \mathop {\rm minimize}\limits_{\bar k \in [0,k]}  \mathop {\max} \limits_{j = 1,...,n_{g_p}}  \left[  \begin{array}{l} \overline g_{p,j} ({\bf u}_{\bar k},\tau_{\bar k}) \vspace{1mm} \\
\displaystyle +  \eta_{c,j} \frac{\partial g_{p,j}}{\partial \tau} \Big |_{({\bf u}_{\bar k},\tau_{\bar k})} \vspace{1mm} \\
\hspace{15mm}( \tau_{k+1} - \tau_{{\bar k}} ) \vspace{1mm} \\
+ (1-  \eta_{c,j}) \overline \kappa_{p,j\tau} \vspace{1mm} \\
\hspace{15mm}\left( \tau_{k+1} - \tau_{\bar k} \right) \end{array}  \right]
\end{array}
\end{equation}

\noindent where (\ref{eq:kstarLUccvcostlocMN}) chooses the most recent iterate that satisfies (\ref{eq:minfeas}) and is robustly guaranteed to not have a cost value greater than the value at any of the other feasible iterates.

\begin{figure*}
\begin{center}
\includegraphics[width=16cm]{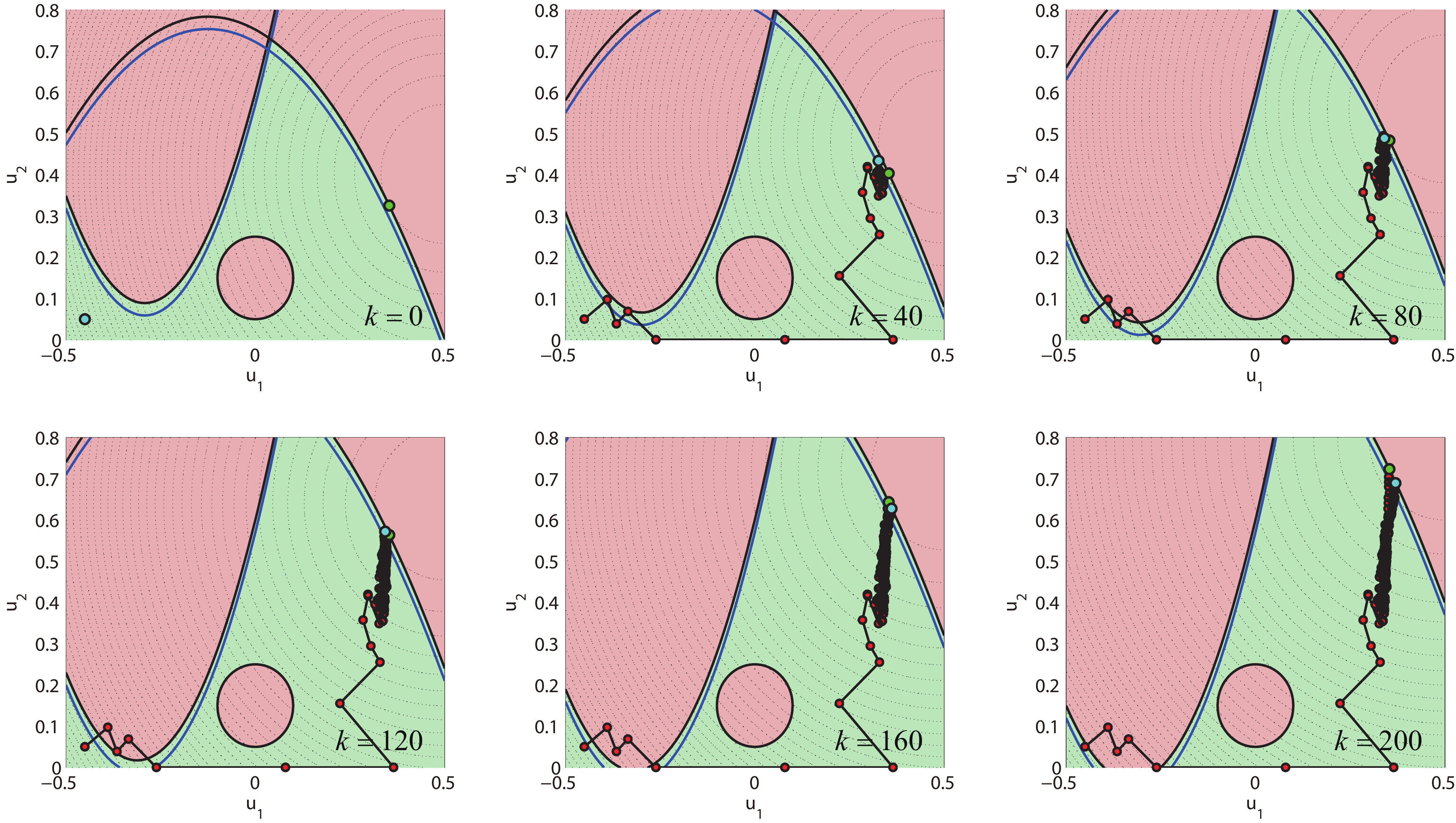}
\caption{Chain of experiments generated by applying the modified SCFO methodology to Problem (\ref{eq:exdeg}) for the ($-$) scenario with measurement/estimation error in the experimental functions accounted for. The effect of the noise in the constraints is illustrated by showing the $3\sigma$ back-offs in blue.}
\label{fig:degG}
\end{center}
\end{figure*} 

\subsubsection{Modifying the Lipschitz Consistency Check}

If Algorithm 1 is to be run in order to verify the consistency of the Lipschitz constants, then one must modify the bounds (\ref{eq:lipcheck1U})-(\ref{eq:lipcheck2L}) to make them implementable in the presence of measurement or estimation error. The modifications we propose here are the following:

\vspace{-2mm}
\begin{equation}\label{eq:lipcheck1UMN}
\begin{array}{l}
\underline g_{p,j} ({\bf u}_{k_2},\tau_{k_2}) \leq \overline g_{p,j} ({\bf u}_{k_1},\tau_{k_1}) \vspace{1mm} \\
\hspace{20mm} \displaystyle + \mathop {\max} \left[ \begin{array}{l} \underline \kappa_{p,j\tau} \left( \tau_{k_2} - \tau_{k_1} \right) \vspace{1mm} \\ 
\overline \kappa_{p,j\tau} \left( \tau_{k_2} - \tau_{k_1} \right) \end{array} \right] \vspace{1mm} \\
\hspace{20mm}\displaystyle + \sum_{i=1}^{n_u} \mathop {\max} \left[ \begin{array}{l} \underline \kappa_{p,ji} ( u_{k_2,i} - u_{k_1,i} ), \vspace{1mm} \\ \overline \kappa_{p,ji} ( u_{k_2,i} - u_{{k_1},i} ) \end{array} \right],
\end{array}
\end{equation}

\vspace{-2mm}
\begin{equation}\label{eq:lipcheck1LMN}
\begin{array}{l}
\overline g_{p,j} ({\bf u}_{k_2},\tau_{k_2}) \geq \underline g_{p,j} ({\bf u}_{k_1},\tau_{k_1}) \vspace{1mm} \\
\hspace{20mm} \displaystyle + \mathop {\min} \left[ \begin{array}{l} \underline \kappa_{p,j\tau} \left( \tau_{k_2} - \tau_{k_1} \right) \vspace{1mm} \\ 
\overline \kappa_{p,j\tau} \left( \tau_{k_2} - \tau_{k_1} \right) \end{array} \right] \vspace{1mm} \\
\hspace{20mm}\displaystyle + \sum_{i=1}^{n_u} \mathop {\min} \left[ \begin{array}{l} \underline \kappa_{p,ji} ( u_{k_2,i} - u_{k_1,i} ), \vspace{1mm} \\ \overline \kappa_{p,ji} ( u_{k_2,i} - u_{{k_1},i} ) \end{array} \right],
\end{array}
\end{equation}

\vspace{-2mm}
\begin{equation}\label{eq:lipcheck1costUMN}
\begin{array}{l}
\underline \phi_{p} ({\bf u}_{k_2},\tau_{k_2}) \leq \overline \phi_{p} ({\bf u}_{k_1},\tau_{k_1}) \vspace{1mm} \\
\hspace{20mm} \displaystyle + \mathop {\max} \left[ \begin{array}{l} \underline \kappa_{\phi,\tau} \left( \tau_{k_2} - \tau_{k_1} \right) \vspace{1mm} \\ 
\overline \kappa_{\phi,\tau} \left( \tau_{k_2} - \tau_{k_1} \right) \end{array} \right] \vspace{1mm} \\
\hspace{20mm}\displaystyle + \sum_{i=1}^{n_u} \mathop {\max} \left[ \begin{array}{l} \underline \kappa_{\phi,i} ( u_{k_2,i} - u_{k_1,i} ), \vspace{1mm} \\ \overline \kappa_{\phi,i} ( u_{k_2,i} - u_{{k_1},i} ) \end{array} \right],
\end{array}
\end{equation}

\vspace{-2mm}
\begin{equation}\label{eq:lipcheck1costLMN}
\begin{array}{l}
\overline \phi_{p} ({\bf u}_{k_2},\tau_{k_2}) \geq \underline \phi_{p} ({\bf u}_{k_1},\tau_{k_1}) \vspace{1mm} \\
\hspace{20mm} \displaystyle + \mathop {\min} \left[ \begin{array}{l} \underline \kappa_{\phi,\tau} \left( \tau_{k_2} - \tau_{k_1} \right) \vspace{1mm} \\ 
\overline \kappa_{\phi,\tau} \left( \tau_{k_2} - \tau_{k_1} \right) \end{array} \right] \vspace{1mm} \\
\hspace{20mm}\displaystyle + \sum_{i=1}^{n_u} \mathop {\min} \left[ \begin{array}{l} \underline \kappa_{\phi,i} ( u_{k_2,i} - u_{k_1,i} ), \vspace{1mm} \\ \overline \kappa_{\phi,i} ( u_{k_2,i} - u_{{k_1},i} ) \end{array} \right],
\end{array}
\end{equation}

\vspace{-2mm}
\begin{equation}\label{eq:lipcheck2UMN}
\begin{array}{l}
\underline \phi_{p} ({\bf u}_{k_2},\tau_{k_2}) \leq \\
\overline \phi_{p} ({\bf u}_{k_1},\tau_{k_1}) + \nabla \phi_p({\bf u}_{k_1},\tau_{k_1})^T \left[ \hspace{-.5mm} \begin{array}{c} {\bf u}_{k_2} - {\bf u}_{k_1} \\ 0 \end{array} \hspace{-.5mm} \right]  \vspace{1mm} \\
+ \mathop {\max} \left[  \begin{array}{l} \underline \kappa_{\phi,\tau} (\tau_{k_2} - \tau_{k_1}), \\ \overline \kappa_{\phi,\tau} (\tau_{k_2} - \tau_{k_1})  \end{array} \right] \vspace{1mm} \\
+\displaystyle \frac{1}{2} \sum_{i_1=1}^{n_u} \sum_{i_2=1}^{n_u} \mathop {\max} \left[ \begin{array}{l} \underline M_{\phi,i_1 i_2} (u_{k_2,i_1} - u_{k_1,i_1}) \\
\hspace{15mm} (u_{k_2,i_2} - u_{{k_1},i_2}), \\ \overline M_{\phi,i_1 i_2} (u_{k_2,i_1} - u_{{k_1},i_1}) \\
\hspace{15mm}(u_{k_2,i_2} - u_{{k_1},i_2}) \end{array} \right],
\end{array}
\end{equation}

\vspace{-2mm}
\begin{equation}\label{eq:lipcheck2LMN}
\begin{array}{l}
\overline \phi_{p} ({\bf u}_{k_2},\tau_{k_2}) \geq \\
\underline \phi_{p} ({\bf u}_{k_1},\tau_{k_1}) + \nabla \phi_p({\bf u}_{k_1},\tau_{k_1})^T \left[ \hspace{-.5mm} \begin{array}{c} {\bf u}_{k_2} - {\bf u}_{k_1} \\ 0 \end{array} \hspace{-.5mm} \right]  \vspace{1mm} \\
+ \mathop {\min} \left[  \begin{array}{l} \underline \kappa_{\phi,\tau} (\tau_{k_2} - \tau_{k_1}), \\ \overline \kappa_{\phi,\tau} (\tau_{k_2} - \tau_{k_1})  \end{array} \right] \vspace{1mm} \\
+\displaystyle \frac{1}{2} \sum_{i_1=1}^{n_u} \sum_{i_2=1}^{n_u} \mathop {\min} \left[ \begin{array}{l} \underline M_{\phi,i_1 i_2} (u_{k_2,i_1} - u_{k_1,i_1}) \\
\hspace{15mm} (u_{k_2,i_2} - u_{{k_1},i_2}), \\ \overline M_{\phi,i_1 i_2} (u_{k_2,i_1} - u_{{k_1},i_1}) \\
\hspace{15mm}(u_{k_2,i_2} - u_{{k_1},i_2}) \end{array} \right],
\end{array}
\end{equation}

\noindent as inconsistency in (\ref{eq:lipcheck1UMN})-(\ref{eq:lipcheck2LMN}) implies inconsistency in (\ref{eq:lipcheck1U})-(\ref{eq:lipcheck2L}). The natural effect of the noise in the consistency check is that less conservative Lipschitz constants suffice for the robust version given by (\ref{eq:lipcheck1UMN})-(\ref{eq:lipcheck2LMN}), as some of the change in function values across the decision variable space are ``absorbed'' by the noise and not attributed to the deterministic changes in the function, the latter being linked to the Lipschitz constants.

It is, however, important to note that the two tasks of checking the consistency of the Lipschitz constants and computing lower and upper bounds become convoluted with these modifications, as the Lipschitz constants will affect the bounds (by Algorithm 2) and the bounds, in turn, may affect the Lipschitz constants (by Algorithm 1). To avoid computing improper bounds when a poor choice of Lipschitz constants is used, it is proposed that the bounds used in (\ref{eq:lipcheck1UMN})-(\ref{eq:lipcheck2LMN}) are those computed by (\ref{eq:costbound1}), as this bound is independent of the Lipschitz constants. Afterwards, Lipschitz constants that are consistent with (\ref{eq:lipcheck1UMN})-(\ref{eq:lipcheck2LMN}) may be used in Algorithm 2.

\subsection{Example}
\label{sec:sec4ex}

Considering again the ($-$) case of Problem (\ref{eq:exdeg}), we corrupt the measurements of the cost and constraints with the additive noise elements

\vspace{-2mm}
\begin{equation}\label{eq:noiseex}
w_\phi, w_1, w_2 \sim \mathcal{N}(0,10^{-4}),
\end{equation}

\noindent and apply a modified version of Algorithm 2 where only the full sets (of cardinality $\overline N$) and the single-element sets (of cardinality 1) are considered in the definition of $\bar {\bf k}$ -- this is done to avoid complexity issues. Using the computed bounds, we then use the modified SCFO implementation as given by (\ref{eq:SCFO1idegLUccvalllocMN}), (\ref{eq:projdeg2}), (\ref{eq:costhighmaxPFlocMN}), and (\ref{eq:kstarLUccvcostlocMN}). Convexity and concavity refinements with $I_{c,1} := \{ 1,2 \}$, $\eta_{c,1} := 0$, $I_{c,2} := \{ 2 \}$, and $\eta_{c,2} := 1$ are used for the constraints, with the cost refined via $I_{v,\phi} := \{ 1,2 \}$ and $\eta_{v,\phi} := 1$. The same local refinements as given in (\ref{eq:liploc1})-(\ref{eq:liploc5}) are used. The noise bounds $\underline W_{\phi, \bar k} := -3 \sigma_{\bar k}/\sqrt{\overline N}$ and $\overline W_{\phi, \bar k} := 3\sigma_{\bar k}/\sqrt{\overline N}$ are employed, with $\underline \Delta_r := 10^{-6}$ in Algorithm 2.

We give the results in Figs. \ref{fig:degG} and \ref{fig:degGcost}, from which we can see that the modified SCFO are able to obtain good performance without violating the constraints even in the presence of measurement/estimation error.

\begin{figure}
\begin{center}
\includegraphics[width=8cm]{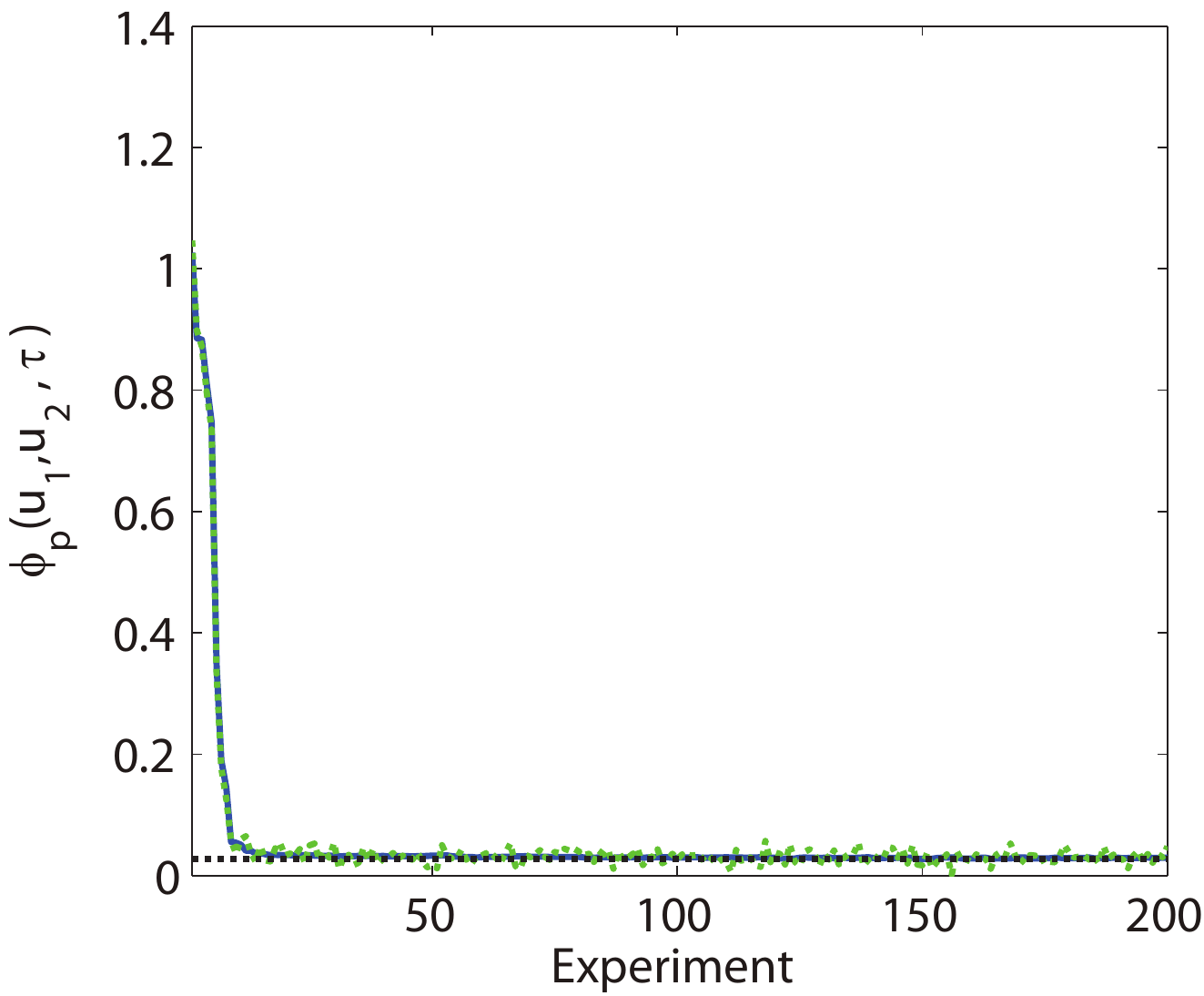}
\caption{Cost function values obtained by the modified SCFO methodology for Problem (\ref{eq:exdeg}) for the ($-$) scenario with measurement/estimation error in the experimental functions accounted for. The effect of the noise in the cost is illustrated via the dotted green line, which represents the corrupted function values.}
\label{fig:degGcost}
\end{center}
\end{figure}

\section{Gradient Estimation and Robust Projection}
\label{sec:gradest}

As an extension to the issues discussed in the previous section, let us now discuss the handling of errors in the \emph{gradient estimates} $\nabla \hat \phi_p ({\bf u}_{\bar k},\tau_{\bar k})$ and $\nabla \hat  g_{p,j} ({\bf u}_{\bar k},\tau_{\bar k})$, which is what one will have to work with in practice since the true gradients $\nabla \phi_p ({\bf u}_{\bar k},\tau_{\bar k})$ and $\nabla g_{p,j} ({\bf u}_{\bar k},\tau_{\bar k})$ will never be available exactly\footnote{We will use the notations $\nabla \hat \phi_p ({\bf u}_{\bar k},\tau_{\bar k})$ and $\nabla \hat  g_{p,j} ({\bf u}_{\bar k},\tau_{\bar k})$ to refer to the estimates of $\nabla \phi_p ({\bf u}_{\bar k},\tau_{\bar k})$ and $\nabla  g_{p,j} ({\bf u}_{\bar k},\tau_{\bar k})$, and \emph{not} to the gradients of $\hat \phi_p ({\bf u}_{\bar k},\tau_{\bar k})$ and $\hat  g_{p,j} ({\bf u}_{\bar k},\tau_{\bar k})$ as given in (\ref{eq:costnoise}) and (\ref{eq:connoise}).}. An additional implementation difficulty, and one that we admittedly will not treat here in detail, is that the additive noise model (\ref{eq:costnoise})-(\ref{eq:connoise}) for the function values does not extend to the gradient case in the measurement sense -- i.e., one does not simply ``measure'' a gradient with additive noise. One may, however, consider obtaining a gradient estimate with some uncertainty bounds that are guaranteed to contain the true gradient. At the time of writing, methods for doing this are not very developed -- we refer the reader to our previous work \cite{Bunin2013a} for one possible approach, as well as a regularization-based algorithm that provides bounds on a gradient estimate and that is being employed in the latest version of the SCFO solver \cite{SCFOug}.

Throughout this discussion, we will assume to have at our disposal estimation algorithms $\Omega_{{\bf u}_{\bar k},\tau_{\bar k}}$, $\Omega_{{\bf u}_{\bar k},\tau_{k+1}}$ that take all of the available measured/estimated data and return estimates of the gradient at $({\bf u}_{\bar k},\tau_{\bar k})$ and $({\bf u}_{\bar k},\tau_{k+1})$, together with the bounds on these estimates:

\vspace{-2mm}
\begin{equation}\label{eq:gradestcost}
\begin{array}{l}
\left[ \nabla \hat \phi_p ({\bf u}_{\bar k},\tau_{\bar k}),  \nabla \underline \phi_p ({\bf u}_{\bar k},\tau_{\bar k}),  \nabla \overline \phi_p ({\bf u}_{\bar k},\tau_{\bar k})  \right] \vspace{1mm} \\
\hspace{15mm} = \Omega_{{\bf u}_{\bar k},\tau_{\bar k}} \left( \begin{array}{l} {\bf u}_0,...,{\bf u}_k, \tau_0,...,\tau_k, \\ \hat \phi_p ({\bf u}_0,\tau_0),...,\hat \phi_p ({\bf u}_k,\tau_k) \end{array} \right),
\end{array}
\end{equation}

\vspace{-2mm}
\begin{equation}\label{eq:gradestcon}
\begin{array}{l}
\left[ \nabla \hat g_{p,j} ({\bf u}_{\bar k},\tau_{\bar k}),  \nabla \underline g_{p,j} ({\bf u}_{\bar k},\tau_{\bar k}),  \nabla \overline g_{p,j} ({\bf u}_{\bar k},\tau_{\bar k})  \right] \vspace{1mm} \\
\hspace{15mm} = \Omega_{{\bf u}_{\bar k},\tau_{\bar k}} \left( \begin{array}{l} {\bf u}_0,...,{\bf u}_k, \tau_0,...,\tau_k, \\ \hat g_{p,j} ({\bf u}_0,\tau_0),...,\hat g_{p,j} ({\bf u}_k,\tau_k) \end{array} \right),
\end{array}
\end{equation}

\vspace{-2mm}
\begin{equation}\label{eq:gradestcost2}
\begin{array}{l}
\left[ \nabla \hat \phi_p ({\bf u}_{\bar k},\tau_{k+1}),  \nabla \underline \phi_p ({\bf u}_{\bar k},\tau_{k+1}),  \nabla \overline \phi_p ({\bf u}_{\bar k},\tau_{k+1})  \right] \vspace{1mm} \\
\hspace{15mm} = \Omega_{{\bf u}_{\bar k},\tau_{k+1}} \left( \begin{array}{l} {\bf u}_0,...,{\bf u}_k, \tau_0,...,\tau_{k+1}, \\ \hat \phi_p ({\bf u}_0,\tau_0),...,\hat \phi_p ({\bf u}_k,\tau_k) \end{array} \right),
\end{array}
\end{equation}

\vspace{-2mm}
\begin{equation}\label{eq:gradestcon2}
\begin{array}{l}
\left[ \nabla \hat g_{p,j} ({\bf u}_{\bar k},\tau_{k+1}),  \nabla \underline g_{p,j} ({\bf u}_{\bar k},\tau_{k+1}),  \nabla \overline g_{p,j} ({\bf u}_{\bar k},\tau_{k+1})  \right] \vspace{1mm} \\
\hspace{13mm} = \Omega_{{\bf u}_{\bar k},\tau_{k+1}} \left( \begin{array}{l} {\bf u}_0,...,{\bf u}_k, \tau_0,...,\tau_{k+1}, \\ \hat g_{p,j} ({\bf u}_0,\tau_0),...,\hat g_{p,j} ({\bf u}_k,\tau_{k}) \end{array} \right).
\end{array}
\end{equation}

\noindent Given the current state of the art with regard to the gradient estimation problem, we will make no strong statistical assumptions on the estimates, apart from the general statement that $\nabla \hat \phi_p ({\bf u}_{\bar k},\tau_{\bar k})$, $\nabla \hat g_{p,j} ({\bf u}_{\bar k},\tau_{\bar k})$, $\nabla \hat \phi_p ({\bf u}_{\bar k},\tau_{k+1})$, and $\nabla \hat g_{p,j} ({\bf u}_{\bar k},\tau_{k+1})$ are, in some statistical sense, ``best'' estimates. We will also assume that 

\vspace{-2mm}
\begin{equation}\label{eq:gradincost}
\nabla \phi_p ({\bf u}_{\bar k},\tau_{\bar k}) \in \left[  \nabla \underline \phi_p ({\bf u}_{\bar k},\tau_{\bar k}),  \nabla \overline \phi_p ({\bf u}_{\bar k},\tau_{\bar k}) \right],
\end{equation}

\vspace{-2mm}
\begin{equation}\label{eq:gradincon}
\nabla g_{p,j} ({\bf u}_{\bar k},\tau_{\bar k}) \in \left[  \nabla \underline g_{p,j} ({\bf u}_{\bar k},\tau_{\bar k}),  \nabla \overline g_{p,j} ({\bf u}_{\bar k},\tau_{\bar k}) \right],
\end{equation}

\vspace{-2mm}
\begin{equation}\label{eq:gradincost2}
\nabla \phi_p ({\bf u}_{\bar k},\tau_{k+1}) \in \left[  \nabla \underline \phi_p ({\bf u}_{\bar k},\tau_{k+1}),  \nabla \overline \phi_p ({\bf u}_{\bar k},\tau_{k+1}) \right],
\end{equation}

\vspace{-2mm}
\begin{equation}\label{eq:gradincon2}
\nabla g_{p,j} ({\bf u}_{\bar k},\tau_{k+1}) \in \left[  \nabla \underline g_{p,j} ({\bf u}_{\bar k},\tau_{k+1}),  \nabla \overline g_{p,j} ({\bf u}_{\bar k},\tau_{k+1}) \right]
\end{equation}

\noindent all hold with sufficiently high probability. Finally, we will assume the estimates and their bounds to be consistent with the Lipschitz constants, i.e., that

\vspace{-2mm}
\begin{equation}\label{eq:gradlipconsist}
\begin{array}{l}
\underline \kappa_{p,ji} < \displaystyle \frac{\partial \underline g_{p,j}}{\partial u_i} \Big |_{({\bf u}_{\bar k},\tau_{\bar k})} \leq \displaystyle \frac{\partial \overline g_{p,j}}{\partial u_i} \Big |_{({\bf u}_{\bar k},\tau_{\bar k})} < \overline \kappa_{p,ji}, \vspace{1mm} \\
\underline \kappa_{p,j\tau} \leq \displaystyle \frac{\partial \underline g_{p,j}}{\partial \tau} \Big |_{({\bf u}_{\bar k},\tau_{\bar k})} \leq \displaystyle \frac{\partial \overline g_{p,j}}{\partial \tau} \Big |_{({\bf u}_{\bar k},\tau_{\bar k})}  \leq \overline \kappa_{p,j\tau}, \vspace{1mm} \\
\underline \kappa_{\phi,i} \leq \displaystyle \frac{\partial \underline \phi_{p}}{\partial u_i} \Big |_{({\bf u}_{\bar k},\tau_{\bar k})} \leq \displaystyle \frac{\partial \overline \phi_{p}}{\partial u_i} \Big |_{({\bf u}_{\bar k},\tau_{\bar k})} \leq \overline \kappa_{\phi,i}, \vspace{1mm} \\
\underline \kappa_{\phi,\tau} \leq \displaystyle \frac{\partial \underline \phi_{p}}{\partial \tau} \Big |_{({\bf u}_{\bar k},\tau_{\bar k})} \leq \displaystyle \frac{\partial \overline \phi_{p}}{\partial \tau} \Big |_{({\bf u}_{\bar k},\tau_{\bar k})} \leq \overline \kappa_{\phi,\tau}, \vspace{1mm} \\
\underline \kappa_{p,ji} < \displaystyle \frac{\partial \underline g_{p,j}}{\partial u_i} \Big |_{({\bf u}_{\bar k},\tau_{k+1})} \leq \displaystyle \frac{\partial \overline g_{p,j}}{\partial u_i} \Big |_{({\bf u}_{\bar k},\tau_{k+1})} < \overline \kappa_{p,ji}, \vspace{1mm} \\
\underline \kappa_{p,j\tau} \leq \displaystyle \frac{\partial \underline g_{p,j}}{\partial \tau} \Big |_{({\bf u}_{\bar k},\tau_{k+1})} \leq \displaystyle \frac{\partial \overline g_{p,j}}{\partial \tau} \Big |_{({\bf u}_{\bar k},\tau_{k+1})}  \leq \overline \kappa_{p,j\tau}, \vspace{1mm} \\
\underline \kappa_{\phi,i} \leq \displaystyle \frac{\partial \underline \phi_{p}}{\partial u_i} \Big |_{({\bf u}_{\bar k},\tau_{k+1})} \leq \displaystyle \frac{\partial \overline \phi_{p}}{\partial u_i} \Big |_{({\bf u}_{\bar k},\tau_{k+1})} \leq \overline \kappa_{\phi,i}, \vspace{1mm} \\
\underline \kappa_{\phi,\tau} \leq \displaystyle \frac{\partial \underline \phi_{p}}{\partial \tau} \Big |_{({\bf u}_{\bar k},\tau_{k+1})} \leq \displaystyle \frac{\partial \overline \phi_{p}}{\partial \tau} \Big |_{({\bf u}_{\bar k},\tau_{k+1})} \leq \overline \kappa_{\phi,\tau}.
\end{array}
\end{equation}

Given only the estimates and bounds provided by (\ref{eq:gradestcost})-(\ref{eq:gradestcon2}), we are now interested in formulating a robust version of the SCFO that has, until now, only utilized the true (unavailable) gradients.

\subsection{A Robust Version of the Projection}

In proposing an implementable version of (\ref{eq:projdeg2}), one could, of course, simply replace the true experimental gradients by their estimated versions:

\vspace{-4mm}
\begin{equation}\label{eq:projdeg2est}
\begin{array}{rl}
\bar {\bf u}_{k+1}^* := \;\;\;\;\;\;\;\;\; & \vspace{1mm}\\
 {\rm arg} \mathop {\rm minimize}\limits_{{\bf u}} & \| {\bf u} - {\bf u}_{k+1}^* \|_2^2  \vspace{1mm} \\
{\rm{subject}}\;{\rm{to}} & \nabla \hat g_{p,j} ({\bf u}_{k^*},\tau_{k+1})^T \left[ \hspace{-1mm} \begin{array}{c} {\bf u} - {\bf u}_{k^*} \\ 0 \end{array} \hspace{-1mm} \right] \leq -\delta_{g_p,j}, \vspace{1mm} \\
& \hspace{-6mm} \forall j: \begin{array}{l} \overline g_{p,j} ({\bf u}_{k^*},\tau_{k^*}) \vspace{1mm} \\ \displaystyle + \eta_{c,j} \frac{\partial g_{p,j}}{\partial \tau} \Big |_{({\bf u}_{k^*},\tau_{k^*})} ( \tau_{k+1} - \tau_{{k^*}} ) \vspace{1mm}  \\ + (1 - \eta_{c,j}) \overline \kappa_{p,j\tau} (\tau_{k+1} - \tau_{k^*} ) \geq -\epsilon_{p,j} \end{array} \vspace{1mm} \\
 & \nabla g_{j} ({\bf u}_{k^*})^T ({\bf u} - {\bf u}_{k^*}) \leq -\delta_{g,j}, \vspace{1mm} \\
& \forall j : g_{j}({\bf u}_{k^*}) \geq -\epsilon_{j} \vspace{1mm} \\
 & \nabla \hat \phi_{p} ({\bf u}_{k^*},\tau_{k+1})^T  \left[ \begin{array}{c} {\bf u} - {\bf u}_{k^*} \\ 0 \end{array} \right] \leq -\delta_{\phi} \vspace{1mm} \\
 & {\bf u}^L \preceq {\bf u} \preceq {\bf u}^U.
\end{array}
\end{equation}

\noindent While such an implementation may yield desired results in the long run -- if, for example, the estimates are unbiased -- there may be significant short-term losses that would result from using particularly bad estimates without accounting for their uncertainty or, in statistical terms, their \emph{variance}, which is in some sense represented by the size of the box corresponding to the lower and upper bounds. 

Here, we propose to use the standard robust approach of guaranteeing that the SCFO hold for the true gradients by ensuring that they hold for \emph{all of the gradients in the uncertainty set}. In other words, we propose to replace (\ref{eq:projdeg2est}) by the more robust

\vspace{-2mm}
\begin{equation}\label{eq:projdeg2rob}
\begin{array}{rl}
\bar {\bf u}_{k+1}^* := \;\;\;\;\;\;\;\;\; & \vspace{.5mm}\\
 {\rm arg} \mathop {\rm minimize}\limits_{{\bf u}} & \| {\bf u} - {\bf u}_{k+1}^* \|_2^2 \vspace{.5mm} \\
{\rm{subject}}\;{\rm{to}} & \nabla \tilde g_{p,j} ({\bf u}_{k^*},\tau_{k+1})^T \left[ \hspace{-1mm} \begin{array}{c} {\bf u} - {\bf u}_{k^*} \vspace{1mm} \\ 0 \end{array} \hspace{-1mm} \right] \leq -\delta_{g_p,j}, \\
& \forall \nabla \tilde g_{p,j} ({\bf u}_{k^*},\tau_{k+1}) \in \vspace{.5mm} \\
&\hspace{2mm} \left[  \nabla \underline g_{p,j} ({\bf u}_{k^*},\tau_{k+1}),  \nabla \overline g_{p,j} ({\bf u}_{k^*},\tau_{k+1}) \right], \vspace{.5mm} \\
& \hspace{-6mm} \forall j: \begin{array}{l} \overline g_{p,j} ({\bf u}_{k^*},\tau_{k^*}) \vspace{.5mm} \\ \displaystyle + \eta_{c,j} \frac{\partial \overline g_{p,j}}{\partial \tau} \Big |_{({\bf u}_{k^*},\tau_{k^*})} ( \tau_{k+1} - \tau_{{k^*}} ) \vspace{.5mm}  \\ + (1 - \eta_{c,j}) \overline \kappa_{p,j\tau} (\tau_{k+1} - \tau_{k^*} ) \geq -\epsilon_{p,j} \end{array} \vspace{.5mm} \\
 & \nabla g_{j} ({\bf u}_{k^*})^T ({\bf u} - {\bf u}_{k^*}) \leq -\delta_{g,j}, \vspace{.5mm} \\
& \forall j : g_{j}({\bf u}_{k^*}) \geq -\epsilon_{j} \vspace{.5mm} \\
 &  \nabla \tilde \phi_{p} ({\bf u}_{k^*},\tau_{k+1})^T  \left[ \begin{array}{c} {\bf u} - {\bf u}_{k^*} \\ 0 \end{array} \right] \leq -\delta_{\phi}, \vspace{.5mm} \\
& \forall \nabla \tilde \phi_{p} ({\bf u}_{k^*},\tau_{k+1}) \in \vspace{.5mm} \\
& \hspace{6mm} \left[  \nabla \underline \phi_p ({\bf u}_{k^*},\tau_{k+1}),  \nabla \overline \phi_p ({\bf u}_{k^*},\tau_{k+1}) \right] \vspace{.5mm} \\
 & {\bf u}^L \preceq {\bf u} \preceq {\bf u}^U, 
\end{array}
\end{equation}

\noindent as any $\bar {\bf u}_{k+1}^*$ that satisfies the constraints of (\ref{eq:projdeg2rob}) is guaranteed to satisfy the constraints of (\ref{eq:projdeg2}) as well.

It should be clear that (\ref{eq:projdeg2rob}) is not numerically implementable as given since it has a semi-infinite number of constraints -- i.e., the gradient uncertainty set, while bounded, contains an infinite number of gradients. However, due to the box nature of the uncertainty set, one may easily reformulate (\ref{eq:projdeg2rob}) by adding auxiliary slack variables $s$ (contained in the vector ${\bf s}_\phi$ and matrix ${\bf S}$):

\vspace{-2mm}
\begin{equation}\label{eq:projdeg2robslack}
\begin{array}{rl}
\bar {\bf u}_{k+1}^* := \;\;\;\;\;\;\;\;\; & \vspace{1mm}\\
 {\rm arg} \mathop {\rm minimize}\limits_{{\bf u},{\bf s}_\phi, {\bf S}} & \| {\bf u} - {\bf u}_{k+1}^* \|_2^2 \vspace{1mm}  \\
 {\rm{subject}}\;{\rm{to}} & \displaystyle \sum_{i=1}^{n_u} s_{ji}  \leq -\delta_{g_p,j} \vspace{1mm} \\
& \displaystyle \frac{\partial \underline g_{p,j}}{\partial u_i} \Big |_{({\bf u}_{k^*}, \tau_{k+1})} (u_i - u_{k^*,i}) \leq s_{ji} \vspace{1mm} \\
& \displaystyle \frac{\partial \overline g_{p,j}}{\partial u_i} \Big |_{({\bf u}_{k^*}, \tau_{k+1})} (u_i - u_{k^*,i}) \leq s_{ji}, \vspace{1mm} \\
& \forall i = 1,...,n_u, \vspace{1mm} \\
& \hspace{-6mm} \forall j: \begin{array}{l} \overline g_{p,j} ({\bf u}_{k^*},\tau_{k^*}) \vspace{1mm} \\ \displaystyle + \eta_{c,j} \frac{\partial \overline g_{p,j}}{\partial \tau} \Big |_{({\bf u}_{k^*},\tau_{k^*})} ( \tau_{k+1} - \tau_{{k^*}} ) \vspace{1mm}  \\ + (1 - \eta_{c,j}) \overline \kappa_{p,j\tau} (\tau_{k+1} - \tau_{k^*} ) \geq -\epsilon_{p,j} \end{array} \vspace{1mm} \\
 & \nabla g_{j} ({\bf u}_{k^*})^T ({\bf u} - {\bf u}_{k^*}) \leq -\delta_{g,j}, \vspace{1mm} \\
& \forall j : g_{j}({\bf u}_{k^*}) \geq -\epsilon_{j} \vspace{1mm} \\
 & \displaystyle  \sum_{i=1}^{n_u} s_{\phi,i}  \leq -\delta_{\phi} \vspace{1mm} \\
& \displaystyle \frac{\partial \underline \phi_{p}}{\partial u_i} \Big |_{({\bf u}_{k^*}, \tau_{k+1})} (u_i - u_{k^*,i}) \leq s_{\phi,i} \vspace{1mm} \\
& \displaystyle \frac{\partial \overline \phi_{p}}{\partial u_i} \Big |_{({\bf u}_{k^*}, \tau_{k+1})} (u_i - u_{k^*,i}) \leq s_{\phi,i} \vspace{1mm} \\
 & {\bf u}^L \preceq {\bf u} \preceq {\bf u}^U. 
\end{array}
\end{equation}

\noindent Clearly, (\ref{eq:projdeg2robslack}) is a standard quadratic program with no more than $2n_{g_p}(1+n_u)+n_g+4n_u+2$ linear constraints. We prove the equivalence of (\ref{eq:projdeg2robslack}) and (\ref{eq:projdeg2rob}) in the following lemma.

\begin{lemma}[Equivalence of Projections (\ref{eq:projdeg2robslack}) and (\ref{eq:projdeg2rob})]
\label{lemma:equiv}
Problems (\ref{eq:projdeg2robslack}) and (\ref{eq:projdeg2rob}) are equivalent in the following sense:

\vspace{-2mm}
\begin{equation}\label{eq:equiv}
\begin{array}{l}
{\bf u}\;{\rm feasible\;for\;} (\ref{eq:projdeg2rob}) \Leftrightarrow \\
\hspace{15mm} \exists {\bf s}_\phi, {\bf S} : {\bf u}, {\bf s}_\phi, {\bf S}\;{\rm feasible\;for\;} (\ref{eq:projdeg2robslack}).
\end{array}
\end{equation}

\end{lemma}
\begin{proof}
Let us first prove the forward implication

\vspace{-2mm}
\begin{equation}\label{eq:equivlr}
\begin{array}{l}
{\bf u}\;{\rm feasible\;for\;} (\ref{eq:projdeg2rob}) \Rightarrow \\
\hspace{15mm} \exists {\bf s}_\phi, {\bf S} : {\bf u}, {\bf s}_\phi, {\bf S}\;{\rm feasible\;for\;} (\ref{eq:projdeg2robslack}).
\end{array}
\end{equation}

To do this, consider the inequality constraints -- on the cost only, for brevity's sake -- in summation form:

\vspace{-2mm}
\begin{equation}\label{eq:sumform}
\begin{array}{l}
\displaystyle \sum_{i=1}^{n_u} \frac{\partial \tilde \phi_p}{\partial u_i} \Big |_{({\bf u}_{k^*},\tau_{k+1})} (u_i - u_{k^*,i}) \leq -\delta_\phi, \\
\hspace{-2mm} \forall \nabla \tilde \phi_{p} ({\bf u}_{k^*},\tau_{k+1}) \in  \left[  \nabla \underline \phi_p ({\bf u}_{k^*},\tau_{k+1}),  \nabla \overline \phi_p ({\bf u}_{k^*},\tau_{k+1}) \right].
\end{array}
\end{equation}

Since the terms inside the summation are linear in $u_i$, the following bound must hold:

\vspace{-2mm}
\begin{equation}\label{eq:sumbounds2}
\begin{array}{l}
\displaystyle \frac{\partial \tilde \phi_p}{\partial u_i} \Big |_{({\bf u}_{k^*},\tau_{k+1})} (u_{i} - u_{k^*,i}) \leq \vspace{1mm}\\
\hspace{10mm} \mathop {\max} \left[ \begin{array}{l} \displaystyle \frac{\partial \underline \phi_p}{\partial u_i} \Big |_{({\bf u}_{k^*},\tau_{k+1})} (u_{i} - u_{k^*,i}), \vspace{1mm} \\ \displaystyle \frac{\partial \overline \phi_p}{\partial u_i} \Big |_{({\bf u}_{k^*},\tau_{k+1})} (u_{i} - u_{k^*,i}) \end{array} \right],
\end{array}
\end{equation}

\noindent with the maximum term in (\ref{eq:sumbounds2}) determined by the sign of $u_{i} - u_{k^*,i}$.

We may thus replace the semi-infinite family of constraints in (\ref{eq:sumform}) by their worst-case maximum:

\vspace{-4mm}
\begin{equation}\label{eq:sumformworst}
\displaystyle \sum_{i=1}^{n_u} \mathop {\max} \left[ \begin{array}{l} \displaystyle \frac{\partial \underline \phi_p}{\partial u_i} \Big |_{({\bf u}_{k^*},\tau_{k+1})} (u_{i} - u_{k^*,i}), \vspace{1mm} \\ \displaystyle \frac{\partial \overline \phi_p}{\partial u_i} \Big |_{({\bf u}_{k^*},\tau_{k+1})} (u_{i} - u_{k^*,i}) \end{array} \right] \leq -\delta_\phi.
\end{equation}

By (\ref{eq:sumbounds2}), it is clear that any ${\bf u}$ that satisfies (\ref{eq:sumformworst}) satisfies (\ref{eq:sumform}) as well. Likewise, since the gradient realization of (\ref{eq:sumformworst}) belongs to the gradient uncertainty set in (\ref{eq:sumform}), it follows that any ${\bf u}$ that satisfies (\ref{eq:sumform}) naturally satisfies (\ref{eq:sumformworst}). (\ref{eq:sumform}) and (\ref{eq:sumformworst}) are thus equivalent. The feasibility of ${\bf u}$ for (\ref{eq:projdeg2rob}) then implies the feasibility of ${\bf u}$ for (\ref{eq:sumform}) and, consequently, for (\ref{eq:sumformworst}).

Having established this, let us choose the slack variables ${\bf s}_\phi$ as

\vspace{-4mm}
\begin{equation}\label{eq:slackchoice2}
s_{\phi,i} := \mathop {\max} \left[ \begin{array}{l} \displaystyle \frac{\partial \underline \phi_p}{\partial u_i} \Big |_{({\bf u}_{k^*},\tau_{k+1})} (u_{i} - u_{k^*,i}), \vspace{1mm} \\ \displaystyle \frac{\partial \overline \phi_p}{\partial u_i} \Big |_{({\bf u}_{k^*},\tau_{k+1})} (u_{i} - u_{k^*,i}) \end{array} \right].
\end{equation}

Substituting this choice into (\ref{eq:sumformworst}) clearly shows that the constraint

\vspace{-4mm}
\begin{equation}\label{eq:consat}
\displaystyle \sum_{i=1}^{n_u} s_{\phi,i}  \leq -\delta_{\phi}
\end{equation}

\noindent is satisfied. Furthermore, it is easily seen that this choice satisfies the constraints

\vspace{-4mm}
\begin{equation}\label{eq:consat2}
\begin{array}{l}
\displaystyle \frac{\partial \underline \phi_{p}}{\partial u_i} \Big |_{({\bf u}_{k^*}, \tau_{k+1})} (u_i - u_{k^*,i}) \leq s_{\phi,i} \vspace{1mm} \\
\displaystyle \frac{\partial \overline \phi_{p}}{\partial u_i} \Big |_{({\bf u}_{k^*}, \tau_{k+1})} (u_i - u_{k^*,i}) \leq s_{\phi,i}
\end{array}
\end{equation}

\noindent as well.

An entirely analogous procedure for the experimental constraint functions and the analogous choice of slack variables ${\bf S}$,

\vspace{-2mm}
\begin{equation}\label{eq:slackchoicecon2}
s_{ji} := \mathop {\max} \left[ \begin{array}{l} \displaystyle \frac{\partial \underline g_{p,j}}{\partial u_i} \Big |_{({\bf u}_{k^*},\tau_{k+1})} (u_{i} - u_{k^*,i}), \vspace{1mm} \\ \displaystyle \frac{\partial \overline g_{p,j}}{\partial u_i} \Big |_{({\bf u}_{k^*},\tau_{k+1})} (u_{i} - u_{k^*,i}) \end{array} \right],
\end{equation}

\noindent leads to the constraints

\vspace{-2mm}
\begin{equation}\label{eq:consatcon}
\displaystyle \sum_{i=1}^{n_u} s_{ji}  \leq -\delta_{g_p,j},
\end{equation}

\vspace{-2mm}
\begin{equation}\label{eq:consat2con}
\begin{array}{l}
\displaystyle \frac{\partial \underline g_{p,j}}{\partial u_i} \Big |_{({\bf u}_{k^*}, \tau_{k+1})} (u_i - u_{k^*,i}) \leq s_{ji} \vspace{1mm} \\
\displaystyle \frac{\partial \overline g_{p,j}}{\partial u_i} \Big |_{({\bf u}_{k^*}, \tau_{k+1})} (u_i - u_{k^*,i}) \leq s_{ji}
\end{array}
\end{equation}

\noindent being satisfied.

As such, we have proven that for any ${\bf u}$ that is feasible for (\ref{eq:projdeg2rob}) there exist ${\bf s}_\phi, {\bf S}$ -- namely, those given by  (\ref{eq:slackchoice2}) and (\ref{eq:slackchoicecon2}) -- so that the extended set ${\bf u}, {\bf s}_\phi, {\bf S}$ is feasible for (\ref{eq:projdeg2robslack}).

We will now prove the backward implication

\vspace{-2mm}
\begin{equation}\label{eq:equivrl}
\begin{array}{l}
{\bf u}\;{\rm feasible\;for\;} (\ref{eq:projdeg2rob}) \Leftarrow \\
\hspace{15mm} \exists {\bf s}_\phi, {\bf S} : {\bf u}, {\bf s}_\phi, {\bf S}\;{\rm feasible\;for\;} (\ref{eq:projdeg2robslack}).
\end{array}
\end{equation}

In this case, the constraints (\ref{eq:consat}) and (\ref{eq:consat2}) are known to hold, which clearly implies that (\ref{eq:sumformworst}) holds as well. By the equivalence of (\ref{eq:sumform}) and (\ref{eq:sumformworst}), it follows that semi-infinite set of constraints (\ref{eq:sumform}) is also satisfied. An identical analysis for the experimental constraint functions shows that the analogous semi-infinite sets of constraints for these functions hold as well by virtue of (\ref{eq:consatcon}) and (\ref{eq:consat2con}). Since these semi-infinite sets are the constraints of (\ref{eq:projdeg2rob}), we have thereby proven (\ref{eq:equivrl}) and, consequently, (\ref{eq:equiv}). \qed

\end{proof}

The practical implication of Lemma \ref{lemma:equiv} is that all ${\bf u}$ that are feasible solutions for the tractable convex problem (\ref{eq:projdeg2robslack}) are also feasible solutions for the semi-infinite and intractable problem (\ref{eq:projdeg2rob}). As such, no ${\bf u}$ is lost by solving the reformulated (\ref{eq:projdeg2robslack}) -- additionally, if a feasible ${\bf u}$ exists for (\ref{eq:projdeg2rob}), then we are guaranteed to find a feasible ${\bf u}$ for (\ref{eq:projdeg2robslack}) provided the availability of a reliable convex programming solver \cite{Boyd2008}.

\subsection{Infeasibility Due to Robustness}

As discussed in \cite{Bunin2013SIAM}, the infeasibility of the standard projection for a sufficiently low choice of projection parameters essentially implies that the algorithm has converged very close to a Fritz John (FJ) point. Unfortunately, this is no longer the case for the robust projection where lower and upper bounds on the gradients are used instead. In fact, it is easily shown that the projection may be infeasible at any experimental iteration regardless of how far one is from an FJ point. For example, suppose that the gradient bounds are sufficiently loose and one has the following

\vspace{-2mm}
\begin{equation}\label{eq:infeasbounds}
\nabla \underline \phi_p ({\bf u}_{k^*},\tau_{k+1}) = - \nabla \overline \phi_p ({\bf u}_{k^*},\tau_{k+1}).
\end{equation}

It is readily seen that any ${\bf u}$ that satisfies the constraint (\ref{eq:sumform}) for $\nabla \tilde \phi_p ({\bf u}_{k^*},\tau_{k+1}) = \nabla \underline \phi_p ({\bf u}_{k^*},\tau_{k+1})$ cannot possibly satisfy it for $\nabla \tilde \phi_p ({\bf u}_{k^*},\tau_{k+1}) = \nabla \overline \phi_p ({\bf u}_{k^*},\tau_{k+1})$ as well, as the two constraint values will have opposite signs. As both need to be satisfied in the robust implementation, the constraint set will thus be infeasible and, consequently, Projection (\ref{eq:projdeg2robslack}) will admit no solution. Setting ${\bf u}_{k+1} := {\bf u}_{k^*}$ in this case may not be a good option since the gradient bounds may simply be bad due to lack of data and would not be an indicator of proximity to an FJ point.

The trade-off that comes up in such scenarios is the desire to continue to move and to adapt the decision variables so as to optimize the system while maintaining some notion of robustness. What is proposed here is to artificially tighten the gradient uncertainty bounds until the projection becomes feasible. Note, however, that, as infeasibility in the projection may be caused both by large projection parameters and conservative gradient bounds, a systematic way to reduce both in tandem is needed. The projection algorithm provided in \cite{Bunin2013SIAM} is thus modified accordingly.

\vspace{2mm}
\noindent {\bf{Algorithm 3 -- Projection with Automatic Choice of Projection Parameters and Gradient Robustness}}
\vspace{2mm}

\begin{enumerate}
\item Set ${\boldsymbol \epsilon}_p := \overline {\boldsymbol \epsilon}_p$, ${\boldsymbol \epsilon} := \overline {\boldsymbol \epsilon}$, $\boldsymbol{\delta}_{g_p} := \boldsymbol{\overline \delta}_{g_p}$, $\boldsymbol{\delta}_{g} := \boldsymbol{\overline \delta}_{g}$, and $\delta_\phi := \overline \delta_\phi$, where $\overline \epsilon_{p,j} := \overline \delta_{g_p,j} \approx - \mathop {\min} \limits_{({\bf u},\tau) \in \mathcal{I}_\tau} g_{p,j} ({\bf u},\tau)$, $\overline \epsilon_{j} := \overline \delta_{g,j} \approx - \mathop {\min} \limits_{{\bf u} \in \mathcal{I}} g_{j} ({\bf u})$, and $\overline \delta_\phi \approx \phi_p ({\bf u}_0,\tau_0) - \mathop {\min} \limits_{({\bf u},\tau) \in \mathcal{I}_\tau} \phi_p ({\bf u},\tau)$. Set $\underline P := 0, \overline P := 1$.
\item Check the feasibility of (\ref{eq:projdeg2est}) for the given choice of ${\boldsymbol \epsilon}_{p}, {\boldsymbol \epsilon}, {\boldsymbol \delta}_{g_p}, {\boldsymbol \delta}_{g}, \delta_\phi$ by solving a linear programming feasibility problem. If no solution exists and $\delta_\phi \geq \overline \delta_\phi / 2^{10}$, set ${\boldsymbol \epsilon}_p := {\boldsymbol \epsilon}_p/2$, ${\boldsymbol \epsilon} := {\boldsymbol \epsilon}/2$, $\boldsymbol{\delta}_{g_p} := \boldsymbol{\delta}_{g_p}/2$, $\boldsymbol{\delta}_{g} := \boldsymbol{\delta}_{g}/2$, $\delta_\phi := \delta_\phi/2$, and repeat this step. Otherwise, proceed to Step 3.

\item If $\delta_\phi < \overline \delta_\phi / 2^{10}$, terminate with $\bar {\bf u}_{k+1}^* := {\bf u}_{k^*}$. Otherwise, set $P := 0.5 \underline P + 0.5 \overline P$ and define the tightened bounds as

\vspace{-2mm}
\begin{equation}\label{eq:boundtight}
\begin{array}{l}
\nabla \underline \phi_p^t ({\bf u}_{k^*},\tau_{k+1}) := \nabla \hat \phi_p ({\bf u}_{k^*},\tau_{k+1}) \vspace{1mm} \\
\hspace{5mm} + P \left[ \nabla \underline \phi_p ({\bf u}_{k^*},\tau_{k+1}) - \nabla \hat \phi_p ({\bf u}_{k^*},\tau_{k+1}) \right], \vspace{1mm} \\
\nabla \overline \phi_p^t ({\bf u}_{k^*},\tau_{k+1}) := \nabla \hat \phi_p ({\bf u}_{k^*},\tau_{k+1}) \vspace{1mm} \\
\hspace{5mm} + P \left[ \nabla \overline \phi_p ({\bf u}_{k^*},\tau_{k+1}) - \nabla \hat \phi_p ({\bf u}_{k^*},\tau_{k+1}) \right],
\end{array}
\end{equation}

$$
\begin{array}{l}
\nabla \underline g_{p,j}^t ({\bf u}_{k^*},\tau_{k+1}) := \nabla \hat g_{p,j} ({\bf u}_{k^*},\tau_{k+1}) \vspace{1mm} \\
\hspace{5mm} + P \left[ \nabla \underline g_{p,j} ({\bf u}_{k^*},\tau_{k+1}) - \nabla \hat g_{p,j} ({\bf u}_{k^*},\tau_{k+1}) \right],  \vspace{1mm} \\
\nabla \overline g_{p,j}^t ({\bf u}_{k^*},\tau_{k+1}) := \nabla \hat g_{p,j} ({\bf u}_{k^*},\tau_{k+1}) \vspace{1mm} \\
\hspace{5mm} + P \left[ \nabla \overline g_{p,j} ({\bf u}_{k^*},\tau_{k+1}) - \nabla \hat g_{p,j} ({\bf u}_{k^*},\tau_{k+1}) \right].
\end{array}
$$

\item Check the feasibility of the robust projection with the tightened bounds for the obtained ${\boldsymbol \epsilon}_{p}, {\boldsymbol \epsilon}, {\boldsymbol \delta}_{g_p}, {\boldsymbol \delta}_{g}, \delta_\phi$:

\vspace{-2mm}
\begin{equation}\label{eq:projdeg2robslackt}
\hspace{-6mm}\begin{array}{rl}
\bar {\bf u}_{k+1}^* := \;\;\;\;\;\;\;\;\; & \vspace{1mm}\\
 {\rm arg} \mathop {\rm minimize}\limits_{{\bf u},{\bf s}_\phi, {\bf S}} & \| {\bf u} - {\bf u}_{k+1}^* \|_2^2 \vspace{1mm}  \\
 {\rm{subject}}\;{\rm{to}} & \displaystyle \sum_{i=1}^{n_u} s_{ji}  \leq -\delta_{g_p,j} \vspace{1mm} \\
& \displaystyle \frac{\partial \underline g_{p,j}^t}{\partial u_i} \Big |_{({\bf u}_{k^*}, \tau_{k+1})} (u_i - u_{k^*,i}) \leq s_{ji} \vspace{1mm} \\
& \displaystyle \frac{\partial \overline g_{p,j}^t}{\partial u_i} \Big |_{({\bf u}_{k^*}, \tau_{k+1})} (u_i - u_{k^*,i}) \leq s_{ji}, \vspace{1mm} \\
& \forall i = 1,...,n_u, \vspace{1mm} \\
& \hspace{-6mm} \forall j: \begin{array}{l} \overline g_{p,j} ({\bf u}_{k^*},\tau_{k^*}) \vspace{1mm} \\ \displaystyle + \eta_{c,j} \frac{\partial \overline g_{p,j}}{\partial \tau} \Big |_{({\bf u}_{k^*},\tau_{k^*})} ( \tau_{k+1} - \tau_{{k^*}} ) \vspace{1mm}  \\ + (1 - \eta_{c,j}) \overline \kappa_{p,j\tau} (\tau_{k+1} - \tau_{k^*} ) \geq -\epsilon_{p,j} \end{array} \vspace{1mm} \\
 & \nabla g_{j} ({\bf u}_{k^*})^T ({\bf u} - {\bf u}_{k^*}) \leq -\delta_{g,j}, \vspace{1mm} \\
& \forall j : g_{j}({\bf u}_{k^*}) \geq -\epsilon_{j} \vspace{1mm} \\
 & \displaystyle  \sum_{i=1}^{n_u} s_{\phi,i}  \leq -\delta_{\phi} \vspace{1mm} \\
& \displaystyle \frac{\partial \underline \phi_{p}^t}{\partial u_i} \Big |_{({\bf u}_{k^*}, \tau_{k+1})} (u_i - u_{k^*,i}) \leq s_{\phi,i} \vspace{1mm} \\
& \displaystyle \frac{\partial \overline \phi_{p}^t}{\partial u_i} \Big |_{({\bf u}_{k^*}, \tau_{k+1})} (u_i - u_{k^*,i}) \leq s_{\phi,i} \vspace{1mm} \\
 & {\bf u}^L \preceq {\bf u} \preceq {\bf u}^U. 
\end{array}
\end{equation}

\noindent If (\ref{eq:projdeg2robslackt}) is infeasible, set $\overline P := P$. Otherwise, set $\underline P := P$. If $\overline P - \underline P < 0.01$, proceed to Step 5. Otherwise, return to Step 3. 

\item Set $P := 0.5\underline P$, define the tightened bounds as in (\ref{eq:boundtight}), and solve (\ref{eq:projdeg2robslackt}) to obtain the projected target $\bar {\bf u}_{k+1}^*$. Terminate.

\end{enumerate}

Algorithm 3 can be seen as consisting of two parts. In the first, appropriate projection parameters are chosen based on the non-robust projection carried out with the estimated gradients. Once a sufficiently low set of projection parameters is found, the second part consists in adding as much robustness as possible for this particular set by finding the largest possible gradient uncertainty set that still maintains the feasibility of the projection. Since this way of ``pushing robustness to its limits'' can have performance drawbacks (as will be shown later), we take the heuristic approach of setting $P := 0.5 \underline P$ to remove some of the robustness.

\subsection{Accounting for Gradient Uncertainty in the Feasibility Conditions}

Consider the feasibility-guaranteeing condition (\ref{eq:SCFO1idegLUccvalllocMN}). From the inequalities

\vspace{-2mm}
\begin{equation}\label{eq:gradprodbounddeg}
\begin{array}{l}
\displaystyle \frac{\partial g_{p,j}}{\partial \tau} \Big |_{({\bf u}_{\bar k},\tau_{\bar k})} ( \tau_{k+1} - \tau_{\bar k} ) \leq  \frac{\partial \overline g_{p,j}}{\partial \tau} \Big |_{({\bf u}_{\bar k},\tau_{\bar k})} (\tau_{k+1} - \tau_{\bar k} ),
\end{array}
\end{equation}

\vspace{-2mm}
\begin{equation}\label{eq:gradprodbound}
\hspace{-2mm}\begin{array}{l}
\displaystyle \frac{\partial g_{p,j}}{\partial u_i} \Big |_{({\bf u}_{\bar k},\tau_{\bar k})} ( u_{k^*,i} + K_k (\bar u_{k+1,i}^* - u_{k^*,i} ) - u_{{\bar k},i} ) \vspace{1mm} \\
 \leq \mathop {\max} \left[ \hspace{-1mm} \begin{array}{l} \displaystyle \frac{\partial \underline g_{p,j}}{\partial u_i} \Big |_{({\bf u}_{\bar k},\tau_{\bar k})} \hspace{-1mm} ( u_{k^*,i} + K_k (\bar u_{k+1,i}^* - u_{k^*,i} ) - u_{{\bar k},i} ), \vspace{1mm} \\ \displaystyle \frac{\partial \overline g_{p,j}}{\partial u_i} \Big |_{({\bf u}_{\bar k},\tau_{\bar k})} \hspace{-1mm} ( u_{k^*,i} + K_k (\bar u_{k+1,i}^* - u_{k^*,i} ) - u_{{\bar k},i} )  \end{array} \hspace{-1mm}  \right]
\end{array}
\end{equation}

\noindent we may propose a version of (\ref{eq:SCFO1idegLUccvalllocMN}) that accounts for gradient uncertainty:

\vspace{-2mm}
\begin{equation}\label{eq:SCFO1idegLUccvalllocMNgrad}
\hspace{-5mm}\begin{array}{l}
\mathop {\min} \limits_{\bar k = 0,...,k} \left[ \hspace{-1mm} \begin{array}{l} \overline g_{p,j} ({\bf u}_{\bar k},\tau_{\bar k}) \vspace{1mm} \\
 \displaystyle +\eta_{c,j}^{\bar k} \frac{\partial \overline g_{p,j}}{\partial \tau} \Big |_{({\bf u}_{\bar k},\tau_{\bar k})} ( \tau_{k+1} - \tau_{{\bar k}} ) \vspace{1mm} \\
 \displaystyle + (1-\eta_{c,j}^{\bar k}) \overline \kappa_{p,j\tau}^{\bar k} \left( \tau_{k+1} - \tau_{\bar k} \right) \vspace{1mm}\\
\displaystyle   + \sum_{i \in I_{c,j}^{\bar k}} \mathop {\max} \left[ \begin{array}{l} \displaystyle \frac{\partial \underline g_{p,j}}{\partial u_i} \Big |_{({\bf u}_{\bar k},\tau_{\bar k})} ( u_{k^*,i} + \\ \hspace{2mm} K_k (\bar u_{k+1,i}^* - u_{k^*,i} ) - u_{\bar k,i} ), \vspace{1mm}\\ \displaystyle \frac{\partial \overline g_{p,j}}{\partial u_i} \Big |_{({\bf u}_{\bar k},\tau_{\bar k})} ( u_{k^*,i} + \\ \hspace{2mm} K_k (\bar u_{k+1,i}^* - u_{k^*,i} ) - u_{\bar k,i} ) \end{array} \right] \vspace{1mm} \\
 + \displaystyle \sum_{i \not \in I_{c,j}^{\bar k}} \mathop {\max} \left[ \begin{array}{l} \underline \kappa_{p,ji}^{\bar k} ( u_{k^*,i} + \\ \hspace{2mm} K_k (\bar u_{k+1,i}^* - u_{k^*,i} ) - u_{\bar k,i} ), \vspace{1mm}\\ \overline \kappa_{p,ji}^{\bar k} ( u_{k^*,i} + \\ \hspace{2mm} K_k (\bar u_{k+1,i}^* - u_{k^*,i} ) - u_{\bar k,i} ) \end{array} \right] \end{array} \hspace{-1mm} \right] \leq 0,
\end{array}
\end{equation}

\noindent which is implementable and robust since the satisfaction of (\ref{eq:SCFO1idegLUccvalllocMNgrad}) implies the satisfaction of (\ref{eq:SCFO1idegLUccvalllocMN}).

\subsection{Accounting for Gradient Uncertainty in the Cost Decrease Conditions}

For the sufficient condition on cost decrease as given by (\ref{eq:SCFO7idegLUloc}), we consider first the inner product of the gradient with the change in decision variables in summation form:

\vspace{-2mm}
\begin{equation}\label{eq:costdecsum}
\begin{array}{l}
 \nabla \phi_p({\bf u}_{k^*},\tau_{k+1})^T \left[ \begin{array}{c} \bar {\bf u}_{k+1}^* - {\bf u}_{k^*} \\ 0 \end{array} \right] = \vspace{1mm} \\
\hspace{25mm} \displaystyle \sum_{i=1}^{n_u} \frac{\partial \phi_p}{\partial u_i} \Big |_{({\bf u}_{k^*},\tau_{k+1})} (\bar u_{k+1,i}^* - u_{k^*,i}).
\end{array}
\end{equation}

Using bounds analogous to that of (\ref{eq:gradprodbound}), we may thus upper bound this term as

\vspace{-2mm}
\begin{equation}\label{eq:costdecsum2}
\begin{array}{l}
 \nabla \phi_p({\bf u}_{k^*},\tau_{k+1})^T \left[ \begin{array}{c} \bar {\bf u}_{k+1}^* - {\bf u}_{k^*} \\ 0 \end{array} \right] \leq \vspace{1mm} \\
\hspace{10mm} \displaystyle \sum_{i=1}^{n_u} \mathop {\max} \left[ \begin{array}{l}\displaystyle \frac{\partial \underline \phi_p}{\partial u_i} \Big |_{({\bf u}_{k^*},\tau_{k+1})} (\bar u_{k+1,i}^* - u_{k^*,i}), \vspace{1mm} \\ \displaystyle \frac{\partial \overline \phi_p}{\partial u_i} \Big |_{({\bf u}_{k^*},\tau_{k+1})} (\bar u_{k+1,i}^* - u_{k^*,i}) \end{array} \right],
\end{array}
\end{equation}

\noindent which allows for the robust version of (\ref{eq:SCFO7idegLUloc}):

\vspace{-2mm}
\begin{equation}\label{eq:SCFO7idegLUlocgrad}
\begin{array}{l}
\displaystyle  \sum_{i=1}^{n_u} \mathop {\max} \left[ \begin{array}{l}\displaystyle \frac{\partial \underline \phi_p}{\partial u_i} \Big |_{({\bf u}_{k^*},\tau_{k+1})} (\bar u_{k+1,i}^* - u_{k^*,i}), \vspace{1mm} \\ \displaystyle \frac{\partial \overline \phi_p}{\partial u_i} \Big |_{({\bf u}_{k^*},\tau_{k+1})} (\bar u_{k+1,i}^* - u_{k^*,i}) \end{array} \right] \vspace{2mm} \\
\displaystyle + \frac{K_k}{2} \sum_{i_1=1}^{n_u} \sum_{i_2=1}^{n_u} \mathop {\max} \left[ \begin{array}{l} \underline M_{\phi,i_1 i_2}^{k^*} (\bar u_{k+1,i_1}^* - u_{{k^*},i_1}) \vspace{1mm} \\
\hspace{11mm} (\bar u_{k+1,i_2}^* - u_{{k^*},i_2}), \vspace{1mm} \\ \overline M_{\phi,i_1 i_2}^{k^*} (\bar u_{k+1,i_1}^* - u_{{k^*},i_1}) \\
\hspace{11mm}(\bar u_{k+1,i_2}^* - u_{{k^*},i_2}) \end{array} \right]  \leq 0.
\end{array}
\end{equation}

Satisfying (\ref{eq:SCFO7idegLUlocgrad}) implies satisfying (\ref{eq:SCFO7idegLUloc}). However, it may occur that very conservative gradient bounds -- as those, for example, in (\ref{eq:infeasbounds}) -- would lead to the terms in (\ref{eq:SCFO7idegLUlocgrad}) being positive for any $K_k \in [0,1]$, thereby making the satisfaction of this condition impossible. To avoid this, we propose to use the tightened gradient bounds instead:

\vspace{-2mm}
\begin{equation}\label{eq:SCFO7idegLUlocgradt}
\begin{array}{l}
\displaystyle  \sum_{i=1}^{n_u} \mathop {\max} \left[ \begin{array}{l}\displaystyle \frac{\partial \underline \phi_p^t}{\partial u_i} \Big |_{({\bf u}_{k^*},\tau_{k+1})} (\bar u_{k+1,i}^* - u_{k^*,i}), \vspace{1mm} \\ \displaystyle \frac{\partial \overline \phi_p^t}{\partial u_i} \Big |_{({\bf u}_{k^*},\tau_{k+1})} (\bar u_{k+1,i}^* - u_{k^*,i}) \end{array} \right] \vspace{2mm} \\
\displaystyle + \frac{K_k}{2} \sum_{i_1=1}^{n_u} \sum_{i_2=1}^{n_u} \mathop {\max} \left[ \begin{array}{l} \underline M_{\phi,i_1 i_2}^{k^*} (\bar u_{k+1,i_1}^* - u_{{k^*},i_1}) \vspace{1mm} \\
\hspace{11mm} (\bar u_{k+1,i_2}^* - u_{{k^*},i_2}), \vspace{1mm} \\ \overline M_{\phi,i_1 i_2}^{k^*} (\bar u_{k+1,i_1}^* - u_{{k^*},i_1}) \\
\hspace{11mm}(\bar u_{k+1,i_2}^* - u_{{k^*},i_2}) \end{array} \right]  \leq 0.
\end{array}
\end{equation}

\noindent As using more conservative gradient bounds will almost always lead to smaller acceptable $K_k$, it is also for this reason that we employ the heuristic $P := 0.5 \underline P$ in Algorithm 3, as this allows for larger steps while compromising only part of the robustness.

For the necessary conditions for cost decrease, we are forced to modify (\ref{eq:costhighmaxPFlocMN}) to account for gradient uncertainty. Here, one must make modifications with respect to both the lower and upper bounds on the cost at the future and reference (respectively) experimental iterates, as we are interested in excluding those points for which the lower bound at the former is superior to the upper bound at the latter. Let us consider the following lower bound analogues to (\ref{eq:gradprodbounddeg}) and (\ref{eq:gradprodbound}):

\vspace{-2mm}
\begin{equation}\label{eq:gradprodbounddegcost}
\displaystyle \frac{\partial \phi_{p}}{\partial \tau} \Big |_{({\bf u}_{\bar k},\tau_{\bar k})} ( \tau_{k+1} - \tau_{\bar k} ) \geq  \frac{\partial \underline \phi_{p}}{\partial \tau} \Big |_{({\bf u}_{\bar k},\tau_{\bar k})} (\tau_{k+1} - \tau_{\bar k} ),
\end{equation}

\vspace{-2mm}
\begin{equation}\label{eq:gradprodboundcost}
\hspace{-2mm}\begin{array}{l}
\displaystyle \frac{\partial \phi_{p}}{\partial u_i} \Big |_{({\bf u}_{\bar k},\tau_{\bar k})} ( u_{k^*,i} + K_k (\bar u_{k+1,i}^* - u_{k^*,i} ) - u_{{\bar k},i} ) \vspace{2mm} \\
 \geq \mathop {\min} \left[ \hspace{-1mm} \begin{array}{l} \displaystyle \frac{\partial \underline \phi_{p}}{\partial u_i} \Big |_{({\bf u}_{\bar k},\tau_{\bar k})} \hspace{-1mm} ( u_{k^*,i} + K_k (\bar u_{k+1,i}^* - u_{k^*,i} ) - u_{{\bar k},i} ), \vspace{1mm} \\ \displaystyle \frac{\partial \overline \phi_{p}}{\partial u_i} \Big |_{({\bf u}_{\bar k},\tau_{\bar k})} \hspace{-1mm} ( u_{k^*,i} + K_k (\bar u_{k+1,i}^* - u_{k^*,i} ) - u_{{\bar k},i} )  \end{array} \hspace{-1mm}  \right],
\end{array}
\end{equation}

\noindent which lead to the robust version of (\ref{eq:costhighmaxPFlocMN}):

\vspace{-2mm}
\begin{equation}\label{eq:costhighmaxPFlocMNgrad}
\hspace{-2mm}\begin{array}{l}
\mathop {\max} \limits_{\bar k = 0,...,k}\left[ \begin{array}{l}
\displaystyle \underline \phi_{p} ({\bf u}_{\bar k},\tau_{\bar k})  +\eta_{v,\phi}^{\bar k} \frac{\partial \underline \phi_{p}}{\partial \tau} \Big |_{({\bf u}_{\bar k},\tau_{\bar k})} ( \tau_{k+1} - \tau_{{\bar k}} ) \vspace{1mm} \\
\displaystyle + (1-\eta_{v,\phi}^{\bar k}) \underline \kappa_{\phi,\tau}^{\bar k} \left( \tau_{k+1} - \tau_{\bar k} \right) \vspace{1mm}\\
\displaystyle   + \sum_{i \in I_{v,\phi}^{\bar k}} \mathop {\min} \left[ \begin{array}{l} \displaystyle \frac{\partial \underline \phi_{p}}{\partial u_i} \Big |_{({\bf u}_{\bar k},\tau_{\bar k})} ( u_{k^*,i} +\\ \hspace{2mm} K_k(\bar u_{k+1,i}^* - u_{k^*,i}) - u_{{\bar k},i} ), \vspace{1mm}\\ \displaystyle \frac{\partial \overline \phi_{p}}{\partial u_i} \Big |_{({\bf u}_{\bar k},\tau_{\bar k})} ( u_{k^*,i} +\\ \hspace{2mm} K_k(\bar u_{k+1,i}^* - u_{k^*,i}) - u_{{\bar k},i} ) \end{array} \right] \vspace{1mm} \\
+ \displaystyle \sum_{i \not \in I_{v,\phi}^{\bar k}} \mathop {\min} \left[ \begin{array}{l} \underline \kappa_{\phi,i}^{\bar k} ( u_{k^*,i} +\\ \hspace{3mm} K_k(\bar u_{k+1,i}^* - u_{k^*,i}) - u_{{\bar k},i} ), \vspace{1mm}\\ \overline \kappa_{\phi,i}^{\bar k} ( u_{k^*,i} +\\ \hspace{3mm} K_k(\bar u_{k+1,i}^* - u_{k^*,i}) - u_{{\bar k},i} ) \end{array} \right] 
\end{array} \right] \vspace{1mm}\\
\displaystyle \leq \mathop {\min}_{\tilde k = 0,...,k} \left[ \hspace{-1mm} \begin{array}{l} \displaystyle \overline \phi_{p} ({\bf u}_{\tilde k},\tau_{\tilde k})  +\eta_{c,\phi}^{\tilde k} \frac{\partial \overline \phi_{p}}{\partial \tau} \Big |_{({\bf u}_{\tilde k},\tau_{\tilde k})} ( \tau_{k+1} - \tau_{{\tilde k}} ) \vspace{1mm} \\
\hspace{0mm} \displaystyle + (1-\eta_{c,\phi}^{\tilde k}) \overline \kappa_{\phi,\tau}^{\tilde k} \left( \tau_{k+1} - \tau_{\tilde k} \right) \vspace{1mm}\\
\hspace{0mm}\displaystyle   + \sum_{i \in I_{c,\phi}^{\tilde k}} \mathop {\max} \left[ \begin{array}{l} \displaystyle \frac{\partial \underline \phi_{p}}{\partial u_i} \Big |_{({\bf u}_{\tilde k},\tau_{\tilde k})} ( u_{k^*,i} - u_{{\tilde k},i} ), \\ \displaystyle  \frac{\partial \overline \phi_{p}}{\partial u_i} \Big |_{({\bf u}_{\tilde k},\tau_{\tilde k})} ( u_{k^*,i} - u_{{\tilde k},i} )   \end{array} \right] \vspace{1mm} \\
\hspace{0mm} + \displaystyle \sum_{i \not \in I_{c,\phi}^{\tilde k}} \mathop {\max} \left[ \begin{array}{l} \underline \kappa_{\phi,i}^{\tilde k} ( u_{k^*,i} - u_{{\tilde k},i} ), \vspace{1mm}\\ \overline \kappa_{\phi,i}^{\tilde k} ( u_{k^*,i} - u_{{\tilde k},i} ) \end{array} \right] \end{array} \right].
\end{array}
\end{equation}

Clearly, any point that would be excluded for failing to satisfy (\ref{eq:costhighmaxPFlocMNgrad}) would also be excluded for failing to satisfy (\ref{eq:costhighmaxPFlocMN}).

\subsection{Choosing a Reference Point While Accounting for Gradient Uncertainty}

In choosing a reference point, one must further modify the subproblems of (\ref{eq:kstarLUccvcostlocMN}) and (\ref{eq:kstar2LUccvlocMN}). Using bounds analogous to (\ref{eq:gradprodbounddeg}) and (\ref{eq:gradprodbounddegcost}), one obtains

\vspace{-2mm}
\begin{equation}\label{eq:kstarLUccvcostlocMNgrad}
\begin{array}{rl}
k^* := \;\;\;\;\;\;\;\;\;\;\;\;\;\;& \vspace{1mm} \\
{\rm arg} \mathop {\rm maximize}\limits_{\bar k \in [0,k]} & \bar k \vspace{1mm}  \\
{\rm{subject}}\;{\rm{to}} & \overline g_{p,j} ({\bf u}_{\bar k},\tau_{\bar k}) \vspace{1mm} \\
& \displaystyle + \eta_{c,j} \frac{\partial \overline g_{p,j}}{\partial \tau} \Big |_{({\bf u}_{\bar k},\tau_{\bar k})} ( \tau_{k+1} - \tau_{{\bar k}} ) \vspace{1mm} \\
&  + (1- \eta_{c,j}) \overline \kappa_{p,j\tau} \left( \tau_{k+1} - \tau_{\bar k} \right) \leq 0, \vspace{1mm} \\
& \forall j = 1,...,n_{g_p} \vspace{1mm} \\
& \displaystyle \underline \phi_p ({\bf u}_{\bar k},\tau_{\bar k}) +  \eta_{v,\phi}^{\bar k, k} \frac{\partial \underline \phi_{p}}{\partial \tau} \Big |_{({\bf u}_{\bar k},\tau_{\bar k})} ( \tau_{k} - \tau_{{\bar k}} ) \vspace{1mm} \\
&+ (1-  \eta_{v,\phi}^{\bar k, k}) \underline \kappa_{\phi,\tau}^{\bar k, k} \left( \tau_{k} - \tau_{\bar k} \right) \leq \vspace{1mm} \\
& \mathop {\min} \limits_{\tilde k \in {\bf k}_f} \left[ \begin{array}{l} \overline \phi_p ({\bf u}_{\tilde k},\tau_{\tilde k}) + \vspace{1mm} \\
\displaystyle \eta_{c,\phi}^{\tilde k, k} \frac{\partial \overline \phi_{p}}{\partial \tau} \Big |_{({\bf u}_{\tilde k},\tau_{\tilde k})} ( \tau_{k} - \tau_{{\tilde k}} ) \vspace{1mm} \\
+ (1- \eta_{c,\phi}^{\tilde k, k}) \overline \kappa_{\phi,\tau}^{\tilde k, k} \left( \tau_{k} - \tau_{\tilde k} \right) \end{array} \right],
\end{array}
\end{equation}

\vspace{-2mm}
\begin{equation}\label{eq:kfeas4}
{\bf k}_f = \left\{ \bar k : \begin{array}{l} \overline g_{p,j} ({\bf u}_{\bar k},\tau_{\bar k}) \vspace{1mm} \\  + \displaystyle \eta_{c,j} \frac{\partial \overline g_{p,j}}{\partial \tau} \Big |_{({\bf u}_{\bar k},\tau_{\bar k})} ( \tau_{k+1} - \tau_{{\bar k}} ) \vspace{1mm} \\ + (1-  \eta_{c,j}) \overline \kappa_{p,j\tau} \left( \tau_{k+1} - \tau_{\bar k} \right)  \leq 0, \vspace{1mm} \\ \forall j = 1,...,n_{g_p} \end{array} \right\},
\end{equation}

\vspace{-2mm}
\begin{equation}\label{eq:kstar2LUccvlocMNgrad}
\begin{array}{l}
k^*  := \\
\displaystyle {\rm arg} \mathop {\rm minimize}\limits_{\bar k \in [0,k]}  \mathop {\max} \limits_{j = 1,...,n_{g_p}}  \left[  \begin{array}{l} \overline g_{p,j} ({\bf u}_{\bar k},\tau_{\bar k}) \vspace{1mm} \\
\displaystyle +  \eta_{c,j} \frac{\partial \overline g_{p,j}}{\partial \tau} \Big |_{({\bf u}_{\bar k},\tau_{\bar k})} \vspace{1mm} \\
\hspace{15mm}( \tau_{k+1} - \tau_{{\bar k}} ) \vspace{1mm} \\
+ (1-  \eta_{c,j}) \overline \kappa_{p,j\tau} \vspace{1mm} \\
\hspace{15mm}\left( \tau_{k+1} - \tau_{\bar k} \right) \end{array}  \right].
\end{array}
\end{equation}

Here, (\ref{eq:kstarLUccvcostlocMNgrad}) selects the most recent point that is guaranteed to be robustly feasible and that cannot be guaranteed to have a cost function value superior to that at any other feasible point. (\ref{eq:kstar2LUccvlocMNgrad}) is a possible alternative for when no robustly feasible point is available, and simply selects the point with the lowest upper constraint value.

\subsection{Accounting for Gradient Uncertainty in the Computation of Lower and Upper Bounds on the Experimental Function Values}

As certain methods of calculating the lower and upper bounds on the experimental function values (discussed in Section \ref{sec:bounds}) use the function derivatives, these too must be modified to account for gradient uncertainty. Consider first the robust version of (\ref{eq:costbound2}):

\vspace{-2mm}
\begin{equation}\label{eq:costbound2grad}
\begin{array}{l}
\displaystyle \frac{1}{N} \sum_{\tilde k \in \bar {\bf k}} \hat \phi_p ({\bf u}_{\tilde k},\tau_{\tilde k}) \vspace{1mm} \\
\displaystyle - \frac{1}{N} \sum_{\tilde k \in \bar {\bf k}}  \eta_{c,\phi}^{\bar k, \tilde k} \mathop {\max} \left[ \begin{array}{l} \displaystyle \frac{\partial \underline \phi_{p}}{\partial \tau} \Big |_{({\bf u}_{\bar k},\tau_{\bar k})} ( \tau_{\tilde k} - \tau_{{\bar k}} ) , \vspace{1mm} \\ \displaystyle \frac{\partial \overline \phi_{p}}{\partial \tau} \Big |_{({\bf u}_{\bar k},\tau_{\bar k})} ( \tau_{\tilde k} - \tau_{{\bar k}} ) \end{array} \right] \vspace{1mm}  \\
\displaystyle - \frac{1}{N} \sum_{\tilde k \in \bar {\bf k}} (1-  \eta_{c,\phi}^{\bar k, \tilde k}) \mathop {\max} \left[ \begin{array}{l} \underline \kappa_{\phi,\tau}^{\bar k, \tilde k} \left( \tau_{\tilde k} - \tau_{\bar k} \right), \vspace{1mm} \\ \overline \kappa_{\phi,\tau}^{\bar k, \tilde k} \left( \tau_{\tilde k} - \tau_{\bar k} \right) \end{array} \right]  - \overline W_{\phi, \bar k} \vspace{1mm} \\
\hspace{0mm} \leq \phi_p ({\bf u}_{\bar k},\tau_{\bar k}) \leq \vspace{1mm} \\
\displaystyle \frac{1}{N} \sum_{\tilde k \in \bar {\bf k}} \hat \phi_p ({\bf u}_{\tilde k},\tau_{\tilde k}) \vspace{1mm} \\
\displaystyle -  \frac{1}{N} \sum_{\tilde k \in \bar {\bf k}}  \eta_{v,\phi}^{\bar k, \tilde k} \mathop {\min} \left[ \begin{array}{l} \displaystyle \frac{\partial \underline \phi_{p}}{\partial \tau} \Big |_{({\bf u}_{\bar k},\tau_{\bar k})} ( \tau_{\tilde k} - \tau_{{\bar k}} ) , \vspace{1mm} \\ \displaystyle \frac{\partial \overline \phi_{p}}{\partial \tau} \Big |_{({\bf u}_{\bar k},\tau_{\bar k})} ( \tau_{\tilde k} - \tau_{{\bar k}} ) \end{array} \right] \vspace{1mm} \\
- \displaystyle \frac{1}{N} \sum_{\tilde k \in \bar {\bf k}} (1-  \eta_{v,\phi}^{\bar k, \tilde k}) \mathop {\min} \left[ \begin{array}{l} \underline \kappa_{\phi,\tau}^{\bar k, \tilde k} \left( \tau_{\tilde k} - \tau_{\bar k} \right), \vspace{1mm} \\ \overline \kappa_{\phi,\tau}^{\bar k, \tilde k} \left( \tau_{\tilde k} - \tau_{\bar k} \right) \end{array} \right]  - \underline W_{\phi, \bar k},
\end{array}
\end{equation}

\noindent the robustness of which is easily seen by the fact that the resulting bounds must be more conservative than the bounds of the version with the exact derivatives (\ref{eq:costbound2}).

Consider as well the bounds (\ref{eq:lipboundcostU2}) and (\ref{eq:lipboundcostL2}), the robust versions of which are given as

\vspace{-2mm}
\begin{equation}\label{eq:lipboundcostU2grad}
\begin{array}{l}
\phi_{p} ({\bf u}_{\bar k},\tau_{\bar k}) < \overline \phi_{p} ({\bf u}_{\tilde k},\tau_{\tilde k}) \vspace{1mm} \\
\displaystyle \hspace{15mm} + \eta_{c,\phi}^{\bar k, \tilde k} \mathop {\max} \left[ \begin{array}{l} \displaystyle \frac{\partial \underline \phi_{p}}{\partial \tau} \Big |_{({\bf u}_{\tilde k},\tau_{\tilde k})} ( \tau_{\bar k} - \tau_{\tilde k} ) , \vspace{1mm} \\ \displaystyle \frac{\partial \overline \phi_{p}}{\partial \tau} \Big |_{({\bf u}_{\tilde k},\tau_{\tilde k})} ( \tau_{\bar k} - \tau_{\tilde k} ) \end{array} \right] \vspace{1mm} \\
\displaystyle  \hspace{15mm} + (1-\eta_{c,\phi}^{\bar k, \tilde k})  \mathop {\max} \left[ \begin{array}{l} \underline \kappa_{\phi,\tau}^{\bar k, \tilde k} ( \tau_{\bar k} - \tau_{\tilde k} ), \vspace{1mm} \\ \overline \kappa_{\phi,\tau}^{\bar k, \tilde k} ( \tau_{\bar k} - \tau_{\tilde k} ) \end{array} \right] \vspace{1mm}  \\
 \displaystyle \hspace{15mm} + \sum_{i \in I_{c,\phi}^{\bar k, \tilde k}} \mathop {\max} \left[ \begin{array}{l} \displaystyle \frac{\partial \underline \phi_p}{\partial u_i} \Big |_{({\bf u}_{\tilde k},\tau_{\tilde k})} ( u_{\bar k,i} - u_{\tilde k,i} ), \vspace{1mm} \\ \displaystyle \frac{\partial \overline \phi_p}{\partial u_i} \Big |_{({\bf u}_{\tilde k},\tau_{\tilde k})} ( u_{\bar k,i} - u_{\tilde k,i} ) \end{array} \right] \vspace{1mm} \\
\displaystyle \hspace{15mm} + \sum_{i \not \in I_{c,\phi}^{\bar k, \tilde k}} \mathop {\max} \left[ \begin{array}{l} \underline \kappa_{\phi,i}^{\bar k, \tilde k} ( u_{\bar k,i} - u_{\tilde k,i} ), \vspace{1mm} \\ \overline \kappa_{\phi,i}^{\bar k, \tilde k} ( u_{\bar k,i} - u_{\tilde k,i} ) \end{array} \right],
\end{array}
\end{equation}

\vspace{-2mm}
\begin{equation}\label{eq:lipboundcostL2grad}
\begin{array}{l}
\phi_{p} ({\bf u}_{\bar k},\tau_{\bar k}) > \underline \phi_{p} ({\bf u}_{\tilde k},\tau_{\tilde k}) \vspace{1mm} \\
\displaystyle \hspace{15mm} + \eta_{v,\phi}^{\bar k, \tilde k} \mathop {\min} \left[ \begin{array}{l} \displaystyle \frac{\partial \underline \phi_{p}}{\partial \tau} \Big |_{({\bf u}_{\tilde k},\tau_{\tilde k})} ( \tau_{\bar k} - \tau_{\tilde k} ) , \vspace{1mm} \\ \displaystyle \frac{\partial \overline \phi_{p}}{\partial \tau} \Big |_{({\bf u}_{\tilde k},\tau_{\tilde k})} ( \tau_{\bar k} - \tau_{\tilde k} ) \end{array} \right] \vspace{1mm} \\
\displaystyle  \hspace{15mm} + (1-\eta_{v,\phi}^{\bar k, \tilde k})  \mathop {\min} \left[ \begin{array}{l} \underline \kappa_{\phi,\tau}^{\bar k, \tilde k} ( \tau_{\bar k} - \tau_{\tilde k} ), \vspace{1mm} \\ \overline \kappa_{\phi,\tau}^{\bar k, \tilde k} ( \tau_{\bar k} - \tau_{\tilde k} ) \end{array} \right] \vspace{1mm}  \\
 \displaystyle \hspace{15mm} + \sum_{i \in I_{v,\phi}^{\bar k, \tilde k}} \mathop {\min} \left[ \begin{array}{l} \displaystyle \frac{\partial \underline \phi_p}{\partial u_i} \Big |_{({\bf u}_{\tilde k},\tau_{\tilde k})} ( u_{\bar k,i} - u_{\tilde k,i} ), \vspace{1mm} \\ \displaystyle \frac{\partial \overline \phi_p}{\partial u_i} \Big |_{({\bf u}_{\tilde k},\tau_{\tilde k})} ( u_{\bar k,i} - u_{\tilde k,i} ) \end{array} \right] \vspace{1mm} \\
\displaystyle \hspace{15mm} + \sum_{i \not \in I_{v,\phi}^{\bar k, \tilde k}} \mathop {\min} \left[ \begin{array}{l} \underline \kappa_{\phi,i}^{\bar k, \tilde k} ( u_{\bar k,i} - u_{\tilde k,i} ), \vspace{1mm} \\ \overline \kappa_{\phi,i}^{\bar k, \tilde k} ( u_{\bar k,i} - u_{\tilde k,i} ) \end{array} \right].
\end{array}
\end{equation}

\noindent These follow from (\ref{eq:lipboundcostU2}) and (\ref{eq:lipboundcostL2}) by the same reasoning as before.

\subsection{Gradient Uncertainty in the Lipschitz Consistency Check}

Of the inequalities considered in the robust version of Algorithm 1, only (\ref{eq:lipcheck2UMN}) and (\ref{eq:lipcheck2LMN}) use derivatives and thus require modification. Using the same concepts as before, we modify them as follows:

\vspace{-2mm}
\begin{equation}\label{eq:lipcheck2UMNgrad}
\begin{array}{l}
\underline \phi_{p} ({\bf u}_{k_2},\tau_{k_2}) \leq \\
\overline \phi_{p} ({\bf u}_{k_1},\tau_{k_1}) + \mathop {\max} \left[ \begin{array}{l} \underline \kappa_{\phi,\tau} ( \tau_{k_2} - \tau_{k_1} ) ,  \vspace{1mm}\\ \overline \kappa_{\phi,\tau} ( \tau_{k_2} - \tau_{k_1} ) \end{array} \right]  \vspace{1mm} \\
\displaystyle + \sum_{i=1}^{n_u} \mathop {\max} \left[ \begin{array}{l} \displaystyle \frac{\partial \underline \phi_{p}}{\partial u_i} \Big |_{({\bf u}_{k_1},\tau_{k_1})} ( u_{k_2,i} - u_{k_1,i} ) , \\ \displaystyle \frac{\partial \overline \phi_{p}}{\partial u_i} \Big |_{({\bf u}_{k_1},\tau_{k_1})} ( u_{k_2,i} - u_{k_1,i} ) \end{array} \right] \vspace{1mm} \\
+\displaystyle \frac{1}{2} \sum_{i_1=1}^{n_u} \sum_{i_2=1}^{n_u} \mathop {\max} \left[ \begin{array}{l} \underline M_{\phi,i_1 i_2} (u_{k_2,i_1} - u_{k_1,i_1}) \\
\hspace{20mm} (u_{k_2,i_2} - u_{{k_1},i_2}), \\ \overline M_{\phi,i_1 i_2} (u_{k_2,i_1} - u_{{k_1},i_1}) \\
\hspace{20mm}(u_{k_2,i_2} - u_{{k_1},i_2}) \end{array} \right],
\end{array}
\end{equation}

\vspace{-2mm}
\begin{equation}\label{eq:lipcheck2LMNgrad}
\begin{array}{l}
\overline \phi_{p} ({\bf u}_{k_2},\tau_{k_2}) \geq \\
\underline \phi_{p} ({\bf u}_{k_1},\tau_{k_1}) + \mathop {\min} \left[ \begin{array}{l} \underline \kappa_{\phi,\tau} ( \tau_{k_2} - \tau_{k_1} ) , \vspace{1mm} \\ \overline \kappa_{\phi,\tau} ( \tau_{k_2} - \tau_{k_1} ) \end{array} \right]  \vspace{1mm} \\
\displaystyle + \sum_{i=1}^{n_u} \mathop {\min} \left[ \begin{array}{l} \displaystyle \frac{\partial \underline \phi_{p}}{\partial u_i} \Big |_{({\bf u}_{k_1},\tau_{k_1})} ( u_{k_2,i} - u_{k_1,i} ) , \\ \displaystyle \frac{\partial \overline \phi_{p}}{\partial u_i} \Big |_{({\bf u}_{k_1},\tau_{k_1})} ( u_{k_2,i} - u_{k_1,i} ) \end{array} \right]
\end{array}
\end{equation}

$$
\begin{array}{l}
+\displaystyle \frac{1}{2} \sum_{i_1=1}^{n_u} \sum_{i_2=1}^{n_u} \mathop {\min} \left[ \begin{array}{l} \underline M_{\phi,i_1 i_2} (u_{k_2,i_1} - u_{k_1,i_1}) \\
\hspace{20mm} (u_{k_2,i_2} - u_{{k_1},i_2}), \\ \overline M_{\phi,i_1 i_2} (u_{k_2,i_1} - u_{{k_1},i_1}) \\
\hspace{20mm}(u_{k_2,i_2} - u_{{k_1},i_2}) \end{array} \right].
\end{array}
$$

\noindent It follows that any Lipschitz constants that are inconsistent for (\ref{eq:lipcheck2UMNgrad}) and (\ref{eq:lipcheck2LMNgrad}) will be inconsistent for (\ref{eq:lipcheck2UMN}) and (\ref{eq:lipcheck2LMN}) as well.

\begin{figure*}
\begin{center}
\includegraphics[width=16cm]{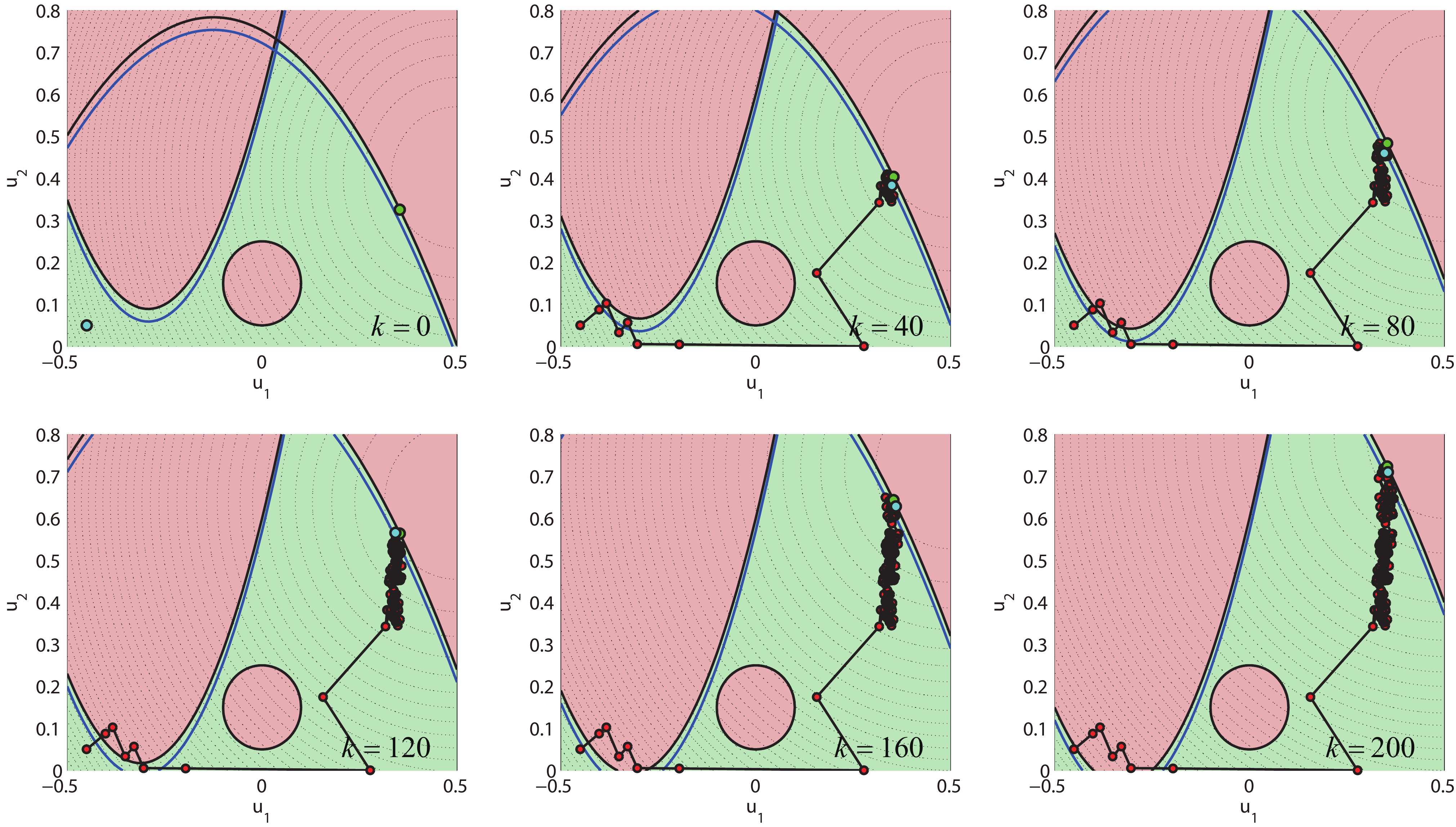}
\caption{Chain of experiments generated by applying the modified SCFO methodology to Problem (\ref{eq:exdeg}) for the ($-$) scenario with estimation error in the experimental function gradients accounted for. $\alpha_\sigma = 0.05$ is used here.}
\label{fig:degH}
\end{center}
\end{figure*}

\begin{figure*}
\begin{center}
\includegraphics[width=16cm]{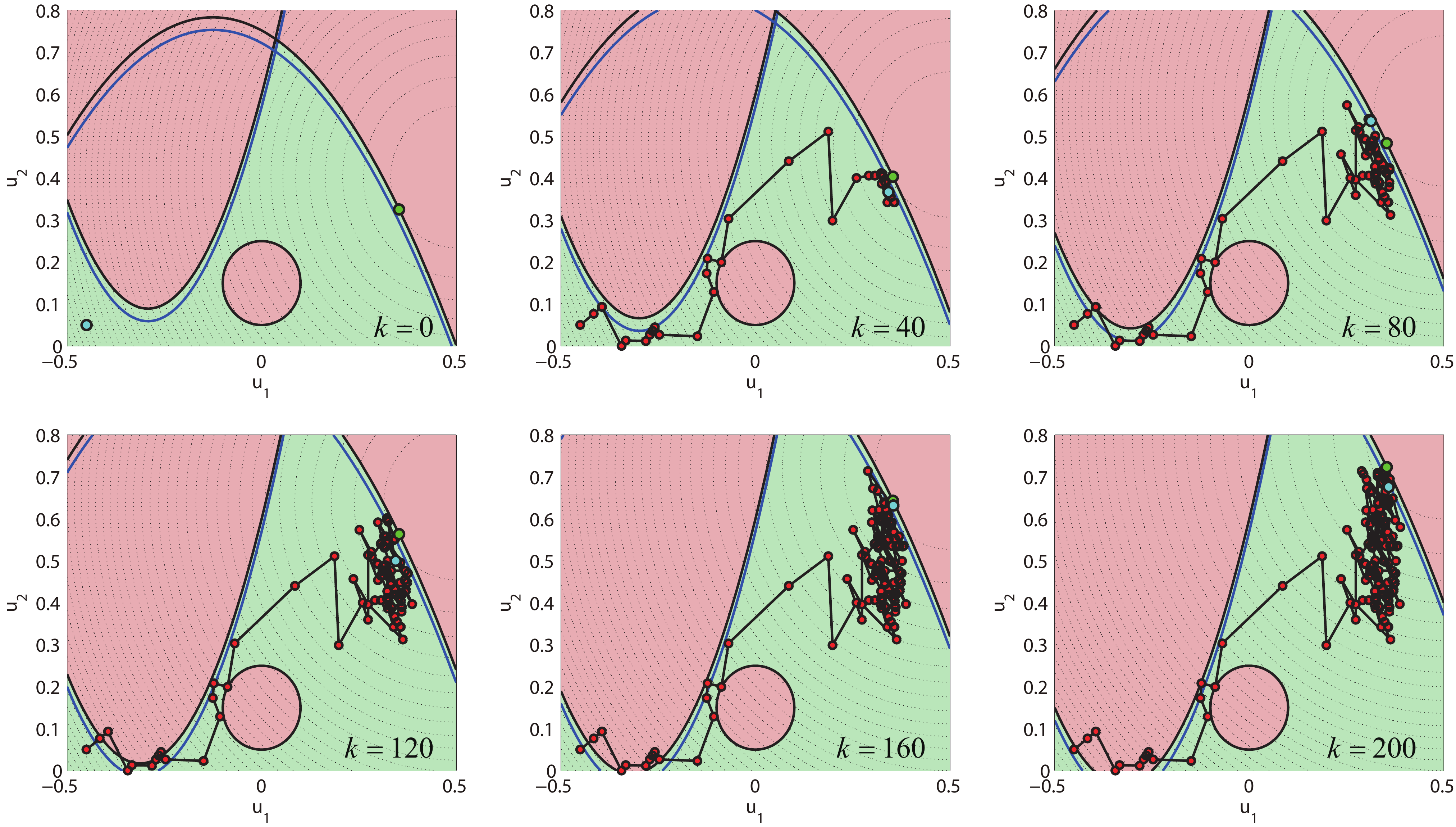}
\caption{Chain of experiments generated by applying the modified SCFO methodology to Problem (\ref{eq:exdeg}) for the ($-$) scenario with estimation error in the experimental function gradients accounted for. $\alpha_\sigma = 0.15$ is used here.}
\label{fig:degHa}
\end{center}
\end{figure*}

\begin{figure*}
\begin{center}
\includegraphics[width=16cm]{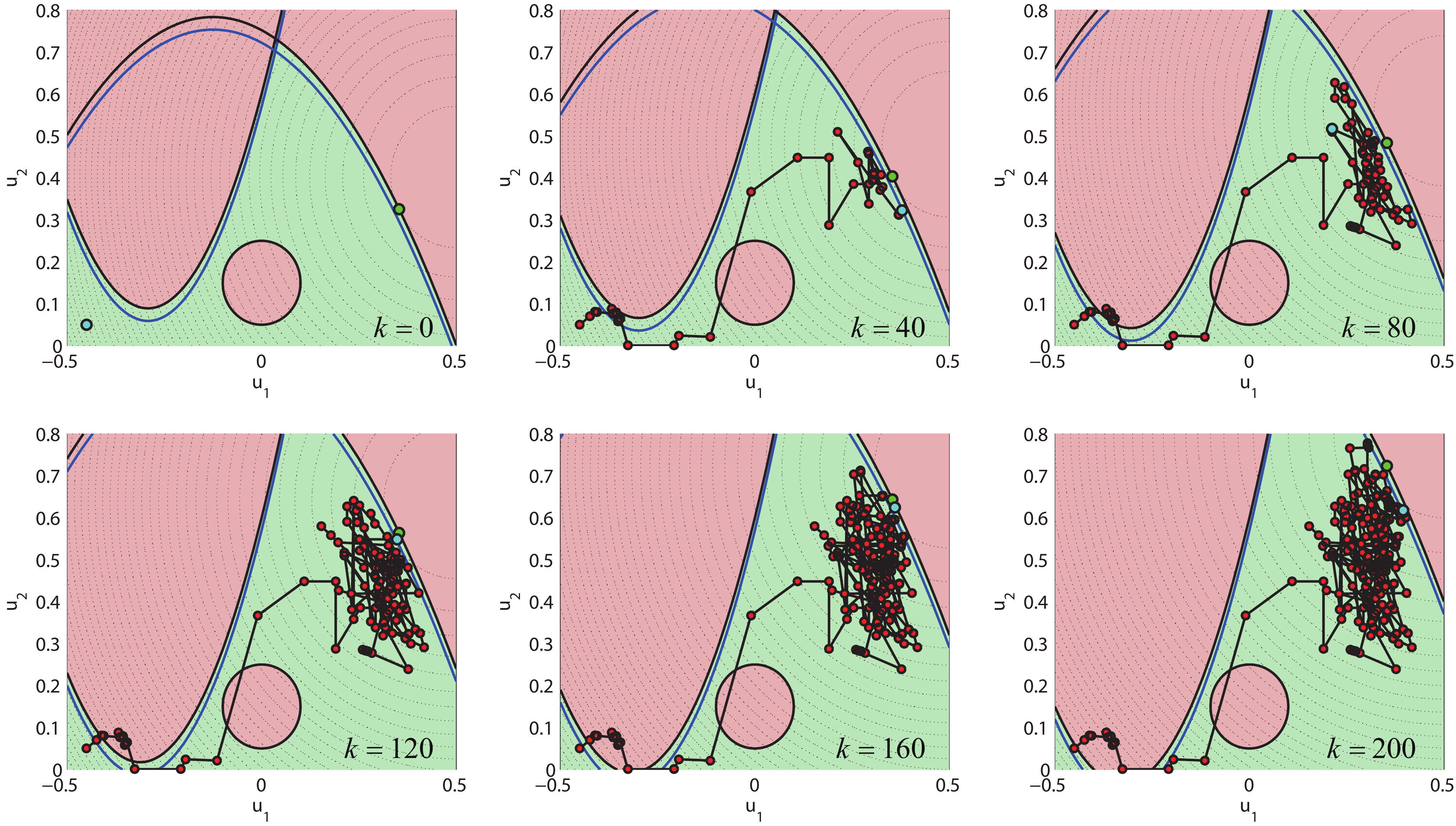}
\caption{Chain of experiments generated by applying the modified SCFO methodology to Problem (\ref{eq:exdeg}) for the ($-$) scenario with estimation error in the experimental function gradients accounted for. $\alpha_\sigma = 0.25$ is used here.}
\label{fig:degHb}
\end{center}
\end{figure*}

\subsection{Example}
\label{sec:gradestex}

Considering the example of Section \ref{sec:sec4ex}, we suppose that the gradient estimation algorithm is able to obtain derivative estimates that are the true derivatives corrupted by additive uniform noise, i.e.:

\vspace{-2mm}
\begin{equation}\label{eq:gradestcor}
\begin{array}{l}
\displaystyle \frac{\partial \hat \phi_p}{\partial u_i} \Big |_{({\bf u}_{\bar k}, \tau_{\bar k})} := \vspace{1mm} \\
\hspace{10mm} \displaystyle \frac{\partial \phi_p}{\partial u_i} \Big |_{({\bf u}_{\bar k}, \tau_{\bar k})} + \alpha_{\sigma} (\overline \kappa_{\phi,i} - \underline \kappa_{\phi,i} ) \mathcal{U} [-1,1], \vspace{1mm}\\
\displaystyle \frac{\partial \hat \phi_p}{\partial \tau} \Big |_{({\bf u}_{\bar k}, \tau_{\bar k})} := \vspace{1mm} \\
\hspace{10mm} \displaystyle \frac{\partial \phi_p}{\partial \tau} \Big |_{({\bf u}_{\bar k}, \tau_{\bar k})} + \alpha_{\sigma} (\overline \kappa_{\phi,\tau} - \underline \kappa_{\phi,\tau} ) \mathcal{U} [-1,1], \vspace{1mm} \\
\displaystyle \frac{\partial \hat \phi_p}{\partial u_i} \Big |_{({\bf u}_{\bar k}, \tau_{ k+1})} := \vspace{1mm} \\
\hspace{10mm} \displaystyle \frac{\partial \phi_p}{\partial u_i} \Big |_{({\bf u}_{\bar k}, \tau_{k+1})} + \alpha_{\sigma} (\overline \kappa_{\phi,i} - \underline \kappa_{\phi,i} ) \mathcal{U} [-1,1], \vspace{1mm}\\
\displaystyle \frac{\partial \hat \phi_p}{\partial \tau} \Big |_{({\bf u}_{\bar k}, \tau_{k+1})} := \vspace{1mm} \\
\hspace{10mm} \displaystyle \frac{\partial \phi_p}{\partial \tau} \Big |_{({\bf u}_{\bar k}, \tau_{k+1})} + \alpha_{\sigma} (\overline \kappa_{\phi,\tau} - \underline \kappa_{\phi,\tau} ) \mathcal{U} [-1,1],
\end{array}
\end{equation}

\noindent with analogous formulas for the constraints. Here, $\alpha_\sigma > 0$ is used to control the size of the uncertainty, with the scaling taken care of by multiplying with the range that the derivatives may take, which is simply the difference between the upper and lower global Lipschitz constants.

Valid bounds are then computed as

\vspace{-2mm}
\begin{equation}\label{eq:gradestbounds}
\begin{array}{l}
\displaystyle \frac{\partial \underline \phi_p}{\partial u_i} \Big |_{({\bf u}_{\bar k}, \tau_{\bar k})} := \displaystyle \frac{\partial \hat \phi_p}{\partial u_i} \Big |_{({\bf u}_{\bar k}, \tau_{\bar k})} - \alpha_{\sigma} (\overline \kappa_{\phi,i} - \underline \kappa_{\phi,i} ), \vspace{1mm}\\
\displaystyle \frac{\partial \overline \phi_p}{\partial u_i} \Big |_{({\bf u}_{\bar k}, \tau_{\bar k})} := \displaystyle \frac{\partial \hat \phi_p}{\partial u_i} \Big |_{({\bf u}_{\bar k}, \tau_{\bar k})} + \alpha_{\sigma} (\overline \kappa_{\phi,i} - \underline \kappa_{\phi,i} ), \vspace{1mm}\\
\displaystyle \frac{\partial \underline \phi_p}{\partial \tau} \Big |_{({\bf u}_{\bar k}, \tau_{\bar k})} := \displaystyle \frac{\partial \hat \phi_p}{\partial \tau} \Big |_{({\bf u}_{\bar k}, \tau_{\bar k})} - \alpha_{\sigma} (\overline \kappa_{\phi,\tau} - \underline \kappa_{\phi,\tau} ), \vspace{1mm}\\
\displaystyle \frac{\partial \overline \phi_p}{\partial \tau} \Big |_{({\bf u}_{\bar k}, \tau_{\bar k})} := \displaystyle \frac{\partial \hat \phi_p}{\partial \tau} \Big |_{({\bf u}_{\bar k}, \tau_{\bar k})} + \alpha_{\sigma} (\overline \kappa_{\phi,\tau} + \underline \kappa_{\phi,\tau} ),
\end{array}
\end{equation}

\vspace{-2mm}
$$
\begin{array}{l}

\displaystyle \frac{\partial \underline \phi_p}{\partial u_i} \Big |_{({\bf u}_{\bar k}, \tau_{k+1})} := \displaystyle \frac{\partial \hat \phi_p}{\partial u_i} \Big |_{({\bf u}_{\bar k}, \tau_{k+1})} - \alpha_{\sigma} (\overline \kappa_{\phi,i} - \underline \kappa_{\phi,i} ), \vspace{1mm}\\
\displaystyle \frac{\partial \overline \phi_p}{\partial u_i} \Big |_{({\bf u}_{\bar k}, \tau_{k+1})} := \displaystyle \frac{\partial \hat \phi_p}{\partial u_i} \Big |_{({\bf u}_{\bar k}, \tau_{k+1})} + \alpha_{\sigma} (\overline \kappa_{\phi,i} - \underline \kappa_{\phi,i} ), \vspace{1mm} \\
\displaystyle \frac{\partial \underline \phi_p}{\partial \tau} \Big |_{({\bf u}_{\bar k}, \tau_{k+1})} := \displaystyle \frac{\partial \hat \phi_p}{\partial \tau} \Big |_{({\bf u}_{\bar k}, \tau_{k+1})} - \alpha_{\sigma} (\overline \kappa_{\phi,\tau} - \underline \kappa_{\phi,\tau} ), \vspace{1mm}\\
\displaystyle \frac{\partial \overline \phi_p}{\partial \tau} \Big |_{({\bf u}_{\bar k}, \tau_{k+1})} := \displaystyle \frac{\partial \hat \phi_p}{\partial \tau} \Big |_{({\bf u}_{\bar k}, \tau_{k+1})} + \alpha_{\sigma} (\overline \kappa_{\phi,\tau} + \underline \kappa_{\phi,\tau} ).
\end{array}
$$

Results are provided for increasing values of $\alpha_\sigma$ (0.05, 0.15, and 0.25) in Figs. \ref{fig:degH}-\ref{fig:degHb}, with the corresponding cost values shown in Fig. \ref{fig:degHcost}. As expected, uncertainty in the gradient estimates does not affect the feasibility of the experimental iterates, as this is guaranteed by the robust modifications. However, performance with regard to convergence speed does degrade with increasing uncertainty, in that the general trend remains the same but with the variance of the iterates increasing with $\alpha_\sigma$. Additionally, it is more difficult for the implementation to ``maneuver'' around the experimental concave constraint as the uncertainty increases, which leads to an increase in the number of experiments needed to get to the neighborhood of the optimum.

We note that the gradient estimation algorithm supposed for this example is artificial in nature and used only for its simplicity -- for a look at how real estimation algorithms would perform in the SCFO context, the interested reader is referred to the examples in \cite{SCFOug}.

\begin{figure}
\begin{center}
\includegraphics[width=8cm]{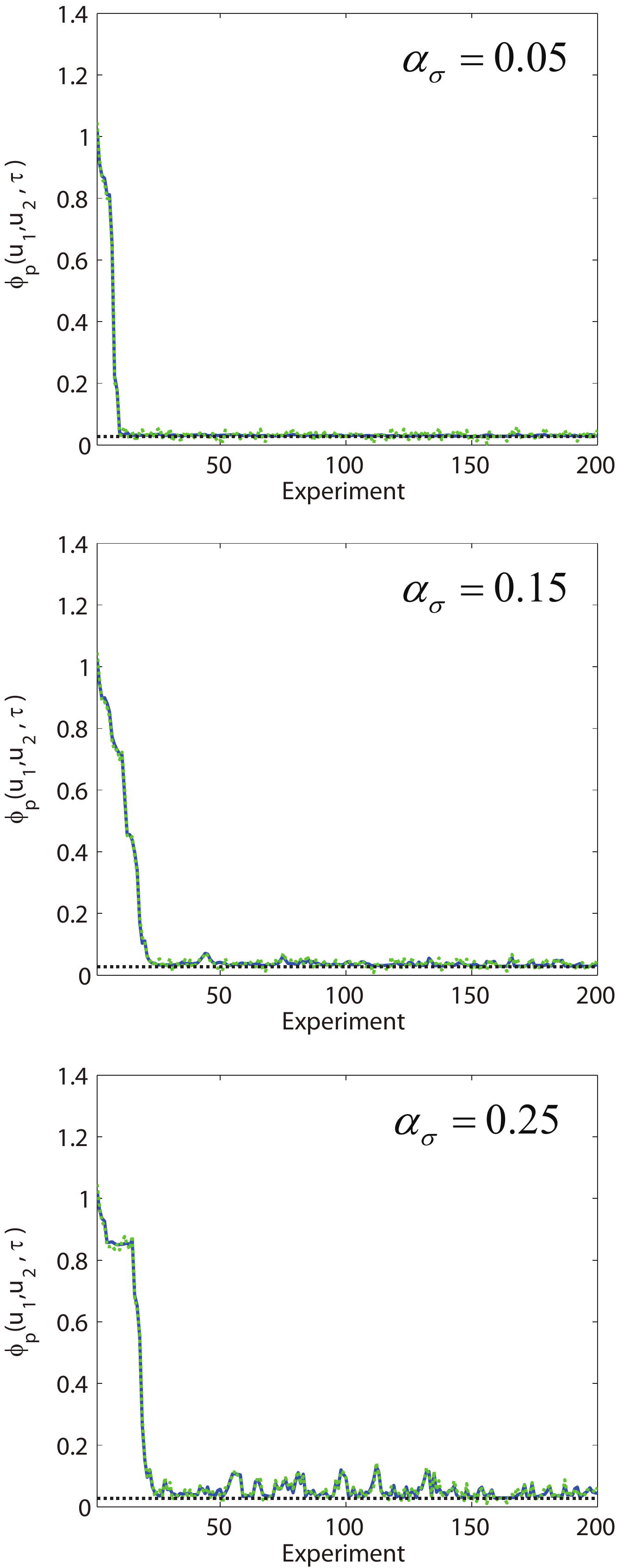}
\caption{Cost function values obtained by the modified SCFO methodology for Problem (\ref{eq:exdeg}) for the ($-$) scenario with estimation error in the experimental function gradients accounted for.}
\label{fig:degHcost}
\end{center}
\end{figure}

\section{Guarantee of Sufficient Excitation}
\label{sec:suffexc}

The proper management of the information-optimization tradeoff is a common requirement in both experimental and derivative-free optimization frameworks. Stated in simple terms, this basically means that one cannot \emph{only} optimize since to optimize well one needs to sample the experimental space in a manner that is adequate for getting the information necessary for optimization. In the SCFO framework, this means sufficiently perturbing the decision variables in all directions so as to be able to allow the gradient estimation algorithm to properly estimate the full gradients of the experimental functions in question \cite{Rodger2010a,Bunin2013a}. Similarly, derivative-free optimization algorithms often need to ensure the more general property of well-poisedness so as to be able to build proper interpolation or regression models from the obtained data \cite{Conn2008,Conn2009}. Finally, methods employing parametric models of the experimental functions may also want to perturb the decision variables in a way that makes identification of these parameters possible \cite{Pfaff2001,Pfaff2006}. 

Regardless of the context, the basic requirement is mathematically the same -- one should be able, when needed, to sufficiently excite the decision variables without incurring any significant drawbacks. As temporary economic losses reflected by increase in the cost function are, in the general case, inevitable once one overrides optimization to gather information, we will not consider these as ``significant''. What is significant, however, are constraint violations, which are not admissible and must be avoided even while perturbing the system for information. This issue becomes particularly relevant when one approaches a constrained optimum.

In this section, we will propose to solve this problem by ensuring that there always exists a ball of radius $\delta_e > 0$ around the reference iterate ${\bf u}_{k^*}$:

\vspace{-2mm}
\begin{equation}\label{eq:feasball}
\mathcal{B}_{e,k^*} = \{ {\bf u} : \| {\bf u} - {\bf u}_{k^*} \|_2 \leq \delta_e \},
\end{equation}

\noindent such that all ${\bf u} \in \mathcal{B}_{e,k^*}$ are guaranteed to be feasible. The following two points justify this choice of approach:

\begin{itemize}
\item By ensuring the existence of a feasible \emph{ball}, one is able to perturb in all directions and to generate a sample set with any geometry.
\item Enforcing that the ball be of a certain radius $\delta_e$ allows for one to tune the size of the excitations with respect to the noise or estimation error, as perturbations that are too small may, for example, lead to overly corrupted gradient estimates \cite{Marchetti2010,Rodger2010a}.  
\end{itemize}

It should be clear that guaranteeing the existence of such a ball will be impossible if a constraint is approached too closely -- one cannot sit on an active constraint and perturb in absolutely any direction without losing feasibility. From this, it follows that a natural way to guarantee that an entire ball is feasible is by adding constraint \emph{back-offs}, denoted here by $b > 0$, to the constraints of the original problem (\ref{eq:mainprobdeg}):

\vspace{-2mm}
\begin{equation}\label{eq:mainprobback}
\begin{array}{rll}
\mathop {{\rm{minimize}}}\limits_{\bf{u}} & \phi_p ({\bf{u}},\tau) & \\
{\rm{subject}}\hspace{1mm}{\rm{to}} & g_{p,j}({\bf{u}},\tau) + b_{p,j}^{k^*,k} \leq 0, & j = 1,...,n_{g_p} \\
 & g_{j}({\bf{u}}) + b_{j}^{k^*} \leq 0, & j = 1,...,n_{g} \\
& {\bf{u}}^L + {\bf b}_u \preceq {\bf u} \preceq {\bf u}^U - {\bf b}_u, &
\end{array}
\end{equation}

\noindent which will allow us to achieve our stated goal of guaranteeing a feasible sufficient-excitation ball at the cost of potential suboptimality, which will naturally result due to the feasible space being tightened. So as to work with a very general case, we make the back-offs for the experimental constraints \emph{iteration-dependent} with respect to both $k^*$ and $k$ so as to properly handle degradation effects and to exploit some of the relaxation theory developed in the earlier sections. For the numerical constraints, it will be shown that iteration dependence with respect to $k^*$ only is needed since degradation is absent in these constraints. As the bound constraints are not subject to degradation and cannot be relaxed, we make their back-offs constant and independent of iteration. The basic idea of adding a back-off to guarantee the existence of a feasible ball is illustrated geometrically in Fig. \ref{fig:minballSIAM}.

\begin{figure}
\begin{center}
\includegraphics[width=8cm]{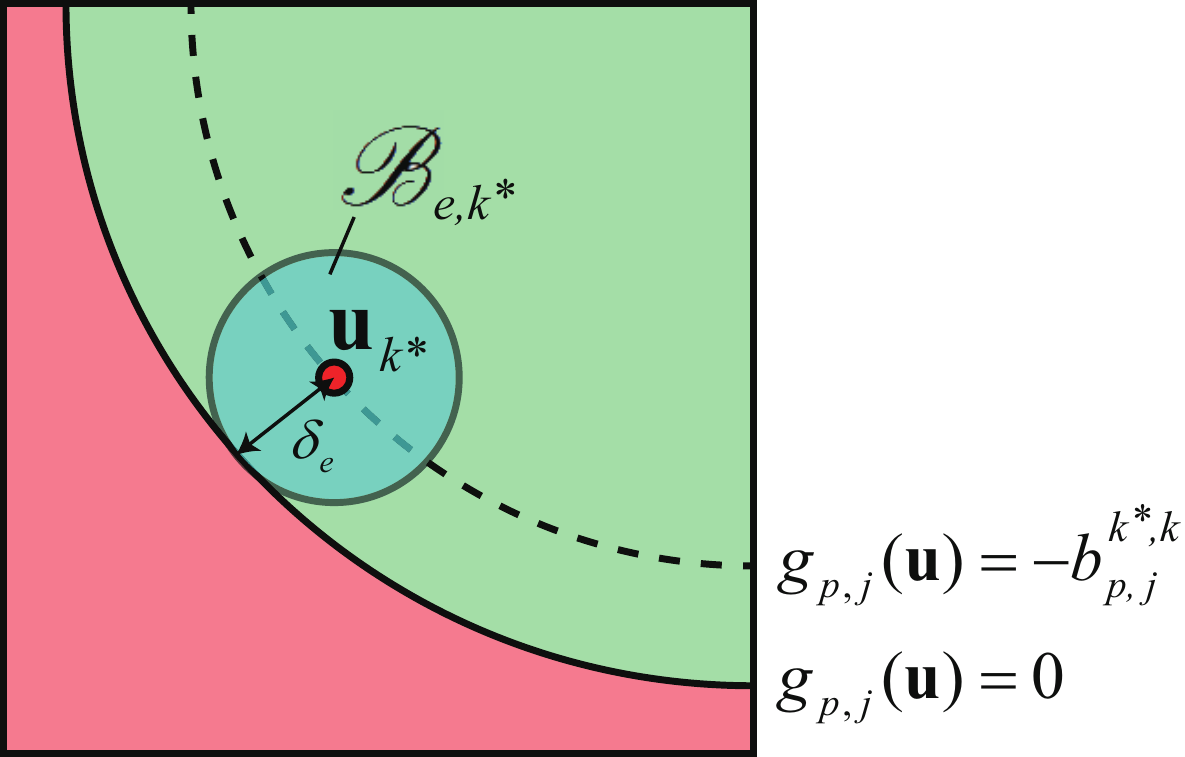}
\caption{Guaranteeing the existence of a sufficient-excitation ball by adding a constraint back-off. A single experimental constraint is considered here.}
\label{fig:minballSIAM}
\end{center}
\end{figure} 

In what follows, we will derive the back-off values sufficient to guarantee the feasibility of $\mathcal{B}_{e,k^*}$, and then proceed to describe how the SCFO should be modified further to account for solving (\ref{eq:mainprobback}) instead of (\ref{eq:mainprobdeg}).

\subsection{Determining the Sufficient Back-offs}

Although we seek to guarantee

\vspace{-2mm}
\begin{equation}\label{eq:feasibility}
\begin{array}{l}
g_{p,j} ({\bf u},\tau_{k+1}) \leq 0, \;\; \forall j = 1,...,n_{g_p}, \vspace{1mm} \\
g_{j} ({\bf u}) \leq 0, \;\; \forall j = 1,...,n_{g}, \vspace{1mm}\\ 
{\bf u}^U \preceq {\bf u} \preceq {\bf u}^U 
\end{array}
\end{equation}

\noindent $\forall {\bf u} \in \mathcal{B}_{e,k^*}$, we may proceed by deriving the appropriate back-off for each constraint separately as the back-offs do not ``interact'' -- i.e., a back-off that guarantees the existence of a ball that is feasible with respect to one constraint in no way influences the value of the back-off needed to guarantee the existence of such a ball with respect to another.

We will first carry out the theoretical analysis for a given experimental constraint $g_{p,j}$. So as to derive back-offs that reduce conservatism by exploiting local relaxations, we define the following local temporal experimental space:

\vspace{-2mm}
\begin{equation}\label{eq:locspace3}
\mathcal{I}_{\tau}^{e,k^*} = \mathcal{B}_{e,k^*} \times \left\{ \tau : \tau_{k^*} \leq \tau \leq \tau_{k+1} \right\},
\end{equation}

\noindent and the corresponding local Lipschitz constants

\vspace{-2mm}
\begin{equation}\label{eq:lipcondegLUlocball}
\begin{array}{l}
\displaystyle \underline \kappa_{p,ji} \leq \underline \kappa_{p,ji}^{e,k^*} < \frac{\partial g_{p,j}}{\partial u_i} \Big |_{({\bf u},\tau)} < \overline \kappa_{p,ji}^{e,k^*} \leq \overline \kappa_{p,ji}, \\
\hspace{50mm} \forall( {\bf u},\tau) \in \mathcal{I}_\tau^{e,k^*},
\end{array}
\end{equation}

\vspace{-2mm}
\begin{equation}\label{eq:lipdegLUloc2ball}
\begin{array}{l}
\displaystyle \underline \kappa_{p,j\tau} \leq \underline \kappa_{p,j\tau}^{e,k^*} \leq \frac{\partial g_{p,j}}{\partial \tau} \Big |_{({\bf u},\tau)} \leq \overline \kappa_{p,j\tau}^{e,k^*} \leq \overline \kappa_{p,j\tau}, \\
\hspace{50mm} \forall ({\bf u},\tau) \in \mathcal{I}_\tau^{e,k^*}.
\end{array}
\end{equation}

Likewise, we will define the concavity index sets $\eta_{c,j}^{e,k^*}$ and $I_{c,j}^{e,k^*}$ in the same manner as $\eta_{c,j}^{\bar k}$ and $I_{c,j}^{\bar k}$ but with respect to $\mathcal{I}_\tau^{e,k^*}$ rather than $\mathcal{I}_\tau^{\bar k}$. By virtue of all of the same principles as before, this allows us to state the bound

\vspace{-2mm}
\begin{equation}\label{eq:ballbound}
\begin{array}{l} 
g_{p,j} ({\bf u},\tau_{k+1}) \leq \overline g_{p,j} ({\bf u}_{k^*},\tau_{k^*}) \vspace{1mm} \\
\hspace{10mm} \displaystyle +\eta_{c,j}^{e,k^*} \frac{\partial \overline g_{p,j}}{\partial \tau} \Big |_{({\bf u}_{k^*},\tau_{k^*})} ( \tau_{k+1} - \tau_{{k^*}} ) \vspace{1mm} \\
\hspace{10mm} \displaystyle + (1-\eta_{c,j}^{e,k^*}) \overline \kappa_{p,j\tau}^{e,k^*} \left( \tau_{k+1} - \tau_{k^*} \right) \vspace{1mm}\\
\hspace{10mm} \displaystyle   + \sum_{i \in I_{c,j}^{e,k^*}} \mathop {\max} \left[ \begin{array}{l} \displaystyle \frac{\partial \underline g_{p,j}}{\partial u_i} \Big |_{({\bf u}_{k^*},\tau_{k^*})} ( u_{i} - u_{k^*,i} ), \vspace{1mm}\\ \displaystyle \frac{\partial \overline g_{p,j}}{\partial u_i} \Big |_{({\bf u}_{k^*},\tau_{k^*})} ( u_{i} - u_{k^*,i} ) \end{array} \right] \vspace{1mm} \\
\hspace{10mm} + \displaystyle \sum_{i \not \in I_{c,j}^{e,k^*}} \mathop {\max} \left[ \begin{array}{l} \underline \kappa_{p,ji}^{e,k^*} ( u_{i} - u_{k^*,i} ), \vspace{1mm}\\ \overline \kappa_{p,ji}^{e,k^*} ( u_{i} - u_{k^*,i} ) \end{array} \right]
\end{array}
\end{equation}

\noindent for all ${\bf u} \in \mathcal{B}_{e,k^*}$.

The major theoretical result of this section is given in the following theorem.

\begin{theorem}[Sufficient back-off for an experimental constraint function]
\label{thm:backoff}
Let the reference point ${\bf u}_{k^*}$ satisfy the backed off robust constraint $\overline g_{p,j} ({\bf u}_{k^*},\tau_{k^*}) + b_{p,j}^{k^*,k} \leq 0$. Defining

\vspace{-2mm}
\begin{equation}\label{eq:kappastar}
\begin{array}{l}
\kappa_{p,ji}^{m,k^*} = \left\{ \begin{array}{ll}  \mathop {\max} \left[ \begin{array}{l} \displaystyle \Bigg |  \frac{\partial \underline g_{p,j}}{\partial u_i} \Big |_{({\bf u}_{k^*},\tau_{k^*})} \Bigg |,\vspace{1mm} \\ \displaystyle \Bigg | \frac{\partial \overline g_{p,j}}{\partial u_i} \Big |_{({\bf u}_{k^*},\tau_{k^*})} \Bigg | \end{array} \right], & i \in I_{c,j}^{e,k^*} \vspace{1mm} \\  \mathop {\max} \left[  | \underline \kappa_{p,ji}^{e,k^*} |, | \overline \kappa_{p,ji}^{e,k^*} | \right], & i \not \in I_{c,j}^{e,k^*}  \end{array} \right . \vspace{1mm} \\
\kappa_{p,j}^{m,k^*} = \left[ \kappa_{p,j1}^{m,k^*}\; \hdots \; \kappa_{p,jn_u}^{m,k^*}  \right]^T
\end{array}
\end{equation}

\noindent and setting the back-off as

\vspace{-2mm}
\begin{equation}\label{eq:suffback}
\begin{array}{lll}
b_{p,j}^{k^*,k} & = & \displaystyle \eta_{c,j}^{e,k^*} \frac{\partial \overline g_{p,j}}{\partial \tau} \Big |_{({\bf u}_{k^*},\tau_{k^*})} ( \tau_{k+1} - \tau_{{k^*}} ) \vspace{1mm} \\
& & + (1-\eta_{c,j}^{e,k^*}) \overline \kappa_{p,j\tau}^{e,k^*} \left( \tau_{k+1} - \tau_{k^*} \right) + \delta_e \| \kappa_{p,j}^{m,k^*} \|_2
\end{array}
\end{equation}

\noindent guarantees that

\vspace{-2mm}
\begin{equation}\label{eq:feasguar}
g_{p,j} ({\bf u},\tau_{k+1}) \leq 0, \;\; \forall {\bf u} \in \mathcal{B}_{e,k^*}.
\end{equation}

\end{theorem}
\begin{proof}

We start with the following general analysis for any given $\underline x, \overline x, y \in \mathbb{R}$:

\vspace{-2mm}
\begin{equation}\label{eq:xy}
\begin{array}{l}
\underline x y \leq |\underline x| |y|, \; \overline x y \leq |\overline x| |y| \vspace{1mm} \\
\Rightarrow \underline x y \leq \mathop {\max} [ |\underline x|, |\overline x| ] |y|, \; \overline x y \leq \mathop {\max} [ |\underline x|, |\overline x| ] |y| \vspace{1mm} \\
\Leftrightarrow \mathop {\max} [\underline x y, \overline x y] \leq \mathop {\max} [ |\underline x|, |\overline x| ] |y|,
\end{array}
\end{equation}

\noindent from which we have

\vspace{-2mm}
\begin{equation}\label{eq:xy2}
\begin{array}{l}
\mathop {\max} \left[ \begin{array}{l} \displaystyle \frac{\partial \underline g_{p,j}}{\partial u_i} \Big |_{({\bf u}_{k^*},\tau_{k^*})} ( u_{i} - u_{k^*,i} ), \vspace{1mm}\\ \displaystyle \frac{\partial \overline g_{p,j}}{\partial u_i} \Big |_{({\bf u}_{k^*},\tau_{k^*})} ( u_{i} - u_{k^*,i} ) \end{array} \right] \vspace{1mm} \\
\hspace{30mm} \leq \kappa_{p,ji}^{m,k^*} | u_{i} - u_{k^*,i} |, \;\; i \in I_{c,j}^{e,k^*} \vspace{1mm} \\
\mathop {\max}  \left[ \begin{array}{l} \underline \kappa_{p,ji} (u_{i} - u_{k^*,i}), \vspace{1mm} \\ \overline \kappa_{p,ji} (u_{i} - u_{k^*,i}) \end{array} \right] \hspace{-1mm} \leq \kappa_{p,ji}^{m,k^*} | u_{i} - u_{k^*,i} |, \;\; i \not \in I_{c,j}^{e,k^*}.
\end{array}
\end{equation}

This then allows us to extend (\ref{eq:ballbound}) as

\vspace{-2mm}
\begin{equation}\label{eq:nsbound2}
\begin{array}{lll} 
g_{p,j} ({\bf u},\tau_{k+1}) & \leq & \overline g_{p,j} ({\bf u}_{k^*},\tau_{k^*}) \\
& & \displaystyle +\eta_{c,j}^{e,k^*} \frac{\partial \overline g_{p,j}}{\partial \tau} \Big |_{({\bf u}_{k^*},\tau_{k^*})} ( \tau_{k+1} - \tau_{{k^*}} ) \vspace{1mm} \\
& & \displaystyle + (1-\eta_{c,j}^{e,k^*}) \overline \kappa_{p,j\tau}^{e,k^*} \left( \tau_{k+1} - \tau_{k^*} \right) \vspace{1mm}\\
& & \displaystyle   + \sum_{i = 1}^{n_u} \kappa_{p,ji}^{m,k^*} | u_{i} - u_{k^*,i} |.
\end{array}
\end{equation}

Let us denote by $\mathcal{K}_{k^*}$ the set of decision variables for which the right-hand side of (\ref{eq:nsbound2}) is nonpositive:

\begin{equation}\label{eq:lipregionK}
\mathcal{K}_{k^*} = \left\{ {\bf u} : \begin{array}{l} \overline g_{p,j} ({\bf u}_{k^*},\tau_{k^*}) \\
 \displaystyle +\eta_{c,j}^{e,k^*} \frac{\partial \overline g_{p,j}}{\partial \tau} \Big |_{({\bf u}_{k^*},\tau_{k^*})} ( \tau_{k+1} - \tau_{{k^*}} ) \vspace{1mm} \\
 \displaystyle + (1-\eta_{c,j}^{e,k^*}) \overline \kappa_{p,j\tau}^{e,k^*} \left( \tau_{k+1} - \tau_{k^*} \right) \vspace{1mm}\\
 \displaystyle   + \sum_{i = 1}^{n_u} \kappa_{p,ji}^{m,k^*} | u_{i} - u_{k^*,i} | \leq 0 \end{array} \right\}.
\end{equation}

\noindent It follows that $g_{p,j} ({\bf u},\tau_{k+1}) \leq 0, \; \forall {\bf u} \in \mathcal{K}_{k^*} \cap \mathcal{B}_{e,k^*}$.

We will complete the proof by essentially inscribing the ball $\mathcal{B}_{e,k^*}$ inside $\mathcal{K}_{k^*}$ and showing that $\mathcal{B}_{e,k^*} \subseteq \mathcal{K}_{k^*}$ for the back-off proposed in (\ref{eq:suffback}) -- this, of course, implies that $g_{p,j} ({\bf u},\tau_{k+1}) \leq 0, \; \forall {\bf u} \in \mathcal{B}_{e,k^*}$ as well, which is our desired result. However, since $\mathcal{B}_{e,k^*}$ is a hypersphere and $\mathcal{K}_{k^*}$ is a convex set, it is sufficient\footnote{This is easily proven by considering the fact that any point in the interior of the sphere may be described as a convex combination of two boundary points. If all boundary points of $\mathcal{B}_{e,k^*}$ belong to the convex set $\mathcal{K}_{k^*}$, their convex combinations (the interior points) must as well.} to show that ${\rm bd} \left( \mathcal{B}_{e,k^*} \right) \subseteq \mathcal{K}_{k^*}$ since ${\rm bd} \left( \mathcal{B}_{e,k^*} \right) \subseteq \mathcal{K}_{k^*} \Leftrightarrow \mathcal{B}_{e,k^*} \subseteq \mathcal{K}_{k^*}$.

Let us parameterize the boundary points of $\mathcal{B}_{e,k^*}$ as

\vspace{-2mm}
\begin{equation}\label{eq:boundpoint}
{\bf u}_{\rm bd} = {\bf u}_{k^*} + \delta_e \delta {\bf u},
\end{equation}

\noindent where $\delta {\bf u} \in \mathbb{R}^{n_u}$ is a unit vector. Since the mapping between $\delta {\bf u}$ and ${\bf u}_{\rm bd}$ is bijective, the desired result may be obtained by proving that for any choice of $\delta {\bf u}$ the resulting boundary point belongs to $\mathcal{K}_{k^*}$. Substituting ${\bf u}_{\rm bd}$ for ${\bf u}$ in (\ref{eq:lipregionK}) yields

\vspace{-2mm}
\begin{equation}\label{eq:lipregionK2}
\begin{array}{ll}
\displaystyle \overline g_{p,j} ({\bf u}_{k^*},\tau_{k^*}) & \displaystyle + \eta_{c,j}^{e,k^*} \frac{\partial \overline g_{p,j}}{\partial \tau} \Big |_{({\bf u}_{k^*},\tau_{k^*})} ( \tau_{k+1} - \tau_{{k^*}} ) \vspace{1mm} \\
& \displaystyle  + (1-\eta_{c,j}^{e,k^*}) \overline \kappa_{p,j\tau}^{e,k^*} \left( \tau_{k+1} - \tau_{k^*} \right) \vspace{1mm} \\
&\displaystyle + \delta_e \sum_{i=1}^{n_u} \kappa_{p,ji}^{m,k^*} | \delta u_i | \leq 0.
\end{array}
\end{equation}

To prove that (\ref{eq:lipregionK2}) holds for all $\delta {\bf u}$, it is sufficient to only consider the direction that maximizes the value of the right-hand side or, more specifically, the direction that maximizes the summation term. In other words, we need to prove that

\vspace{-2mm}
\begin{equation}\label{eq:lipregionK4}
\begin{array}{ll}
\displaystyle \overline g_{p,j} ({\bf u}_{k^*},\tau_{k^*}) & \displaystyle + \eta_{c,j}^{e,k^*} \frac{\partial \overline g_{p,j}}{\partial \tau} \Big |_{({\bf u}_{k^*},\tau_{k^*})} ( \tau_{k+1} - \tau_{{k^*}} ) \vspace{1mm} \\
& \displaystyle  + (1-\eta_{c,j}^{e,k^*}) \overline \kappa_{p,j\tau}^{e,k^*} \left( \tau_{k+1} - \tau_{k^*} \right) \vspace{1mm} \\
&\displaystyle + \delta_e \mathop {\max} \limits_{\| \delta {\bf u} \|_2 = 1} \sum_{i=1}^{n_u} \kappa_{p,ji}^{m,k^*} | \delta u_i | \leq 0.
\end{array}
\end{equation}

To make the analysis easier, we will first show that we may restrict our consideration to $\delta {\bf u} \in \mathbb{R}^{n_u}_+$ since the point where the maximum is attained is nonunique and since there will always be one in the positive orthant. To prove this, suppose that there exists a direction $\delta \tilde {\bf u} \not \in \mathbb{R}^{n_u}_+$, $\| \delta \tilde {\bf u} \|_2 = 1$ such that

\vspace{-2mm}
\begin{equation}\label{eq:nonposdu}
\sum_{i=1}^{n_u} \kappa_{p,ji}^{m,k^*} | \delta \tilde u_i | > \mathop {\max} \limits_{{\footnotesize \begin{array}{c} \| \delta {\bf u} \|_2 = 1 \\ \delta {\bf u} \in \mathbb{R}^{n_u}_+ \end{array} }} \sum_{i=1}^{n_u} \kappa_{p,ji}^{m,k^*} | \delta u_i |.
\end{equation}

\noindent Clearly, the choice $\delta u_i := | \delta \tilde u_i |, \; i = 1,...,n_u$ shows that there exists a unit vector $\delta {\bf u} \in \mathbb{R}^{n_u}_+$ with a value equal to the left-hand side of (\ref{eq:nonposdu}), thereby proving that (\ref{eq:nonposdu}) cannot hold for any choice of $\delta \tilde {\bf u}$. As such, we may simplify (\ref{eq:lipregionK4}) further:

\vspace{-2mm}
\begin{equation}\label{eq:lipregionK5}
\begin{array}{ll}
\displaystyle \overline g_{p,j} ({\bf u}_{k^*},\tau_{k^*}) & \displaystyle + \eta_{c,j}^{e,k^*} \frac{\partial \overline g_{p,j}}{\partial \tau} \Big |_{({\bf u}_{k^*},\tau_{k^*})} ( \tau_{k+1} - \tau_{{k^*}} ) \vspace{1mm} \\
& \displaystyle  + (1-\eta_{c,j}^{e,k^*}) \overline \kappa_{p,j\tau}^{e,k^*} \left( \tau_{k+1} - \tau_{k^*} \right) \vspace{1mm} \\
&\displaystyle + \delta_e \mathop {\max} \limits_{{\footnotesize \begin{array}{c} \| \delta {\bf u} \|_2 = 1 \\ \delta {\bf u} \in \mathbb{R}^{n_u}_+ \end{array} }} \left( \kappa_{p,j}^{m,k^*} \right)^T  \delta {\bf u} \leq 0,
\end{array}
\end{equation}

\noindent where we have removed the absolute value due to the positive-orthant restriction and have written the resulting summation in vector form.

To evaluate the maximum term analytically, we may consider a simpler case by taking the following steps:

\vspace{-2mm}
\begin{equation}\label{eq:maxterm}
\begin{array}{lll}
\mathop {\max} \limits_{{\footnotesize \begin{array}{c} \| \delta {\bf u} \|_2 = 1 \\ \delta {\bf u} \in \mathbb{R}^{n_u}_+ \end{array} }} \left( \kappa_{p,j}^{m,k^*} \right)^T  \delta {\bf u} & = & \mathop {\max} \limits_{{\footnotesize \begin{array}{c} \| \delta {\bf u} \|_2^2 = 1 \\ \delta {\bf u} \in \mathbb{R}^{n_u}_+ \end{array} }} \left( \kappa_{p,j}^{m,k^*} \right)^T  \delta {\bf u} \\
& \leq & \mathop {\max} \limits_{{\footnotesize \begin{array}{c} \| \delta {\bf u} \|_2^2 \leq 1 \end{array} }} \left( \kappa_{p,j}^{m,k^*} \right)^T  \delta {\bf u},
\end{array}
\end{equation}

\noindent as squaring the unit-norm constraint does not affect the directions considered and relaxing the restrictions by relaxing the equality and removing the positivity constraint can only increase the maximum value. As the last term is equivalent to maximizing a linear function over a unit ball, writing out its stationarity conditions readily yields the following maximum point:

\vspace{-2mm}
\begin{equation}\label{eq:argmax}
\begin{array}{rl}
\displaystyle \frac{\kappa_{p,j}^{m,k^*}}{\| \kappa_{p,j}^{m,k^*} \|_2} = {\rm arg} \mathop {\rm maximize}\limits_{\delta {\bf u}} & \left( \kappa_{p,j}^{m,k^*} \right)^T  \delta {\bf u}  \\
{\rm{subject}}\;{\rm{to}} & \delta {\bf u}^T \delta {\bf u} \leq 1,
\end{array}
\end{equation}

\noindent provided that $\kappa_{p,j}^{m,k^*} \neq {\bf 0}$. This allows us to advance (\ref{eq:maxterm}) to state:

\vspace{-2mm}
\begin{equation}\label{eq:maxterm2}
\mathop {\max} \limits_{{\footnotesize \begin{array}{c} \| \delta {\bf u} \|_2 = 1 \\ \delta {\bf u} \in \mathbb{R}^{n_u}_+ \end{array} }} \left( \kappa_{p,j}^{m,k^*} \right)^T  \delta {\bf u} \leq \| \kappa_{p,j}^{m,k^*} \|_2,
\end{equation}

\noindent which we note is valid for any $\kappa_{p,j}^{m,k^*}$ (the bound holding trivially in the somewhat pathological case of $\kappa_{p,j}^{m,k^*} = {\bf 0}$).

Using this result to push (\ref{eq:lipregionK5}) further, we obtain

\vspace{-2mm}
\begin{equation}\label{eq:lipregionK6}
\begin{array}{l}
\displaystyle \overline g_{p,j} ({\bf u}_{k^*},\tau_{k^*})  + \eta_{c,j}^{e,k^*} \frac{\partial \overline g_{p,j}}{\partial \tau} \Big |_{({\bf u}_{k^*},\tau_{k^*})} ( \tau_{k+1} - \tau_{{k^*}} ) \vspace{1mm} \\
 + (1-\eta_{c,j}^{e,k^*}) \overline \kappa_{p,j\tau}^{e,k^*} \left( \tau_{k+1} - \tau_{k^*} \right) + \delta_e \| \kappa_{p,j}^{m,k^*} \|_2 \leq 0,
\end{array}
\end{equation}

\noindent which must clearly hold given the choice of back-off in (\ref{eq:suffback}) and the assumed condition $\overline g_{p,j} ({\bf u}_{k^*},\tau_{k^*}) + b_{p,j}^{k^*,k} \leq 0$. Since $(\ref{eq:lipregionK6}) \Rightarrow (\ref{eq:lipregionK5}) \Rightarrow (\ref{eq:lipregionK4}) \Rightarrow (\ref{eq:lipregionK2})$, we have proven that ${\bf u}_{\rm bd} \in \mathcal {K}_{k^*}$ for all possible ${\bf u}_{\rm bd}$ and thus that $\mathcal{B}_{e,k^*} \subseteq \mathcal{K}_{k^*}$, which, since  $g_{p,j} ({\bf u},\tau_{k+1}) \leq 0, \; \forall {\bf u} \in \mathcal{K}_{k^*} \cap \mathcal{B}_{e,k^*}$, finally proves that $g_{p,j} ({\bf u},\tau_{k+1}) \leq 0, \; \forall {\bf u} \in \mathcal{B}_{e,k^*}$. \qed

\end{proof}

We now proceed with the simpler derivations for the numerical and bound constraints.

\begin{theorem}[Sufficient back-off for a numerical constraint function]
\label{thm:backoff2}
Let the reference point ${\bf u}_{k^*}$ satisfy the backed off numerical constraint $g_{j} ({\bf u}_{k^*}) + b_{j}^{k^*} \leq 0$. Setting the back-off as

\vspace{-2mm}
\begin{equation}\label{eq:suffbacknum}
b_{j}^{k^*} = \mathop {\max} \limits_{{\bf u} \in \mathcal{B}_{e,k^*}} \; g_j ({\bf u}) - g_j ({\bf u}_{k^*})
\end{equation}

\noindent guarantees that

\vspace{-2mm}
\begin{equation}\label{eq:feasguarnum}
g_{j} ({\bf u}) \leq 0, \;\; \forall {\bf u} \in \mathcal{B}_{e,k^*}.
\end{equation}

\end{theorem}

\begin{proof}
Substituting (\ref{eq:suffbacknum}) into the assumed condition $g_{j} ({\bf u}_{k^*}) + b_{j}^{k^*} \leq 0$ yields

\vspace{-2mm}
\begin{equation}\label{eq:suffbacknum1}
\mathop {\max} \limits_{{\bf u} \in \mathcal{B}_{e,k^*}} \; g_j ({\bf u}) \leq 0,
\end{equation}

\noindent which is equivalent to (\ref{eq:feasguarnum}). \qed

\end{proof}

\begin{corollary}[Sufficient back-off for a bound constraint]
\label{cor:backoff3}
Let the reference point ${\bf u}_{k^*}$ satisfy the backed off bound constraints $u^L_i + b_{u,i} \leq u_{k^*,i} \leq u^U_i - b_{u,i}$. Setting the back-off as

\vspace{-2mm}
\begin{equation}\label{eq:suffbackbound}
b_{u,i} = \delta_e
\end{equation}

\noindent guarantees that

\vspace{-2mm}
\begin{equation}\label{eq:feasguarbound}
{\bf u}^L \preceq {\bf u} \preceq {\bf u}^U, \;\; \forall {\bf u} \in \mathcal{B}_{e,k^*}.
\end{equation}

\end{corollary}

\begin{proof}
As the bound constraints are just special cases of the general numerical constraints, we will place them into the canonical form $g_j ({\bf u}) \leq 0$:

\vspace{-2mm}
\begin{equation}\label{eq:boundcanon}
u^L_i - u_i \leq 0, \;\; u_i - u^U_i \leq 0,
\end{equation}

\noindent and then apply the result of Theorem \ref{thm:backoff2}:

\vspace{-2mm}
\begin{equation}\label{eq:boundconmax}
\begin{array}{lll}
b_{u,i} & = & \mathop {\max} \limits_{{\bf u} \in \mathcal{B}_{e,k^*}} \; \left( u^L_i - u_i \right) - \left( u^L_i - u_{k^*,i} \right) \\
& = & \mathop {\max} \limits_{{\bf u} \in \mathcal{B}_{e,k^*}} \; \left( u_i - u^U_i \right) - \left( u_{k^*,i} - u^U_i \right) = \delta_e. \;\;\;\; \qed
\end{array}
\end{equation}

\end{proof}

\subsection{Choosing the Reference Point to Satisfy the Back-offs}

The back-offs derived in (\ref{eq:suffback}), (\ref{eq:suffbacknum}), and (\ref{eq:suffbackbound}) guarantee that both the reference point ${\bf u}_{k^*}$ and all the points in the ball around it, $\mathcal{B}_{e,k^*}$, satisfy the problem constraints provided that

\vspace{-2mm}
\begin{equation}\label{eq:backedoff}
\begin{array}{ll}
\overline g_{p,j} ({\bf u}_{k^*},\tau_{k^*}) + b_{p,j}^{k^*,k} \leq 0, & \forall j = 1,...,n_{g_p} \\
g_{j} ({\bf u}_{k^*}) + b_{j}^{k^*} \leq 0, & \forall j = 1,...,n_{g} \\
{\bf u}^L + {\bf b}_u \preceq {\bf u}_{k^*} \preceq {\bf u}^U - {\bf b}_u, & 
\end{array}
\end{equation}

\noindent or, if we substitute in the results of (\ref{eq:suffback}), (\ref{eq:suffbacknum}), and (\ref{eq:suffbackbound}) directly,

\vspace{-2mm}
\begin{equation}\label{eq:backedoff2}
\begin{array}{l}
\overline g_{p,j} ({\bf u}_{k^*},\tau_{k^*}) + \displaystyle \eta_{c,j}^{e,k^*} \frac{\partial \overline g_{p,j}}{\partial \tau} \Big |_{({\bf u}_{k^*},\tau_{k^*})} ( \tau_{k+1} - \tau_{{k^*}} ) \vspace{1mm} \\
+ (1-\eta_{c,j}^{e,k^*}) \overline \kappa_{p,j\tau}^{e,k^*} \left( \tau_{k+1} - \tau_{k^*} \right) + \delta_e \| \kappa_{p,j}^{m,k^*} \|_2 \leq 0, \vspace{1mm} \\
\hspace{50mm} \forall j = 1,...,n_{g_p} \vspace{1mm} \\
\mathop {\max} \limits_{{\bf u} \in \mathcal{B}_{e,k^*}} g_j ({\bf u}) \leq 0, \hspace{25mm} \forall j = 1,...,n_{g} \vspace{1mm} \\
{\bf u}^L + \delta_e {\bf 1}  \preceq {\bf u}_{k^*} \preceq {\bf u}^U - \delta_e {\bf 1}.
\end{array}
\end{equation}

We may thus use these conditions as constraints in the search for a reference point and modify the search of (\ref{eq:kstarLUccvcostlocMNgrad}) as

\vspace{-2mm}
\begin{equation}\label{eq:kstarLUccvcostlocMNgradBO}
\begin{array}{rl}
k^* := \;\;\;\;\;\;\;\;\;\;\;\;\;\;& \vspace{1mm} \\
{\rm arg} \mathop {\rm maximize}\limits_{\bar k \in [0,k]} & \bar k \vspace{1mm} \\
{\rm{subject}}\;{\rm{to}} & \overline g_{p,j} ({\bf u}_{\bar k},\tau_{\bar k}) \vspace{1mm} \\
& \displaystyle + \displaystyle \eta_{c,j}^{e,\bar k} \frac{\partial \overline g_{p,j}}{\partial \tau} \Big |_{({\bf u}_{\bar k},\tau_{\bar k})} ( \tau_{k+1} - \tau_{{\bar k}} ) \vspace{1mm} \\
&  + (1-\eta_{c,j}^{e,\bar k}) \overline \kappa_{p,j\tau}^{e,\bar k} \left( \tau_{k+1} - \tau_{\bar k} \right) \vspace{1mm} \\
& + \delta_e \| \kappa_{p,j}^{m,\bar k} \|_2 \leq 0, \;\; \forall j = 1,...,n_{g_p} \vspace{1mm} \\
& \mathop {\max} \limits_{{\bf u} \in \mathcal{B}_{e,\bar k}} g_j ({\bf u}) \leq 0, \;\; \forall j = 1,...,n_{g} \vspace{1mm} \\
& {\bf u}^L + \delta_e {\bf 1}  \preceq {\bf u}_{\bar k} \preceq {\bf u}^U - \delta_e {\bf 1} \vspace{1mm} \\
& \displaystyle \underline \phi_p ({\bf u}_{\bar k},\tau_{\bar k}) +  \eta_{v,\phi}^{\bar k, k} \frac{\partial \underline \phi_{p}}{\partial \tau} \Big |_{({\bf u}_{\bar k},\tau_{\bar k})} ( \tau_{k} - \tau_{{\bar k}} ) \vspace{1mm} \\
&+ (1-  \eta_{v,\phi}^{\bar k, k}) \underline \kappa_{\phi,\tau}^{\bar k, k} \left( \tau_{k} - \tau_{\bar k} \right) \leq \vspace{1mm} \\
& \mathop {\min} \limits_{\tilde k \in {\bf k}_f} \left[ \begin{array}{l} \overline \phi_p ({\bf u}_{\tilde k},\tau_{\tilde k}) + \vspace{1mm} \\
\displaystyle  \eta_{c,\phi}^{\tilde k, k} \frac{\partial \overline \phi_{p}}{\partial \tau} \Big |_{({\bf u}_{\tilde k},\tau_{\tilde k})} ( \tau_{k} - \tau_{{\tilde k}} ) \vspace{1mm} \\
+ (1-  \eta_{c,\phi}^{\tilde k, k}) \overline \kappa_{\phi,\tau}^{\tilde k, k} \left( \tau_{k} - \tau_{\tilde k} \right) \end{array} \right],
\end{array}
\end{equation}

\vspace{-2mm}
\begin{equation}\label{eq:kfeas5}
{\bf k}_f = \left\{ \bar k : 
\begin{array}{l}
\overline g_{p,j} ({\bf u}_{\bar k},\tau_{\bar k}) \vspace{1mm} \\ 
+ \displaystyle \eta_{c,j}^{e,\bar k} \frac{\partial \overline g_{p,j}}{\partial \tau} \Big |_{({\bf u}_{\bar k},\tau_{\bar k})} ( \tau_{k+1} - \tau_{{\bar k}} ) \vspace{1mm} \\ 
+ (1-\eta_{c,j}^{e,\bar k}) \overline \kappa_{p,j\tau}^{e,\bar k} \left( \tau_{k+1} - \tau_{\bar k} \right) \vspace{1mm} \\ 
+ \delta_e \| \kappa_{p,j}^{m,\bar k} \|_2  \leq 0, \;\; \forall j = 1,...,n_{g_p}; \vspace{1mm} \\
 \mathop {\max} \limits_{{\bf u} \in \mathcal{B}_{e,\bar k}} g_j ({\bf u}) \leq 0, \;\; \forall j = 1,...,n_{g}; \vspace{1mm} \\
{\bf u}^L + \delta_e {\bf 1}  \preceq {\bf u}_{\bar k} \preceq {\bf u}^U - \delta_e {\bf 1}
\end{array} \right\}.
\end{equation}

It may of course occur that no experimental iterate satisfy the restrictions of (\ref{eq:kstarLUccvcostlocMNgradBO}). Naturally, this is more likely in the back-off case since the restrictions are made tighter, and choosing both overly conservative Lipschitz constants or a $\delta_e$ value that is too large can lead to infeasibility in the reference point search. If one cannot obtain a ${\bf u}_{k^*}$ by reducing the conservativism of the Lipschitz constants or choosing to work with a smaller $\delta_e$, the natural option, if one insists on continuing to optimize, would be to choose as the reference that iteration at which the maximum violation is minimized:

\vspace{-2mm}
\begin{equation}\label{eq:kstar2LUccvlocMNgradBO}
k^*  := {\rm arg} \mathop {\rm minimize}\limits_{\bar k \in [0,k]}  \mathop {\max} \left[ g_{p,m}^{\bar k}, g_{m}^{\bar k}, u_{L,m}^{\bar k}, u_{U,m}^{\bar k}  \right],
\end{equation}

\noindent where

\vspace{-2mm}
\begin{equation}\label{eq:viols}
\begin{array}{l}
g_{p,m}^{\bar k} = \mathop {\max} \limits_{j = 1,...,n_{g_p}} \left[ \begin{array}{l}
\overline g_{p,j} ({\bf u}_{\bar k},\tau_{\bar k}) \vspace{1mm} \\ 
+ \displaystyle \eta_{c,j}^{e,\bar k} \frac{\partial \overline g_{p,j}}{\partial \tau} \Big |_{({\bf u}_{\bar k},\tau_{\bar k})} ( \tau_{k+1} - \tau_{{\bar k}} ) \vspace{1mm} \\ 
+ (1-\eta_{c,j}^{e,\bar k}) \overline \kappa_{p,j\tau}^{e,\bar k} \left( \tau_{k+1} - \tau_{\bar k} \right) \vspace{1mm} \\ 
+ \delta_e \| \kappa_{p,j}^{m,\bar k} \|_2 \end{array} \right] \vspace{1mm} \\
g_{m}^{\bar k} = \mathop {\max} \limits_{j = 1,...,n_g} \mathop {\max} \limits_{{\bf u} \in \mathcal{B}_{e,\bar k}} g_j ({\bf u}) \vspace{1mm} \\
u_{L,m}^{\bar k} = \mathop {\max} \limits_{i = 1,...,n_u} \left( u^L_i + \delta_e - u_{\bar k,i} \right)  \vspace{1mm} \\
u_{U,m}^{\bar k} = \mathop {\max} \limits_{i = 1,...,n_u} \left( u_{\bar k,i} + \delta_e - u^U_i  \right).
\end{array}
\end{equation}

To close this subsection, we note that computing the maximum of $g_j$ over $\mathcal{B}_{e,\bar k}$ may not be trivial if the numerical constraint $g_j$ is not concave, and numerical global optimization methods may be required to obtain this value. Alternatively, one may also take an upper concave relaxation of the function and maximize this instead, which may solve numerical issues but may introduce more conservative back-offs.

\subsection{Adding Back-offs to the Feasibility Conditions}

In principle, the condition (\ref{eq:SCFO1idegLUccvalllocMNgrad}), together with (\ref{eq:SCFO2i}), will guarantee feasibility without any additional modifications. However, the goal of the back-offs is to leave some feasible space specifically intended for perturbations -- as already mentioned, the problem being solved once these are added is no longer (\ref{eq:mainprobdeg}) but (\ref{eq:mainprobback}). As such, we are not really interested in generating a chain of experiments that satisfy the constraints of (\ref{eq:mainprobdeg}) but violate the constraints of (\ref{eq:mainprobback}). 

In the ideal case, we would expect the future experimental iterate ${\bf u}_{k+1}$ to become the next reference ${\bf u}_{k^*}$, and so would like to be able to ensure that the appropriate back-offs are ensured for it in advance. In other words, we would like to be able to ensure, \emph{a priori}, that

\vspace{-2mm}
\begin{equation}\label{eq:aprioriback}
\begin{array}{ll}
\overline g_{p,j} ({\bf u}_{k+1},\tau_{k+1}) + b_{p,j}^{k+1,k+1} \leq 0, & \forall j = 1,...,n_{g_p} \vspace{1mm} \\
g_{j} ({\bf u}_{k+1}) + b_{j}^{k+1} \leq 0, & \forall j = 1,...,n_{g} \vspace{1mm} \\
{\bf u}^L + {\bf b}_u \preceq {\bf u}_{k+1} \preceq {\bf u}^U - {\bf b}_u, &
\end{array}
\end{equation}

\noindent where we have simply made the substitutions $k^* \rightarrow k+1$ (assuming that $k+1$ becomes the next reference) and $k \rightarrow k+1$ (indicating a shift in indices) in (\ref{eq:backedoff}). Substituting in the back-off expressions yields the one-iteration-ahead analogue of (\ref{eq:backedoff2}):

\vspace{-2mm}
\begin{equation}\label{eq:aprioriback2a}
\begin{array}{l}
\overline g_{p,j} ({\bf u}_{k+1},\tau_{k+1}) \vspace{1mm} \\
+ \displaystyle \eta_{c,j}^{e,k+1} \frac{\partial \overline g_{p,j}}{\partial \tau} \Big |_{({\bf u}_{k+1},\tau_{k+1})} ( \tau_{k+2} - \tau_{{k+1}} ) \vspace{1mm} \\
+ (1-\eta_{c,j}^{e,k+1}) \overline \kappa_{p,j\tau}^{e,k+1} \left( \tau_{k+2} - \tau_{k+1} \right) \vspace{1mm} \\
 + \delta_e \| \kappa_{p,j}^{m,k+1} \|_2 \leq 0, \;\; \forall j = 1,...,n_{g_p},
\end{array}
\end{equation}

\vspace{-2mm}
\begin{equation}\label{eq:aprioriback2b}
\mathop {\max} \limits_{{\bf u} \in \mathcal{B}_{e,k+1}} g_j ({\bf u}) \leq 0, \;\; \forall j = 1,...,n_{g},
\end{equation}

\vspace{-2mm}
\begin{equation}\label{eq:aprioriback2c}
{\bf u}^L + \delta_e {\bf 1}  \preceq {\bf u}_{k+1} \preceq {\bf u}^U - \delta_e {\bf 1}.
\end{equation}

Let us consider (\ref{eq:aprioriback2b}) and (\ref{eq:aprioriback2c}), as these are relatively simple. For (\ref{eq:aprioriback2c}), we note that one may ensure these constraints by simply modifying the constraints in the projection (shown later). For (\ref{eq:aprioriback2b}), one may modify Condition (\ref{eq:SCFO2i}) to

\vspace{-2mm}
\begin{equation}\label{eq:SCFO2iback}
\mathop {\max} \limits_{{\bf u} \in \mathcal{B}_{K}} g_{j}({\bf u}) \leq 0, \;\; \forall j = 1,...,n_g,
\end{equation}

\noindent where

\vspace{-2mm}
\begin{equation}\label{eq:ballsearch}
\mathcal{B}_{K} = \{ {\bf u} : \| {\bf u} - {\bf u}_{k^*} - K_k \left( \bar {\bf u}_{k+1}^* - {\bf u}_{k^*} \right) \|_2 \leq \delta_e \}.
\end{equation}

It is important to note that the computational burden of the line search, considered up to now to be negligible, may increase significantly depending on the computational burden of evaluating $\mathop {\max} \limits_{{\bf u} \in \mathcal{B}_{K}} g_{j}({\bf u})$, as this computation would now have to be done for every $K_k$ considered in the line search. Again, it may be of practical interest to upper bound $g_j$ by a simpler function for which the computation is cheap if this issue arises.

Condition (\ref{eq:aprioriback2a}) is a bit more difficult to grasp. First, note that while enforcing (\ref{eq:SCFO2iback}) automatically enforces (\ref{eq:SCFO2i}), enforcing (\ref{eq:aprioriback2a}) \emph{does not imply} that the standard feasibility condition $g_{p,j} ({\bf u}_{k+1},\tau_{k+1}) \leq 0$ is satisfied. In fact, all that it implies is that $g_{p,j} ({\bf u},\tau_{k+2}) \leq 0, \; \forall {\bf u} \in \mathcal{B}_K$ and, in turn, that $g_{p,j} ({\bf u}_{k+1},\tau_{k+2}) \leq 0$, which may be possible to satisfy while allowing a violation at $({\bf u}_{k+1},\tau_{k+1})$ due to degradation effects. As such, (\ref{eq:aprioriback2a}) is an additional condition that should be added to the line search, and not a replacement.

In handling (\ref{eq:aprioriback2a}), note that the bound $\overline g_{p,j} ({\bf u}_{k+1},\tau_{k+1})$ may be obtained as the left-hand side of (\ref{eq:SCFO1idegLUccvalllocMNgrad}), with 

\vspace{-2mm}
\begin{equation}\label{eq:SCFO1idegLUccvalllocMNgradgpk1}
\begin{array}{l}
\overline g_{p,j} ({\bf u}_{k+1},\tau_{k+1}) = \vspace{1mm} \\
\mathop {\min} \limits_{\bar k = 0,...,k} \left[ \hspace{-1mm} \begin{array}{l} \overline g_{p,j} ({\bf u}_{\bar k},\tau_{\bar k}) \displaystyle +\eta_{c,j}^{\bar k} \frac{\partial \overline g_{p,j}}{\partial \tau} \Big |_{({\bf u}_{\bar k},\tau_{\bar k})} ( \tau_{k+1} - \tau_{{\bar k}} ) \vspace{1mm} \\
 \displaystyle + (1-\eta_{c,j}^{\bar k}) \overline \kappa_{p,j\tau}^{\bar k} \left( \tau_{k+1} - \tau_{\bar k} \right) \vspace{1mm}\\
\displaystyle   + \sum_{i \in I_{c,j}^{\bar k}} \mathop {\max} \left[ \begin{array}{l} \displaystyle \frac{\partial \underline g_{p,j}}{\partial u_i} \Big |_{({\bf u}_{\bar k},\tau_{\bar k})} ( u_{k^*,i} + \\ \hspace{2mm} K_k (\bar u_{k+1,i}^* - u_{k^*,i} ) - u_{\bar k,i} ), \vspace{1mm}\\ \displaystyle \frac{\partial \overline g_{p,j}}{\partial u_i} \Big |_{({\bf u}_{\bar k},\tau_{\bar k})} ( u_{k^*,i} + \\ \hspace{2mm} K_k (\bar u_{k+1,i}^* - u_{k^*,i} ) - u_{\bar k,i} ) \end{array} \right] \vspace{1mm} \\
 + \displaystyle \sum_{i \not \in I_{c,j}^{\bar k}} \mathop {\max} \left[ \begin{array}{l} \underline \kappa_{p,ji}^{\bar k} ( u_{k^*,i} + \\ \hspace{2mm} K_k (\bar u_{k+1,i}^* - u_{k^*,i} ) - u_{\bar k,i} ), \vspace{1mm}\\ \overline \kappa_{p,ji}^{\bar k} ( u_{k^*,i} + \\ \hspace{2mm} K_k (\bar u_{k+1,i}^* - u_{k^*,i} ) - u_{\bar k,i} ) \end{array} \right] \end{array} \hspace{-1mm} \right].
\end{array}
\end{equation}

\noindent The portion of (\ref{eq:aprioriback2a}) that corresponds to the back-off is more troublesome, however, as $\eta_{c,j}^{e,k+1}$, $\frac{\partial \overline g_{p,j}}{\partial \tau} \Big |_{({\bf u}_{k+1},\tau_{k+1})}$, $\overline \kappa_{p,j\tau}^{e,k+1}$, and $\kappa_{p,j}^{m,k+1}$ are all functions of $K_k$ and must be evaluated at each point in the line search. While the difficulty is not conceptual -- in theory, one could obtain all of these values given a certain $K_k$ -- it is relevant in practice, since certain components require calling a gradient estimation algorithm and thus could lead to large computational burdens. Removing the concavity and locality relaxations would solve this issue, however, as the need to estimate gradients would be removed and the above-stated components would simply be constants.

The addition to (\ref{eq:SCFO1idegLUccvalllocMNgrad}) is then given as

\vspace{-2mm}
\begin{equation}\label{eq:SCFO1idegLUccvalllocMNgradback}
\begin{array}{l}
\mathop {\min} \limits_{\bar k = 0,...,k} \left[ \hspace{-1mm} \begin{array}{l} \overline g_{p,j} ({\bf u}_{\bar k},\tau_{\bar k}) \displaystyle +\eta_{c,j}^{\bar k} \frac{\partial \overline g_{p,j}}{\partial \tau} \Big |_{({\bf u}_{\bar k},\tau_{\bar k})} ( \tau_{k+1} - \tau_{{\bar k}} ) \vspace{1mm} \\
 \displaystyle + (1-\eta_{c,j}^{\bar k}) \overline \kappa_{p,j\tau}^{\bar k} \left( \tau_{k+1} - \tau_{\bar k} \right) \vspace{1mm}\\
\displaystyle   + \sum_{i \in I_{c,j}^{\bar k}} \mathop {\max} \left[ \begin{array}{l} \displaystyle \frac{\partial \underline g_{p,j}}{\partial u_i} \Big |_{({\bf u}_{\bar k},\tau_{\bar k})} ( u_{k^*,i} + \\ \hspace{2mm} K_k (\bar u_{k+1,i}^* - u_{k^*,i} ) - u_{\bar k,i} ), \vspace{1mm}\\ \displaystyle \frac{\partial \overline g_{p,j}}{\partial u_i} \Big |_{({\bf u}_{\bar k},\tau_{\bar k})} ( u_{k^*,i} + \\ \hspace{2mm} K_k (\bar u_{k+1,i}^* - u_{k^*,i} ) - u_{\bar k,i} ) \end{array} \right] \vspace{1mm} \\
 + \displaystyle \sum_{i \not \in I_{c,j}^{\bar k}} \mathop {\max} \left[ \begin{array}{l} \underline \kappa_{p,ji}^{\bar k} ( u_{k^*,i} + \\ \hspace{2mm} K_k (\bar u_{k+1,i}^* - u_{k^*,i} ) - u_{\bar k,i} ), \vspace{1mm}\\ \overline \kappa_{p,ji}^{\bar k} ( u_{k^*,i} + \\ \hspace{2mm} K_k (\bar u_{k+1,i}^* - u_{k^*,i} ) - u_{\bar k,i} ) \end{array} \right] \end{array} \hspace{-1mm} \right] \vspace{1mm} \\
+ \displaystyle \eta_{c,j}^{K} \frac{\partial \overline g_{p,j}}{\partial \tau} \Big |_{\left( {\bf u}_{k^*} + K_k ( \bar {\bf u}_{k+1}^* - {\bf u}_{k^*} ),\tau_{k+1}\right)} ( \tau_{k+2} - \tau_{{k+1}} ) \vspace{1mm} \\
+ (1-\eta_{c,j}^{K}) \overline \kappa_{p,j\tau}^{K} \left( \tau_{k+2} - \tau_{k+1} \right) + \delta_e \| \kappa_{p,j}^{m,K} \|_2  \leq 0, \vspace{1mm} \\
\hspace{50mm} \forall j = 1,...,n_{g_p},
\end{array}
\end{equation}

\noindent where, defining

\vspace{-2mm}
\begin{equation}\label{eq:locspace4}
\mathcal{I}_{\tau}^{K} = \mathcal{B}_{K} \times \left\{ \tau : \tau_{k+1} \leq \tau \leq \tau_{k+2} \right\},
\end{equation}

\noindent we define $\eta_{c,j}^K$ and $\overline \kappa_{p,j\tau}^{K}$ as the analogues of $\eta_{c,j}^{e,k^*}$ and $\overline \kappa_{p,j\tau}^{e,k^*}$ over $\mathcal{I}_{\tau}^K$ (as opposed to $\mathcal{I}_{\tau}^{e,k^*}$). Likewise, $\kappa_{p,j}^{m,K}$ is defined as

\vspace{-2mm}
\begin{equation}\label{eq:kappastarM}
\begin{array}{l}
\kappa_{p,ji}^{m,K} = \left\{ \begin{array}{l}  \mathop {\max} \left[ \begin{array}{l} \displaystyle \Bigg |  \frac{\partial \underline g_{p,j}}{\partial u_i} \Big |_{\left( {\bf u}_{k^*} + K_k ( \bar {\bf u}_{k+1}^* - {\bf u}_{k^*} ),\tau_{k+1}\right)} \Bigg |,\vspace{1mm} \\ \displaystyle \Bigg | \frac{\partial \overline g_{p,j}}{\partial u_i} \Big |_{\left({\bf u}_{k^*} + K_k ( \bar {\bf u}_{k+1}^* - {\bf u}_{k^*} ),\tau_{k+1}\right)} \Bigg | \end{array} \right], \vspace{1mm} \\
\hspace{47mm} i \in I_{c,j}^{K} \vspace{1mm} \\  \mathop {\max} \left[  | \underline \kappa_{p,ji}^{K} |, | \overline \kappa_{p,ji}^{K} | \right], \hspace{17mm} i \not \in I_{c,j}^{K}  \end{array} \right . \vspace{1mm} \\
\kappa_{p,j}^{m,K} = \left[ \kappa_{p,j1}^{m,K}\; \hdots \; \kappa_{p,jn_u}^{m,K}  \right]^T,
\end{array}
\end{equation}

\noindent with $I_{c,j}^K$ denoting the concavity indices over $\mathcal{I}_\tau^{K}$ and $\underline \kappa_{p,ji}^{K}, \overline \kappa_{p,ji}^{K}$ defined as

\vspace{-2mm}
\begin{equation}\label{eq:lipcondegLUlocK}
\begin{array}{l}
\displaystyle \underline \kappa_{p,ji} \leq \underline \kappa_{p,ji}^{K} < \frac{\partial g_{p,j}}{\partial u_i} \Big |_{({\bf u},\tau)} < \overline \kappa_{p,ji}^{K} \leq \overline \kappa_{p,ji}, \vspace{1mm} \\
\hspace{45mm} \forall ({\bf u},\tau) \in \mathcal{I}_\tau^{K}.
\end{array}
\end{equation}

Note that an additional bit of information, in the form of $\tau_{k+2}$, is required in this implementation, as the user should specify not only the time of the upcoming experiment at $\tau_{k+1}$ but also the time of the one after.

Practically, it may occur that the degradation term

\vspace{-2mm}
$$
\begin{array}{l}
\displaystyle \eta_{c,j}^{K} \frac{\partial \overline g_{p,j}}{\partial \tau} \Big |_{\left({\bf u}_{k^*} + K_k ( \bar {\bf u}_{k+1}^* - {\bf u}_{k^*} ),\tau_{k+1} \right)} ( \tau_{k+2} - \tau_{{k+1}} ) \vspace{1mm} \\
+ (1-\eta_{c,j}^{K}) \overline \kappa_{p,j\tau}^{K} \left( \tau_{k+2} - \tau_{k+1} \right)
\end{array}
$$

\noindent render the condition (\ref{eq:SCFO1idegLUccvalllocMNgradback}) infeasible. In this case, we suggest using an alternative to  (\ref{eq:SCFO1idegLUccvalllocMNgradback}) that ignores this term:

\vspace{-2mm}
\begin{equation}\label{eq:SCFO1idegLUccvalllocMNgradback2}
\begin{array}{l}
\mathop {\min} \limits_{\bar k = 0,...,k} \left[ \hspace{-1mm} \begin{array}{l} \overline g_{p,j} ({\bf u}_{\bar k},\tau_{\bar k}) \displaystyle +\eta_{c,j}^{\bar k} \frac{\partial \overline g_{p,j}}{\partial \tau} \Big |_{({\bf u}_{\bar k},\tau_{\bar k})} ( \tau_{k+1} - \tau_{{\bar k}} ) \vspace{1mm} \\
 \displaystyle + (1-\eta_{c,j}^{\bar k}) \overline \kappa_{p,j\tau}^{\bar k} \left( \tau_{k+1} - \tau_{\bar k} \right) \vspace{1mm}\\
\displaystyle   + \sum_{i \in I_{c,j}^{\bar k}} \mathop {\max} \left[ \begin{array}{l} \displaystyle \frac{\partial \underline g_{p,j}}{\partial u_i} \Big |_{({\bf u}_{\bar k},\tau_{\bar k})} ( u_{k^*,i} + \\ \hspace{2mm} K_k (\bar u_{k+1,i}^* - u_{k^*,i} ) - u_{\bar k,i} ), \vspace{1mm}\\ \displaystyle \frac{\partial \overline g_{p,j}}{\partial u_i} \Big |_{({\bf u}_{\bar k},\tau_{\bar k})} ( u_{k^*,i} + \\ \hspace{2mm} K_k (\bar u_{k+1,i}^* - u_{k^*,i} ) - u_{\bar k,i} ) \end{array} \right] \vspace{1mm} \\
 + \displaystyle \sum_{i \not \in I_{c,j}^{\bar k}} \mathop {\max} \left[ \begin{array}{l} \underline \kappa_{p,ji}^{\bar k} ( u_{k^*,i} + \\ \hspace{2mm} K_k (\bar u_{k+1,i}^* - u_{k^*,i} ) - u_{\bar k,i} ), \vspace{1mm}\\ \overline \kappa_{p,ji}^{\bar k} ( u_{k^*,i} + \\ \hspace{2mm} K_k (\bar u_{k+1,i}^* - u_{k^*,i} ) - u_{\bar k,i} ) \end{array} \right] \end{array} \hspace{-1mm} \right] \vspace{1mm} \\
\hspace{25mm} + \delta_e \| \kappa_{p,j}^{m,K} \|_2  \leq 0, \;\; \forall j = 1,...,n_{g_p}.
\end{array}
\end{equation}

\begin{figure*}
\begin{center}
\includegraphics[width=16cm]{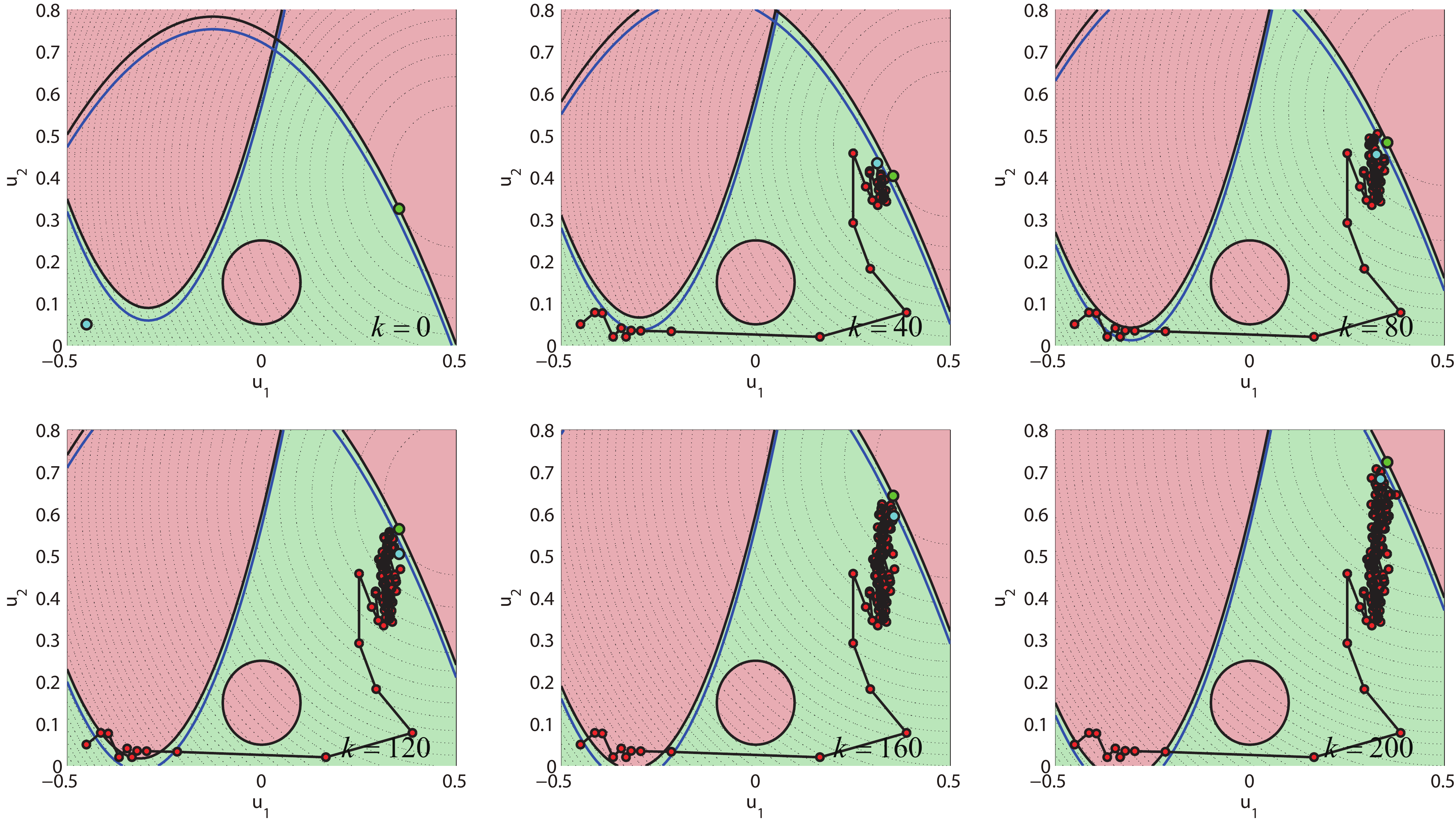}
\caption{Chain of experiments generated by applying the modified SCFO methodology to Problem (\ref{eq:exdeg}) for the ($-$) scenario with a sufficient excitation requirement of $\delta_e = 0.02$ added.}
\label{fig:degIa}
\end{center}
\end{figure*}

\subsection{Adding Back-offs in the Projection}

Since the feasibility conditions have essentially been offset by the back-offs, it makes sense that the projection conditions of the SCFO are offset as well to avoid premature convergence to a backed off constraint. As already mentioned, the bound constraints are also offset so as to ensure satisfaction of (\ref{eq:aprioriback2c}):

\vspace{-2mm}
\begin{equation}\label{eq:projdeg2robslackBO}
\begin{array}{rl}
\bar {\bf u}_{k+1}^* := \;\;\;\;\;\;\;\;\; & \vspace{1mm}\\
 {\rm arg} \mathop {\rm minimize}\limits_{{\bf u},{\bf s}_\phi, {\bf S}} & \| {\bf u} - {\bf u}_{k+1}^* \|_2^2 \vspace{1mm}  \\
 {\rm{subject}}\;{\rm{to}} & \displaystyle \sum_{i=1}^{n_u} s_{ji}  \leq -\delta_{g_p,j} \vspace{1mm} \\
& \displaystyle \frac{\partial \underline g_{p,j}}{\partial u_i} \Big |_{({\bf u}_{k^*}, \tau_{k+1})} (u_i - u_{k^*,i}) \leq s_{ji} \vspace{1mm} \\
& \displaystyle \frac{\partial \overline g_{p,j}}{\partial u_i} \Big |_{({\bf u}_{k^*}, \tau_{k+1})} (u_i - u_{k^*,i}) \leq s_{ji}, \vspace{1mm} \\
& \forall i = 1,...,n_u, \vspace{1mm} \\
&  \forall j: \overline g_{p,j} ({\bf u}_{k^*},\tau_{k^*}) + b_{p,j}^{k^*,k} \geq -\epsilon_{p,j} \vspace{1mm} \\
 & \nabla g_{j} ({\bf u}_{k^*})^T ({\bf u} - {\bf u}_{k^*}) \leq -\delta_{g,j}, \vspace{1mm} \\
& \forall j : g_{j}({\bf u}_{k^*}) + b_j^{k^*,k} \geq -\epsilon_{j} \vspace{1mm} \\
 & \displaystyle \sum_{i=1}^{n_u} s_{\phi,i}  \leq -\delta_{\phi} \vspace{1mm} \\
& \displaystyle \frac{\partial \underline \phi_{p}}{\partial u_i} \Big |_{({\bf u}_{k^*}, \tau_{k+1})} (u_i - u_{k^*,i}) \leq s_{\phi,i} \vspace{1mm} \\
& \displaystyle \frac{\partial \overline \phi_{p}}{\partial u_i} \Big |_{({\bf u}_{k^*}, \tau_{k+1})} (u_i - u_{k^*,i}) \leq s_{\phi,i} \vspace{1mm} \\
 & {\bf u}^L + \delta_e {\bf 1} \preceq {\bf u} \preceq {\bf u}^U - \delta_e {\bf 1}, 
\end{array}
\end{equation}

\noindent or, with the expressions (\ref{eq:suffback}) and (\ref{eq:suffbacknum}) substituted in for the back-offs:

\vspace{-2mm}
\begin{equation}\label{eq:projdeg2robslackBO2}
\begin{array}{rl}
\bar {\bf u}_{k+1}^* := \;\;\;\;\;\;\;\;\; & \vspace{1mm}\\
 {\rm arg} \mathop {\rm minimize}\limits_{{\bf u},{\bf s}_\phi, {\bf S}} & \| {\bf u} - {\bf u}_{k+1}^* \|_2^2 \vspace{1mm}  \\
 {\rm{subject}}\;{\rm{to}} & \displaystyle \sum_{i=1}^{n_u} s_{ji}  \leq -\delta_{g_p,j} \vspace{1mm} \\
& \displaystyle \frac{\partial \underline g_{p,j}}{\partial u_i} \Big |_{({\bf u}_{k^*}, \tau_{k+1})} (u_i - u_{k^*,i}) \leq s_{ji} \vspace{1mm} \\
& \displaystyle \frac{\partial \overline g_{p,j}}{\partial u_i} \Big |_{({\bf u}_{k^*}, \tau_{k+1})} (u_i - u_{k^*,i}) \leq s_{ji}, \vspace{1mm} \\
& \forall i = 1,...,n_u, \vspace{1mm} \\
&  \forall j: \begin{array}{l} \overline g_{p,j} ({\bf u}_{k^*},\tau_{k^*}) \vspace{1mm} \\
 +  \displaystyle \eta_{c,j}^{e,k^*} \frac{\partial \overline g_{p,j}}{\partial \tau} \Big |_{({\bf u}_{k^*},\tau_{k^*})} ( \tau_{k+1} - \tau_{{ k^*}} ) \vspace{1mm} \\ 
 + (1-\eta_{c,j}^{e,k^*}) \overline \kappa_{p,j\tau}^{e,k^*} \left( \tau_{k+1} - \tau_{k^*} \right) \vspace{1mm} \\ 
 + \delta_e \| \kappa_{p,j}^{m,k^*} \|_2 \geq -\epsilon_{p,j} \end{array}
\end{array}
\end{equation}

$$
\begin{array}{rl}
 & \nabla g_{j} ({\bf u}_{k^*})^T ({\bf u} - {\bf u}_{k^*}) \leq -\delta_{g,j}, \vspace{1mm} \\
& \forall j : \mathop {\max} \limits_{{\bf u} \in \mathcal{B}_{e,k^*}} g_{j}({\bf u}) \geq -\epsilon_{j} \vspace{1mm} \\
 & \displaystyle \sum_{i=1}^{n_u} s_{\phi,i}  \leq -\delta_{\phi} \vspace{1mm} \\
& \displaystyle \frac{\partial \underline \phi_{p}}{\partial u_i} \Big |_{({\bf u}_{k^*}, \tau_{k+1})} (u_i - u_{k^*,i}) \leq s_{\phi,i} \vspace{1mm} \\
& \displaystyle \frac{\partial \overline \phi_{p}}{\partial u_i} \Big |_{({\bf u}_{k^*}, \tau_{k+1})} (u_i - u_{k^*,i}) \leq s_{\phi,i} \vspace{1mm} \\
 & {\bf u}^L + \delta_e {\bf 1} \preceq {\bf u} \preceq {\bf u}^U - \delta_e {\bf 1}. 
\end{array}
$$

\subsection{Example}
\label{sec:BOex}

We build on the $\alpha_\sigma = 0.05$ example of Section \ref{sec:gradestex} by adding the requirement that feasible perturbations of norm $\delta_e$ always be possible, and incorporate the theory discussed in this section into the implementation. 

Due to the concavity (or absence thereof) being global in the constraint functions, we simply set $I_{c,j}^{e,k^*} := I_{c,j}$, but set $I_{c,j}^K := \varnothing$ to simplify the line search and avoid gradient bound computations during the search. The local Lipschitz constants are computed in a manner analogous to (\ref{eq:liploc1})-(\ref{eq:liploc5}). The maximum of $g_1$ over a ball may be upper bounded by its maximum over an inscribing box, for which the maximum may be easily evaluated due to the separability of $g_1$:

\vspace{-2mm}
\begin{equation}\label{eq:maxterm10}
\begin{array}{l}
\mathop {\max} \limits_{{\bf u} \in \mathcal{B}_{e,k^*}} \left( -u_1^2 - (u_2-0.15)^2 + 0.01 \right) \\
\leq \mathop {\max} \limits_{{\bf u} \in {\rm bbox} ({\bf u}_{k^*} \pm \delta_e {\bf 1})} \left( -u_1^2 - (u_2-0.15)^2 + 0.01 \right) \\
= - \mathop {\min} \limits_{u_1 \in [u_1-\delta_e,u_1+\delta_e]} u_1^2 \\
\hspace{3.4mm}- \mathop {\min} \limits_{u_2 \in [u_2-\delta_e,u_2+\delta_e]} (u_2-0.15)^2 + 0.01,
\end{array}
\end{equation}

\noindent where the separate minimum terms may be easily solved analytically, the minimum occurring at either $u_i-\delta_e$ or $u_i+\delta_e$ or at the value that brings the term to 0, if this value belongs to ${\rm bbox} ({\bf u}_{k^*} \pm \delta_e {\bf 1})$. The same is done to easily compute the maximum over $\mathcal{B}_{K}$ in the line search.

Sufficient excitation is enforced whenever the $\| {\bf u}_{k+1} - {\bf u}_{k^*} \|_2$ value for the ${\bf u}_{k+1}$ found by the standard implementation is inferior to $\delta_e$. In this example, we do this in a somewhat brute manner by simply picking a random unit vector $\delta {\bf u}_r$ and redefining ${\bf u}_{k+1}$ as

\vspace{-2mm}
\begin{equation}\label{eq:uredef}
{\bf u}_{k+1} := {\bf u}_{k^*} + \delta_e \delta {\bf u}_r,
\end{equation}

\noindent more intelligent perturbation schemes of course being possible.

Results for $\delta_e = 0.02$ are provided in Fig. \ref{fig:degIa}, which show that feasibility is maintained throughout despite the sufficient excitation requirement, which is shown to be enforced in Fig. \ref{fig:degIperta}. The cost function values are given in Fig. \ref{fig:degIcosta}.

\begin{figure}
\begin{center}
\includegraphics[width=8cm]{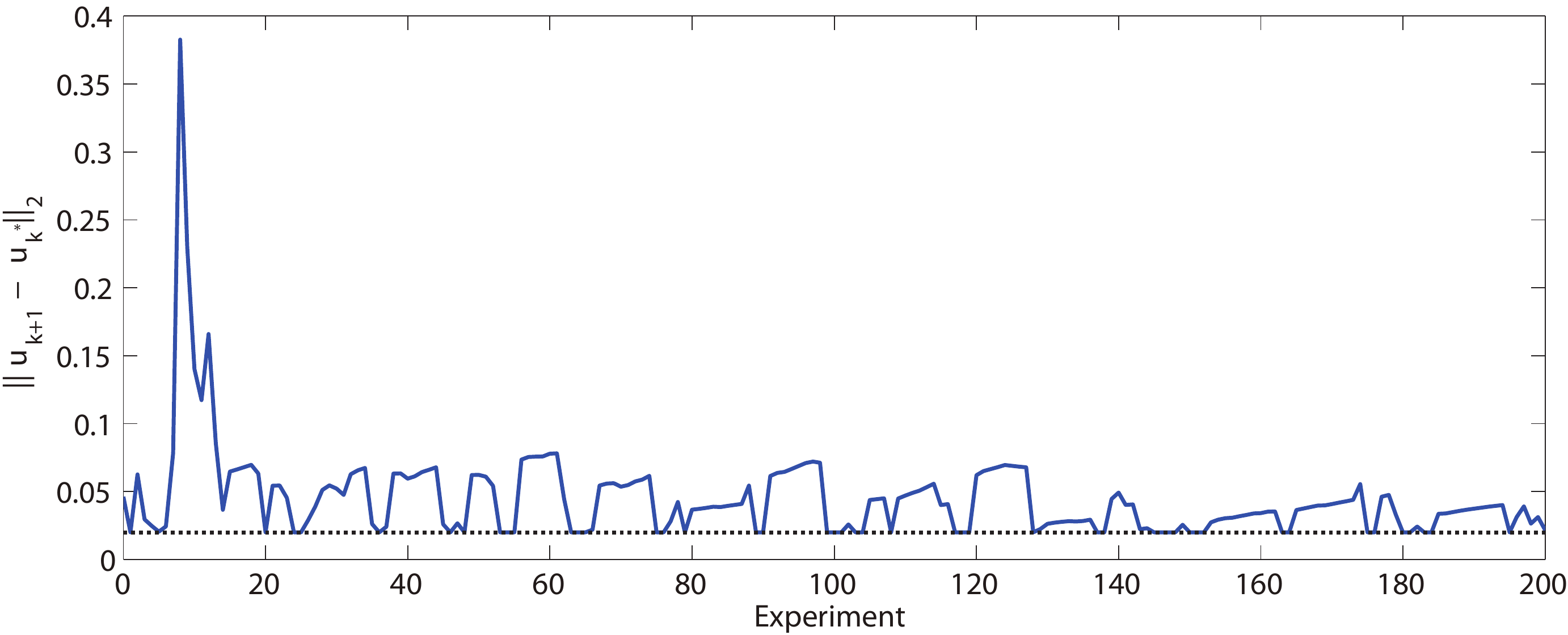}
\caption{Size of the steps taken between experimental iterations. The dashed black line gives the value of $\delta_e$.}
\label{fig:degIperta}
\end{center}
\end{figure}

\begin{figure}
\begin{center}
\includegraphics[width=8cm]{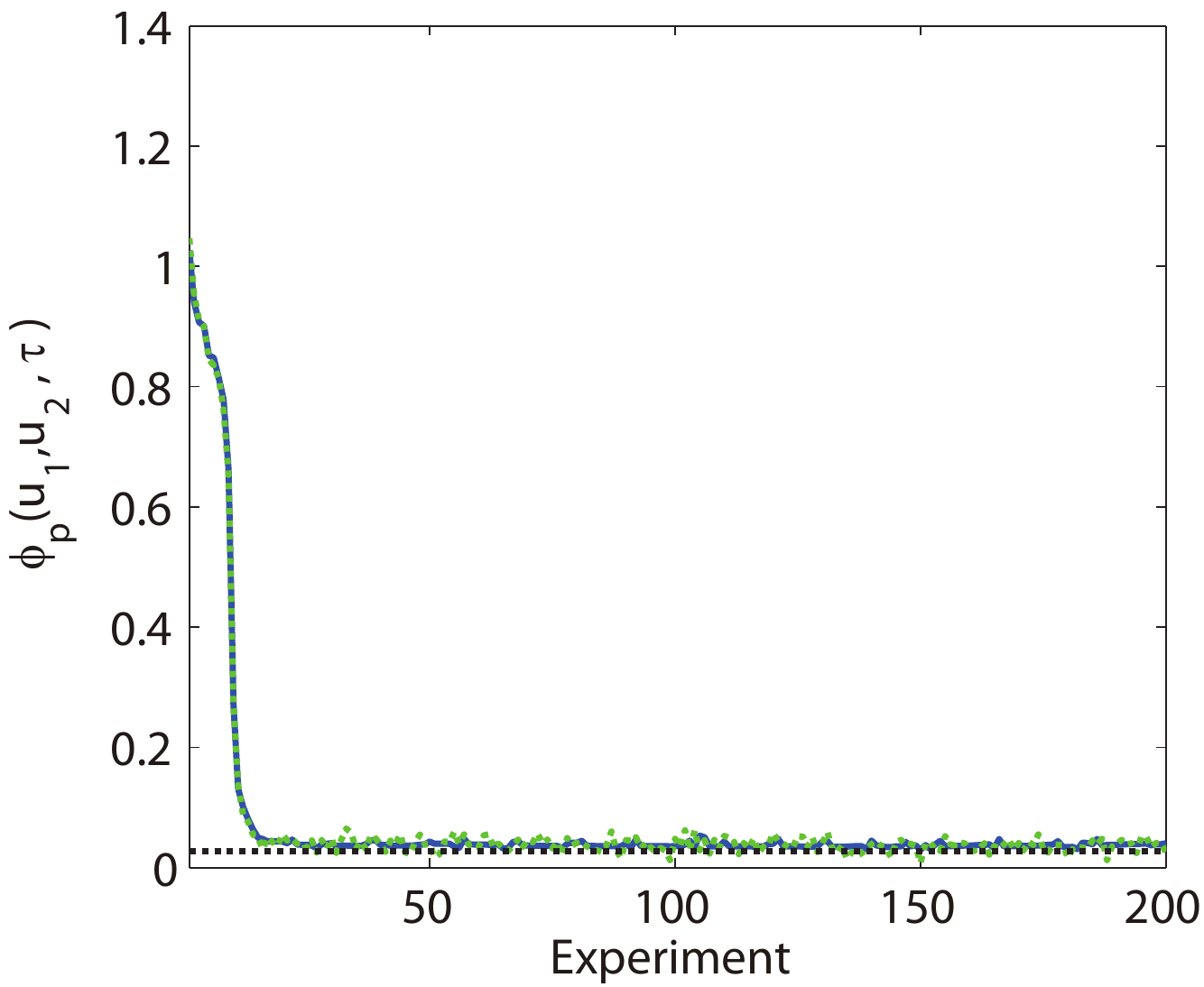}
\caption{Cost function values obtained by the modified SCFO methodology for Problem (\ref{eq:exdeg}) for the ($-$) scenario with a sufficient excitation requirement of $\delta_e = 0.02$ enforced.}
\label{fig:degIcosta}
\end{center}
\end{figure}

\section{Accomodating Soft Constraints}
\label{sec:softcon}

Up to now, we have developed all of our theory for the case where the experimental and numerical constraints were considered inviolable, in that we had the requirement

\vspace{-2mm}
\begin{equation}\label{eq:inviol}
\begin{array}{ll}
g_{p,j} ({\bf u}_k,\tau_k) \leq 0, \; & \forall j = 1,...,n_{g_p} \\
g_{j} ({\bf u}_k) \leq 0, \; & \forall j = 1,...,n_{g}
\end{array}
\end{equation}

\noindent for all experimental iterations $k$. However, it may often occur that one can violate certain constraints temporarily during the convergence process. Supposing an allowable violation for each constraint and denoting this by $d \geq 0$, this is tantamount to saying that one may allow

\vspace{-2mm}
\begin{equation}\label{eq:inviolslack}
\begin{array}{ll}
g_{p,j} ({\bf u}_k,\tau_k) \leq d_{p,j},\; & \forall j = 1,...,n_{g_p} \\
g_{j} ({\bf u}_k) \leq d_j,\; & \forall j = 1,...,n_{g}
\end{array}
\end{equation}

\noindent for some iterations.

Such a relaxation is of practical interest as it generally leads to faster convergence speed, the SCFO being allowed to take greater steps due to less stringent feasibility conditions. However, some care should be taken to manage the slacks $d$, as one needs to ensure that the violations are, in some sense, bounded and do not last for an arbitrarily large number of experiments since they are constrained to be temporary by definition.

Here, we propose to ensure these properties by making the slacks iteration-dependent

\vspace{-2mm}
\begin{equation}\label{eq:inviolslackiter}
\begin{array}{ll}
g_{p,j} ({\bf u}_k,\tau_k) \leq d_{p,j}^{k},\; & \forall j = 1,...,n_{g_p} \\
g_{j} ({\bf u}_k) \leq d_j^k, \; & \forall j = 1,...,n_{g}
\end{array}
\end{equation}

\noindent and managing them as follows:

\vspace{-2mm}
\begin{equation}\label{eq:slackmanage}
\begin{array}{l}
d_{p,j}^0 := \overline d_{p,j}, \;d_j^0 := \overline d_j \vspace{1mm} \\
d_{p,j}^{k+1} := \left\{ \begin{array}{ll} \beta_{p,j} d_{p,j}^k, & \;\;\; \overline g_{p,j} ({\bf u}_k, \tau_k) > 0 \\ d_{p,j}^k, & \;\;\; \overline g_{p,j} ({\bf u}_k, \tau_k) \leq 0 \end{array} \right . \vspace{1mm} \\
d_{j}^{k+1} := \left\{ \begin{array}{ll} \beta_j d_{j}^k, & \;\;\; g_{j} ({\bf u}_k) > 0 \\ d_{j}^k, & \;\;\; g_{j} ({\bf u}_k) \leq 0, \end{array} \right . 
\end{array}
\end{equation}

\noindent where $\overline d_{p,j}$ and $\overline d_{j}$ are the maximum allowable violations while $\beta_{p,j}, \beta_j \in [0,1)$ are slack reduction constants. The temporary nature of the violations is then quantified by

\vspace{-3mm}
\begin{equation}\label{eq:violsum}
\begin{array}{l}
\displaystyle \sum_{\bar k =0}^{\infty} \mathop {\max} \left[ 0, g_{p,j} ({\bf u}_{\bar k}, \tau_{\bar k}) \right]   \leq d_{p,j}^{S} \vspace{1mm} \\
\displaystyle \sum_{\bar k =0}^{\infty} \mathop {\max} \left[ 0, g_{j} ({\bf u}_{\bar k}) \right]   \leq d_{j}^{S},
\end{array}
\end{equation}

\noindent where $d_{p,j}^{S}, d_{j}^{S} \geq 0$ are some user-specified constants and may be thought of as upper bounds on the ``discrete violation integrals''.

We now derive sufficiently low values of $\beta_{p,j}, \beta_j$ to guarantee that the bounds (\ref{eq:violsum}) are satisfied.

\begin{theorem}[Slack reduction constant for a given experimental constraint]
\label{thm:beta}
Let $I_{\beta}$ denote the set of experimental iteration indices, ordered from smallest to largest, where the constraint cannot be proven to be satisfied:

\vspace{-2mm}
\begin{equation}\label{eq:Ibeta}
I_\beta = \left\{  \bar k : \overline g_{p,j} ({\bf u}_{\bar k}, \tau_{\bar k}) > 0  \right\}.
\end{equation}

\noindent If the update strategy (\ref{eq:slackmanage}) is applied with

\vspace{-2mm}
\begin{equation}\label{eq:betamax}
\beta_{p,j} \leq \frac{d_{p,j}^S - \overline d_{p,j}}{d_{p,j}^S},
\end{equation}

\noindent then the bound for $g_{p,j}$ in (\ref{eq:violsum}) is satisfied.

\end{theorem}
\begin{proof}
Let us rewrite the violation integral as

\vspace{-2mm}
\begin{equation}\label{eq:violsum2}
\begin{array}{l}
\displaystyle \sum_{\bar k =0}^{\infty} \mathop {\max} \left[ 0, g_{p,j} ({\bf u}_{\bar k}, \tau_{\bar k}) \right] \\
\hspace{10mm} \displaystyle = \sum_{i = 0}^{\# I_\beta - 1} \mathop {\max} [ 0, g_{p,j} ({\bf u}_{I_\beta (i)}, \tau_{I_\beta (i)}) ],
\end{array}
\end{equation}

\noindent with $\# I_\beta$ denoting the cardinality of $I_\beta$ and $I_\beta (i-1)$ denoting its $i$-th element. The validity of (\ref{eq:violsum2}) follows from the fact that $g_{p,j} ({\bf u}_{\bar k}, \tau_{\bar k}) \leq \overline g_{p,j} ({\bf u}_{\bar k}, \tau_{\bar k})$, which implies that any non-zero term in the left-hand-side summation must also appear on the right-hand side as it will be indexed by $I_\beta$.

From the requirement (\ref{eq:inviolslackiter}) and the update law (\ref{eq:slackmanage}), we have that

\vspace{-2mm}
\begin{equation}\label{eq:update}
\begin{array}{l}
 g_{p,j} ({\bf u}_{I_\beta (0)}, \tau_{I_\beta (0)}) \leq \overline d_{p,j} \\
 g_{p,j} ({\bf u}_{I_\beta (1)}, \tau_{I_\beta (1)}) \leq \beta_{p,j} \overline d_{p,j} \\
\hspace{10mm} \vdots \\
 g_{p,j} ({\bf u}_{I_\beta (i)}, \tau_{I_\beta (i)}) \leq \beta_{p,j}^i \overline d_{p,j},
\end{array}
\end{equation}

\noindent which then allows us to upper bound the summation further as

\vspace{-2mm}
\begin{equation}\label{eq:violsum3}
\begin{array}{lll}
\displaystyle \sum_{\bar k =0}^{\infty} \mathop {\max} \left[ 0, g_{p,j} ({\bf u}_{\bar k}, \tau_{\bar k}) \right] & \leq &  \displaystyle \sum_{i = 0}^{\# I_\beta  - 1} \mathop {\max} [0, \beta_{p,j}^i \overline d_{p,j}  ] \vspace{1mm} \\
& = & \displaystyle \overline d_{p,j}  \sum_{i = 0}^{\# I_\beta  - 1} \beta_{p,j}^i \vspace{1mm}  \\
& \leq & \overline d_{p,j} \displaystyle \sum_{i = 0}^{\infty} \beta_{p,j}^i \vspace{1mm} \\
& = & \displaystyle \frac{\overline d_{p,j}}{1-\beta_{p,j}},
\end{array}
\end{equation}

\noindent which takes its largest value when $\beta_{p,j}$ is maximized. Substituting in the maximum value given by the imposed bound in (\ref{eq:betamax}), we see that

\vspace{-2mm}
\begin{equation}\label{eq:violsum4}
\displaystyle \sum_{\bar k =0}^{\infty} \mathop {\max} \left[ 0, g_{p,j} ({\bf u}_{\bar k}, \tau_{\bar k}) \right] \leq \frac{\overline d_{p,j}}{1-\frac{d_{p,j}^S - \overline d_{p,j}}{d_{p,j}^S}} = d_{p,j}^S. \;\;\;\;\; \qed
\end{equation}

\end{proof}

The appropriate $\beta_j$ bound for a given numerical constraint may be derived in identical fashion, with

\vspace{-2mm}
\begin{equation}\label{eq:betamaxnum}
\beta_{j} \leq \frac{d_{j}^S - \overline d_{j}}{d_{j}^S}
\end{equation}

\noindent guaranteeing the second inequality in (\ref{eq:violsum}).

\subsection{Accounting for Slack in the Feasibility Conditions}

The modifications to the feasibility conditions are straightforward and simply involve replacing the zeros on the right-hand sides with the slacks at the next iteration:

\vspace{-2mm}
\begin{equation}\label{eq:SCFO1idegLUccvalllocMNgradslack}
\begin{array}{l}
\mathop {\min} \limits_{\bar k = 0,...,k} \left[ \hspace{-1mm} \begin{array}{l} \overline g_{p,j} ({\bf u}_{\bar k},\tau_{\bar k}) \displaystyle +\eta_{c,j}^{\bar k} \frac{\partial \overline g_{p,j}}{\partial \tau} \Big |_{({\bf u}_{\bar k},\tau_{\bar k})} ( \tau_{k+1} - \tau_{{\bar k}} ) \vspace{1mm} \\
 \displaystyle + (1-\eta_{c,j}^{\bar k}) \overline \kappa_{p,j\tau}^{\bar k} \left( \tau_{k+1} - \tau_{\bar k} \right) \vspace{1mm}\\
\displaystyle   + \sum_{i \in I_{c,j}^{\bar k}} \mathop {\max} \left[ \begin{array}{l} \displaystyle \frac{\partial \underline g_{p,j}}{\partial u_i} \Big |_{({\bf u}_{\bar k},\tau_{\bar k})} ( u_{k^*,i} + \\ \hspace{2mm} K_k (\bar u_{k+1,i}^* - u_{k^*,i} ) - u_{\bar k,i} ), \vspace{1mm}\\ \displaystyle \frac{\partial \overline g_{p,j}}{\partial u_i} \Big |_{({\bf u}_{\bar k},\tau_{\bar k})} ( u_{k^*,i} + \\ \hspace{2mm} K_k (\bar u_{k+1,i}^* - u_{k^*,i} ) - u_{\bar k,i} ) \end{array} \right] \vspace{1mm} \\
 + \displaystyle \sum_{i \not \in I_{c,j}^{\bar k}} \mathop {\max} \left[ \begin{array}{l} \underline \kappa_{p,ji}^{\bar k} ( u_{k^*,i} + \\ \hspace{2mm} K_k (\bar u_{k+1,i}^* - u_{k^*,i} ) - u_{\bar k,i} ), \vspace{1mm}\\ \overline \kappa_{p,ji}^{\bar k} ( u_{k^*,i} + \\ \hspace{2mm} K_k (\bar u_{k+1,i}^* - u_{k^*,i} ) - u_{\bar k,i} ) \end{array} \right] \end{array} \hspace{-1mm} \right] \vspace{1mm} \\
\hspace{40mm} \leq d_{p,j}^{k+1}, \;\; \forall j = 1,...,n_{g_p}
\end{array}
\end{equation}

\vspace{-2mm}
\begin{equation}\label{eq:SCFO1idegLUccvalllocMNgradbackslack}
\begin{array}{l}
\mathop {\min} \limits_{\bar k = 0,...,k} \left[ \hspace{-1mm} \begin{array}{l} \overline g_{p,j} ({\bf u}_{\bar k},\tau_{\bar k}) \displaystyle +\eta_{c,j}^{\bar k} \frac{\partial \overline g_{p,j}}{\partial \tau} \Big |_{({\bf u}_{\bar k},\tau_{\bar k})} ( \tau_{k+1} - \tau_{{\bar k}} ) \vspace{1mm} \\
 \displaystyle + (1-\eta_{c,j}^{\bar k}) \overline \kappa_{p,j\tau}^{\bar k} \left( \tau_{k+1} - \tau_{\bar k} \right) \vspace{1mm}\\
\displaystyle   + \sum_{i \in I_{c,j}^{\bar k}} \mathop {\max} \left[ \begin{array}{l} \displaystyle \frac{\partial \underline g_{p,j}}{\partial u_i} \Big |_{({\bf u}_{\bar k},\tau_{\bar k})} ( u_{k^*,i} + \\ \hspace{2mm} K_k (\bar u_{k+1,i}^* - u_{k^*,i} ) - u_{\bar k,i} ), \vspace{1mm}\\ \displaystyle \frac{\partial \overline g_{p,j}}{\partial u_i} \Big |_{({\bf u}_{\bar k},\tau_{\bar k})} ( u_{k^*,i} + \\ \hspace{2mm} K_k (\bar u_{k+1,i}^* - u_{k^*,i} ) - u_{\bar k,i} ) \end{array} \right] \vspace{1mm} \\
 + \displaystyle \sum_{i \not \in I_{c,j}^{\bar k}} \mathop {\max} \left[ \begin{array}{l} \underline \kappa_{p,ji}^{\bar k} ( u_{k^*,i} + \\ \hspace{2mm} K_k (\bar u_{k+1,i}^* - u_{k^*,i} ) - u_{\bar k,i} ), \vspace{1mm}\\ \overline \kappa_{p,ji}^{\bar k} ( u_{k^*,i} + \\ \hspace{2mm} K_k (\bar u_{k+1,i}^* - u_{k^*,i} ) - u_{\bar k,i} ) \end{array} \right] \end{array} \hspace{-1mm} \right] \vspace{1mm} \\
+ \displaystyle \eta_{c,j}^{K} \frac{\partial \overline g_{p,j}}{\partial \tau} \Big |_{\left({\bf u}_{k^*} + K_k ( \bar {\bf u}_{k+1}^* - {\bf u}_{k^*} ),\tau_{k+1}\right)} ( \tau_{k+2} - \tau_{{k+1}} ) \vspace{1mm} \\
+ (1-\eta_{c,j}^{K}) \overline \kappa_{p,j\tau}^{K} \left( \tau_{k+2} - \tau_{k+1} \right) + \delta_e \| \kappa_{p,j}^{m,K} \|_2  \leq d_{p,j}^{k+1}, \vspace{1mm} \\
\hspace{50mm} \forall j = 1,...,n_{g_p},
\end{array}
\end{equation}

\vspace{-2mm}
\begin{equation}\label{eq:SCFO2ibackslack}
\mathop {\max} \limits_{{\bf u} \in \mathcal{B}_{K}} g_{j}({\bf u}) \leq d_j^{k+1}, \;\; \forall j = 1,...,n_g.
\end{equation}

As before, the degradation term may be removed from (\ref{eq:SCFO1idegLUccvalllocMNgradbackslack}) if this condition is too restricting:

\vspace{-2mm}
\begin{equation}\label{eq:SCFO1idegLUccvalllocMNgradbackslack2}
\begin{array}{l}
\mathop {\min} \limits_{\bar k = 0,...,k} \left[ \hspace{-1mm} \begin{array}{l} \overline g_{p,j} ({\bf u}_{\bar k},\tau_{\bar k}) \displaystyle +\eta_{c,j}^{\bar k} \frac{\partial \overline g_{p,j}}{\partial \tau} \Big |_{({\bf u}_{\bar k},\tau_{\bar k})} ( \tau_{k+1} - \tau_{{\bar k}} ) \vspace{1mm} \\
 \displaystyle + (1-\eta_{c,j}^{\bar k}) \overline \kappa_{p,j\tau}^{\bar k} \left( \tau_{k+1} - \tau_{\bar k} \right) \vspace{1mm}\\
\displaystyle   + \sum_{i \in I_{c,j}^{\bar k}} \mathop {\max} \left[ \begin{array}{l} \displaystyle \frac{\partial \underline g_{p,j}}{\partial u_i} \Big |_{({\bf u}_{\bar k},\tau_{\bar k})} ( u_{k^*,i} + \\ \hspace{2mm} K_k (\bar u_{k+1,i}^* - u_{k^*,i} ) - u_{\bar k,i} ), \vspace{1mm}\\ \displaystyle \frac{\partial \overline g_{p,j}}{\partial u_i} \Big |_{({\bf u}_{\bar k},\tau_{\bar k})} ( u_{k^*,i} + \\ \hspace{2mm} K_k (\bar u_{k+1,i}^* - u_{k^*,i} ) - u_{\bar k,i} ) \end{array} \right] \vspace{1mm} \\
 + \displaystyle \sum_{i \not \in I_{c,j}^{\bar k}} \mathop {\max} \left[ \begin{array}{l} \underline \kappa_{p,ji}^{\bar k} ( u_{k^*,i} + \\ \hspace{2mm} K_k (\bar u_{k+1,i}^* - u_{k^*,i} ) - u_{\bar k,i} ), \vspace{1mm}\\ \overline \kappa_{p,ji}^{\bar k} ( u_{k^*,i} + \\ \hspace{2mm} K_k (\bar u_{k+1,i}^* - u_{k^*,i} ) - u_{\bar k,i} ) \end{array} \right] \end{array} \hspace{-1mm} \right] \vspace{1mm} \\
\hspace{20mm} + \delta_e \| \kappa_{p,j}^{m,K} \|_2  \leq d_{p,j}^{k+1}, \;\; \forall  j = 1,...,n_{g_p}.
\end{array}
\end{equation}

\subsection{Accounting for Slack in the Choice of Reference Point}

\begin{figure*}
\begin{center}
\includegraphics[width=16cm]{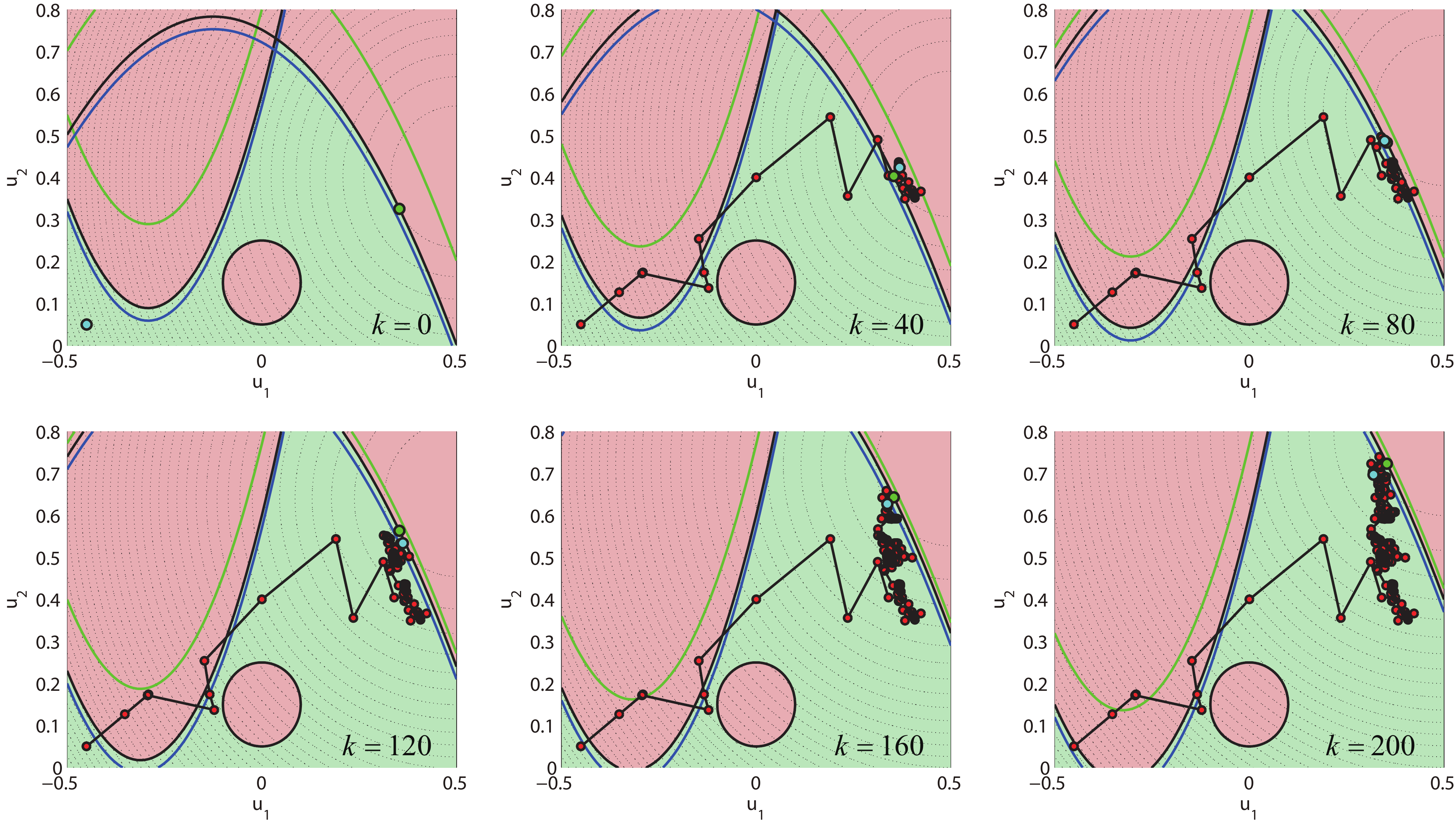}
\caption{Chain of experiments generated by applying the modified SCFO methodology to Problem (\ref{eq:exdeg}) for the ($-$) scenario with slacks added to the experimental constraints (shown in green).}
\label{fig:degJ}
\end{center}
\end{figure*}

Similar modifications are made to (\ref{eq:kstarLUccvcostlocMNgradBO}) and (\ref{eq:kfeas5}):

\vspace{-2mm}
\begin{equation}\label{eq:kstarLUccvcostlocMNgradBOslack}
\begin{array}{rl}
k^* := \;\;\;\;\;\;\;\;\;\;\;\;\;\;& \vspace{1mm} \\
{\rm arg} \mathop {\rm maximize}\limits_{\bar k \in [0,k]} & \bar k \vspace{1mm}  \\
{\rm{subject}}\;{\rm{to}} & \overline g_{p,j} ({\bf u}_{\bar k},\tau_{\bar k}) \vspace{1mm} \\
& \displaystyle + \displaystyle \eta_{c,j}^{e,\bar k} \frac{\partial \overline g_{p,j}}{\partial \tau} \Big |_{({\bf u}_{\bar k},\tau_{\bar k})} ( \tau_{k+1} - \tau_{{\bar k}} ) \vspace{1mm} \\
&  + (1-\eta_{c,j}^{e,\bar k}) \overline \kappa_{p,j\tau}^{e,\bar k} \left( \tau_{k+1} - \tau_{\bar k} \right) \vspace{1mm} \\
& + \delta_e \| \kappa_{p,j}^{m,\bar k} \|_2 \leq d_{p,j}^{k+1}, \;\; \forall j = 1,...,n_{g_p}  \vspace{1mm} \\
& \mathop {\max} \limits_{{\bf u} \in \mathcal{B}_{e,\bar k}} g_j ({\bf u}) \leq d_j^{k+1}, \;\; \forall j = 1,...,n_{g} \vspace{1mm} \\
& {\bf u}^L + \delta_e {\bf 1}  \preceq {\bf u}_{\bar k} \preceq {\bf u}^U - \delta_e {\bf 1} \vspace{1mm} \\
& \displaystyle \underline \phi_p ({\bf u}_{\bar k},\tau_{\bar k}) +  \eta_{v,\phi}^{\bar k, k} \frac{\partial \underline \phi_{p}}{\partial \tau} \Big |_{({\bf u}_{\bar k},\tau_{\bar k})} ( \tau_{k} - \tau_{{\bar k}} ) \vspace{1mm} \\
&+ (1-  \eta_{v,\phi}^{\bar k, k}) \underline \kappa_{\phi,\tau}^{\bar k, k} \left( \tau_{k} - \tau_{\bar k} \right) \leq \vspace{1mm} \\
& \mathop {\min} \limits_{\tilde k \in {\bf k}_f} \left[ \begin{array}{l} \overline \phi_p ({\bf u}_{\tilde k},\tau_{\tilde k}) + \vspace{1mm} \\
\displaystyle \eta_{c,\phi}^{\tilde k, k} \frac{\partial \overline \phi_{p}}{\partial \tau} \Big |_{({\bf u}_{\tilde k},\tau_{\tilde k})} ( \tau_{k} - \tau_{{\tilde k}} ) \vspace{1mm} \\
+ (1-  \eta_{c,\phi}^{\tilde k, k}) \overline \kappa_{\phi,\tau}^{\tilde k, k} \left( \tau_{k} - \tau_{\tilde k} \right) \end{array} \right],
\end{array}
\end{equation}

\vspace{-2mm}
\begin{equation}\label{eq:kfeas6}
{\bf k}_f = \left\{ \bar k : 
\begin{array}{l}
\overline g_{p,j} ({\bf u}_{\bar k},\tau_{\bar k}) \vspace{1mm} \\ 
+ \displaystyle \eta_{c,j}^{e,\bar k} \frac{\partial \overline g_{p,j}}{\partial \tau} \Big |_{({\bf u}_{\bar k},\tau_{\bar k})} ( \tau_{k+1} - \tau_{{\bar k}} ) \vspace{1mm} \\ 
+ (1-\eta_{c,j}^{e,\bar k}) \overline \kappa_{p,j\tau}^{e,\bar k} \left( \tau_{k+1} - \tau_{\bar k} \right) \vspace{1mm} \\ 
+ \delta_e \| \kappa_{p,j}^{m,\bar k} \|_2  \leq d_{p,j}^{k+1}, \;\; \forall j = 1,...,n_{g_p}; \vspace{1mm} \\
 \mathop {\max} \limits_{{\bf u} \in \mathcal{B}_{e,\bar k}} g_j ({\bf u}) \leq d_j^{k+1}, \;\; \forall j = 1,...,n_{g}; \vspace{1mm} \\
{\bf u}^L + \delta_e {\bf 1}  \preceq {\bf u}_{\bar k} \preceq {\bf u}^U - \delta_e {\bf 1}
\end{array} \right\},
\end{equation}

\noindent as this would yield a ${\bf u}_{k^*}$ that would guarantee the existence of a $K_k \geq 0$ that satisfies, at least, the relaxed condition of (\ref{eq:SCFO1idegLUccvalllocMNgradbackslack2}).

In the case that no $K_k$ satisfying these restrictions can be found, one may proceed with (\ref{eq:kstar2LUccvlocMNgradBO}) with the following modifications made to (\ref{eq:viols}):

\vspace{-2mm}
\begin{equation}\label{eq:violsmod}
\begin{array}{l}
g_{p,m}^{\bar k} = \mathop {\max} \limits_{j = 1,...,n_{g_p}} \left[ \begin{array}{l}
\overline g_{p,j} ({\bf u}_{\bar k},\tau_{\bar k}) \vspace{1mm} \\ 
+ \displaystyle \eta_{c,j}^{e,\bar k} \frac{\partial \overline g_{p,j}}{\partial \tau} \Big |_{({\bf u}_{\bar k},\tau_{\bar k})} ( \tau_{k+1} - \tau_{{\bar k}} ) \vspace{1mm} \\ 
+ (1-\eta_{c,j}^{e,\bar k}) \overline \kappa_{p,j\tau}^{e,\bar k} \left( \tau_{k+1} - \tau_{\bar k} \right) \vspace{1mm} \\ 
+ \delta_e \| \kappa_{p,j}^{m,\bar k} \|_2 - d_{p,j}^{k+1} \end{array} \right] \vspace{1mm} \\
g_{m}^{\bar k} = \mathop {\max} \limits_{j = 1,...,n_g} \mathop {\max} \limits_{{\bf u} \in \mathcal{B}_{e,\bar k}} \left( g_j ({\bf u}) - d_{j}^{k+1} \right) \vspace{1mm} \\
u_{L,m}^{\bar k} = \mathop {\max} \limits_{i = 1,...,n_u} \left( u^L_i + \delta_e - u_{\bar k,i} \right) \vspace{1mm} \\
u_{U,m}^{\bar k} = \mathop {\max} \limits_{i = 1,...,n_u} \left( u_{\bar k,i} + \delta_e - u^U_i  \right).
\end{array}
\end{equation}

\subsection{Accounting for Slack in the Projection}

Finally, we make the same modification to the Boolean triggers in the projection (\ref{eq:projdeg2robslackBO2}):

\vspace{-2mm}
\begin{equation}\label{eq:projdeg2robslackBO2slack}
\begin{array}{rl}
\bar {\bf u}_{k+1}^* := \;\;\;\;\;\;\;\;\; & \vspace{1mm}\\
 {\rm arg} \mathop {\rm minimize}\limits_{{\bf u},{\bf s}_\phi, {\bf S}} & \| {\bf u} - {\bf u}_{k+1}^* \|_2^2 \vspace{1mm}  \\
 {\rm{subject}}\;{\rm{to}} & \displaystyle \sum_{i=1}^{n_u} s_{ji}  \leq -\delta_{g_p,j} \vspace{1mm} \\
& \displaystyle \frac{\partial \underline g_{p,j}}{\partial u_i} \Big |_{({\bf u}_{k^*}, \tau_{k+1})} (u_i - u_{k^*,i}) \leq s_{ji} \vspace{1mm} \\
& \displaystyle \frac{\partial \overline g_{p,j}}{\partial u_i} \Big |_{({\bf u}_{k^*}, \tau_{k+1})} (u_i - u_{k^*,i}) \leq s_{ji}, \vspace{1mm} \\
& \forall i = 1,...,n_u, \vspace{1mm} \\
&  \forall j: \begin{array}{l} \overline g_{p,j} ({\bf u}_{k^*},\tau_{k^*}) \vspace{1mm} \\
 +  \displaystyle \eta_{c,j}^{e,k^*} \frac{\partial \overline g_{p,j}}{\partial \tau} \Big |_{({\bf u}_{k^*},\tau_{k^*})} ( \tau_{k+1} - \tau_{{ k^*}} ) \vspace{1mm} \\ 
 + (1-\eta_{c,j}^{e,k^*}) \overline \kappa_{p,j\tau}^{e,k^*} \left( \tau_{k+1} - \tau_{k^*} \right) \vspace{1mm} \\ 
 + \delta_e \| \kappa_{p,j}^{m,k^*} \|_2 \geq -\epsilon_{p,j}+d_{p,j}^{k+1} \end{array} \vspace{1mm} \\
 & \nabla g_{j} ({\bf u}_{k^*})^T ({\bf u} - {\bf u}_{k^*}) \leq -\delta_{g,j}, \vspace{1mm} \\
& \forall j : \mathop {\max} \limits_{{\bf u} \in \mathcal{B}_{e,k^*}} g_{j}({\bf u}) \geq -\epsilon_{j}+d_{j}^{k+1} \vspace{1mm} \\
 & \displaystyle \sum_{i=1}^{n_u} s_{\phi,i}  \leq -\delta_{\phi} \vspace{1mm} \\
& \displaystyle \frac{\partial \underline \phi_{p}}{\partial u_i} \Big |_{({\bf u}_{k^*}, \tau_{k+1})} (u_i - u_{k^*,i}) \leq s_{\phi,i} \vspace{1mm} \\
& \displaystyle \frac{\partial \overline \phi_{p}}{\partial u_i} \Big |_{({\bf u}_{k^*}, \tau_{k+1})} (u_i - u_{k^*,i}) \leq s_{\phi,i} \vspace{1mm} \\
 & {\bf u}^L + \delta_e {\bf 1} \preceq {\bf u} \preceq {\bf u}^U - \delta_e {\bf 1}, 
\end{array}
\end{equation}

\noindent which simply ``delays'' the projection into the locally feasible space for those constraints that are relaxed.

\subsection{Example}

We build further on the previous example of Section \ref{sec:BOex} by adding the slacks

\vspace{-2mm}
\begin{equation}\label{eq:exslacks}
\begin{array}{l}
\overline d_{p,1} = \overline d_{p,2} = 0.2 \vspace{1mm} \\
d_{p,1}^S = 5, \; d_{p,2}^S = 10,
\end{array}
\end{equation}

\noindent to the experimental constraints only, and apply the proposed modified SCFO with $\beta_{p,j}$ taken as the maximum value given by the bound (\ref{eq:betamax}).

The results are shown in Fig. \ref{fig:degJ}, with the corresponding cost function values given in Fig. \ref{fig:degJcost}. As expected, we see two key benefits of allowing temporary violations -- the first being that it is possible to converge to the optimum faster, and the second that it is also possible to temporarily obtain a cost that is better than the one at the optimum, thus leading to greater overall gains. Fig. \ref{fig:degJcon} shows that the sums of the violations are well below\footnote{While we do not explore this option here, it should be clear that one could propose more elegant slack reduction techniques that would likely lead to even better performance.} the limits set by $d_{p,j}^S$.

\begin{figure}
\begin{center}
\includegraphics[width=8cm]{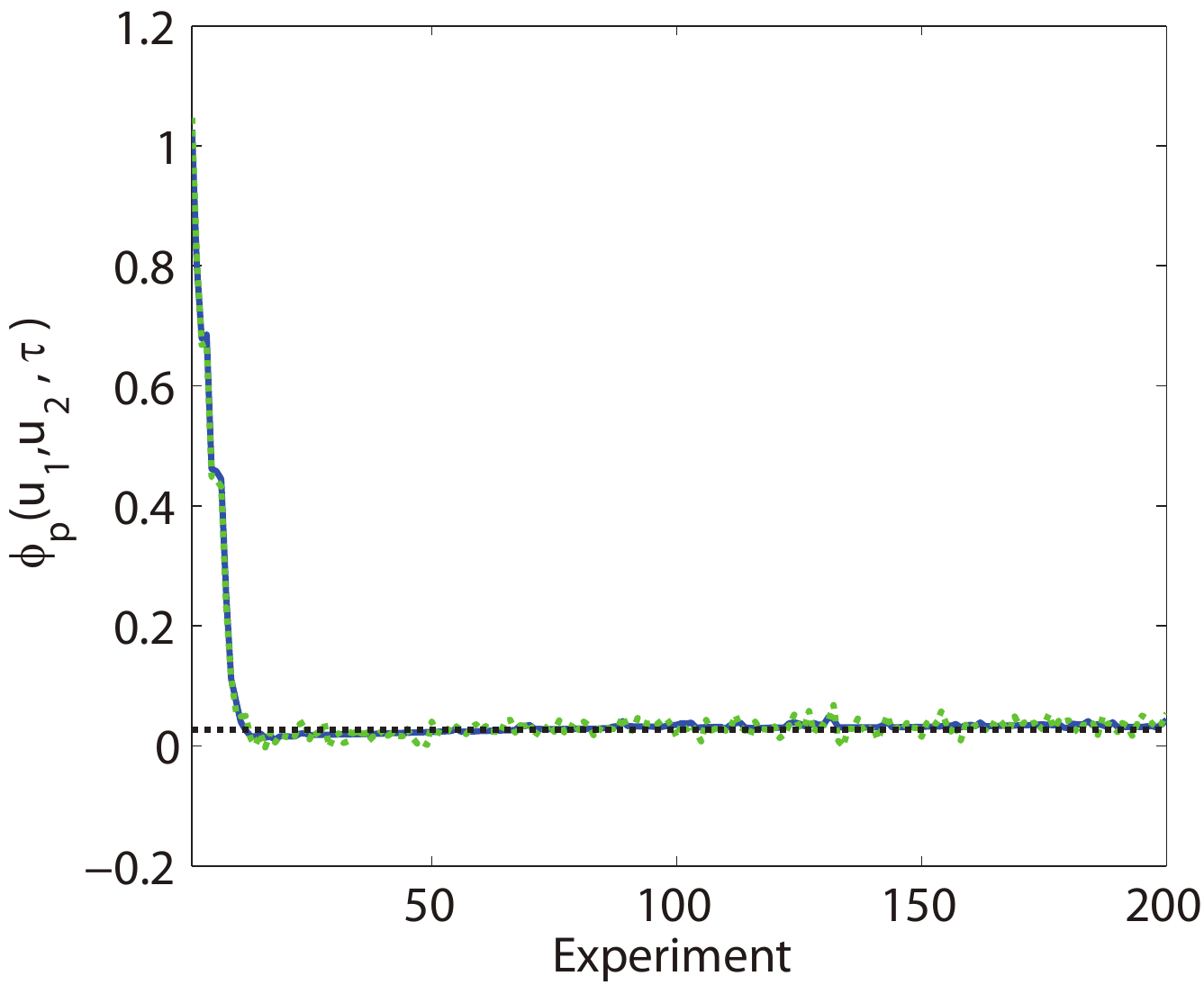}
\caption{Cost function values obtained by the modified SCFO methodology for Problem (\ref{eq:exdeg}) for the ($-$) scenario when slacks are added to the experimental constraints.}
\label{fig:degJcost}
\end{center}
\end{figure} 

\begin{figure}
\begin{center}
\includegraphics[width=8cm]{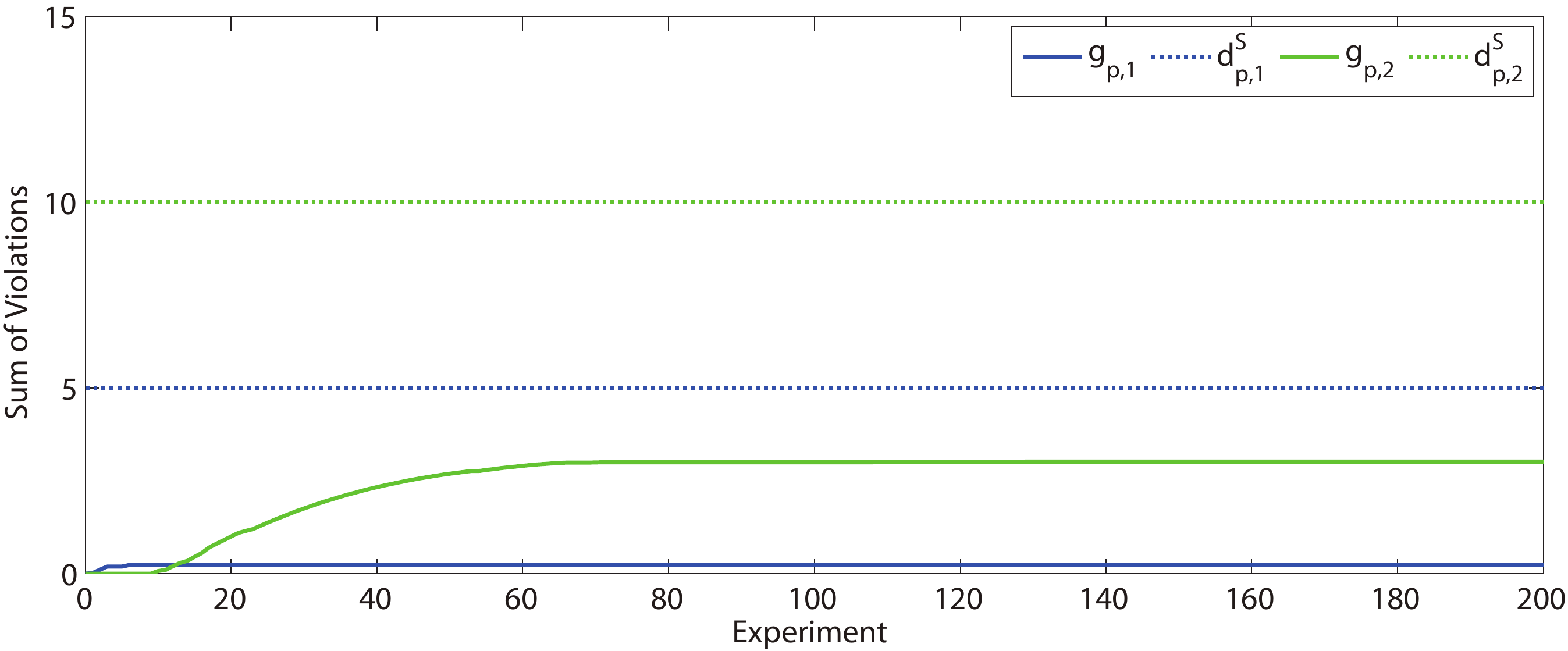}
\caption{The values of the sums of experimental constraint violations over the course of optimization.}
\label{fig:degJcon}
\end{center}
\end{figure} 

\section{Accomodating a Numerical Cost Function}
\label{sec:numcost}

In some experimental optimization problems it may occur that the cost function is numerical and can be easily evaluated without requiring experiments -- see, e.g., the second application example in \cite{SCFOug}. When this is so, we may simplify the SCFO considerably by treating $\phi$ (a known function) instead of $\phi_p$ (an unknown function), and exploiting the fact that derivative information on the function is available, that there is no degradation, and that the Lipschitz constants for the cost are no longer required for the desired properties to be enforced.

Most notably, the line search may be modified from finding the largest $K_k$ that satisfies the SCFO to finding the $K_k$ that minimizes the cost while satisfying (\ref{eq:SCFO1idegLUccvalllocMNgradslack})-(\ref{eq:SCFO2ibackslack}):

\vspace{-2mm}
\begin{equation}\label{eq:linesearchnum}
\begin{array}{rl}
K_k := {\rm arg} \mathop {\rm minimize}\limits_{K \in [0,1]} & \phi \left( {\bf u}_{k^*} + K (\bar {\bf u}_{k+1}^* - {\bf u}_{k^*}) \right)  \\
{\rm{subject}}\;{\rm{to}} & (\ref{eq:SCFO1idegLUccvalllocMNgradslack}){\rm -}(\ref{eq:SCFO2ibackslack}).
\end{array}
\end{equation}

\noindent Because this line search is guaranteed to yield a cost decrease if one exists -- unlike the line search incorporating (\ref{eq:SCFO7idegLUlocgrad}), which can fail to yield an actual decrease in the case of erroneous Lipschitz constants -- the necessary condition for cost decrease (\ref{eq:costhighmaxPFlocMNgrad}) is not needed.

In choosing the reference point as in (\ref{eq:kstarLUccvcostlocMNgradBOslack}) and (\ref{eq:kfeas6}), one no longer needs to use the heuristic of taking the latest feasible experimental iterate that is guaranteed to not have a cost value that is worse than that an another feasible experimental iterate. Instead, one may simply choose the feasible iterate with the lowest cost value, i.e.:

\vspace{-2mm}
\begin{equation}\label{eq:kstarLUccvcostlocMNgradBOslacknum}
\begin{array}{rl}
k^* := \;\;\;\;\;\;\;\;\;\;\;\;\;\;& \vspace{1mm} \\
{\rm arg} \mathop {\rm minimize}\limits_{\bar k \in [0,k]} & \phi({\bf u}_{\bar k}) \vspace{1mm}  \\
{\rm{subject}}\;{\rm{to}} & \overline g_{p,j} ({\bf u}_{\bar k},\tau_{\bar k}) \vspace{1mm} \\
& \displaystyle + \displaystyle \eta_{c,j}^{e,\bar k} \frac{\partial \overline g_{p,j}}{\partial \tau} \Big |_{({\bf u}_{\bar k},\tau_{\bar k})} ( \tau_{k+1} - \tau_{{\bar k}} ) \vspace{1mm} \\
&  + (1-\eta_{c,j}^{e,\bar k}) \overline \kappa_{p,j\tau}^{e,\bar k} \left( \tau_{k+1} - \tau_{\bar k} \right) \vspace{1mm} \\
& + \delta_e \| \kappa_{p,j}^{m,\bar k} \|_2 \leq d_{p,j}^{k+1}, \;\; \forall j = 1,...,n_{g_p} \vspace{1mm} \\
& \mathop {\max} \limits_{{\bf u} \in \mathcal{B}_{e,\bar k}} g_j ({\bf u}) \leq d_j^{k+1}, \;\; \forall j = 1,...,n_{g} \vspace{1mm} \\
& {\bf u}^L + \delta_e {\bf 1}  \preceq {\bf u}_{\bar k} \preceq {\bf u}^U - \delta_e {\bf 1}.
\end{array}
\end{equation}

Finally, the projection (\ref{eq:projdeg2robslackBO2slack}) may be simplified to

\vspace{-2mm}
\begin{equation}\label{eq:projdeg2robslackBO2slacknum}
\begin{array}{rl}
\bar {\bf u}_{k+1}^* := \;\;\;\;\;\;\;\;\; & \vspace{1mm}\\
 {\rm arg} \mathop {\rm minimize}\limits_{{\bf u},{\bf S}} & \| {\bf u} - {\bf u}_{k+1}^* \|_2^2 \vspace{1mm}  \\
 {\rm{subject}}\;{\rm{to}} & \displaystyle \sum_{i=1}^{n_u} s_{ji}  \leq -\delta_{g_p,j} \vspace{1mm} \\
& \displaystyle \frac{\partial \underline g_{p,j}}{\partial u_i} \Big |_{({\bf u}_{k^*}, \tau_{k+1})} (u_i - u_{k^*,i}) \leq s_{ji} \vspace{1mm} \\
& \displaystyle \frac{\partial \overline g_{p,j}}{\partial u_i} \Big |_{({\bf u}_{k^*}, \tau_{k+1})} (u_i - u_{k^*,i}) \leq s_{ji}, \vspace{1mm} \\
& \forall i = 1,...,n_u, \vspace{1mm} \\
&  \forall j: \begin{array}{l} \overline g_{p,j} ({\bf u}_{k^*},\tau_{k^*}) \vspace{1mm} \\
 +  \displaystyle \eta_{c,j}^{e,k^*} \frac{\partial \overline g_{p,j}}{\partial \tau} \Big |_{({\bf u}_{k^*},\tau_{k^*})} ( \tau_{k+1} - \tau_{{ k^*}} ) \vspace{1mm} \\ 
 + (1-\eta_{c,j}^{e,k^*}) \overline \kappa_{p,j\tau}^{e,k^*} \left( \tau_{k+1} - \tau_{k^*} \right) \vspace{1mm} \\ 
 + \delta_e \| \kappa_{p,j}^{m,k^*} \|_2 \geq -\epsilon_{p,j}+d_{p,j}^{k+1} \end{array} \vspace{1mm} \\
 & \nabla g_{j} ({\bf u}_{k^*})^T ({\bf u} - {\bf u}_{k^*}) \leq -\delta_{g,j}, \vspace{1mm} \\
& \forall j : \mathop {\max} \limits_{{\bf u} \in \mathcal{B}_{e,k^*}} g_{j}({\bf u}) \geq -\epsilon_{j}+d_{j}^{k+1} \vspace{1mm} \\
 & \nabla \phi ({\bf u}_{k^*})^T ({\bf u} - {\bf u}_{k^*})  \leq -\delta_{\phi} \vspace{1mm} \\
 & {\bf u}^L + \delta_e {\bf 1} \preceq {\bf u} \preceq {\bf u}^U - \delta_e {\bf 1}. 
\end{array}
\end{equation}

\section{Complete Implementable Form of the SCFO}
\label{sec:complete}

In this section, we will simply summarize the results derived so far to give the SCFO in their ready-to-code and fully implementable form, where by ``fully implementable'' we mean that all of the different issues raised in this document are simultaneously accounted for. Additionally, we will restate all presented algorithms fully so as to include the proposed modifications -- some of which came several sections later in the text after the algorithm was first stated -- directly inside the algorithm statement. This is done to ease the task of coding for those readers who are intent on writing their own SCFO implementation or modifying the one currently available \cite{SCFOug}.

We start by noting that the practical SCFO implementation may be seen as consisting of an initialization step followed by six iterative steps, the latter being carried out at every experimental iteration:

\begin{enumerate}
\item (Initialization -- carried out once prior to any optimization) One defines all the elements of the experimental optimization problem that one expects to remain constant (i.e., $\overline \epsilon_{p,j}$, $\overline \epsilon_{j}$, $\overline \delta_{g_p,j}$, $\overline \delta_{g,j}$, $\overline \delta_\phi$, $\underline \kappa$, $\overline \kappa$, $\underline M$, $\overline M$, $I_{c}$, $I_{v}$, $\eta$, $\underline W$, $\overline W$, $\delta_e$, $\overline d$, $d^S$).
\item (Pretreatment of collected data) Using the measurements obtained up until the current iteration, one runs the Lipschitz consistency check algorithm and then uses the consistent Lipschitz constants to compute lower and upper bounds on the measured experimental function values.
\item (Choice of reference point) The reference point ${\bf u}_{k^*}$ is chosen as the point that meets the feasibility requirements while being likely to have the best cost function value obtained in all experiments carried out so far.
\item (Computation of the optimal target) The target ${\bf u}_{k+1}^*$ is found by some prescribed optimization method -- see \cite{Bunin2013SIAM} for a review of such methods.
\item (Projection of optimal target) The target ${\bf u}_{k+1}^*$ is projected to obtain the projected target $\bar {\bf u}_{k+1}^*$ satisfying a portion of the SCFO in some robust sense.
\item (Filtering of adaptation step) The filter gain $K_k$ is found by a line search to define ${\bf u}_{k+1}$ as a function of ${\bf u}_{k^*}$, $\bar {\bf u}_{k+1}^*$, and $K_k$. If the excitation obtained by changing the decision variables is insufficient, this choice may be overriden by a sufficiently exciting ${\bf u}_{k+1}$. 
\item (Experimental application) An experiment defined by the computed ${\bf u}_{k+1}$ is carried out, new measurements are obtained, and the cycle restarts with Step 2.
\end{enumerate}

\noindent Note that this scheme does not have a termination criterion. This is intentional and reflects the philosophy discussed by, for example, Box and Draper in \cite{Box:69}, in that in experimental optimization there is little reason to stop continuous improvement, especially when effects like degradation may continually force the algorithm to pursue a moving optimum. From an algorithmic viewpoint, it also does not make sense to terminate the SCFO adaptations since these adaptations are designed to obtain better and better performance until this becomes impossible. Of course, one may nevertheless propose certain termination criteria when the current level of optimality is known to be sufficient, and the most recent version of the SCFO solver  accomodates this \cite{SCFOug} -- we do not, however, address this here.

The rest of this section is devoted to going through Steps 2, 3, 5, and 6 and giving the relevant formulas used in the different subroutines of each in their finalized, complete forms.

\subsection{Pretreatment of Measured Data}

In the pretreatment of the data collected, one has the tasks of (a) verifying that the Lipschitz constants provided are consistent with the data and, if not, making the constants consistent, and (b) computing lower and upper bounds on the true experimental function values. As mentioned previously, the two tasks are convoluted since the Lipschitz consistency check uses the computed lower and upper bounds and the lower and upper bounds may in turn be a function of the Lipschitz constants. To deconvolute the two, we make it so that the bounds used in the Lipschitz consistency check are independent of the Lipschitz constants, i.e., we use the bounds given in (\ref{eq:costbound1}):

\vspace{-2mm}
\begin{equation}\label{eq:costbound1a}
\begin{array}{l}
\underline {\underline \phi}_p ({\bf u}_{\bar k},\tau_{\bar k}) = \hat \phi_p({\bf u}_{\bar k},\tau_{\bar k}) - \overline w_{\phi, \bar k} \vspace{1mm} \\
\overline {\overline \phi}_p ({\bf u}_{\bar k},\tau_{\bar k}) = \hat \phi_p({\bf u}_{\bar k},\tau_{\bar k}) - \underline w_{\phi, \bar k},
\end{array}
\end{equation}

\noindent where the double lines denote that these are more conservative bounds than what would be obtained if we were to run Algorithm 2 with the correct Lipschitz constants.

We now state the finalized version of Algorithm 1.
\newline
\newline
{\bf Algorithm 1F -- Lipschitz Consistency Check (Finalized)}
\begin{enumerate}
\item Set $a := 1$.
\item Check the validity of

\vspace{-2mm}
\begin{equation}\label{eq:lipcheck1costUMNf}
\begin{array}{l}
\underline {\underline \phi}_{p} ({\bf u}_{k_2},\tau_{k_2}) \leq \overline {\overline \phi}_{p} ({\bf u}_{k_1},\tau_{k_1}) \vspace{1mm} \\
\hspace{20mm} \displaystyle + \mathop {\max} \left[ \begin{array}{l} \underline \kappa_{\phi,\tau} \left( \tau_{k_2} - \tau_{k_1} \right) \vspace{1mm} \\ 
\overline \kappa_{\phi,\tau} \left( \tau_{k_2} - \tau_{k_1} \right) \end{array} \right] \vspace{1mm} \\
\hspace{20mm}\displaystyle + \sum_{i=1}^{n_u} \mathop {\max} \left[ \begin{array}{l} \underline \kappa_{\phi,i} ( u_{k_2,i} - u_{k_1,i} ), \vspace{1mm} \\ \overline \kappa_{\phi,i} ( u_{k_2,i} - u_{{k_1},i} ) \end{array} \right]
\end{array}
\end{equation}

\noindent and

\vspace{-2mm}
\begin{equation}\label{eq:lipcheck1costLMNf}
\begin{array}{l}
\overline {\overline \phi}_{p} ({\bf u}_{k_2},\tau_{k_2}) \geq \underline {\underline \phi}_{p} ({\bf u}_{k_1},\tau_{k_1}) \vspace{1mm} \\
\hspace{20mm} \displaystyle + \mathop {\min} \left[ \begin{array}{l} \underline \kappa_{\phi,\tau} \left( \tau_{k_2} - \tau_{k_1} \right) \vspace{1mm} \\ 
\overline \kappa_{\phi,\tau} \left( \tau_{k_2} - \tau_{k_1} \right) \end{array} \right] \vspace{1mm} \\
\hspace{20mm}\displaystyle + \sum_{i=1}^{n_u} \mathop {\min} \left[ \begin{array}{l} \underline \kappa_{\phi,i} ( u_{k_2,i} - u_{k_1,i} ), \vspace{1mm} \\ \overline \kappa_{\phi,i} ( u_{k_2,i} - u_{{k_1},i} ) \end{array} \right]
\end{array}
\end{equation}
 
for every combination $( k_1, k_2 ) \in \{0,...,k\} \times \{0,...,k\}$. If these inequalities are satisfied for every $( k_1, k_2 )$, then terminate. Otherwise, proceed to Step 3.
\item If $a \leq 5$, then for each $i = 1,...,n_u$ set

\vspace{-2mm}
$$
\begin{array}{rcl}
\underline \kappa_{\phi,i} & := & 2^{-{\rm sign}\;\underline \kappa_{\phi,i}}\underline \kappa_{\phi,i}, \vspace{1mm} \\
\overline \kappa_{\phi,i} & := & 2^{{\rm sign}\;\overline \kappa_{\phi,i}}\overline \kappa_{\phi,i}, \vspace{1mm} \\
\underline \kappa_{\phi,\tau} & := & 2^{-{\rm sign}\;\underline \kappa_{\phi,\tau}}\underline \kappa_{\phi,\tau}, \vspace{1mm} \\
\overline \kappa_{\phi,\tau} & := & 2^{{\rm sign}\;\overline \kappa_{\phi,\tau}}\overline \kappa_{\phi,\tau}.
\end{array}
$$

If $5 < a \leq 10$, set

\vspace{-2mm}
$$
\begin{array}{rcl}
\underline \kappa_{\phi,i} & := & -2 \mathop {\max} \left[ | \underline \kappa_{\phi,i} |, | \overline \kappa_{\phi,i} | \right], \vspace{1mm} \\ 
\overline \kappa_{\phi,i} & := & 2 \mathop {\max} \left[ | \underline \kappa_{\phi,i} |, | \overline \kappa_{\phi,i} | \right], \vspace{1mm} \\ 
\underline \kappa_{\phi,\tau} & := & -2 \mathop {\max} \left[ | \underline \kappa_{\phi,\tau} |, | \overline \kappa_{\phi,\tau} | \right], \vspace{1mm} \\
\overline \kappa_{\phi,\tau} & := & 2 \mathop {\max} \left[ | \underline \kappa_{\phi,\tau} |, | \overline \kappa_{\phi,\tau} | \right].
\end{array}
$$

If $10 < a$, set

\vspace{-2mm}
$$
\begin{array}{rcl}
\underline \kappa_{\phi,i} & := & 2^{a-10} \underline \kappa_{\phi,i}, \vspace{1mm} \\
\overline \kappa_{\phi,i} & := & 2^{a-10} \overline \kappa_{\phi,i}, \vspace{1mm} \\
\underline \kappa_{\phi,\tau} & := & 2^{a-10} \underline \kappa_{\phi,\tau}, \vspace{1mm} \\
\overline \kappa_{\phi,\tau} & := & 2^{a-10} \overline \kappa_{\phi,\tau}.
\end{array}
$$

\item Set $a := a+1$ and return to Step 2.

\end{enumerate}

For the sake of completeness, let us also state the version for the higher-order Lipschitz constants $M$.
\newline
\newline
{\bf Algorithm 1FA -- Higher-Order Lipschitz Consistency Check (Finalized)}
\begin{enumerate}
\item Set $a := 1$.
\item Check the validity of

\vspace{-2mm}
\begin{equation}\label{eq:lipcheck2UMNgradf}
\begin{array}{l}
\underline {\underline \phi}_{p} ({\bf u}_{k_2},\tau_{k_2}) \leq \vspace{1mm} \\
\overline {\overline \phi}_{p} ({\bf u}_{k_1},\tau_{k_1}) + \mathop {\max} \left[ \begin{array}{l}  \underline \kappa_{\phi,\tau} ( \tau_{k_2} - \tau_{k_1} ) , \vspace{1mm} \\ \overline \kappa_{\phi,\tau} ( \tau_{k_2} - \tau_{k_1} ) \end{array} \right]  \vspace{1mm} \\
\displaystyle + \sum_{i=1}^{n_u} \mathop {\max} \left[ \begin{array}{l} \displaystyle \frac{\partial \underline \phi_{p}}{\partial u_i} \Big |_{({\bf u}_{k_1},\tau_{k_1})} ( u_{k_2,i} - u_{k_1,i} ) , \vspace{1mm} \\ \displaystyle \frac{\partial \overline \phi_{p}}{\partial u_i} \Big |_{({\bf u}_{k_1},\tau_{k_1})} ( u_{k_2,i} - u_{k_1,i} ) \end{array} \right] \vspace{1mm} \\
+\displaystyle \frac{1}{2} \sum_{i_1=1}^{n_u} \sum_{i_2=1}^{n_u} \mathop {\max} \left[ \begin{array}{l} \underline M_{\phi,i_1 i_2} (u_{k_2,i_1} - u_{k_1,i_1}) \vspace{1mm} \\
\hspace{15mm} (u_{k_2,i_2} - u_{{k_1},i_2}), \\ \overline M_{\phi,i_1 i_2} (u_{k_2,i_1} - u_{{k_1},i_1}) \vspace{1mm} \\
\hspace{15mm}(u_{k_2,i_2} - u_{{k_1},i_2}) \end{array} \right]
\end{array}
\end{equation}

\noindent and

\vspace{-2mm}
\begin{equation}\label{eq:lipcheck2LMNgradf}
\begin{array}{l}
\overline {\overline \phi}_{p} ({\bf u}_{k_2},\tau_{k_2}) \geq \\
\underline {\underline \phi}_{p} ({\bf u}_{k_1},\tau_{k_1}) + \mathop {\min} \left[ \begin{array}{l}  \underline \kappa_{\phi,\tau} ( \tau_{k_2} - \tau_{k_1} ) , \vspace{1mm} \\ \overline \kappa_{\phi,\tau} ( \tau_{k_2} - \tau_{k_1} ) \end{array} \right]  \vspace{1mm} \\
\displaystyle + \sum_{i=1}^{n_u} \mathop {\min} \left[ \begin{array}{l} \displaystyle \frac{\partial \underline \phi_{p}}{\partial u_i} \Big |_{({\bf u}_{k_1},\tau_{k_1})} ( u_{k_2,i} - u_{k_1,i} ) , \vspace{1mm} \\ \displaystyle \frac{\partial \overline \phi_{p}}{\partial u_i} \Big |_{({\bf u}_{k_1},\tau_{k_1})} ( u_{k_2,i} - u_{k_1,i} ) \end{array} \right] \vspace{1mm} \\
+\displaystyle \frac{1}{2} \sum_{i_1=1}^{n_u} \sum_{i_2=1}^{n_u} \mathop {\min} \left[ \begin{array}{l} \underline M_{\phi,i_1 i_2} (u_{k_2,i_1} - u_{k_1,i_1}) \vspace{1mm} \\
\hspace{15mm} (u_{k_2,i_2} - u_{{k_1},i_2}), \\ \overline M_{\phi,i_1 i_2} (u_{k_2,i_1} - u_{{k_1},i_1}) \vspace{1mm} \\
\hspace{15mm}(u_{k_2,i_2} - u_{{k_1},i_2}) \end{array} \right]
\end{array}
\end{equation}
 
for every combination $( k_1, k_2 ) \in \{ 0,...,k \} \times \{ 0,..., k\}$. If these inequalities are satisfied for every $( k_1, k_2 )$, then terminate. Otherwise, proceed to Step 3.
\item If $a \leq 5$, then, for each $i_1 = 1,...,n_u, \; i_2 = 1,...,n_u$, set

\vspace{-2mm}
$$
\begin{array}{rcl}
\underline M_{\phi,i_1 i_2} & := & 2^{-{\rm sign}\;\underline M_{\phi,i_1 i_2}}\underline M_{\phi,i_1 i_2}, \vspace{1mm} \\
\overline M_{\phi,i_1 i_2} & := & 2^{{\rm sign}\;\overline M_{\phi,i_1 i_2}}\overline M_{\phi,i_1 i_2}.
\end{array}
$$

If $5 < a \leq 10$, set

\vspace{-2mm}
$$
\begin{array}{rcl}
\underline M_{\phi,i_1 i_2} & := & -2 \mathop {\max} \left[ | \underline M_{\phi,i_1 i_2} |, | \overline M_{\phi,i_1 i_2} | \right], \vspace{1mm} \\ 
\overline M_{\phi,i_1 i_2} & := & 2 \mathop {\max} \left[ | \underline M_{\phi,i_1 i_2} |, | \overline M_{\phi,i_1 i_2} | \right].
\end{array}
$$

If $10 < a$, set

\vspace{-2mm}
$$
\begin{array}{rcl}
\underline M_{\phi,i_1 i_2} & := & 2^{a-10} \underline M_{\phi,i_1 i_2}, \vspace{1mm} \\
\overline M_{\phi,i_1 i_2} & := & 2^{a-10} \overline M_{\phi,i_1 i_2}.
\end{array}
$$

\item Set $a := a+1$ and return to Step 2.

\end{enumerate}

To obtain consistent Lipschitz constants for the experimental constraints, one applies Algorithm 1F to the functions $g_{p,j}$ as well.

Following this treatment of the Lipschitz constants, we may now apply the finalized version of Algorithm 2.
\newline
\newline
{\bf Algorithm 2F -- Lower and Upper Bounds for Experimental Function Values (Finalized)}
\begin{enumerate}
\item Set $\bar k := 0$. Choose $\Delta_r, \underline \Delta_r \in \mathbb{R}_{++}$ such that $\Delta_r > \underline \Delta_r$.
\item If $\bar k > k$, go to Step 6. Otherwise, proceed to Step 3.
\item Let $\overline N$ denote the number of iterates with the decision variable values ${\bf u}_{\bar k}$. Set $\underline \phi_p ({\bf u}_{\bar k},\tau_{\bar k}) := -\infty$ and $\overline \phi_p ({\bf u}_{\bar k},\tau_{\bar k}) := \infty$. 
\item For $N = 1,...,\overline N$:
\begin{enumerate}
\item Generate the $\left( \begin{array}{c} \overline N \\ N \end{array} \right)$ index sets that correspond to different combinations of iterations with the same decision variable values. For each set:
\begin{enumerate}
\item Define $\bar {\bf k}$ as the corresponding index set and compute the candidate lower and upper bounds as

\vspace{-2mm}
\begin{equation}\label{eq:costbound2gradLf}
\hspace{-4mm}\begin{array}{l}
\underline \phi_{p,test} := \displaystyle \frac{1}{N} \sum_{\tilde k \in \bar {\bf k}} \hat \phi_p ({\bf u}_{\tilde k},\tau_{\tilde k}) \vspace{1mm} \\
\displaystyle - \frac{1}{N} \sum_{\tilde k \in \bar {\bf k}}  \eta_{c,\phi}^{\bar k, \tilde k} \mathop {\max} \left[ \begin{array}{l} \displaystyle \frac{\partial \underline \phi_{p}}{\partial \tau} \Big |_{({\bf u}_{\bar k},\tau_{\bar k})} ( \tau_{\tilde k} - \tau_{{\bar k}} ) , \vspace{1mm} \\ \displaystyle \frac{\partial \overline \phi_{p}}{\partial \tau} \Big |_{({\bf u}_{\bar k},\tau_{\bar k})} ( \tau_{\tilde k} - \tau_{{\bar k}} ) \end{array} \right] \vspace{1mm}  \\
\displaystyle - \frac{1}{N} \sum_{\tilde k \in \bar {\bf k}} (1-  \eta_{c,\phi}^{\bar k, \tilde k}) \mathop {\max} \left[ \begin{array}{l} \underline \kappa_{\phi,\tau}^{\bar k, \tilde k} \left( \tau_{\tilde k} - \tau_{\bar k} \right), \vspace{1mm} \\ \overline \kappa_{\phi,\tau}^{\bar k, \tilde k} \left( \tau_{\tilde k} - \tau_{\bar k} \right) \end{array} \right] \vspace{1mm} \\
  - \overline W_{\phi, \bar k},
\end{array}
\end{equation}

\vspace{-2mm}
\begin{equation}\label{eq:costbound2gradUf}
\hspace{-4mm}\begin{array}{l}
\overline \phi_{p,test} := \displaystyle \frac{1}{N} \sum_{\tilde k \in \bar {\bf k}} \hat \phi_p ({\bf u}_{\tilde k},\tau_{\tilde k}) \vspace{1mm} \\
\displaystyle -  \frac{1}{N} \sum_{\tilde k \in \bar {\bf k}}  \eta_{v,\phi}^{\bar k, \tilde k} \mathop {\min} \left[ \begin{array}{l} \displaystyle \frac{\partial \underline \phi_{p}}{\partial \tau} \Big |_{({\bf u}_{\bar k},\tau_{\bar k})} ( \tau_{\tilde k} - \tau_{{\bar k}} ) , \vspace{1mm} \\ \displaystyle \frac{\partial \overline \phi_{p}}{\partial \tau} \Big |_{({\bf u}_{\bar k},\tau_{\bar k})} ( \tau_{\tilde k} - \tau_{{\bar k}} ) \end{array} \right] \vspace{1mm} \\
- \displaystyle \frac{1}{N} \sum_{\tilde k \in \bar {\bf k}} (1-  \eta_{v,\phi}^{\bar k, \tilde k}) \mathop {\min} \left[ \begin{array}{l} \underline \kappa_{\phi,\tau}^{\bar k, \tilde k} \left( \tau_{\tilde k} - \tau_{\bar k} \right), \vspace{1mm} \\ \overline \kappa_{\phi,\tau}^{\bar k, \tilde k} \left( \tau_{\tilde k} - \tau_{\bar k} \right) \end{array} \right] \vspace{1mm} \\  
- \underline W_{\phi, \bar k}.
\end{array}
\end{equation}

\item If $\underline \phi_{p,test} >  \underline \phi_p ({\bf u}_{\bar k},\tau_{\bar k})$, set $\underline \phi_p ({\bf u}_{\bar k},\tau_{\bar k}) := \underline \phi_{p,test}$. If $\overline \phi_{p,test} < \overline \phi_p ({\bf u}_{\bar k},\tau_{\bar k})$, set $\overline \phi_p ({\bf u}_{\bar k},\tau_{\bar k}) := \overline \phi_{p,test}$.
\end{enumerate}
\end{enumerate}
\item Set $\bar k := \bar k + 1$ and return to Step 2.
\item If $\Delta_r \leq \underline \Delta_r$, terminate. Otherwise, set $\Delta_r := 0, \; \bar k := 0$ and go to Step 7.
\item If $\bar k > k$, return to Step 6. Otherwise, proceed to Step 8.
\item For ${\tilde k} := \{ 0,...,k \} \setminus \bar k$:

\begin{enumerate}
\item Compute the lower and upper bound candidate values as

\vspace{-2mm}
\begin{equation}\label{eq:lipboundcostU2gradf}
\begin{array}{l}
\overline \phi_{p,test} := \overline \phi_{p} ({\bf u}_{\tilde k},\tau_{\tilde k}) \vspace{1mm} \\
\displaystyle \hspace{5mm} + \eta_{c,\phi}^{\bar k, \tilde k} \mathop {\max} \left[ \begin{array}{l} \displaystyle \frac{\partial \underline \phi_{p}}{\partial \tau} \Big |_{({\bf u}_{\tilde k},\tau_{\tilde k})} ( \tau_{\bar k} - \tau_{\tilde k} ) , \vspace{1mm} \\ \displaystyle \frac{\partial \overline \phi_{p}}{\partial \tau} \Big |_{({\bf u}_{\tilde k},\tau_{\tilde k})} ( \tau_{\bar k} - \tau_{\tilde k} ) \end{array} \right] \vspace{1mm} \\
\displaystyle  \hspace{5mm} + (1-\eta_{c,\phi}^{\bar k, \tilde k})  \mathop {\max} \left[ \begin{array}{l} \underline \kappa_{\phi,\tau}^{\bar k, \tilde k} ( \tau_{\bar k} - \tau_{\tilde k} ), \vspace{1mm} \\ \overline \kappa_{\phi,\tau}^{\bar k, \tilde k} ( \tau_{\bar k} - \tau_{\tilde k} ) \end{array} \right] \vspace{1mm} \\
 \displaystyle \hspace{5mm} + \sum_{i \in I_{c,\phi}^{\bar k, \tilde k}} \mathop {\max} \left[ \begin{array}{l} \displaystyle \frac{\partial \underline \phi_p}{\partial u_i} \Big |_{({\bf u}_{\tilde k},\tau_{\tilde k})} ( u_{\bar k,i} - u_{\tilde k,i} ), \vspace{1mm} \\ \displaystyle \frac{\partial \overline \phi_p}{\partial u_i} \Big |_{({\bf u}_{\tilde k},\tau_{\tilde k})} ( u_{\bar k,i} - u_{\tilde k,i} ) \end{array} \right] \vspace{1mm} \\
\displaystyle \hspace{5mm} + \sum_{i \not \in I_{c,\phi}^{\bar k, \tilde k}} \mathop {\max} \left[ \begin{array}{l} \underline \kappa_{\phi,i}^{\bar k, \tilde k} ( u_{\bar k,i} - u_{\tilde k,i} ), \vspace{1mm} \\ \overline \kappa_{\phi,i}^{\bar k, \tilde k} ( u_{\bar k,i} - u_{\tilde k,i} ) \end{array} \right],
\end{array}
\end{equation}

\vspace{-2mm}
\begin{equation}\label{eq:lipboundcostL2gradf}
\begin{array}{l}
\underline \phi_{p,test} := \underline \phi_{p} ({\bf u}_{\tilde k},\tau_{\tilde k}) \vspace{1mm} \\
\displaystyle \hspace{5mm} + \eta_{v,\phi}^{\bar k, \tilde k} \mathop {\min} \left[ \begin{array}{l} \displaystyle \frac{\partial \underline \phi_{p}}{\partial \tau} \Big |_{({\bf u}_{\tilde k},\tau_{\tilde k})} ( \tau_{\bar k} - \tau_{\tilde k} ) , \vspace{1mm} \\ \displaystyle \frac{\partial \overline \phi_{p}}{\partial \tau} \Big |_{({\bf u}_{\tilde k},\tau_{\tilde k})} ( \tau_{\bar k} - \tau_{\tilde k} ) \end{array} \right] \vspace{1mm} \\
\displaystyle  \hspace{5mm} + (1-\eta_{v,\phi}^{\bar k, \tilde k})  \mathop {\min} \left[ \begin{array}{l} \underline \kappa_{\phi,\tau}^{\bar k, \tilde k} ( \tau_{\bar k} - \tau_{\tilde k} ), \vspace{1mm} \\ \overline \kappa_{\phi,\tau}^{\bar k, \tilde k} ( \tau_{\bar k} - \tau_{\tilde k} ) \end{array} \right] \vspace{1mm}  \\
 \displaystyle \hspace{5mm} + \sum_{i \in I_{v,\phi}^{\bar k, \tilde k}} \mathop {\min} \left[ \begin{array}{l} \displaystyle \frac{\partial \underline \phi_p}{\partial u_i} \Big |_{({\bf u}_{\tilde k},\tau_{\tilde k})} ( u_{\bar k,i} - u_{\tilde k,i} ), \vspace{1mm} \\ \displaystyle \frac{\partial \overline \phi_p}{\partial u_i} \Big |_{({\bf u}_{\tilde k},\tau_{\tilde k})} ( u_{\bar k,i} - u_{\tilde k,i} ) \end{array} \right] \vspace{1mm} \\
\displaystyle \hspace{5mm} + \sum_{i \not \in I_{v,\phi}^{\bar k, \tilde k}} \mathop {\min} \left[ \begin{array}{l} \underline \kappa_{\phi,i}^{\bar k, \tilde k} ( u_{\bar k,i} - u_{\tilde k,i} ), \vspace{1mm} \\ \overline \kappa_{\phi,i}^{\bar k, \tilde k} ( u_{\bar k,i} - u_{\tilde k,i} ) \end{array} \right].
\end{array}
\end{equation} 

\item Set 

\vspace{-2mm}
\begin{equation}\label{eq:Deltar}
\Delta_r := \mathop {\max} \left[  \begin{array}{l} \Delta_r,  \underline \phi_{p,test} - \underline \phi_p ({\bf u}_{\bar k},\tau_{\bar k}), \vspace{1mm} \\ \overline \phi_p ({\bf u}_{\bar k},\tau_{\bar k}) - \overline \phi_{p,test} \end{array} \right].
\end{equation}

\item If $\underline \phi_{p,test} > \underline \phi_p ({\bf u}_{\bar k},\tau_{\bar k})$, set $\underline \phi_p ({\bf u}_{\bar k},\tau_{\bar k}) := \underline \phi_{p,test}$. If $\overline \phi_{p,test} < \overline \phi_p ({\bf u}_{\bar k},\tau_{\bar k})$, set $\overline \phi_p ({\bf u}_{\bar k},\tau_{\bar k}) := \overline \phi_{p,test}$.

\end{enumerate}

\item Set $\bar k := \bar k+1$ and return to Step 7.

\end{enumerate}

The same algorithm is then carried out for the experimental constraint functions $g_{p,j}$.

It is important to note the slight discrepancy between Algorithms 1F and 2F, in that the latter uses local Lipschitz constants while the former uses global ones. To have compatability between the two, it would be necessary either for both to use local or for both to use global. How this compatability is established is left up to the user. For example, one may simply choose to work with global constants throughout since local relaxations may not be reliable in many cases, in which case one would use, e.g., $\underline \kappa_{\phi,i}^{\bar k, \tilde k} = \underline \kappa_{\phi,i}$ in Algorithm 2F. Alternatively, one may modify the consistency check of Algorithm 1F to be local in nature, although this may be a bit complex and so we do not treat the modification here. Finally, one may dispense with compatability and simply use the global constants in Algorithm 1F and the local ones in Algorithm 2F -- in this case, no sort of consistency check would be carried out for the local constants, which would most likely be obtained from a model.

\subsection{Choosing the Reference Point}

With consistent Lipschitz constants obtained and the experimental function values bounded, we now proceed to choose the point to use as the reference for optimization. In the case where the cost function is experimental in nature, this is taken as the decision variable vector at the most recent iterate that is robustly guaranteed to satisfy the problem constraints with some back-off to allow for excitation if needed:

\vspace{-2mm}
\begin{equation}\label{eq:kstarLUccvcostlocMNgradBOslackf}
\begin{array}{rl}
k^* := \;\;\;\;\;\;\;\;\;\;\;\;\;\;& \\
{\rm arg} \mathop {\rm maximize}\limits_{\bar k \in [0,k]} & \bar k  \vspace{1mm} \\
{\rm{subject}}\;{\rm{to}} & \overline g_{p,j} ({\bf u}_{\bar k},\tau_{\bar k}) \vspace{1mm} \\
& \displaystyle + \displaystyle \eta_{c,j}^{e,\bar k} \frac{\partial \overline g_{p,j}}{\partial \tau} \Big |_{({\bf u}_{\bar k},\tau_{\bar k})} ( \tau_{k+1} - \tau_{{\bar k}} ) \vspace{1mm} \\
&  + (1-\eta_{c,j}^{e,\bar k}) \overline \kappa_{p,j\tau}^{e,\bar k} \left( \tau_{k+1} - \tau_{\bar k} \right) \vspace{1mm} \\
& + \delta_e \| \kappa_{p,j}^{m,\bar k} \|_2 \leq d_{p,j}^{k+1}, \;\; \forall j = 1,...,n_{g_p} \vspace{1mm} \\
& \mathop {\max} \limits_{{\bf u} \in \mathcal{B}_{e,\bar k}} g_j ({\bf u}) \leq d_j^{k+1}, \;\; \forall j = 1,...,n_{g} \vspace{1mm} \\
& {\bf u}^L + \delta_e {\bf 1}  \preceq {\bf u}_{\bar k} \preceq {\bf u}^U - \delta_e {\bf 1} \vspace{1mm} \\
& \displaystyle \underline \phi_p ({\bf u}_{\bar k},\tau_{\bar k}) +  \eta_{v,\phi}^{\bar k, k} \frac{\partial \underline \phi_{p}}{\partial \tau} \Big |_{({\bf u}_{\bar k},\tau_{\bar k})} ( \tau_{k} - \tau_{{\bar k}} ) \vspace{1mm} \\
&+ (1-  \eta_{v,\phi}^{\bar k, k}) \underline \kappa_{\phi,\tau}^{\bar k, k} \left( \tau_{k} - \tau_{\bar k} \right) \leq \vspace{1mm} \\
& \mathop {\min} \limits_{\tilde k \in {\bf k}_f} \left[ \begin{array}{l} \overline \phi_p ({\bf u}_{\tilde k},\tau_{\tilde k}) + \vspace{1mm} \\
\displaystyle  \eta_{c,\phi}^{\tilde k, k} \frac{\partial \overline \phi_{p}}{\partial \tau} \Big |_{({\bf u}_{\tilde k},\tau_{\tilde k})} ( \tau_{k} - \tau_{{\tilde k}} ) \vspace{1mm} \\
+ (1-  \eta_{c,\phi}^{\tilde k, k}) \overline \kappa_{\phi,\tau}^{\tilde k, k} \left( \tau_{k} - \tau_{\tilde k} \right) \end{array} \right],
\end{array}
\end{equation}

\vspace{-2mm}
\begin{equation}\label{eq:kfeas6f}
{\bf k}_f = \left\{ \bar k : 
\begin{array}{l}
\overline g_{p,j} ({\bf u}_{\bar k},\tau_{\bar k}) \vspace{1mm} \\ 
+ \displaystyle \eta_{c,j}^{e,\bar k} \frac{\partial \overline g_{p,j}}{\partial \tau} \Big |_{({\bf u}_{\bar k},\tau_{\bar k})} ( \tau_{k+1} - \tau_{{\bar k}} ) \vspace{1mm} \\ 
+ (1-\eta_{c,j}^{e,\bar k}) \overline \kappa_{p,j\tau}^{e,\bar k} \left( \tau_{k+1} - \tau_{\bar k} \right) \vspace{1mm} \\ 
+ \delta_e \| \kappa_{p,j}^{m,\bar k} \|_2  \leq d_{p,j}^{k+1}, \;\;\forall j = 1,...,n_{g_p}; \vspace{1mm} \\
 \mathop {\max} \limits_{{\bf u} \in \mathcal{B}_{e,\bar k}} g_j ({\bf u}) \leq d_j^{k+1}, \;\; \forall j = 1,...,n_{g}; \vspace{1mm} \\
{\bf u}^L + \delta_e {\bf 1}  \preceq {\bf u}_{\bar k} \preceq {\bf u}^U - \delta_e {\bf 1}
\end{array} \right\}.
\end{equation}

If no point satisfying these restrictions can be found, then one should either cease running experiments, as feasibility cannot be guaranteed, or choose a point that violates the constraints the least in the worst case:

\vspace{-2mm}
\begin{equation}\label{eq:kstar2LUccvlocMNgradBOf}
k^*  := {\rm arg} \mathop {\rm minimize}\limits_{\bar k \in [0,k]}  \mathop {\max} \left[ g_{p,m}^{\bar k}, g_{m}^{\bar k}, u_{L,m}^{\bar k}, u_{U,m}^{\bar k}  \right],
\end{equation}

\noindent where

\vspace{-2mm}
\begin{equation}\label{eq:violsf}
\begin{array}{l}
g_{p,m}^{\bar k} = \mathop {\max} \limits_{j = 1,...,n_{g_p}} \left[ \begin{array}{l}
\overline g_{p,j} ({\bf u}_{\bar k},\tau_{\bar k}) \vspace{1mm} \\ 
+ \displaystyle \eta_{c,j}^{e,\bar k} \frac{\partial \overline g_{p,j}}{\partial \tau} \Big |_{({\bf u}_{\bar k},\tau_{\bar k})} ( \tau_{k+1} - \tau_{{\bar k}} ) \vspace{1mm} \\ 
+ (1-\eta_{c,j}^{e,\bar k}) \overline \kappa_{p,j\tau}^{e,\bar k} \left( \tau_{k+1} - \tau_{\bar k} \right) \vspace{1mm} \\ 
+ \delta_e \| \kappa_{p,j}^{m,\bar k} \|_2 - d_{p,j}^{k+1} \end{array} \right] \vspace{1mm} \\
g_{m}^{\bar k} = \mathop {\max} \limits_{j = 1,...,n_g} \mathop {\max} \limits_{{\bf u} \in \mathcal{B}_{e,\bar k}} \left( g_j ({\bf u}) - d_{j}^{k+1} \right) \vspace{1mm} \\
u_{L,m}^{\bar k} = \mathop {\max} \limits_{i = 1,...,n_u} \left( u^L_i + \delta_e - u_{\bar k,i} \right) \vspace{1mm} \\
u_{U,m}^{\bar k} = \mathop {\max} \limits_{i = 1,...,n_u} \left( u_{\bar k,i} + \delta_e - u^U_i  \right).
\end{array}
\end{equation}

If the cost function is numerical, then one may choose $k^*$ as

\vspace{-2mm}
\begin{equation}\label{eq:kstarLUccvcostlocMNgradBOslacknumf}
\begin{array}{rl}
k^* := \;\;\;\;\;\;\;\;\;\;\;\;\;\;& \vspace{1mm} \\
{\rm arg} \mathop {\rm minimize}\limits_{\bar k \in [0,k]} & \phi({\bf u}_{\bar k}) \vspace{1mm}  \\
{\rm{subject}}\;{\rm{to}} & \overline g_{p,j} ({\bf u}_{\bar k},\tau_{\bar k}) \vspace{1mm} \\
& \displaystyle + \displaystyle \eta_{c,j}^{e,\bar k} \frac{\partial \overline g_{p,j}}{\partial \tau} \Big |_{({\bf u}_{\bar k},\tau_{\bar k})} ( \tau_{k+1} - \tau_{{\bar k}} ) \vspace{1mm} \\
&  + (1-\eta_{c,j}^{e,\bar k}) \overline \kappa_{p,j\tau}^{e,\bar k} \left( \tau_{k+1} - \tau_{\bar k} \right) \vspace{1mm} \\
& + \delta_e \| \kappa_{p,j}^{m,\bar k} \|_2 \leq d_{p,j}^{k+1} , \;\; \forall j = 1,...,n_{g_p} \vspace{1mm} \\
& \mathop {\max} \limits_{{\bf u} \in \mathcal{B}_{e,\bar k}} g_j ({\bf u}) \leq d_j^{k+1}, \;\; \forall j = 1,...,n_{g} \vspace{1mm} \\
& {\bf u}^L + \delta_e {\bf 1}  \preceq {\bf u}_{\bar k} \preceq {\bf u}^U - \delta_e {\bf 1},
\end{array}
\end{equation}

\noindent and resort to (\ref{eq:kstar2LUccvlocMNgradBOf}) if no point satisfying the restrictions can be found.

\subsection{Projection of Optimal Target}

Having chosen a reference ${\bf u}_{k^*}$ and with some optimal target ${\bf u}_{k+1}^*$ supplied by an external optimization method, we now project this point to satisfy the SCFO in some robust sense. To this end, let us state the following finalized version of Algorithm 3.

\vspace{2mm}
\noindent {\bf{Algorithm 3F -- Projection with Automatic Choice of Projection Parameters and Gradient Robustness (Finalized)}}
\vspace{2mm}

\begin{enumerate}
\item Set ${\boldsymbol \epsilon}_p := \overline {\boldsymbol \epsilon}_p$, ${\boldsymbol \epsilon} := \overline {\boldsymbol \epsilon}$, $\boldsymbol{\delta}_{g_p} := \boldsymbol{\overline \delta}_{g_p}$, $\boldsymbol{\delta}_{g} := \boldsymbol{\overline \delta}_{g}$, and $\delta_\phi := \overline \delta_\phi$, where $\overline \epsilon_{p,j} := \overline \delta_{g_p,j} \approx - \mathop {\min} \limits_{({\bf u},\tau) \in \mathcal{I}_\tau} g_{p,j} ({\bf u},\tau)$, $\overline \epsilon_{j} := \overline \delta_{g,j} \approx - \mathop {\min} \limits_{{\bf u} \in \mathcal{I}} g_{j} ({\bf u})$, and $\overline \delta_\phi \approx \phi_p ({\bf u}_0,\tau_0) - \mathop {\min} \limits_{({\bf u},\tau) \in \mathcal{I}_\tau} \phi_p ({\bf u},\tau)$. Set $\underline P := 0, \overline P := 1$.
\item Check the feasibility of

\vspace{-3mm}
\begin{equation}\label{eq:projdeg2robslackBO2slackf}
\begin{array}{rl}
\mathop {\rm minimize}\limits_{{\bf u}} & \| {\bf u} - {\bf u}_{k+1}^* \|_2^2 \vspace{1mm}  \\
 {\rm{subject}}\;{\rm{to}} &  \nabla \hat g_{p,j} ({\bf u}_{k^*}, \tau_{k+1})^T \left[ \hspace{-1mm} \begin{array}{c} {\bf u} - {\bf u}_{k^*} \\ 0 \end{array} \hspace{-1mm} \right] \leq -\delta_{g_p,j} , \vspace{1mm} \\
 & \hspace{-6mm} \forall j: \begin{array}{l} \overline g_{p,j} ({\bf u}_{k^*},\tau_{k^*}) \vspace{1mm}  \\
 +  \displaystyle \eta_{c,j}^{e,k^*} \frac{\partial \overline g_{p,j}}{\partial \tau} \Big |_{({\bf u}_{k^*},\tau_{k^*})} ( \tau_{k+1} - \tau_{{ k^*}} ) \vspace{1mm} \\ 
 + (1-\eta_{c,j}^{e,k^*}) \overline \kappa_{p,j\tau}^{e,k^*} \left( \tau_{k+1} - \tau_{k^*} \right) \vspace{1mm} \\ 
 + \delta_e \| \kappa_{p,j}^{m,k^*} \|_2 \geq -\epsilon_{p,j}+d_{p,j}^{k+1} \end{array} \vspace{1mm} \\
 & \nabla g_{j} ({\bf u}_{k^*})^T ({\bf u} - {\bf u}_{k^*}) \leq -\delta_{g,j}, \vspace{1mm} \\
& \forall j : \mathop {\max} \limits_{{\bf u} \in \mathcal{B}_{e,k^*}} g_{j}({\bf u}) \geq -\epsilon_{j}+d_{j}^{k+1} \vspace{1mm} \\
& \nabla \hat \phi_{p} ({\bf u}_{k^*}, \tau_{k+1})^T \left[ \hspace{-1mm} \begin{array}{c} {\bf u} - {\bf u}_{k^*} \\ 0 \end{array} \hspace{-1mm} \right] \leq -\delta_{\phi} , \vspace{1mm} \\
 & {\bf u}^L + \delta_e {\bf 1} \preceq {\bf u} \preceq {\bf u}^U - \delta_e {\bf 1}
\end{array}
\end{equation}

\noindent for the given choice of ${\boldsymbol \epsilon}_{p}, {\boldsymbol \epsilon}, {\boldsymbol \delta}_{g_p}, {\boldsymbol \delta}_{g}, \delta_\phi$ by solving a linear programming feasibility problem if the cost function is experimental. If the cost function is numerical, check the feasibility of

\vspace{-3mm}
\begin{equation}\label{eq:projdeg2robslackBO2slackfnum}
\begin{array}{rl}
\mathop {\rm minimize}\limits_{{\bf u}} & \| {\bf u} - {\bf u}_{k+1}^* \|_2^2 \vspace{1mm}  \\
 {\rm{subject}}\;{\rm{to}} & \nabla \hat g_{p,j} ({\bf u}_{k^*}, \tau_{k+1})^T \left[ \hspace{-1mm} \begin{array}{c} {\bf u} - {\bf u}_{k^*} \\ 0 \end{array} \hspace{-1mm} \right] \leq -\delta_{g_p,j} , \vspace{1mm} \\
 & \hspace{-6mm} \forall j: \begin{array}{l} \overline g_{p,j} ({\bf u}_{k^*},\tau_{k^*}) \vspace{1mm}  \\
 +  \displaystyle \eta_{c,j}^{e,k^*} \frac{\partial \overline g_{p,j}}{\partial \tau} \Big |_{({\bf u}_{k^*},\tau_{k^*})} ( \tau_{k+1} - \tau_{{ k^*}} ) \vspace{1mm} \\ 
 + (1-\eta_{c,j}^{e,k^*}) \overline \kappa_{p,j\tau}^{e,k^*} \left( \tau_{k+1} - \tau_{k^*} \right) \vspace{1mm} \\ 
 + \delta_e \| \kappa_{p,j}^{m,k^*} \|_2 \geq -\epsilon_{p,j}+d_{p,j}^{k+1} \end{array} 
\end{array}
\end{equation}

$$
\begin{array}{rl}
 & \nabla g_{j} ({\bf u}_{k^*})^T ({\bf u} - {\bf u}_{k^*}) \leq -\delta_{g,j}, \vspace{1mm} \\
& \forall j : \mathop {\max} \limits_{{\bf u} \in \mathcal{B}_{e,k^*}} g_{j}({\bf u}) \geq -\epsilon_{j}+d_{j}^{k+1} \vspace{1mm} \\
& \nabla \phi ({\bf u}_{k^*})^T ({\bf u} - {\bf u}_{k^*}) \leq -\delta_{\phi} , \vspace{1mm} \\
 & {\bf u}^L + \delta_e {\bf 1} \preceq {\bf u} \preceq {\bf u}^U - \delta_e {\bf 1}
\end{array}
$$

\noindent instead. If no solution exists and $\delta_\phi \geq \overline \delta_\phi / 2^{10}$, set ${\boldsymbol \epsilon}_p := {\boldsymbol \epsilon}_p/2$, ${\boldsymbol \epsilon} := {\boldsymbol \epsilon}/2$, $\boldsymbol{\delta}_{g_p} := \boldsymbol{\delta}_{g_p}/2$, $\boldsymbol{\delta}_{g} := \boldsymbol{\delta}_{g}/2$, $\delta_\phi := \delta_\phi/2$, and repeat this step. Otherwise, proceed to Step 3.

\item If $\delta_\phi < \overline \delta_\phi / 2^{10}$, terminate with $\bar {\bf u}_{k+1}^* := {\bf u}_{k^*}$. Otherwise, set $P := 0.5 \underline P + 0.5 \overline P$ and define the tightened bounds as

\vspace{-2mm}
\begin{equation}\label{eq:boundtightf}
\begin{array}{l}
\nabla \underline \phi_p^t ({\bf u}_{k^*},\tau_{k+1}) := \nabla \hat \phi_p ({\bf u}_{k^*},\tau_{k+1}) \vspace{1mm} \\
\hspace{5mm} + P \left[ \nabla \underline \phi_p ({\bf u}_{k^*},\tau_{k+1}) - \nabla \hat \phi_p ({\bf u}_{k^*},\tau_{k+1}) \right], \vspace{1mm} \\
\nabla \overline \phi_p^t ({\bf u}_{k^*},\tau_{k+1}) := \nabla \hat \phi_p ({\bf u}_{k^*},\tau_{k+1}) \vspace{1mm} \\
\hspace{5mm} + P \left[ \nabla \overline \phi_p ({\bf u}_{k^*},\tau_{k+1}) - \nabla \hat \phi_p ({\bf u}_{k^*},\tau_{k+1}) \right], \vspace{1mm} \\
\nabla \underline g_{p,j}^t ({\bf u}_{k^*},\tau_{k+1}) := \nabla \hat g_{p,j} ({\bf u}_{k^*},\tau_{k+1}) \vspace{1mm} \\
\hspace{5mm} + P \left[ \nabla \underline g_{p,j} ({\bf u}_{k^*},\tau_{k+1}) - \nabla \hat g_{p,j} ({\bf u}_{k^*},\tau_{k+1}) \right],  \vspace{1mm} \\
\nabla \overline g_{p,j}^t ({\bf u}_{k^*},\tau_{k+1}) := \nabla \hat g_{p,j} ({\bf u}_{k^*},\tau_{k+1}) \vspace{1mm} \\
\hspace{5mm} + P \left[ \nabla \overline g_{p,j} ({\bf u}_{k^*},\tau_{k+1}) - \nabla \hat g_{p,j} ({\bf u}_{k^*},\tau_{k+1}) \right].
\end{array}
\end{equation}

\item Check the feasibility of the robust projection with the tightened bounds for the obtained ${\boldsymbol \epsilon}_{p}, {\boldsymbol \epsilon}, {\boldsymbol \delta}_{g_p}, {\boldsymbol \delta}_{g}, \delta_\phi$:

\vspace{-2mm}
\begin{equation}\label{eq:projdeg2robslackBO2slackff}
\begin{array}{rl}
\bar {\bf u}_{k+1}^* := \;\;\;\;\;\;\;\;\; & \vspace{1mm}\\
 {\rm arg} \mathop {\rm minimize}\limits_{{\bf u},{\bf s}_\phi, {\bf S}} & \| {\bf u} - {\bf u}_{k+1}^* \|_2^2 \vspace{1mm}  \\
 {\rm{subject}}\;{\rm{to}} & \displaystyle \sum_{i=1}^{n_u} s_{ji}  \leq -\delta_{g_p,j} \vspace{1mm} \\
& \displaystyle \frac{\partial \underline g_{p,j}^t}{\partial u_i} \Big |_{({\bf u}_{k^*}, \tau_{k+1})} (u_i - u_{k^*,i}) \leq s_{ji} \vspace{1mm} \\
& \displaystyle \frac{\partial \overline g_{p,j}^t}{\partial u_i} \Big |_{({\bf u}_{k^*}, \tau_{k+1})} (u_i - u_{k^*,i}) \leq s_{ji}, \vspace{1mm} \\
& \forall i = 1,...,n_u, \vspace{1mm} \\
&  \hspace{-6mm} \forall j: \begin{array}{l} \overline g_{p,j} ({\bf u}_{k^*},\tau_{k^*}) \vspace{1mm} \\
 +  \displaystyle \eta_{c,j}^{e,k^*} \frac{\partial \overline g_{p,j}}{\partial \tau} \Big |_{({\bf u}_{k^*},\tau_{k^*})} ( \tau_{k+1} - \tau_{{ k^*}} ) \vspace{1mm} \\ 
 + (1-\eta_{c,j}^{e,k^*}) \overline \kappa_{p,j\tau}^{e,k^*} \left( \tau_{k+1} - \tau_{k^*} \right) \vspace{1mm} \\ 
 + \delta_e \| \kappa_{p,j}^{m,k^*} \|_2 \geq -\epsilon_{p,j}+d_{p,j}^{k+1} \end{array} \vspace{1mm} \\
 & \nabla g_{j} ({\bf u}_{k^*})^T ({\bf u} - {\bf u}_{k^*}) \leq -\delta_{g,j}, \vspace{1mm} \\
& \forall j : \mathop {\max} \limits_{{\bf u} \in \mathcal{B}_{e,k^*}} g_{j}({\bf u}) \geq -\epsilon_{j}+d_{j}^{k+1}
\end{array}
\end{equation}

$$
\begin{array}{rl}
 & \displaystyle \sum_{i=1}^{n_u} s_{\phi,i}  \leq -\delta_{\phi} \vspace{1mm} \\
& \displaystyle \frac{\partial \underline \phi_{p}^t}{\partial u_i} \Big |_{({\bf u}_{k^*}, \tau_{k+1})} (u_i - u_{k^*,i}) \leq s_{\phi,i} \vspace{1mm} \\
& \displaystyle \frac{\partial \overline \phi_{p}^t}{\partial u_i} \Big |_{({\bf u}_{k^*}, \tau_{k+1})} (u_i - u_{k^*,i}) \leq s_{\phi,i} \vspace{1mm} \\
 & {\bf u}^L + \delta_e {\bf 1} \preceq {\bf u} \preceq {\bf u}^U - \delta_e {\bf 1}, 
\end{array}
$$

\noindent or, if the cost is numerical, check the feasibility of

\vspace{-2mm}
\begin{equation}\label{eq:projdeg2robslackBO2slackffnum}
\begin{array}{rl}
\bar {\bf u}_{k+1}^* := \;\;\;\;\;\;\;\;\; & \vspace{1mm}\\
 {\rm arg} \mathop {\rm minimize}\limits_{{\bf u},{\bf S}} & \| {\bf u} - {\bf u}_{k+1}^* \|_2^2 \vspace{1mm}  \\
 {\rm{subject}}\;{\rm{to}} & \displaystyle \sum_{i=1}^{n_u} s_{ji}  \leq -\delta_{g_p,j} \vspace{1mm} \\
& \displaystyle \frac{\partial \underline g_{p,j}^t}{\partial u_i} \Big |_{({\bf u}_{k^*}, \tau_{k+1})} (u_i - u_{k^*,i}) \leq s_{ji} \vspace{1mm} \\
& \displaystyle \frac{\partial \overline g_{p,j}^t}{\partial u_i} \Big |_{({\bf u}_{k^*}, \tau_{k+1})} (u_i - u_{k^*,i}) \leq s_{ji}, \vspace{1mm} \\
& \forall i = 1,...,n_u, \vspace{1mm} \\
&  \hspace{-6mm} \forall j: \begin{array}{l} \overline g_{p,j} ({\bf u}_{k^*},\tau_{k^*}) \vspace{1mm} \\
 +  \displaystyle \eta_{c,j}^{e,k^*} \frac{\partial \overline g_{p,j}}{\partial \tau} \Big |_{({\bf u}_{k^*},\tau_{k^*})} ( \tau_{k+1} - \tau_{{ k^*}} ) \vspace{1mm} \\ 
 + (1-\eta_{c,j}^{e,k^*}) \overline \kappa_{p,j\tau}^{e,k^*} \left( \tau_{k+1} - \tau_{k^*} \right) \vspace{1mm} \\ 
 + \delta_e \| \kappa_{p,j}^{m,k^*} \|_2 \geq -\epsilon_{p,j}+d_{p,j}^{k+1} \end{array} \vspace{1mm} \\
 & \nabla g_{j} ({\bf u}_{k^*})^T ({\bf u} - {\bf u}_{k^*}) \leq -\delta_{g,j}, \vspace{1mm} \\
& \forall j : \mathop {\max} \limits_{{\bf u} \in \mathcal{B}_{e,k^*}} g_{j}({\bf u}) \geq -\epsilon_{j}+d_{j}^{k+1} \vspace{1mm} \\
 & \nabla \phi ({\bf u}_{k^*})^T ({\bf u} - {\bf u}_{k^*})  \leq -\delta_{\phi} \vspace{1mm} \\
 & {\bf u}^L + \delta_e {\bf 1} \preceq {\bf u} \preceq {\bf u}^U - \delta_e {\bf 1}. 
\end{array}
\end{equation}

\noindent If (\ref{eq:projdeg2robslackBO2slackff}) -- or, in the case of the numerical cost, (\ref{eq:projdeg2robslackBO2slackffnum}) -- is infeasible, set $\overline P := P$. Otherwise, set $\underline P := P$. If $\overline P - \underline P < 0.01$, proceed to Step 5. Otherwise, return to Step 3. 

\item Set $P := 0.5\underline P$, define the tightened bounds as in (\ref{eq:boundtightf}), and solve (\ref{eq:projdeg2robslackBO2slackff}) -- or, in the case of a numerical cost, (\ref{eq:projdeg2robslackBO2slackffnum}) -- to obtain the projected target $\bar {\bf u}_{k+1}^*$. Terminate.

\end{enumerate}
 
\subsection{The Line Search in $K_k$}

Having computed $\bar {\bf u}_{k+1}^*$, we now define the next experimental decision variable set ${\bf u}_{k+1}$ via the filter law

\vspace{-2mm}
\begin{equation}\label{eq:inputfilterf}
{\bf u}_{k+1} := {\bf u}_{k^*} + K_k (\bar {\bf u}_{k+1}^* - {\bf u}_{k^*}),
\end{equation}

\noindent with $K_k \in [0,1]$ kept sufficiently small so as to ensure the desired feasibility and cost decrease properties.

In the case of an experimental cost function, we may write the appropriate $K_k$ as the solution to the following univariate line search problem:

\vspace{-2mm}
\begin{equation}\label{eq:linesearchf}
\begin{array}{rl}
K_k := {\rm arg} \mathop {\rm maximize}\limits_{K \in [0,1]} & K  \\
{\rm{subject}}\;{\rm{to}} & \\
& \hspace{-30mm} \mathop {\min} \limits_{\bar k = 0,...,k} \left[ \hspace{-1mm} \begin{array}{l} \overline g_{p,j} ({\bf u}_{\bar k},\tau_{\bar k}) \displaystyle +\eta_{c,j}^{\bar k} \frac{\partial \overline g_{p,j}}{\partial \tau} \Big |_{({\bf u}_{\bar k},\tau_{\bar k})} ( \tau_{k+1} - \tau_{{\bar k}} ) \vspace{1mm} \\
 \displaystyle + (1-\eta_{c,j}^{\bar k}) \overline \kappa_{p,j\tau}^{\bar k} \left( \tau_{k+1} - \tau_{\bar k} \right) \vspace{1mm}\\
\displaystyle   + \sum_{i \in I_{c,j}^{\bar k}} \mathop {\max} \left[ \begin{array}{l} \displaystyle \frac{\partial \underline g_{p,j}}{\partial u_i} \Big |_{({\bf u}_{\bar k},\tau_{\bar k})} ( u_{k^*,i} + \\ \hspace{2mm} K (\bar u_{k+1,i}^* - u_{k^*,i} ) - u_{\bar k,i} ), \vspace{1mm}\\ \displaystyle \frac{\partial \overline g_{p,j}}{\partial u_i} \Big |_{({\bf u}_{\bar k},\tau_{\bar k})} ( u_{k^*,i} + \\ \hspace{2mm} K (\bar u_{k+1,i}^* - u_{k^*,i} ) - u_{\bar k,i} ) \end{array} \right] \vspace{1mm} \\
 + \displaystyle \sum_{i \not \in I_{c,j}^{\bar k}} \mathop {\max} \left[ \begin{array}{l} \underline \kappa_{p,ji}^{\bar k} ( u_{k^*,i} + \\ \hspace{2mm} K (\bar u_{k+1,i}^* - u_{k^*,i} ) - u_{\bar k,i} ), \vspace{1mm}\\ \overline \kappa_{p,ji}^{\bar k} ( u_{k^*,i} + \\ \hspace{2mm} K (\bar u_{k+1,i}^* - u_{k^*,i} ) - u_{\bar k,i} ) \end{array} \right] \end{array} \hspace{-1mm} \right] \vspace{1mm} \\
& \hspace{10mm} \leq d_{p,j}^{k+1}, \; \forall j = 1,...,n_{g_p} \vspace{1mm} \\
& \hspace{-30mm} \mathop {\min} \limits_{\bar k = 0,...,k} \left[ \hspace{-1mm} \begin{array}{l} \overline g_{p,j} ({\bf u}_{\bar k},\tau_{\bar k}) \displaystyle +\eta_{c,j}^{\bar k} \frac{\partial \overline g_{p,j}}{\partial \tau} \Big |_{({\bf u}_{\bar k},\tau_{\bar k})} ( \tau_{k+1} - \tau_{{\bar k}} ) \vspace{1mm} \\
 \displaystyle + (1-\eta_{c,j}^{\bar k}) \overline \kappa_{p,j\tau}^{\bar k} \left( \tau_{k+1} - \tau_{\bar k} \right) \vspace{1mm}\\
\displaystyle   + \sum_{i \in I_{c,j}^{\bar k}} \mathop {\max} \left[ \begin{array}{l} \displaystyle \frac{\partial \underline g_{p,j}}{\partial u_i} \Big |_{({\bf u}_{\bar k},\tau_{\bar k})} ( u_{k^*,i} + \\ \hspace{2mm} K (\bar u_{k+1,i}^* - u_{k^*,i} ) - u_{\bar k,i} ), \vspace{1mm}\\ \displaystyle \frac{\partial \overline g_{p,j}}{\partial u_i} \Big |_{({\bf u}_{\bar k},\tau_{\bar k})} ( u_{k^*,i} + \\ \hspace{2mm} K (\bar u_{k+1,i}^* - u_{k^*,i} ) - u_{\bar k,i} ) \end{array} \right] \vspace{1mm} \\
 + \displaystyle \sum_{i \not \in I_{c,j}^{\bar k}} \mathop {\max} \left[ \begin{array}{l} \underline \kappa_{p,ji}^{\bar k} ( u_{k^*,i} + \\ \hspace{2mm} K (\bar u_{k+1,i}^* - u_{k^*,i} ) - u_{\bar k,i} ), \vspace{1mm}\\ \overline \kappa_{p,ji}^{\bar k} ( u_{k^*,i} + \\ \hspace{2mm} K (\bar u_{k+1,i}^* - u_{k^*,i} ) - u_{\bar k,i} ) \end{array} \right] \end{array} \hspace{-1mm} \right] \vspace{1mm} \\
& \hspace{-30mm} + \displaystyle \eta_{c,j}^{K} \frac{\partial \overline g_{p,j}}{\partial \tau} \Big |_{\left({\bf u}_{k^*} + K ( \bar {\bf u}_{k+1}^* - {\bf u}_{k^*} ),\tau_{k+1}\right)} ( \tau_{k+2} - \tau_{{k+1}} ) \vspace{1mm} \\
& \hspace{-30mm} + (1-\eta_{c,j}^{K}) \overline \kappa_{p,j\tau}^{K} \left( \tau_{k+2} - \tau_{k+1} \right) + \delta_e \| \kappa_{p,j}^{m,K} \|_2  \leq d_{p,j}^{k+1}, \vspace{1mm} \\
& \hspace{20mm} \forall j = 1,...,n_{g_p} \vspace{-.5mm} \\
& \hspace{-30mm} \mathop {\max} \limits_{{\bf u} \in \mathcal{B}_{K}} g_{j}({\bf u}) \leq d_j^{k+1}, \;\; \forall j = 1,...,n_g \vspace{0mm} \\
& \hspace{-30mm} \displaystyle \sum_{i=1}^{n_u} \mathop {\max} \left[ \begin{array}{l}\displaystyle \frac{\partial \underline \phi_p^t}{\partial u_i} \Big |_{({\bf u}_{k^*},\tau_{k+1})} (\bar u_{k+1,i}^* - u_{k^*,i}), \vspace{1mm} \\ \displaystyle \frac{\partial \overline \phi_p^t}{\partial u_i} \Big |_{({\bf u}_{k^*},\tau_{k+1})} (\bar u_{k+1,i}^* - u_{k^*,i}) \end{array} \right] \vspace{.5mm} \\
& \hspace{-32mm} \displaystyle + \frac{K}{2} \sum_{i_1=1}^{n_u} \sum_{i_2=1}^{n_u} \mathop {\max} \left[ \begin{array}{l} \underline M_{\phi,i_1 i_2}^{k^*} (\bar u_{k+1,i_1}^* - u_{{k^*},i_1}) \vspace{0mm} \\
\hspace{16mm} (\bar u_{k+1,i_2}^* - u_{{k^*},i_2}), \vspace{1mm} \\ \overline M_{\phi,i_1 i_2}^{k^*} (\bar u_{k+1,i_1}^* - u_{{k^*},i_1}) \vspace{1mm} \\
\hspace{16mm}(\bar u_{k+1,i_2}^* - u_{{k^*},i_2}) \end{array} \right] \leq 0
\end{array}
\end{equation}

$$
\begin{array}{l}
\mathop {\max} \limits_{\bar k = 0,...,k}\left[ \begin{array}{l}
\displaystyle \underline \phi_{p} ({\bf u}_{\bar k},\tau_{\bar k})  +\eta_{v,\phi}^{\bar k} \frac{\partial \underline \phi_{p}}{\partial \tau} \Big |_{({\bf u}_{\bar k},\tau_{\bar k})} ( \tau_{k+1} - \tau_{{\bar k}} ) \vspace{1mm} \\
\displaystyle + (1-\eta_{v,\phi}^{\bar k}) \underline \kappa_{\phi,\tau}^{\bar k} \left( \tau_{k+1} - \tau_{\bar k} \right) \vspace{1mm}\\
\displaystyle   + \sum_{i \in I_{v,\phi}^{\bar k}} \mathop {\min} \left[ \begin{array}{l} \displaystyle \frac{\partial \underline \phi_{p}}{\partial u_i} \Big |_{({\bf u}_{\bar k},\tau_{\bar k})} ( u_{k^*,i} +\\ \hspace{2mm} K(\bar u_{k+1,i}^* - u_{k^*,i}) - u_{{\bar k},i} ), \vspace{1mm}\\ \displaystyle \frac{\partial \overline \phi_{p}}{\partial u_i} \Big |_{({\bf u}_{\bar k},\tau_{\bar k})} ( u_{k^*,i} +\\ \hspace{2mm} K(\bar u_{k+1,i}^* - u_{k^*,i}) - u_{{\bar k},i} ) \end{array} \right] \vspace{1mm} \\
+ \displaystyle \sum_{i \not \in I_{v,\phi}^{\bar k}} \mathop {\min} \left[ \begin{array}{l} \underline \kappa_{\phi,i}^{\bar k} ( u_{k^*,i} +\\ \hspace{3mm} K(\bar u_{k+1,i}^* - u_{k^*,i}) - u_{{\bar k},i} ), \vspace{1mm}\\ \overline \kappa_{\phi,i}^{\bar k} ( u_{k^*,i} +\\ \hspace{3mm} K(\bar u_{k+1,i}^* - u_{k^*,i}) - u_{{\bar k},i} ) \end{array} \right] 
\end{array} \right] \vspace{1mm}\\
\hspace{1mm}\displaystyle \leq \mathop {\min}_{\tilde k = 0,...,k} \left[ \begin{array}{l} \displaystyle \overline \phi_{p} ({\bf u}_{\tilde k},\tau_{\tilde k})  +\eta_{c,\phi}^{\tilde k} \frac{\partial \overline \phi_{p}}{\partial \tau} \Big |_{({\bf u}_{\tilde k},\tau_{\tilde k})} ( \tau_{k+1} - \tau_{{\tilde k}} ) \vspace{1mm} \\
\hspace{0mm} \displaystyle + (1-\eta_{c,\phi}^{\tilde k}) \overline \kappa_{\phi,\tau}^{\tilde k} \left( \tau_{k+1} - \tau_{\tilde k} \right) \vspace{1mm}\\
\hspace{0mm}\displaystyle   + \sum_{i \in I_{c,\phi}^{\tilde k}} \mathop {\max} \left[ \begin{array}{l} \displaystyle \frac{\partial \underline \phi_{p}}{\partial u_i} \Big |_{({\bf u}_{\tilde k},\tau_{\tilde k})} ( u_{k^*,i} - u_{{\tilde k},i} ), \vspace{1mm} \\ \displaystyle  \frac{\partial \overline \phi_{p}}{\partial u_i} \Big |_{({\bf u}_{\tilde k},\tau_{\tilde k})} ( u_{k^*,i} - u_{{\tilde k},i} )   \end{array} \right] \vspace{1mm} \\
\hspace{0mm} + \displaystyle \sum_{i \not \in I_{c,\phi}^{\tilde k}} \mathop {\max} \left[ \begin{array}{l} \underline \kappa_{\phi,i}^{\tilde k} ( u_{k^*,i} - u_{{\tilde k},i} ), \vspace{1mm}\\ \overline \kappa_{\phi,i}^{\tilde k} ( u_{k^*,i} - u_{{\tilde k},i} ) \end{array} \right] \end{array} \right].
\end{array}
$$

However, in the case that this does not yield a feasible solution due to degradation effects, one may try the relaxed alternative

\vspace{-2mm}
\begin{equation}\label{eq:linesearchf2}
\begin{array}{rl}
K_k := {\rm arg} \mathop {\rm maximize}\limits_{K \in [0,1]} & K  \\
{\rm{subject}}\;{\rm{to}} & \\
& \hspace{-30mm} \mathop {\min} \limits_{\bar k = 0,...,k} \left[ \hspace{-1mm} \begin{array}{l} \overline g_{p,j} ({\bf u}_{\bar k},\tau_{\bar k}) \displaystyle +\eta_{c,j}^{\bar k} \frac{\partial \overline g_{p,j}}{\partial \tau} \Big |_{({\bf u}_{\bar k},\tau_{\bar k})} ( \tau_{k+1} - \tau_{{\bar k}} ) \vspace{1mm} \\
 \displaystyle + (1-\eta_{c,j}^{\bar k}) \overline \kappa_{p,j\tau}^{\bar k} \left( \tau_{k+1} - \tau_{\bar k} \right) \vspace{1mm}\\
\displaystyle   + \sum_{i \in I_{c,j}^{\bar k}} \mathop {\max} \left[ \begin{array}{l} \displaystyle \frac{\partial \underline g_{p,j}}{\partial u_i} \Big |_{({\bf u}_{\bar k},\tau_{\bar k})} ( u_{k^*,i} + \\ \hspace{2mm} K (\bar u_{k+1,i}^* - u_{k^*,i} ) - u_{\bar k,i} ), \vspace{1mm}\\ \displaystyle \frac{\partial \overline g_{p,j}}{\partial u_i} \Big |_{({\bf u}_{\bar k},\tau_{\bar k})} ( u_{k^*,i} + \\ \hspace{2mm} K (\bar u_{k+1,i}^* - u_{k^*,i} ) - u_{\bar k,i} ) \end{array} \right] \vspace{1mm} \\
 + \displaystyle \sum_{i \not \in I_{c,j}^{\bar k}} \mathop {\max} \left[ \begin{array}{l} \underline \kappa_{p,ji}^{\bar k} ( u_{k^*,i} + \\ \hspace{2mm} K (\bar u_{k+1,i}^* - u_{k^*,i} ) - u_{\bar k,i} ), \vspace{1mm}\\ \overline \kappa_{p,ji}^{\bar k} ( u_{k^*,i} + \\ \hspace{2mm} K (\bar u_{k+1,i}^* - u_{k^*,i} ) - u_{\bar k,i} ) \end{array} \right] \end{array} \hspace{-1mm} \right] \vspace{1mm} \\
& \hspace{-10mm} + \delta_e \| \kappa_{p,j}^{m,K} \|_2  \leq d_{p,j}^{k+1}, \; \forall j = 1,...,n_{g_p} \\
& \hspace{-30mm} \mathop {\max} \limits_{{\bf u} \in \mathcal{B}_{K}} g_{j}({\bf u}) \leq d_j^{k+1}, \;\; \forall j = 1,...,n_g 
\end{array}
\end{equation}

$$
\begin{array}{l}
  \displaystyle  \sum_{i=1}^{n_u} \mathop {\max} \left[ \begin{array}{l}\displaystyle \frac{\partial \underline \phi_p^t}{\partial u_i} \Big |_{({\bf u}_{k^*},\tau_{k+1})} (\bar u_{k+1,i}^* - u_{k^*,i}), \vspace{1mm} \\ \displaystyle \frac{\partial \overline \phi_p^t}{\partial u_i} \Big |_{({\bf u}_{k^*},\tau_{k+1})} (\bar u_{k+1,i}^* - u_{k^*,i}) \end{array} \right] \vspace{1mm} \\
 \displaystyle + \frac{K}{2} \sum_{i_1=1}^{n_u} \sum_{i_2=1}^{n_u} \mathop {\max} \left[ \begin{array}{l} \underline M_{\phi,i_1 i_2}^{k^*} (\bar u_{k+1,i_1}^* - u_{{k^*},i_1}) \\
\hspace{16mm} (\bar u_{k+1,i_2}^* - u_{{k^*},i_2}), \\ \overline M_{\phi,i_1 i_2}^{k^*} (\bar u_{k+1,i_1}^* - u_{{k^*},i_1}) \\
\hspace{16mm}(\bar u_{k+1,i_2}^* - u_{{k^*},i_2}) \end{array} \right]   \leq 0 \vspace{1mm} \\
\mathop {\max} \limits_{\bar k = 0,...,k}\left[ \begin{array}{l}
\displaystyle \underline \phi_{p} ({\bf u}_{\bar k},\tau_{\bar k})  +\eta_{v,\phi}^{\bar k} \frac{\partial \underline \phi_{p}}{\partial \tau} \Big |_{({\bf u}_{\bar k},\tau_{\bar k})} ( \tau_{k+1} - \tau_{{\bar k}} ) \vspace{1mm} \\
\displaystyle + (1-\eta_{v,\phi}^{\bar k}) \underline \kappa_{\phi,\tau}^{\bar k} \left( \tau_{k+1} - \tau_{\bar k} \right) \vspace{1mm}\\
\displaystyle   + \sum_{i \in I_{v,\phi}^{\bar k}} \mathop {\min} \left[ \begin{array}{l} \displaystyle \frac{\partial \underline \phi_{p}}{\partial u_i} \Big |_{({\bf u}_{\bar k},\tau_{\bar k})} ( u_{k^*,i} +\\ \hspace{2mm} K(\bar u_{k+1,i}^* - u_{k^*,i}) - u_{{\bar k},i} ), \vspace{1mm}\\ \displaystyle \frac{\partial \overline \phi_{p}}{\partial u_i} \Big |_{({\bf u}_{\bar k},\tau_{\bar k})} ( u_{k^*,i} +\\ \hspace{2mm} K(\bar u_{k+1,i}^* - u_{k^*,i}) - u_{{\bar k},i} ) \end{array} \right] \vspace{1mm} \\
+ \displaystyle \sum_{i \not \in I_{v,\phi}^{\bar k}} \mathop {\min} \left[ \begin{array}{l} \underline \kappa_{\phi,i}^{\bar k} ( u_{k^*,i} +\\ \hspace{3mm} K(\bar u_{k+1,i}^* - u_{k^*,i}) - u_{{\bar k},i} ), \vspace{1mm}\\ \overline \kappa_{\phi,i}^{\bar k} ( u_{k^*,i} +\\ \hspace{3mm} K(\bar u_{k+1,i}^* - u_{k^*,i}) - u_{{\bar k},i} ) \end{array} \right] 
\end{array} \right] \vspace{1mm}\\
\hspace{1mm}\displaystyle \leq \mathop {\min}_{\tilde k = 0,...,k} \left[ \begin{array}{l} \displaystyle \overline \phi_{p} ({\bf u}_{\tilde k},\tau_{\tilde k})  +\eta_{c,\phi}^{\tilde k} \frac{\partial \overline \phi_{p}}{\partial \tau} \Big |_{({\bf u}_{\tilde k},\tau_{\tilde k})} ( \tau_{k+1} - \tau_{{\tilde k}} ) \vspace{1mm} \\
\hspace{0mm} \displaystyle + (1-\eta_{c,\phi}^{\tilde k}) \overline \kappa_{\phi,\tau}^{\tilde k} \left( \tau_{k+1} - \tau_{\tilde k} \right) \vspace{1mm}\\
\hspace{0mm}\displaystyle   + \sum_{i \in I_{c,\phi}^{\tilde k}} \mathop {\max} \left[ \begin{array}{l} \displaystyle \frac{\partial \underline \phi_{p}}{\partial u_i} \Big |_{({\bf u}_{\tilde k},\tau_{\tilde k})} ( u_{k^*,i} - u_{{\tilde k},i} ), \vspace{1mm} \\ \displaystyle  \frac{\partial \overline \phi_{p}}{\partial u_i} \Big |_{({\bf u}_{\tilde k},\tau_{\tilde k})} ( u_{k^*,i} - u_{{\tilde k},i} )   \end{array} \right] \vspace{1mm} \\
\hspace{0mm} + \displaystyle \sum_{i \not \in I_{c,\phi}^{\tilde k}} \mathop {\max} \left[ \begin{array}{l} \underline \kappa_{\phi,i}^{\tilde k} ( u_{k^*,i} - u_{{\tilde k},i} ), \vspace{1mm}\\ \overline \kappa_{\phi,i}^{\tilde k} ( u_{k^*,i} - u_{{\tilde k},i} ) \end{array} \right] \end{array} \right].
\end{array}
$$

If a feasible $K_k$ is still not found, one may simply set $K_k := 0$.

In the case that the cost is numerical in nature, the line search of (\ref{eq:linesearchf}) is replaced by

\vspace{-2mm}
\begin{equation}\label{eq:linesearchfnum}
\begin{array}{rl}
K_k := {\rm arg} \mathop {\rm minimize}\limits_{K \in [0,1]} & \phi \left( {\bf u}_{k^*} + K (\bar {\bf u}_{k+1}^* - {\bf u}_{k^*}) \right)  \\
{\rm{subject}}\;{\rm{to}} & \\
& \hspace{-30mm} \mathop {\min} \limits_{\bar k = 0,...,k} \left[ \hspace{-1mm} \begin{array}{l} \overline g_{p,j} ({\bf u}_{\bar k},\tau_{\bar k}) \displaystyle +\eta_{c,j}^{\bar k} \frac{\partial \overline g_{p,j}}{\partial \tau} \Big |_{({\bf u}_{\bar k},\tau_{\bar k})} ( \tau_{k+1} - \tau_{{\bar k}} ) \vspace{1mm} \\
 \displaystyle + (1-\eta_{c,j}^{\bar k}) \overline \kappa_{p,j\tau}^{\bar k} \left( \tau_{k+1} - \tau_{\bar k} \right) \vspace{1mm}\\
\displaystyle   + \sum_{i \in I_{c,j}^{\bar k}} \mathop {\max} \left[ \begin{array}{l} \displaystyle \frac{\partial \underline g_{p,j}}{\partial u_i} \Big |_{({\bf u}_{\bar k},\tau_{\bar k})} ( u_{k^*,i} + \\ \hspace{2mm} K (\bar u_{k+1,i}^* - u_{k^*,i} ) - u_{\bar k,i} ), \vspace{1mm}\\ \displaystyle \frac{\partial \overline g_{p,j}}{\partial u_i} \Big |_{({\bf u}_{\bar k},\tau_{\bar k})} ( u_{k^*,i} + \\ \hspace{2mm} K (\bar u_{k+1,i}^* - u_{k^*,i} ) - u_{\bar k,i} ) \end{array} \right] \vspace{1mm} \\
 + \displaystyle \sum_{i \not \in I_{c,j}^{\bar k}} \mathop {\max} \left[ \begin{array}{l} \underline \kappa_{p,ji}^{\bar k} ( u_{k^*,i} + \\ \hspace{2mm} K (\bar u_{k+1,i}^* - u_{k^*,i} ) - u_{\bar k,i} ), \vspace{1mm}\\ \overline \kappa_{p,ji}^{\bar k} ( u_{k^*,i} + \\ \hspace{2mm} K (\bar u_{k+1,i}^* - u_{k^*,i} ) - u_{\bar k,i} ) \end{array} \right] \end{array} \hspace{-1mm} \right] \vspace{1mm} \\
& \hspace{10mm}   \leq d_{p,j}^{k+1}, \; \forall j = 1,...,n_{g_p} \vspace{1mm}\\
& \hspace{-30mm} \mathop {\min} \limits_{\bar k = 0,...,k} \left[ \hspace{-1mm} \begin{array}{l} \overline g_{p,j} ({\bf u}_{\bar k},\tau_{\bar k}) \displaystyle +\eta_{c,j}^{\bar k} \frac{\partial \overline g_{p,j}}{\partial \tau} \Big |_{({\bf u}_{\bar k},\tau_{\bar k})} ( \tau_{k+1} - \tau_{{\bar k}} ) \vspace{1mm} \\
 \displaystyle + (1-\eta_{c,j}^{\bar k}) \overline \kappa_{p,j\tau}^{\bar k} \left( \tau_{k+1} - \tau_{\bar k} \right) \vspace{1mm}\\
\displaystyle   + \sum_{i \in I_{c,j}^{\bar k}} \mathop {\max} \left[ \begin{array}{l} \displaystyle \frac{\partial \underline g_{p,j}}{\partial u_i} \Big |_{({\bf u}_{\bar k},\tau_{\bar k})} ( u_{k^*,i} + \\ \hspace{2mm} K (\bar u_{k+1,i}^* - u_{k^*,i} ) - u_{\bar k,i} ), \vspace{1mm}\\ \displaystyle \frac{\partial \overline g_{p,j}}{\partial u_i} \Big |_{({\bf u}_{\bar k},\tau_{\bar k})} ( u_{k^*,i} + \\ \hspace{2mm} K (\bar u_{k+1,i}^* - u_{k^*,i} ) - u_{\bar k,i} ) \end{array} \right] \vspace{1mm} \\
 + \displaystyle \sum_{i \not \in I_{c,j}^{\bar k}} \mathop {\max} \left[ \begin{array}{l} \underline \kappa_{p,ji}^{\bar k} ( u_{k^*,i} + \\ \hspace{2mm} K (\bar u_{k+1,i}^* - u_{k^*,i} ) - u_{\bar k,i} ), \vspace{1mm}\\ \overline \kappa_{p,ji}^{\bar k} ( u_{k^*,i} + \\ \hspace{2mm} K (\bar u_{k+1,i}^* - u_{k^*,i} ) - u_{\bar k,i} ) \end{array} \right] \end{array} \hspace{-1mm} \right] \vspace{1mm} \\
& \hspace{-30mm} + \displaystyle \eta_{c,j}^{K} \frac{\partial \overline g_{p,j}}{\partial \tau} \Big |_{\left({\bf u}_{k^*} + K ( \bar {\bf u}_{k+1}^* - {\bf u}_{k^*} ),\tau_{k+1}\right)} ( \tau_{k+2} - \tau_{{k+1}} ) \vspace{1mm} \\
& \hspace{-30mm} + (1-\eta_{c,j}^{K}) \overline \kappa_{p,j\tau}^{K} \left( \tau_{k+2} - \tau_{k+1} \right) + \delta_e \| \kappa_{p,j}^{m,K} \|_2  \leq d_{p,j}^{k+1}, \\
& \hspace{20mm} \forall j = 1,...,n_{g_p} \\
& \hspace{-30mm} \mathop {\max} \limits_{{\bf u} \in \mathcal{B}_{K}} g_{j}({\bf u}) \leq d_j^{k+1}, \;\; \forall j = 1,...,n_g,
\end{array}
\end{equation}

\noindent or, if this is infeasible, by

\begin{equation}\label{eq:linesearchfnum2}
\begin{array}{rl}
K_k := {\rm arg} \mathop {\rm minimize}\limits_{K \in [0,1]} & \phi \left( {\bf u}_{k^*} + K (\bar {\bf u}_{k+1}^* - {\bf u}_{k^*}) \right)  \\
{\rm{subject}}\;{\rm{to}} & 
\end{array}
\end{equation}

$$
\begin{array}{rl}
& \hspace{0mm} \mathop {\min} \limits_{\bar k = 0,...,k} \left[ \hspace{-1mm} \begin{array}{l} \overline g_{p,j} ({\bf u}_{\bar k},\tau_{\bar k}) \displaystyle +\eta_{c,j}^{\bar k} \frac{\partial \overline g_{p,j}}{\partial \tau} \Big |_{({\bf u}_{\bar k},\tau_{\bar k})} ( \tau_{k+1} - \tau_{{\bar k}} ) \vspace{1mm} \\
 \displaystyle + (1-\eta_{c,j}^{\bar k}) \overline \kappa_{p,j\tau}^{\bar k} \left( \tau_{k+1} - \tau_{\bar k} \right) \vspace{1mm}\\
\displaystyle   + \sum_{i \in I_{c,j}^{\bar k}} \mathop {\max} \left[ \begin{array}{l} \displaystyle \frac{\partial \underline g_{p,j}}{\partial u_i} \Big |_{({\bf u}_{\bar k},\tau_{\bar k})} ( u_{k^*,i} + \\ \hspace{2mm} K (\bar u_{k+1,i}^* - u_{k^*,i} ) - u_{\bar k,i} ), \vspace{1mm}\\ \displaystyle \frac{\partial \overline g_{p,j}}{\partial u_i} \Big |_{({\bf u}_{\bar k},\tau_{\bar k})} ( u_{k^*,i} + \\ \hspace{2mm} K (\bar u_{k+1,i}^* - u_{k^*,i} ) - u_{\bar k,i} ) \end{array} \right] \vspace{1mm} \\
 + \displaystyle \sum_{i \not \in I_{c,j}^{\bar k}} \mathop {\max} \left[ \begin{array}{l} \underline \kappa_{p,ji}^{\bar k} ( u_{k^*,i} + \\ \hspace{2mm} K (\bar u_{k+1,i}^* - u_{k^*,i} ) - u_{\bar k,i} ), \vspace{1mm}\\ \overline \kappa_{p,ji}^{\bar k} ( u_{k^*,i} + \\ \hspace{2mm} K (\bar u_{k+1,i}^* - u_{k^*,i} ) - u_{\bar k,i} ) \end{array} \right] \end{array} \hspace{-1mm} \right] \vspace{1mm} \\
& \hspace{20mm}  + \delta_e \| \kappa_{p,j}^{m,K} \|_2  \leq d_{p,j}^{k+1}, \; \forall j = 1,...,n_{g_p} \\
& \hspace{0mm} \mathop {\max} \limits_{{\bf u} \in \mathcal{B}_{K}} g_{j}({\bf u}) \leq d_j^{k+1}, \;\; \forall j = 1,...,n_g,
\end{array}
$$

\noindent with $K_k := 0$ as a last resort.

Following this step, one obtains the next set of decision variables ${\bf u}_{k+1}$. Ideally, one would apply these directly in the next experiment, but it may occur that the ${\bf u}_{k+1}$ generated, even if satisfying the SCFO, is not sufficiently exciting in some sense -- e.g., the step ${\bf u}_{k+1} - {\bf u}_{k^*}$ is not sufficiently large enough to reject the data corruption due to measurement noise, or the local data set that ${\bf u}_{k+1}$ becomes a member of is not well-poised for regression and a better poised set is needed. If the user deems the generated ${\bf u}_{k+1}$ as not sufficiently exciting, they then have the freedom to override it and replace it by any other ${\bf u} \in \mathcal{B}_{e,k^*}$, as any such ${\bf u}$ is guaranteed to be feasible and may be better from the information-gathering point of view.

\section{Theoretical Properties}
\label{sec:gcproof}

We conclude this document by examining the very important theoretical questions that the development in the previous sections has given rise to -- namely, what can we prove about the implementable version of the SCFO? As should be expected given the number of additional complexities that were introduced, we require stronger assumptions than those in \cite{Bunin2013SIAM} in order to prove anything substantial.

As before, we will be able to show that feasibility -- or relaxed feasibility requirements, when soft constraints are used -- may be guaranteed throughout the optimization process. For convergence to an FJ point, the analysis is complicated significantly by the presence of degradation, the errors in both the experimental function values and the experimental function gradients, and the addition of iteration-dependent back-offs in the optimization problem.

While we could work around all of these issues by making degradation negligible, making the function value/gradient errors arbitrarily small as $k \rightarrow \infty$, and making the back-offs constant -- all of which would essentially allow us to recycle the analysis of \cite{Bunin2013SIAM} -- we are not, at the present moment, satisfied with this approach as it requires a number of strong assumptions. As such, global convergence to an optimum will not be examined in the present document, but will instead be addressed in a more adequate manner in a future version.

\subsection{Proof of Feasible-Side Iterates}

Let us first make the following assumption, which is simply the modified version of Assumption A1:

\begin{enumerate}[]
\item {\bf{A1F}}: The functions $\phi_p$ (or $\phi$, if the cost is numerical), $g_p$, and $g$ are $C^2$ on an open set containing $\mathcal{I}_\tau$.
\end{enumerate}

\noindent As before, Assumption A1F is needed to guarantee the existence of the necessary Lipschitz constants over $\mathcal{I}_\tau$.

To ensure that one can always maintain feasibility, one needs an analogue to Assumption A2 that guarantees that a certain input point, taken here to be ${\bf u}_0$, is known to be feasible regardless of the time $\tau$. In doing so, it will always be possible to choose this point as the reference ${\bf u}_{k^*}$ when a feasibility-guaranteeing reference cannot be found by standard means. However, so that sufficient excitation may be accounted for, one needs to expand this assumption to include the ball $\mathcal{B}_{e,0}$.

\begin{enumerate}[]
\item {\bf{A2F}}: The initial experimental iterate, ${\bf{u}}_0$, and the ball of radius $\delta_e$ around it, $\mathcal{B}_{e,0}$, are strictly feasible with respect to the experimental constraints for all $\tau$ ($g_{p,j} ({\bf{u}},\tau) < 0, \; \forall {\bf u} \in \mathcal{B}_{e,0}, \; \forall j = 1,...,n_{g_p}$), feasible with respect to the numerical constraints ($g_{j} ({\bf{u}}) \leq 0, \; \forall {\bf u} \in \mathcal{B}_{e,0}, \; \forall j = 1,...,n_{g}$), and lies in the compressed experimental space (${\bf u}^L + \delta_e {\bf 1} \preceq {\bf u}_0 \preceq {\bf u}^U - \delta_e {\bf 1}$).
\end{enumerate}

\noindent The practical significance of this assumption is the guarantee that setting ${\bf u}_{k^*} := {\bf u}_0$ will always result in a reference point that both is feasible and guarantees the existence of a feasible $\mathcal{B}_{e,k^*}$. In applications, it may be seen as a sort of ``safe point'' -- a set of decision variables that one knows is suboptimal but which one also knows will always meet safety specifications with some tolerance.

Additional assumptions are required on the correctness of the Lipschitz constants and the concavity relaxations, as well as on the probabilistic validity of the gradient-estimate bounds and the upper bounds on the experimental constraint function values.

\begin{enumerate}[]
\item {\bf{A3}}: The supplied Lipschitz constants $\underline \kappa_{p,ji}$, $\overline \kappa_{p,ji}$, $\underline \kappa_{p,j\tau}$, $\overline \kappa_{p,j\tau}$, $\underline \kappa_{p,ji}^{\bar k}$, $\overline \kappa_{p,ji}^{\bar k}$, $\underline \kappa_{p,j\tau}^{\bar k}$, $\overline \kappa_{p,j\tau}^{\bar k}$, $\underline \kappa_{p,j\tau}^{\bar k_1, \bar k_2}$, $\overline \kappa_{p,j\tau}^{\bar k_1, \bar k_2}$, $\underline \kappa_{p,ji}^{e,k^*}$, $\overline \kappa_{p,ji}^{e,k^*}$, $\underline \kappa_{p,j\tau}^{e,k^*}$, $\overline \kappa_{p,j\tau}^{e,k^*}$, $\underline \kappa_{p,ji}^{K}$, $\overline \kappa_{p,ji}^{K}$ are \emph{correct} in the sense of satisfying (\ref{eq:lipcondegLU}), (\ref{eq:lipdegLU}), (\ref{eq:lipcondegLUloc}), (\ref{eq:lipdegLUloc}), (\ref{eq:lipdegLUloc2}), (\ref{eq:lipcondegLUlocball}), (\ref{eq:lipdegLUloc2ball}), and (\ref{eq:lipcondegLUlocK}).
\item {\bf{A4}}: The supplied degradation concavity indicators $\eta_{c,j}$, $\eta_{c,j}^{e,\bar k}$, $\eta_{c,j}^{\bar k}$, $\eta_{c,j}^K$ and the concavity index sets $I_{c,j}$, $I_{c,j}^{\bar k}$ are \emph{correct} in the sense that the concave relationships that they denote exist on the relevant domains.
\item {\bf{A5}}: $g_{p,j} ({\bf u}_{\bar k},\tau_{\bar k}) \leq \overline g_{p,j} ({\bf u}_{\bar k},\tau_{\bar k})$ almost surely for all $j = 1,...,n_{g_p}$ and for all $\bar k \in [0,k]$.
\item {\bf{A6}}: $\nabla g_{p,j} ({\bf u}_{\bar k},\tau_{\bar k}) \in \left[  \nabla \underline g_{p,j} ({\bf u}_{\bar k},\tau_{\bar k}),  \nabla \overline g_{p,j} ({\bf u}_{\bar k},\tau_{\bar k}) \right]$ almost surely for all $j = 1,...,n_{g_p}$ and for all $\bar k \in [0,k]$.
\end{enumerate}

We may now state the main feasibility result.

\begin{theorem}[SCFO feasibility in the implementable case]
\label{thm:feas}
Let Assumptions A1F, A2F, A3, A4, A5, A6 hold and let the reference point be chosen as follows:

\begin{itemize}
\item ${\bf u}_{k^*}$ is set as the solution to (\ref{eq:kstarLUccvcostlocMNgradBOslackf}) -- or as the solution to (\ref{eq:kstarLUccvcostlocMNgradBOslacknumf}) when the cost is numerical -- if such a solution exists.
\item If (\ref{eq:kstarLUccvcostlocMNgradBOslackf}) -- or, for the case of a numerical cost, (\ref{eq:kstarLUccvcostlocMNgradBOslacknumf}) -- is infeasible, ${\bf u}_{k^*}$ is set as ${\bf u}_0$.
\end{itemize}

\noindent Let ${\bf u}_{k+1}$ then be defined by (\ref{eq:inputfilterf}), with $\bar {\bf u}_{k+1}^*$ computed by Algorithm 3F and $K_k$ chosen as follows:

\begin{itemize}
\item $K_k$ is set as the solution to (\ref{eq:linesearchf}) (resp., to (\ref{eq:linesearchfnum}), for a numerical cost) if such a solution exists.
\item If (\ref{eq:linesearchf}) (resp., (\ref{eq:linesearchfnum})) is infeasible, $K_k$ is set as the solution to (\ref{eq:linesearchf2}) (resp., to (\ref{eq:linesearchfnum2})) if such a solution exists.
\item If both (\ref{eq:linesearchf}) and (\ref{eq:linesearchf2}) (resp., (\ref{eq:linesearchfnum}) and (\ref{eq:linesearchfnum2})) are infeasible, $K_k$ is set as 0.
\end{itemize}

\noindent In the case that the resulting ${\bf u}_{k+1}$ is not sufficiently exciting with respect to user specifications, it is replaced with some other ${\bf u}_{k+1} \in \mathcal{B}_{e,k^*}$.

\noindent It follows that if the slack update law (\ref{eq:slackmanage}) is applied, with the slack-reduction constants satisfying (\ref{eq:betamax}) and (\ref{eq:betamaxnum}), the bounds (\ref{eq:violsum}) hold almost surely for all $j$, with ${\bf u}^L \preceq {\bf u}_k \preceq {\bf u}^U$ for all $k$.

\end{theorem}
\begin{proof} Considering the case where the cost function is experimental and enumerating the algorithmic possibilities, one has:

\begin{enumerate}[(i)]
\item ${\bf u}_{k^*}$ is the solution to (\ref{eq:kstarLUccvcostlocMNgradBOslackf}) and ${\bf u}_{k+1}$ is defined by (\ref{eq:inputfilterf}), with $K_k$ the solution to (\ref{eq:linesearchf}),
\item ${\bf u}_{k^*}$ is the solution to (\ref{eq:kstarLUccvcostlocMNgradBOslackf}) and ${\bf u}_{k+1}$ is defined by (\ref{eq:inputfilterf}), with $K_k$ the solution to (\ref{eq:linesearchf2}),
\item ${\bf u}_{k^*}$ is the solution to (\ref{eq:kstarLUccvcostlocMNgradBOslackf}) and ${\bf u}_{k+1}$ is defined by (\ref{eq:inputfilterf}), with $K_k := 0$,
\item ${\bf u}_{k^*}$ is the solution to (\ref{eq:kstarLUccvcostlocMNgradBOslackf}) and ${\bf u}_{k+1}$ is chosen as some sufficiently exciting point belonging to $\mathcal{B}_{e,k^*}$,
\item ${\bf u}_{k^*} := {\bf u}_0$ and ${\bf u}_{k+1}$ is defined by (\ref{eq:inputfilterf}), with $K_k$ the solution to (\ref{eq:linesearchf}),
\item ${\bf u}_{k^*} := {\bf u}_0$ and ${\bf u}_{k+1}$ is defined by (\ref{eq:inputfilterf}), with $K_k$ the solution to (\ref{eq:linesearchf2}),
\item ${\bf u}_{k^*} := {\bf u}_0$ and ${\bf u}_{k+1}$ is defined by (\ref{eq:inputfilterf}), with $K_k := 0$,
\item ${\bf u}_{k^*} := {\bf u}_0$ and ${\bf u}_{k+1}$ is chosen as some sufficiently exciting point belonging to $\mathcal{B}_{e,k^*}$.
\end{enumerate}

We will proceed to prove the desired result by showing that the conditions needed to employ Theorem \ref{thm:beta} -- i.e., those given by (\ref{eq:inviolslackiter}) -- hold for all $k$. That ${\bf u}^L \preceq {\bf u}_k \preceq {\bf u}^U$ holds for all $k$ will be verified in a much simpler manner.

That these conditions are met for $k = 0$ follows directly from Assumption A2F. For the general $k$, it is required that we consider each of the eight algorithmic scenarios listed above.

Considering Scenario (i), where we assume that $K_k$ satisfies the constraints of (\ref{eq:linesearchf}), we have that

\vspace{-2mm}
\begin{equation}\label{eq:feasproof2}
\begin{array}{l}
\mathop {\min} \limits_{\bar k = 0,...,k} \left[ \hspace{-1mm} \begin{array}{l} \overline g_{p,j} ({\bf u}_{\bar k},\tau_{\bar k}) \displaystyle +\eta_{c,j}^{\bar k} \frac{\partial \overline g_{p,j}}{\partial \tau} \Big |_{({\bf u}_{\bar k},\tau_{\bar k})} ( \tau_{k+1} - \tau_{{\bar k}} ) \vspace{1mm} \\
 \displaystyle + (1-\eta_{c,j}^{\bar k}) \overline \kappa_{p,j\tau}^{\bar k} \left( \tau_{k+1} - \tau_{\bar k} \right) \vspace{1mm}\\
\displaystyle   + \sum_{i \in I_{c,j}^{\bar k}} \mathop {\max} \left[ \begin{array}{l} \displaystyle \frac{\partial \underline g_{p,j}}{\partial u_i} \Big |_{({\bf u}_{\bar k},\tau_{\bar k})} ( u_{k^*,i} + \\ \hspace{2mm} K_k (\bar u_{k+1,i}^* - u_{k^*,i} ) - u_{\bar k,i} ), \vspace{1mm}\\ \displaystyle \frac{\partial \overline g_{p,j}}{\partial u_i} \Big |_{({\bf u}_{\bar k},\tau_{\bar k})} ( u_{k^*,i} + \\ \hspace{2mm} K_k (\bar u_{k+1,i}^* - u_{k^*,i} ) - u_{\bar k,i} ) \end{array} \right] \vspace{1mm} \\
 + \displaystyle \sum_{i \not \in I_{c,j}^{\bar k}} \mathop {\max} \left[ \begin{array}{l} \underline \kappa_{p,ji}^{\bar k} ( u_{k^*,i} + \\ \hspace{2mm} K_k (\bar u_{k+1,i}^* - u_{k^*,i} ) - u_{\bar k,i} ), \vspace{1mm}\\ \overline \kappa_{p,ji}^{\bar k} ( u_{k^*,i} + \\ \hspace{2mm} K_k (\bar u_{k+1,i}^* - u_{k^*,i} ) - u_{\bar k,i} ) \end{array} \right] \end{array} \hspace{-1mm} \right] \vspace{1mm} \\
\hspace{35mm}   \leq d_{p,j}^{k+1}, \; \forall j = 1,...,n_{g_p},
\end{array}
\end{equation}

\noindent which, with the backwards substitution of the filter law (\ref{eq:inputfilterf}), becomes

\vspace{-2mm}
\begin{equation}\label{eq:feasproof3}
\begin{array}{l}
\mathop {\min} \limits_{\bar k = 0,...,k} \left[ \hspace{-1mm} \begin{array}{l} \overline g_{p,j} ({\bf u}_{\bar k},\tau_{\bar k}) \displaystyle +\eta_{c,j}^{\bar k} \frac{\partial \overline g_{p,j}}{\partial \tau} \Big |_{({\bf u}_{\bar k},\tau_{\bar k})} ( \tau_{k+1} - \tau_{{\bar k}} ) \vspace{1mm} \\
 \displaystyle + (1-\eta_{c,j}^{\bar k}) \overline \kappa_{p,j\tau}^{\bar k} \left( \tau_{k+1} - \tau_{\bar k} \right) \vspace{1mm}\\
\displaystyle   + \sum_{i \in I_{c,j}^{\bar k}} \mathop {\max} \left[ \begin{array}{l} \displaystyle \frac{\partial \underline g_{p,j}}{\partial u_i} \Big |_{({\bf u}_{\bar k},\tau_{\bar k})} ( u_{k+1,i} - u_{\bar k,i} ), \vspace{1mm}\\ \displaystyle \frac{\partial \overline g_{p,j}}{\partial u_i} \Big |_{({\bf u}_{\bar k},\tau_{\bar k})} ( u_{k+1,i} - u_{\bar k,i} ) \end{array} \right] \vspace{1mm} \\
 + \displaystyle \sum_{i \not \in I_{c,j}^{\bar k}} \mathop {\max} \left[ \begin{array}{l} \underline \kappa_{p,ji}^{\bar k} ( u_{k+1,i} - u_{\bar k,i} ), \vspace{1mm}\\ \overline \kappa_{p,ji}^{\bar k} ( u_{k+1,i} - u_{\bar k,i} ) \end{array} \right] \end{array} \hspace{-1mm} \right] \vspace{1mm} \\
\hspace{35mm}  \leq d_{p,j}^{k+1}, \; \forall j = 1,...,n_{g_p}.
\end{array}
\end{equation}

Since

\vspace{-2mm}
\begin{equation}\label{eq:feasproof4}
\hspace{-1mm}\begin{array}{l}
g_{p,j} ({\bf u}_{k+1},\tau_{k+1}) \leq \vspace{1mm} \\
\mathop {\min} \limits_{\bar k = 0,...,k} \left[ \hspace{-1mm} \begin{array}{l} \overline g_{p,j} ({\bf u}_{\bar k},\tau_{\bar k}) \displaystyle +\eta_{c,j}^{\bar k} \frac{\partial \overline g_{p,j}}{\partial \tau} \Big |_{({\bf u}_{\bar k},\tau_{\bar k})} ( \tau_{k+1} - \tau_{{\bar k}} ) \vspace{1mm} \\
 \displaystyle + (1-\eta_{c,j}^{\bar k}) \overline \kappa_{p,j\tau}^{\bar k} \left( \tau_{k+1} - \tau_{\bar k} \right) \vspace{1mm}\\
\displaystyle   + \sum_{i \in I_{c,j}^{\bar k}} \mathop {\max} \left[ \begin{array}{l} \displaystyle \frac{\partial \underline g_{p,j}}{\partial u_i} \Big |_{({\bf u}_{\bar k},\tau_{\bar k})} ( u_{k+1,i} - u_{\bar k,i} ), \vspace{1mm}\\ \displaystyle \frac{\partial \overline g_{p,j}}{\partial u_i} \Big |_{({\bf u}_{\bar k},\tau_{\bar k})} ( u_{k+1,i} - u_{\bar k,i} ) \end{array} \right] \vspace{1mm} \\
 + \displaystyle \sum_{i \not \in I_{c,j}^{\bar k}} \mathop {\max} \left[ \begin{array}{l} \underline \kappa_{p,ji}^{\bar k} ( u_{k+1,i} - u_{\bar k,i} ), \vspace{1mm}\\ \overline \kappa_{p,ji}^{\bar k} ( u_{k+1,i} - u_{\bar k,i} ) \end{array} \right] \end{array} \hspace{-1mm} \right],
\end{array}
\end{equation}

\noindent it follows that $g_{p,j} ({\bf u}_{k+1},\tau_{k+1}) \leq d_{p,j}^{k+1}, \;\; \forall j = 1,...,n_{g_p}$, which thus meets the first requirement of (\ref{eq:inviolslackiter}) with an index shift. Also from (\ref{eq:linesearchf}), we have that

\vspace{-2mm}
\begin{equation}\label{eq:feasproof5}
\mathop {\max} \limits_{{\bf u} \in \mathcal{B}_{e,k+1}} g_{j}({\bf u}) \leq d_j^{k+1}, \;\; \forall j = 1,...,n_g
\end{equation}

\noindent is satisfied, which in turn implies $g_{j} ({\bf u}_{k+1}) \leq d_j^{k+1}, \; \forall j = 1,...,n_g$ and satisfies the second requirement of (\ref{eq:inviolslackiter}) with an index shift. To show that ${\bf u}^L \preceq {\bf u}_{k+1} \preceq {\bf u}^U$, it is sufficient to note that since both ${\bf u}^L + \delta_e {\bf 1} \preceq {\bf u}_{k^*} \preceq {\bf u}^U - \delta_e {\bf 1}$ and ${\bf u}^L + \delta_e {\bf 1} \preceq \bar {\bf u}_{k+1}^* \preceq {\bf u}^U - \delta_e {\bf 1}$, the latter following from the constraints in the projection of Algorithm 3F, it must be that ${\bf u}^L + \delta_e {\bf 1} \preceq {\bf u}_{k+1} \preceq {\bf u}^U - \delta_e {\bf 1}$ as $K_k \in [0,1]$.

Scenario (ii) follows an identical analysis, save that instead of (\ref{eq:feasproof2}) we have the more restricting condition

\vspace{-2mm}
\begin{equation}\label{eq:feasproof5a}
\begin{array}{l}
\mathop {\min} \limits_{\bar k = 0,...,k} \left[ \hspace{-1mm} \begin{array}{l} \overline g_{p,j} ({\bf u}_{\bar k},\tau_{\bar k}) \displaystyle +\eta_{c,j}^{\bar k} \frac{\partial \overline g_{p,j}}{\partial \tau} \Big |_{({\bf u}_{\bar k},\tau_{\bar k})} ( \tau_{k+1} - \tau_{{\bar k}} ) \vspace{1mm} \\
 \displaystyle + (1-\eta_{c,j}^{\bar k}) \overline \kappa_{p,j\tau}^{\bar k} \left( \tau_{k+1} - \tau_{\bar k} \right) \vspace{1mm}\\
\displaystyle   + \sum_{i \in I_{c,j}^{\bar k}} \mathop {\max} \left[ \begin{array}{l} \displaystyle \frac{\partial \underline g_{p,j}}{\partial u_i} \Big |_{({\bf u}_{\bar k},\tau_{\bar k})} ( u_{k^*,i} + \\ \hspace{2mm} K_k (\bar u_{k+1,i}^* - u_{k^*,i} ) - u_{\bar k,i} ), \vspace{1mm}\\ \displaystyle \frac{\partial \overline g_{p,j}}{\partial u_i} \Big |_{({\bf u}_{\bar k},\tau_{\bar k})} ( u_{k^*,i} + \\ \hspace{2mm} K_k (\bar u_{k+1,i}^* - u_{k^*,i} ) - u_{\bar k,i} ) \end{array} \right] \vspace{1mm} \\
 + \displaystyle \sum_{i \not \in I_{c,j}^{\bar k}} \mathop {\max} \left[ \begin{array}{l} \underline \kappa_{p,ji}^{\bar k} ( u_{k^*,i} + \\ \hspace{2mm} K_k (\bar u_{k+1,i}^* - u_{k^*,i} ) - u_{\bar k,i} ), \vspace{1mm}\\ \overline \kappa_{p,ji}^{\bar k} ( u_{k^*,i} + \\ \hspace{2mm} K_k (\bar u_{k+1,i}^* - u_{k^*,i} ) - u_{\bar k,i} ) \end{array} \right] \end{array} \hspace{-1mm} \right] \vspace{1mm} \\
\hspace{20mm} + \delta_e \| \kappa_{p,j}^{m,K} \|_2   \leq d_{p,j}^{k+1}, \; \forall j = 1,...,n_{g_p},
\end{array}
\end{equation}

\noindent which implies (\ref{eq:feasproof2}) since $\delta_e \| \kappa_{p,j}^{m,K} \|_2 \geq 0$, thus implying that $g_{p,j} ({\bf u}_{k+1},\tau_{k+1}) \leq d_{p,j}^{k+1}$ for this case as well.

Let us consider Scenario (iv) next, since Scenario (iii) is just a special case of (iv). Since ${\bf u}_{k^*}$ is assumed to solve (\ref{eq:kstarLUccvcostlocMNgradBOslackf}), we know that

\vspace{-2mm}
\begin{equation}\label{eq:feasproof6}
\begin{array}{l}
\overline g_{p,j} ({\bf u}_{k^*},\tau_{k^*}) + \displaystyle \eta_{c,j}^{e,k^*} \frac{\partial \overline g_{p,j}}{\partial \tau} \Big |_{({\bf u}_{k^*},\tau_{k^*})} ( \tau_{k+1} - \tau_{{k^*}} ) \vspace{1mm} \\
 + (1-\eta_{c,j}^{e, k^*}) \overline \kappa_{p,j\tau}^{e, k^*} \left( \tau_{k+1} - \tau_{ k^*} \right) \vspace{1mm} \\ 
+ \delta_e \| \kappa_{p,j}^{m, k^*} \|_2 \leq d_{p,j}^{k+1}, \; \forall j = 1,...,n_{g_p}.
\end{array}
\end{equation}

From Theorem \ref{thm:backoff}, we know that (\ref{eq:feasproof6}) implies that $g_{p,j} ({\bf u},\tau_{k+1}) \leq d_{p,j}^{k+1}, \; \forall {\bf u} \in \mathcal{B}_{e,k^*}$. Since it is assumed that ${\bf u}_{k+1} \in \mathcal{B}_{e,k^*}$, it then follows that $g_{p,j} ({\bf u}_{k+1},\tau_{k+1}) \leq d_{p,j}^{k+1}, \; \forall j = 1,...,n_{g_p}$, which satisfies the first requirement of (\ref{eq:inviolslackiter}) with an index shift. Since 

\vspace{-2mm}
\begin{equation}\label{eq:feasproof7}
\mathop {\max} \limits_{{\bf u} \in \mathcal{B}_{e,k^*}} g_j ({\bf u}) \leq d_j^{k+1}, \;\; \forall j = 1,...,n_{g}
\end{equation}

\noindent must also hold, it follows that $g_j ({\bf u}_{k+1}) \leq d_j^{k+1}, \; \forall j = 1,...,n_{g}$ by the same logic, which meets the second requirement of (\ref{eq:inviolslackiter}) with an index shift. Finally, since ${\bf u}^L + \delta_e {\bf 1} \preceq {\bf u}_{k^*} \preceq {\bf u}^U - \delta_e {\bf 1}$, it is clear that ${\bf u}^L \preceq {\bf u}_{k+1} \preceq {\bf u}^U$ by Corollary \ref{cor:backoff3}.

For Scenario (iii), note that $K_k := 0$ results in ${\bf u}_{k+1} := {\bf u}_{k^*} \in \mathcal{B}_{e,k^*}$, and so the results derived just above for Scenario (iv) may be applied.

The analyses for Scenarios (v) and (vi) are identical to those of Scenarios (i) and (ii), as all are independent of how ${\bf u}_{k^*}$ is chosen since both assume the feasibility of (\ref{eq:linesearchf}) (or (\ref{eq:linesearchf2})).

The proof for Scenarios (vii) and (viii) follows directly from Assumption A2F. 

Having proven that (\ref{eq:inviolslackiter}) must hold for all $k$ regardless of the implementation scenario, the desired result follows from Theorem \ref{thm:beta}. We can only say that this result holds \emph{almost surely} since (\ref{eq:feasproof4}) is a probabilistic bound that is assumed to hold almost surely here. For the case where the cost is numerical, an identical analysis may be repeated to yield the same result. \qed 

\end{proof}

\bibliographystyle{spmpsci}      
\bibliography{implement}
%
%



\end{document}